\documentclass[10pt]{amsart}

\usepackage[latin2]{inputenc}
\usepackage{amsmath}
\usepackage{graphicx}
\usepackage{amssymb}
\usepackage{esint}
\usepackage{color}
\usepackage{amsthm}
\usepackage{epsfig}
\usepackage{enumitem}
\usepackage{mathtools}
\usepackage{esvect}
\usepackage{dsfont}

\usepackage{longtable}
\usepackage{tabu}

\usepackage[dvipsnames]{xcolor}
\usepackage{xxcolor}
\usepackage[linktocpage=true,colorlinks=true,linkcolor=Blue,citecolor=Green]{hyperref}

\usepackage{tikz,pgf}

\usetikzlibrary{calc,intersections,through,backgrounds,patterns,arrows.meta, math,through}
\usepackage[labelformat = brace, position=top]{subcaption}
\usepackage[english]{babel}

\newtheorem{theorem}{Theorem}
\newtheorem{proposition}[theorem]{Proposition}
\newtheorem{lemma}[theorem]{Lemma}

\newtheorem{definition}[theorem]{Definition}
\newtheorem{construction}[theorem]{Construction}
\newtheorem{remark}[theorem]{Remark}
\newtheorem*{theorem*}{Theorem}

\theoremstyle{definition}

\def\Xint#1{\mathchoice
{\XXint\displaystyle\textstyle{#1}}%
{\XXint\textstyle\scriptstyle{#1}}%
{\XXint\scriptstyle\scriptscriptstyle{#1}}%
{\XXint\scriptscriptstyle\scriptscriptstyle{#1}}%
\!\int}
\def\XXint#1#2#3{{\setbox0=\hbox{$#1{#2#3}{\int}$ }
\vcenter{\hbox{$#2#3$ }}\kern-.6\wd0}}

\def\dashint{\Xint-}

\definecolor{Yellow}{rgb}{0.95,0.9,0.0} 
\definecolor{Red}{rgb}{0.8,0.1,0.1}
\definecolor{Green}{rgb}{0.1,0.65,0.2}
\definecolor{Blue}{rgb}{0.1,0.1,0.8}
\definecolor{Purple}{rgb}{0.7,0.1,0.7}
\definecolor{Grey}{rgb}{0.6,0.6,0.6}

\definecolor{YELLOW}{rgb}{0.95,0.9,0.0} 
\definecolor{RED}{rgb}{0.8,0.1,0.1}
\definecolor{GREEN}{rgb}{0.25,0.65,0.1}
\definecolor{BLUE}{rgb}{0.1,0.1,0.8}
\definecolor{PURPLE}{rgb}{0.7,0.1,0.7}

\newcommand{\BV}{\operatorname{BV}} 
\newcommand{\supp}{\operatorname{supp}}

\newcommand{\dist}{\operatorname{dist}} 
 
\DeclareMathOperator*{\esssup}{ess\,sup}

\newcommand{\Id}{\operatorname{Id}}

\newcommand{\interface}{}
\DeclareMathOperator*{\argmin}{arg\,min}

\newcommand{\Rd}[1][d]{{\mathbb{R}^{#1}}}
\allowdisplaybreaks[3]

\newcommand{\dH}{\,\mathrm{d}\mathcal{H}^{d-1}}

\newcommand{\dx}{\,\mathrm{d}x}
\newcommand{\dy}{\,\mathrm{d}y}
\newcommand{\dt}{\,\mathrm{d}t}
\newcommand{\dtildet}{\,\mathrm{d}\tilde t}
\newcommand{\ddt}{\frac{\mathrm{d}}{\mathrm{d}t}}
\newcommand{\dnablachii}{\,\mathrm{d}|\nabla \chi_i|}

\newcommand{\tj}{p}

\newcommand{\xiTwoPh}{\xi}
\newcommand{\xiTrJ}{\xi}
\newcommand{\BTwoPh}{B}
\newcommand{\BTrJ}{B}

\newcommand{\TBV}{T_{\BV}}
\newcommand{\Tstrong}{T_{\mathrm{strong}}}

\renewcommand{\vec}[1]{{\operatorname{#1}}}
\newcommand{\cupdot}{\mathbin{\mathaccent\cdot\cup}}

\begin{document}

\title[Weak-strong uniqueness for multiphase mean curvature flow]{The local structure of the energy landscape in multiphase mean curvature flow: Weak-strong uniqueness and stability of evolutions}

\author{Julian Fischer}
\address{Institute of Science and Technology Austria (IST Austria), Am~Campus~1, 
3400 Klosterneuburg, Austria}
\email{julian.fischer@ist.ac.at}
\author{Sebastian Hensel}
\address{Institute of Science and Technology Austria (IST Austria), Am~Campus~1, 
3400 Klosterneuburg, Austria}
\email{sebastian.hensel@ist.ac.at}
\curraddr{Hausdorff Center for Mathematics, Universit{\"a}t Bonn, Endenicher Allee 62, 53115 Bonn, Germany
(\texttt{sebastian.hensel@hcm.uni-bonn.de})}
\author{Tim Laux}
\address{Hausdorff Center for Mathematics, Universit{\"a}t Bonn, Endenicher Alllee 62, 53115 Bonn, Germany}
\email{tim.laux@hcm.uni-bonn.de}
\author{Theresa M.\ Simon}
\address{Institut f{\"u}r angewandte Mathematik, Universit{\"a}t Bonn, Endenicher Allee 60, 53115 Bonn, Germany}
\email{simon@iam.uni-bonn.de}

\thanks{
Parts of the paper were written during the visit of the authors to the Hausdorff Research Institute for Mathematics (HIM), University of Bonn, in the framework of the trimester program ``Evolution of Interfaces''. The support and the hospitality of HIM are gratefully acknowledged.
This project has received funding from the European Union's Horizon 2020 research and 
innovation programme under the Marie Sk\l{}odowska-Curie Grant Agreement No.\ 665385 
\begin{tabular}{@{}c@{}}\includegraphics[width=3ex]{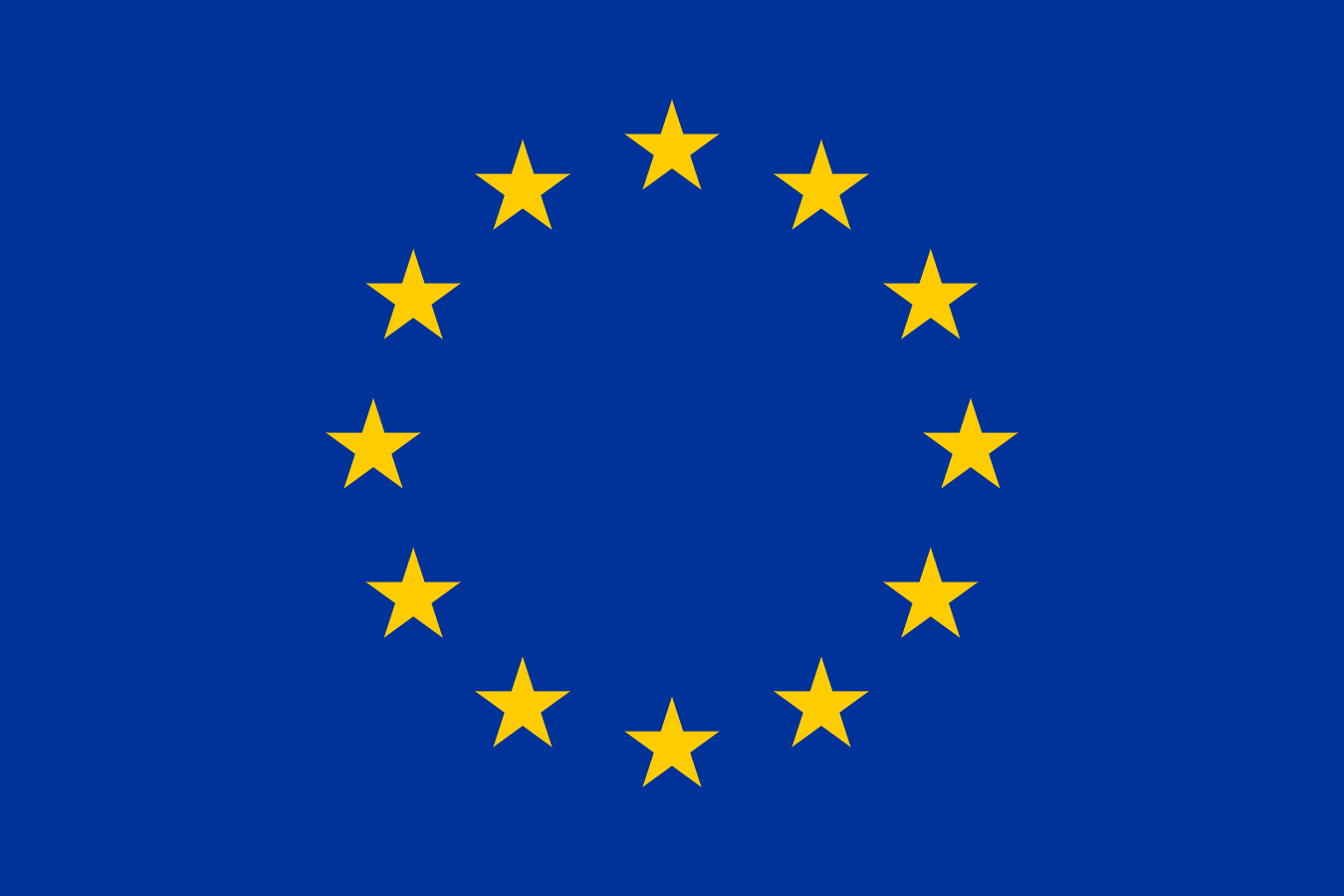}\end{tabular}, from the
European Research Council (ERC) under the European Union's Horizon 2020
research and innovation programme (grant agreement No 948819)
\smash{
\begin{tabular}{@{}c@{}}\includegraphics[width=6ex]{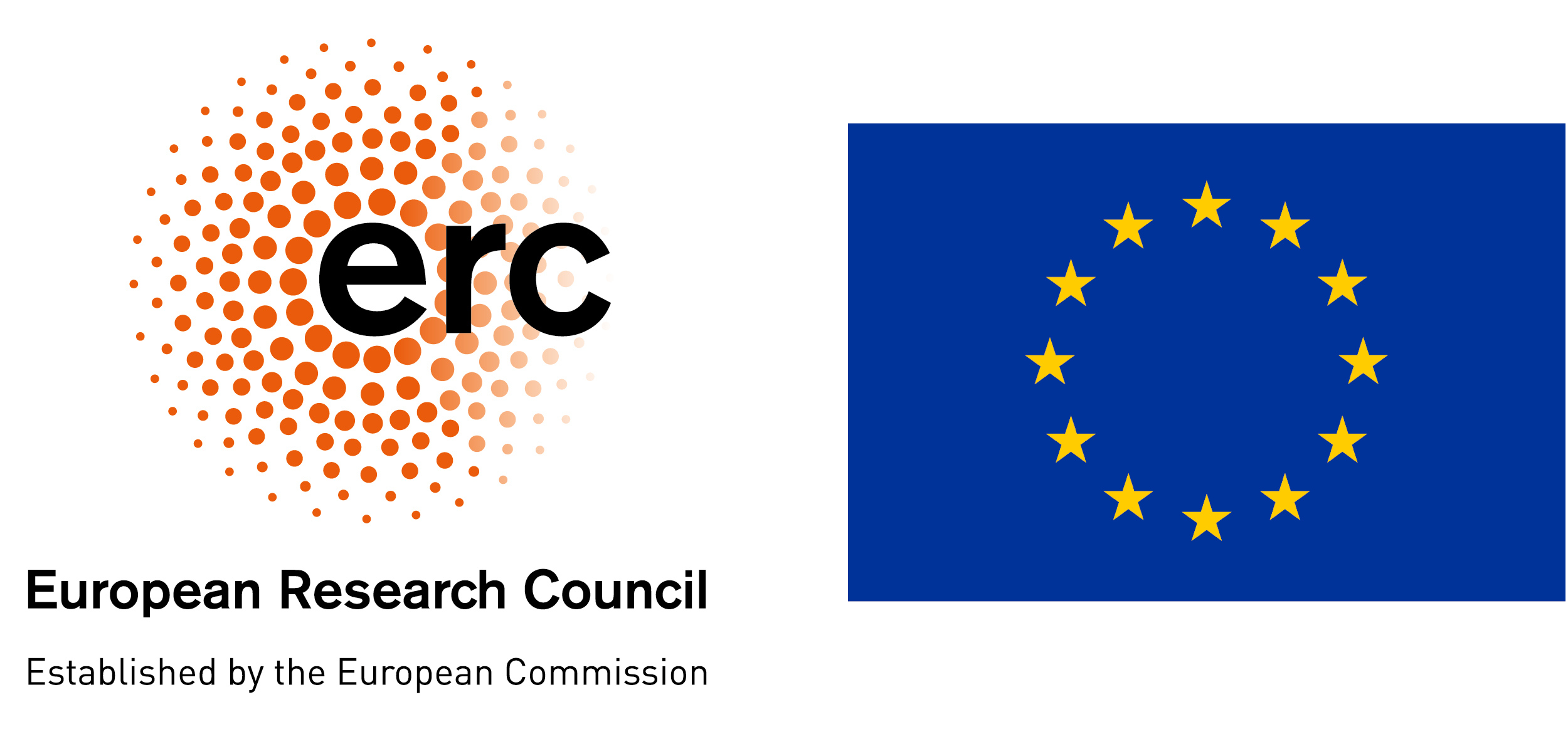}\end{tabular}
},
and from the Deutsche Forschungsgemeinschaft (DFG, German Research Foundation) 
under Germany's Excellence Strategy -- EXC-2047/1 -- 390685813.
}%

\begin{abstract}
We prove that in the absence of topological changes, the notion of $\BV$ solutions to planar multiphase mean curvature flow does not allow for a mechanism for (unphysical) non-uniqueness.
Our approach is based on the local structure of the energy landscape near a classical evolution by mean curvature. Mean curvature flow being the gradient flow of the surface energy functional, we develop a gradient-flow analogue of the notion of calibrations. Just like the existence of a calibration guarantees that one has reached a global minimum in the energy landscape, the existence of a ``gradient flow calibration'' ensures that the route of steepest descent in the energy landscape is unique and stable.
\end{abstract}


\maketitle

\setcounter{tocdepth}{2}
\tableofcontents

\section{Introduction}
In evolution problems for interfaces, the occurrence of topology changes and the 
associated geometric singularities generally limits the applicability of classical 
solution concepts to a finite time horizon, depending on the initial data.
The evolution beyond topology changes can only be described in the framework of 
suitably weakened solution concepts. However, weak concepts may in general suffer 
from an (unphysical) loss of uniqueness of solutions: For example, in the framework 
of Brakke solutions \cite{Brakke} to mean curvature flow (MCF), the interface may 
suddenly disappear at any time (see Figure~\ref{FigureBrakke} for an illustration). 
In particular, Brakke solutions fail to be unique, even prior to the onset of geometric 
singularities in the classical solution.
With the exception of evolution equations subject to a comparison principle such 
as two-phase mean curvature flow \cite{ChenGigaGoto,EvansSpruck}, only few positive 
results on uniqueness of weak solutions for interface evolution problems are known.

In the present work, we establish a weak-strong uniqueness principle for 
distributional solutions (in the framework of finite perimeter sets, a 
solution concept also known as ``$\BV$ solutions'') to planar multiphase 
mean curvature flow: As long as a strong solution to planar multiphase mean 
curvature flow -- in the sense of an evolution of smooth curves meeting at 
triple junctions at an angle of $120^\circ$ -- exists, any
distributional
solution starting from the same initial conditions must coincide with it. 
Note that for regular initial data, strong solutions are known to exist until 
a topology change in the network of evolving curves occurs, see for instance 
\cite{MantegazzaNovagaPludaSchulze}. In particular, our result establishes uniqueness of
distributional solutions to planar multiphase mean curvature flow in the absence of topology changes.

Our weak-strong uniqueness principles also apply to a notion of varifold solutions introduced by 
Kim, Stuvard, and Tonegawa \cite{KimTonegawa,StuvardTonegawaMCF}.
We emphasize that beyond certain topology changes even a mathematically ideal solution concept should not be expected to prevent failure of uniqueness, as inherent instabilities may lead to different evolutions of the system (see Figure~\ref{FigureNonuniqueness}).
Thus, together with the works by Tonegawa et al.\ \cite{KimTonegawa,StuvardTonegawaMCF} on global existence of solutions, our present work shows that this concept of varifold solutions gives rise to a mathematically sound theory of solutions for multiphase mean curvature flow.

The key insight in our present work is the observation that in analogy to 
the notion of calibrations for minimizers of the surface energy functional, 
one may develop a notion of calibrations for its gradient flow.
Just like classical calibrations carry information on the global structure 
of the energy landscape -- namely, a global lower bound for the energy -- , ``gradient 
flow calibrations'' contain information on the local structure of the energy landscape 
near a partition evolving by mean curvature: The existence of a gradient flow calibration 
implies that the path of steepest descent in the energy landscape of the surface energy 
functional is unique and stable with respect to perturbations of the initial condition.\footnote{While 
in the present work ``path of steepest descent'' is to be understood as ``$\BV$ solution to multiphase mean 
curvature flow'', we will give a rigorous statement of this notion at the level of the energy functional in a future work.}

We implement this strategy in general ambient dimension $d\geq 2$ by proving  that the existence of a gradient flow calibration implies an inclusion principle for BV solutions to multiphase mean curvature flow: The existence of a calibration for an evolving partition ensures that the interface of any BV solution must be contained in the corresponding interface of the calibrated partition. This reduces the proof of the desired weak-strong uniqueness principle to the construction of a gradient flow calibration, given a strong solution to multiphase mean curvature flow. We provide this explicit construction in the planar case $d=2$.
However, we would like to emphasize that conceptually the approach carries over to multiple dimensions. In particular, with the techniques used in the present paper it is for example possible to calibrate the smooth evolution of a double bubble; the adaptation of our arguments is elaborated on in the follow-up work \cite{HenselLauxMultiD}. However, as soon as quadruple junctions (typically occurring in three spatial dimensions) are present in the initial data, an additional construction is needed; nevertheless, we expect the principles of our present construction to guide the construction also in this situation.

\begin{figure}
\begin{tikzpicture}[scale=2.0]
\draw[fill=GREEN!70!white,draw=black] (1,1) .. controls (1+0.3,1-0.2) and (2-0.3,0.7+0.3*0.05) .. (2.0,0.7) .. controls (2.0+0.07,0.7+0.1) and (2.2-0.2*0.05,1.1-0.2) .. (2.2,1.1) .. controls (2.2-0.3*0.7,1.1+0.3*0.5) and (1.5+0.3*0.5,1.8-0.3*0.87) .. (1.5,1.8) .. controls (1.5-0.2*1,1.8) and (0.9+0.3*0.81,1.7+0.3*0.57) .. (0.9,1.7) .. controls (0.9+0.35*0.06,1.7-0.35) and (1-0.35*0.06,1+0.35) .. cycle;
\draw[fill=BLUE!70!white,draw=black] (0.2,2.3)  .. controls (0.2+0.3,2.3-0.3) and (0.9-0.3*0.91,1.7+0.3*0.42) .. (0.9,1.7) .. controls (0.9+0.3*0.81,1.7+0.3*0.57) and (1.5-0.2*1,1.8) .. (1.5,1.8) .. controls (1.5+0.3*0.5,1.8+0.3*0.87) and (1.8-0.2,2.3-0.3) .. (1.8,2.3) -- cycle;
\draw[fill=PURPLE!70!white,draw=black] (2.7,1.3) .. controls (2.7-0.3,1.3-0.05) and (2.2+0.2*0.83,1.1+0.2*0.54) .. (2.2,1.1) .. controls (2.2-0.3*0.7,1.1+0.3*0.5) and (1.5+0.3*0.5,1.8-0.3*0.87) .. (1.5,1.8) .. controls (1.5+0.3*0.5,1.8+0.3*0.87) and (1.8-0.2,2.3-0.3) .. (1.8,2.3) -- (2.7,2.3) -- cycle;
\draw[fill=RED!70!white,draw=black] (2.0,0.7) .. controls (2.0+0.07,0.7+0.1) and (2.2-0.2*0.05,1.1-0.2) .. (2.2,1.1) .. controls (2.2+0.2*0.83,1.1+0.2*0.54) and (2.7-0.3,1.3-0.05) .. (2.7,1.3) -- (2.7,0.4) -- (2.1,0.4) .. controls (2.1-0.05,0.4+0.1)  and (2.0+0.15*0.17,0.7-0.15*1.0) .. cycle;
\draw[fill=PURPLE!70!white,draw=black] (2.1,0.4) .. controls (2.1-0.05,0.4+0.1)  and (2.0+0.15*0.17,0.7-0.15*1.0) .. (2.0,0.7) .. controls (2-0.3,0.7+0.3*0.05) and (1+0.3,1-0.2) .. (1,1) .. controls (1-0.25,1-0.25*0.47) and (0.5+0.15,0.4+0.2) .. (0.5,0.4) -- cycle;
\draw[fill=RED!70!white,draw=black] (0.2,2.3) .. controls (0.2+0.3,2.3-0.3) and (0.9-0.3*0.91,1.7+0.3*0.42) .. (0.9,1.7) .. controls (0.9+0.35*0.06,1.7-0.35) and (1-0.35*0.06,1+0.35) .. (1,1) .. controls (1-0.25,1-0.25*0.47) and (0.5+0.15,0.4+0.2) .. (0.5,0.4) -- (0.2,0.4)  -- cycle;
\end{tikzpicture}
\caption{A partition of a planar domain by a network of smooth curves meeting at triple junctions at angles of 120$^\circ$, corresponding to the typical situation in multiphase mean curvature flow with equal surface energies.\label{FigureGrains}}
\end{figure}
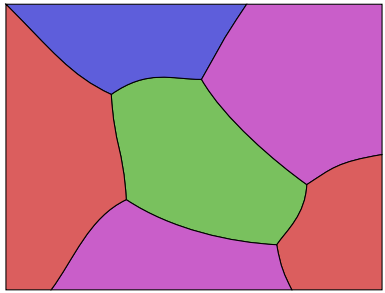

\subsection{Multiphase mean curvature flow}

Mathematically, mean curvature flow is one of the most studied geometric evolution equations. 
Being the gradient flow of the area functional with respect to the $L^2(S_t)$ distance, it 
constitutes the perhaps most natural area-reducing flow for submanifolds.
Its multiphase variant may be seen as the simplest case of mean curvature flow for a 
non-smooth surface, allowing for ``branching'' of the surface (see e.\,g.\ Figure~\ref{FigureGrains}).

Multiphase mean curvature flow also is an important phenomenological model for the motion 
of grain boundaries in polycrystals (``grains'' being the domains in a polycrystal with a 
single crystallographic orientation): Their evolution may be approximated as the gradient 
flow of the surface energy between the different grains, see for instance the seminal work of Mullins \cite{Mullins}.
While in principle the motion of grain boundaries is governed by anisotropic mean curvature flow 
or even more complex evolution equations \cite{SrolovitzUnified,SrolovitzEtAl}, isotropic 
multiphase mean curvature flow may be viewed as an important model case for these equations.
For recent developments in anisotropic and crystalline curvature flows, we refer to 
Caselles and Chambolle \cite{CasellesChambolle} and Chambolle, Morini, and Ponsiglione \cite{ChambolleMoriniPonsiglione}.

The existence theory for solutions to multiphase mean curvature flow is
quite well-developed: Classical solutions to planar multiphase mean curvature flow 
are known to exist (and to be unique) for short times, see Bronsard and Reitich \cite{BronsardReitich}.
For initial configurations close to an equilibrium state, classical solutions exist even 
globally in time, see Kinderlehrer and Liu \cite{KinderlehrerLiu}. In the higher-dimensional 
case, Depner, Garcke, and Kohsaka \cite{DepnerGarckeKohsaka} have shown the local-in-time existence 
of classical solutions for the evolution of a double bubble.
In principle, Brakke's concept of varifold solutions \cite{Brakke} is applicable to 
multiphase mean curvature flow. However, it suffers from the well-known shortcoming of 
exhibiting a drastic and unphysical failure of uniqueness of solutions \cite{Brakke} 
as mentioned above; see Figure~\ref{FigureBrakke} for an illustration.

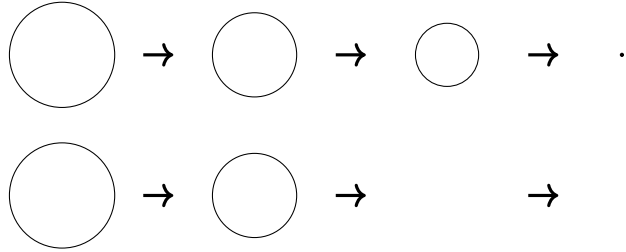
\begin{figure}
\begin{tikzpicture}[xscale=0.7,yscale=0.7]
\draw (-2,0) ellipse (1.0 and 1.0);
\end{tikzpicture}
~
\begin{tikzpicture}[xscale=0.5,yscale=0.7]
\draw[white] (0,-1) -- (0,1);
\draw[very thick,->] (-0.4,0) -- (0.4,0);
\end{tikzpicture}
~
\begin{tikzpicture}[xscale=0.7,yscale=0.7]
\draw[white] (-3,-1) -- (-1,1);
\draw (-2,0) ellipse (0.8 and 0.8);
\end{tikzpicture}
~
\begin{tikzpicture}[xscale=0.5,yscale=0.7]
\draw[white] (0,-1) -- (0,1);
\draw[very thick,->] (-0.4,0) -- (0.4,0);
\end{tikzpicture}
~
\begin{tikzpicture}[xscale=0.7,yscale=0.7]
\draw[white] (-3,-1) -- (-1,1);
\draw (-2,0) ellipse (0.6 and 0.6);
\end{tikzpicture}
~
\begin{tikzpicture}[xscale=0.5,yscale=0.7]
\draw[white] (0,-1) -- (0,1);
\draw[very thick,->] (-0.4,0) -- (0.4,0);
\end{tikzpicture}
\begin{tikzpicture}[xscale=0.7,yscale=0.7]
\draw[white] (-3,-1) -- (-1,1);
\draw[fill=black] (-2,0) circle (0.03);
\end{tikzpicture}
\\~\\
\begin{tikzpicture}[xscale=0.7,yscale=0.7]
\draw (-2,0) ellipse (1.0 and 1.0);
\end{tikzpicture}
~
\begin{tikzpicture}[xscale=0.5,yscale=0.7]
\draw[white] (0,-1) -- (0,1);
\draw[very thick,->] (-0.4,0) -- (0.4,0);
\end{tikzpicture}
~
\begin{tikzpicture}[xscale=0.7,yscale=0.7]
\draw[white] (-3,-1) -- (-1,1);
\draw (-2,0) ellipse (0.8 and 0.8);
\end{tikzpicture}
~
\begin{tikzpicture}[xscale=0.5,yscale=0.7]
\draw[white] (0,-1) -- (0,1);
\draw[very thick,->] (-0.4,0) -- (0.4,0);
\end{tikzpicture}
~
\begin{tikzpicture}[xscale=0.7,yscale=0.7]
\draw[white] (-3,-1) -- (-1,1);
\end{tikzpicture}
~
\begin{tikzpicture}[xscale=0.5,yscale=0.7]
\draw[white] (0,-1) -- (0,1);
\draw[very thick,->] (-0.4,0) -- (0.4,0);
\end{tikzpicture}
\begin{tikzpicture}[xscale=0.7,yscale=0.7]
\draw[white] (-3,-1) -- (-1,1);
\end{tikzpicture}
\caption{Top: An initially circular interface evolving by mean curvature flow. In finite time the interface shrinks to a point and disappears, giving rise to a geometric singularity and a topology change. Bottom: In Brakke solutions to mean curvature flow, the interface may suddenly disappear at any time, leading to a drastic failure of uniqueness of solutions.\label{FigureBrakke}}
\end{figure}

The existence of classical solutions to planar multiphase mean curvature flow 
up to finitely many singular times -- a solution concept that we will refer to 
as ``classical solutions with restarting'' -- has been established by 
Manteganzza, Novaga, Pluda, and Schulze \cite{MantegazzaNovagaPludaSchulze} 
under the assumption that certain types of singularities do not accumulate, 
extending earlier results by Ilmanen, Neves, and Schulze \cite{IlmanenNevesSchulze} 
and Mantegazza, Novaga, and Tortorelli \cite{MantegazzaNovagaTortorelli}. However, 
it is not evident how to generalize this notion of solutions to the higher-dimensional 
case, as it relies on the classification of potential singularities.

In \cite{LauxOtto,LauxOttoBrakke}, a conditional convergence result for an efficient 
numerical scheme for multiphase mean curvature flow -- the thresholding scheme 
of Merriman, Bence, and Osher \cite{MerrimanBenceOsher} -- towards $\BV$ solutions of 
multiphase mean curvature flow has been shown by Otto and the third author, thereby 
also establishing a conditional existence result for $\BV$~solutions. In~\cite{LauxSimon}, a 
conditional convergence result for the Allen-Cahn approximation for multiphase mean curvature 
flow towards $\BV$~solutions has been derived by the third and the fourth author. Both results 
employ an assumption of convergence of the interface area, analogous to the one 
in Luckhaus-Sturzenhecker~\cite{LuckhausSturzenhecker} for the implicit time discretization 
developed by Luckhaus-Sturzenhecker and Almgren-Taylor-Wang~\cite{AlmgrenTaylorWang}.

Kim and Tonegawa~\cite{KimTonegawa}, and Stuvard and Tonegawa~\cite{StuvardTonegawaMCF} have recently introduced a notion of varifold solutions that combines the concept of Brakke solutions with an evolution equation for the different phases. They prove global existence of solutions in arbitrary ambient dimension, only requiring the initial partition to have finite perimeter. By imposing an evolution equation for the phases, their solution concept prevents the sudden unphysical vanishing of the interface that is possible in the framework of Brakke solutions. As their notion of solutions may be viewed as the natural generalization of the concept of $\BV$ solutions to varifolds, one might expect similar uniqueness properties as in the case of $\BV$ solutions. Indeed, we shall also establish a weak-strong uniqueness principle for these varifold-$\BV$ solutions. Finally, we mention that our weak-strong uniqueness principle can be extended to another notion of varifold solutions satisfying a global energy-dissipation inequality in the sense of De Giorgi, see~\cite{hensellaux}.

\begin{table}
\begin{tabular}{llll}
\emph{Solution concept}
&\emph{Topology changes}
&\emph{Uniqueness prior to}
&\emph{Existence}
\\&&\emph{topology changes}&\emph{theory}
\vspace{2.5mm}
\\
classical solutions
&not possible
&yes \cite{BronsardReitich}
&yes (local) \cite{BronsardReitich}
\vspace{1.5mm}
\\
Brakke solutions
&possible
&fails \cite{Brakke}
&yes \cite{Brakke}
\vspace{1.5mm}
\\
classical solutions with
&possible
&yes \cite{MantegazzaNovagaPludaSchulze}
&cond.\footnotemark \cite{MantegazzaNovagaPludaSchulze}
\\
restarting (2D only)
\vspace{1.5mm}
\\
Kim-Stuvard-
&possible
&{\bf yes (Theorem\,\ref{TheoremStabilityVarifoldSolution})}\footnotemark
&yes \cite{KimTonegawa,StuvardTonegawaMCF}
\\
Tonegawa solutions & & &
\vspace{1.5mm}
\\
$\BV$ solutions
&possible
&{\bf yes (Theorem\,\ref{MainResult})}
&cond.\footnotemark \cite{LauxOtto,LauxSimon}
\vspace{2.5mm}
\end{tabular}
\caption{An overview of solutions concepts for multiphase mean curvature flow.
\label{TableSolutionConcepts}}
\end{table}
\footnotetext[2]{Global existence under the assumption that a certain type of singularities does not accumulate.}
\footnotetext[3]{Provided that one starts with a multiplicity one interface.}
\footnotetext[4]{Global existence under an assumption as in Almgren-Taylor-Wang / Luckhaus-Sturzenhecker.}

\subsection{The uniqueness properties of multiphase mean curvature flow}

The uniqueness properties of weak solution concepts for multiphase mean 
curvature flow have remained essentially unexplored. For two-phase mean curvature 
flow, a combination of the level-set formulation by Osher and Sethian \cite{OsherSethian} 
and Ohta, Jasnow, and Kawasaki \cite{OhtaJasnowKawasaki}, and the concept of viscosity 
solutions by Crandall and Lions \cite{CrandallLions} facilitates an existence and 
uniqueness theory for a weak notion of solutions, as 
shown by Chen, Giga, and Goto \cite{ChenGigaGoto} and Evans and Spruck \cite{EvansSpruck}. While 
these viscosity solutions to two-phase mean curvature flow are unique, a given level set 
may ``fatten'' \cite{BarlesSonerSouganidis}, thereby failing to describe an interface and 
indicating the emergence of a non-unique evolution of the surface. Nevertheless, fattening 
is known to not occur prior to the first topology change, provided that one starts with a 
smooth initial surface. Unfortunately, the absence of a comparison principle for multiphase 
mean curvature flow a~priori prevents the applicability of these techniques in the multiphase case.

The example in Figure~\ref{FigureNonuniqueness} shows that after topology changes, the uniqueness of BV solutions to planar multiphase mean curvature flow may fail.
Note that in contrast to the sudden vanishing of the interface in Brakke solutions, this is a case of \emph{physical} non-uniqueness: The failure of uniqueness is caused by a physically unstable situation -- the symmetric configuration of four perfect squares -- , starting from which infinitesimal perturbations may select either of the two evolutions.
This example also shows that a principle of maximal dissipation of energy may fail to single out a unique evolution.

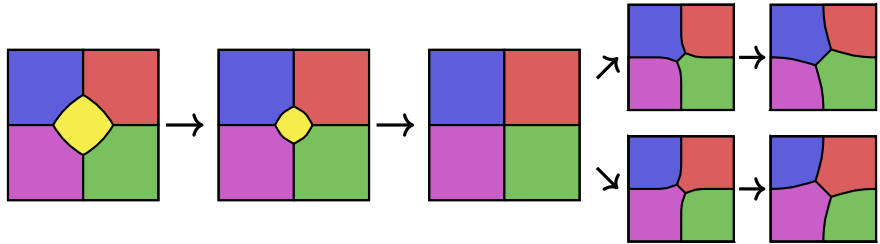
\begin{figure}
\begin{tikzpicture}[scale=1.0]
\begin{scope}[xshift=-4.4cm,yshift=-2.05cm]
\draw[fill=PURPLE!70!white,draw=black,thick] (0,0+0.3+2.3) -- (1,0+0.3+2.3) -- (1,1+0.3+2.3) -- (0,1+0.3+2.3) -- cycle;
\draw[fill=BLUE!70!white,draw=black,thick] (0,2+0.3+2.3) -- (1,2+0.3+2.3) -- (1,1+0.3+2.3) -- (0,1+0.3+2.3) -- cycle;
\draw[fill=GREEN!70!white,draw=black,thick] (2,0+0.3+2.3) -- (1,0+0.3+2.3) -- (1,1+0.3+2.3) -- (2,1+0.3+2.3) -- cycle;
\draw[fill=RED!70!white,draw=black,thick] (2,2+0.3+2.3) -- (2,1+0.3+2.3) -- (1,1+0.3+2.3) -- (1,2+0.3+2.3) -- cycle;
\draw[fill=YELLOW!70!white,draw=black,thick] (1.4,1+0.3+2.3) .. controls (1.4-0.2*0.5,1+0.2*0.86+0.3+2.3) and (1+0.2*0.86,1.4-0.2*0.5+0.3+2.3) .. (1,1.4+0.3+2.3) .. controls (1-0.2*0.86,1.4-0.2*0.5+0.3+2.3) and (0.6+0.2*0.5,1+0.2*0.86+0.3+2.3) .. (0.6,1+0.3+2.3) .. controls (0.6+0.2*0.5,1-0.2*0.86+0.3+2.3) and (1-0.2*0.86,0.6+0.2*0.5+0.3+2.3) .. (1,0.6+0.3+2.3) .. controls (1+0.2*0.86,0.6+0.2*0.5+0.3+2.3) and (1.4-0.2*0.5,1-0.2*0.86+0.3+2.3) .. cycle;
\draw[thick] (0,0+0.3+2.3) -- (2,0+0.3+2.3) -- (2,2+0.3+2.3) -- (0,2+0.3+2.3) -- cycle;
\draw[very thick,->] (2+0.1,1+0.3+2.3) -- (2+0.6,1+0.3+2.3);
\end{scope}
\begin{scope}[xshift=-1.6cm,yshift=-2.05cm]
\draw[fill=PURPLE!70!white,draw=black,thick] (0,0+0.3+2.3) -- (1,0+0.3+2.3) -- (1,1+0.3+2.3) -- (0,1+0.3+2.3) -- cycle;
\draw[fill=BLUE!70!white,draw=black,thick] (0,2+0.3+2.3) -- (1,2+0.3+2.3) -- (1,1+0.3+2.3) -- (0,1+0.3+2.3) -- cycle;
\draw[fill=GREEN!70!white,draw=black,thick] (2,0+0.3+2.3) -- (1,0+0.3+2.3) -- (1,1+0.3+2.3) -- (2,1+0.3+2.3) -- cycle;
\draw[fill=RED!70!white,draw=black,thick] (2,2+0.3+2.3) -- (2,1+0.3+2.3) -- (1,1+0.3+2.3) -- (1,2+0.3+2.3) -- cycle;
\draw[fill=YELLOW!70!white,draw=black,thick] (1.25,1+0.3+2.3) .. controls (1.25-0.2*0.5,1+0.2*0.86+0.3+2.3) and (1+0.2*0.86,1.25-0.2*0.5+0.3+2.3) .. (1,1.25+0.3+2.3) .. controls (1-0.2*0.86,1.25-0.2*0.5+0.3+2.3) and (0.75+0.2*0.5,1+0.2*0.86+0.3+2.3) .. (0.75,1+0.3+2.3) .. controls (0.75+0.2*0.5,1-0.2*0.86+0.3+2.3) and (1-0.2*0.86,0.75+0.2*0.5+0.3+2.3) .. (1,0.75+0.3+2.3) .. controls (1+0.2*0.86,0.75+0.2*0.5+0.3+2.3) and (1.25-0.2*0.5,1-0.2*0.86+0.3+2.3) .. cycle;
\draw[thick] (0,0+0.3+2.3) -- (2,0+0.3+2.3) -- (2,2+0.3+2.3) -- (0,2+0.3+2.3) -- cycle;
\draw[very thick,->] (2+0.1,1+0.3+2.3) -- (2+0.6,1+0.3+2.3);
\end{scope}
\begin{scope}[xshift=1.2cm,yshift=0.25cm]
\draw[fill=PURPLE!70!white,draw=black,thick] (0,0+0.3) -- (1,0+0.3) -- (1,1+0.3) -- (0,1+0.3) -- cycle;
\draw[fill=BLUE!70!white,draw=black,thick] (0,2+0.3) -- (1,2+0.3) -- (1,1+0.3) -- (0,1+0.3) -- cycle;
\draw[fill=GREEN!70!white,draw=black,thick] (2,0+0.3) -- (1,0+0.3) -- (1,1+0.3) -- (2,1+0.3) -- cycle;
\draw[fill=RED!70!white,draw=black,thick] (2,2+0.3) -- (2,1+0.3) -- (1,1+0.3) -- (1,2+0.3) -- cycle;
\draw[thick] (0,0+0.3) -- (2,0+0.3) -- (2,2+0.3) -- (0,2+0.3) -- cycle;
\end{scope}
\begin{scope}[xscale=0.7,yscale=0.7,xshift=13.2cm]
\draw[fill=RED!70!white,draw=black,thick] (0-5.0,0) -- (2-5.0,0) -- (2-5.0,2) -- (0-5.0,2) -- cycle;
\draw[fill=PURPLE!70!white,draw=black,thick] (0-5.0,0) -- (2-5.0,0) -- (0-5.0,2) -- cycle;
\draw[fill=BLUE!70!white,draw=black,thick] (0-5.0,2) -- (1-5.0,2) .. controls (1-5.0,1.55) and (0.85+0.2*0.258-5.0,1.15+0.2*0.97) .. (0.85-5.0,1.15) .. controls (0.85-0.2*0.97-5.0,1.15-0.2*0.258) and (0.45-5.0,1) .. (0-5.0,1) -- cycle;
\draw[fill=GREEN!70!white,draw=black,thick] (2-5.0,0) -- (2-5.0,1) .. controls (1.55-5.0,1) and (1.15+0.2*0.97-5.0,0.85+0.2*0.258) .. (1.15-5.0,0.85) .. controls (1.15-0.2*0.258-5.0,0.85-0.2*0.97) and (1-5.0,0.45) .. (1-5.0,0) -- cycle;
\draw[thick] (0-5.0,0) -- (2-5.0,0) -- (2-5.0,2) -- (0-5.0,2) -- cycle;
\draw[very thick,->] (2.1-7.7,1) -- (2.6-7.7,1);
\end{scope}
\begin{scope}[xscale=0.7,yscale=0.7,xshift=8cm]
\draw[fill=RED!70!white,draw=black,thick] (0-2.5,0) -- (2-2.5,0) -- (2-2.5,2) -- (0-2.5,2) -- cycle;
\draw[fill=PURPLE!70!white,draw=black,thick] (0-2.5,0) -- (2-2.5,0) -- (0-2.5,2) -- cycle;
\draw[fill=BLUE!70!white,draw=black,thick] (0-2.5,2) -- (1-2.5,2) -- (1-2.5,1.5) .. controls (1-2.5,1.2) and (0.92-2.5+0.13*0.258,1.08+0.13*0.97) .. (0.92-2.5,1.08) .. controls (0.92-2.5-0.13*0.97,1.08-0.13*0.258) and (0.8-2.5,1.0) .. (0.5-2.5,1) -- (0-2.5,1) -- cycle;
\draw[fill=GREEN!70!white,draw=black,thick] (2-2.5,0) -- (1-2.5,0) -- (1-2.5,0.5) .. controls (1-2.5,0.8) and (1.08-2.5-0.13*0.258,0.92-0.13*0.97) .. (1.08-2.5,0.92) .. controls (1.08-2.5+0.13*0.97,0.92+0.13*0.258) and (1.2-2.5,1.0) .. (1.5-2.5,1) -- (2-2.5,1) -- cycle;
\draw[thick] (0-2.5,0) -- (2-2.5,0) -- (2-2.5,2) -- (0-2.5,2) -- cycle;
\draw[very thick,->] (2.1-5.2,1+0.4) -- (2.5-5.2,1);
\end{scope}
\begin{scope}[xscale=0.7,yscale=0.7,xshift=3cm,yshift=2.5cm]
\draw[very thick,->] (2.0-0.1,1-0.4) -- (2.4-0.1,1);
\draw[fill=BLUE!70!white,draw=black,thick] (0+2.5,0) -- (2+2.5,0) -- (2+2.5,2) -- (0+2.5,2) -- cycle;
\draw[fill=GREEN!70!white,draw=black,thick] (0+2.5,0) -- (2+2.5,0) -- (2+2.5,2) -- cycle;
\draw[fill=PURPLE!70!white,draw=black,thick] (0+2.5,0) -- (1+2.5,0) -- (1+2.5,0.5) .. controls (1+2.5,0.8) and (0.92+2.5+0.13*0.258,0.92-0.13*0.97) .. (0.92+2.5,0.92) .. controls (0.92+2.5-0.13*0.97,0.92+0.13*0.258) and (0.8+2.5,1.0) .. (0.5+2.5,1) -- (0+2.5,1) -- cycle;
\draw[fill=RED!70!white,draw=black,thick] (2+2.5,2) -- (1+2.5,2) -- (1+2.5,1.5) .. controls (1+2.5,1.2) and (1.08+2.5-0.13*0.258,1.08+0.13*0.97) .. (1.08+2.5,1.08) .. controls (1.08+2.5+0.13*0.97,1.08-0.13*0.258) and (1.2+2.5,1.0) .. (1.5+2.5,1) -- (2+2.5,1) -- cycle;
\draw[thick] (0+2.5,0) -- (2+2.5,0) -- (2+2.5,2) -- (0+2.5,2) -- cycle;
\draw[very thick,->] (2.1+2.5,1) -- (2.6+2.5,1);
\end{scope}
\begin{scope}[xscale=0.7,yscale=0.7,xshift=3.2cm,yshift=2.5cm]
\draw[fill=BLUE!70!white,draw=black,thick] (0+5.0,0) -- (2+5.0,0) -- (2+5.0,2) -- (0+5.0,2) -- cycle;
\draw[fill=GREEN!70!white,draw=black,thick] (0+5.0,0) -- (2+5.0,0) -- (2+5.0,2) -- cycle;
\draw[fill=PURPLE!70!white,draw=black,thick] (0+5.0,0) -- (0+5.0,1) .. controls (0.45+5.0,1) and (0.85-0.2*0.97+5.0,0.85+0.2*0.258) .. (0.85+5.0,0.85) .. controls (0.85+0.2*0.258+5.0,0.85-0.2*0.97) and (1+5.0,0.45) .. (1+5.0,0) -- cycle;
\draw[fill=RED!70!white,draw=black,thick] (2+5.0,2) -- (2+5.0,1) .. controls (1.55+5.0,1) and (1.15+0.2*0.97+5.0,1.15-0.2*0.258) .. (1.15+5.0,1.15) .. controls (1.15-0.2*0.258+5.0,1.15+0.2*0.97) and (1+5.0,1.55) .. (1+5.0,2) -- cycle;
\draw[thick] (0+5.0,0) -- (2+5.0,0) -- (2+5.0,2) -- (0+5.0,2) -- cycle;
\end{scope}
\end{tikzpicture}
\caption{An example of a nonunique evolution of multiphase mean curvature flow, starting from an initial interface consisting only of smooth curves meeting at an angle of 120$^\circ$.\label{FigureNonuniqueness}}
\end{figure}

Our main result -- a uniqueness theorem for $\BV$ solutions to planar multiphase mean curvature flow prior to the first topology change, along with a corresponding result for Kim-Stuvard-Tonegawa varifold-$\BV$ solutions -- is therefore not only the first positive result concerning uniqueness for a weak solution concept to multiphase mean curvature flow, but also optimal for general initial data.
Nevertheless, let us mention that it has been suggested by Ilmanen (see e.\,g.\ \cite{MantegazzaNovagaPludaSchulze}) that the uniqueness properties may be better if one restricts one's attention to \emph{generic} initial data:
For initial data given by a small random perturbation of a fixed multiphase interface, the evolution by mean curvature in the plane is expected to be unique and stable with respect to perturbations for \emph{almost every} perturbation.
The argument in favor of this proposed phenomenon is based on a numerical study classifying the ``stable'' and therefore ``generically occurring'' singularities in planar mean curvature flow \cite{Haettenschweiler}.
In two-phase mean curvature flow, evidence in favor of ``generic'' well-posedness is abundant: For instance, an infinitesimal amount of stochastic noise has been shown to yield selection principles for the evolution, see Dirr, Luckhaus, and Novaga \cite{DirrLuckhausNovaga} and Souganidis and Yip \cite{SouganidisYip}. Furthermore, in the framework of viscosity solutions it is immediate that fattening of level sets must be absent in almost all levels. Finally, a classification of generic singularities has been achieved by Colding and Minicozzi \cite{ColdingMinicozzi2,ColdingMinicozzi}.

\subsection{Classical calibrations and gradient flow calibrations}

The key idea for our weak-strong uniqueness result is a gradient-flow 
analogue of the notion of \emph{calibrations}. The classical concept of 
calibrations is an important tool to deduce lower bounds on the interface 
energy functional for fixed boundary conditions.
Recall that a classical calibration for a candidate minimizer $(\bar \chi_1,\ldots,\bar \chi_P)$ 
of the interface energy functional (for given boundary conditions and with equal surface tensions) 
is a collection of vector fields $\xi_i$, $1\leq i\leq N$, subject to the following three properties:
\begin{itemize}
\item It holds that $|\xi_i-\xi_j|\leq 1$ for all $i$ and $j$.
\item The vector fields are solenoidal, i.\,e.,\ $\nabla \cdot \xi_i=0$ for all $i$.
\item On the interface $\partial  \{\bar\chi_i=1\} \cap \partial\{ \bar \chi_j=1\}$ 
			between the phases $i$ and $j$, $i\neq j$, the vector field $\xi_{i,j}:=\xi_i-\xi_j$ 
			coincides with the outer unit normal vector field of $\partial \{\bar\chi_i=1\}$.
\end{itemize}
The existence of a calibration allows to infer that the partition $(\bar \chi_1,\ldots,\bar \chi_P)$ 
indeed minimizes the interface energy functional among all possible Caccioppoli 
partitions, see \cite[Definition 4.16]{AmbrosioFuscoPallara}, of the underlying 
set $D \subset \Rd[d]$, $d\geq 2$, with the same boundary conditions: For any competitor 
partition $(\chi_1,\ldots,\chi_{P})$, one may compute using the first two defining conditions 
of a calibration (with the abbreviation for the interfaces $I_{i,j}:=\partial^*\{\chi_i=1\}\cap \partial^* \{\chi_j=1\}$)
\begin{align*}
E[\chi]&=\frac{1}{2}\sum_{i,j=1,i\neq j}^P \int_{I_{i,j}} 1 \,\mathrm{d}\mathcal{H}^{d-1}
\geq 
\frac{1}{2} \sum_{i,j=1,i\neq j}^P \int_{I_{i,j}} (\xi_j-\xi_i)\cdot \frac{\nabla \chi_i}{|\nabla \chi_i|} \,\mathrm{d}|\nabla \chi_i|
\\&
= - \sum_{i,j=1,i\neq j}^P \int_{I_{i,j}} \xi_i \cdot \frac{\nabla \chi_i}{|\nabla \chi_i|} \,\mathrm{d}|\nabla \chi_i|
= - \sum_{i=1}^P \int_{D} \xi_i \cdot \frac{\nabla \chi_i}{|\nabla \chi_i|} \,\mathrm{d}|\nabla \chi_i|
\\&
= - \sum_{i=1}^P \int_{\partial D} \chi_i \vec{n}_{\partial D}\cdot \xi_i \,\mathrm{d}\mathcal{H}^{d-1}.
\end{align*}
The third defining condition for a calibration shows that in the previous computation, 
equality is in fact achieved for $(\bar \chi_1,\ldots,\bar \chi_{P})$. 
This proves $E[\chi]\geq E[\bar \chi]$ for all partitions~$\chi$ with the same boundary 
conditions $\chi=\bar \chi$ on $\partial D$.

We recall that a notion of calibrations is also available for free 
discontinuity problems \cite{AlbertiBouchitteDalMaso}, having important 
implications for the numerical computation of minimizers \cite{ChambolleCremersPock}. 
For further applications of the concept of calibrations to anisotropic or nonlocal perimeters or the Steiner problem, we refer to \cite{AlterCasellesChambolle,CasellesChambolleMollNovaga} 
as well as \cite{CabreCalibrations,PagliariHalfspaces} and \cite{CarioniPluda}.

In the present work, we introduce a gradient-flow analogue of the notion of calibrations.
As shown above, the existence of a (classical) calibration ensures that a certain configuration 
is a global minimizer of the energy functional. In a similar spirit, the existence of a 
gradient-flow calibration ensures that the path of steepest descent in the energy landscape 
of the surface energy functional is unique, and moreover that this path is stable with respect 
to perturbations in the initial condition.
In the case of equal surface tensions, a \emph{gradient flow calibration} for a given 
classical solution $\bar \chi=(\bar \chi_1,\ldots,\bar \chi_P)$ to multiphase mean curvature 
flow on~$\Rd$ consists of the following objects:
\begin{subequations}
\begin{itemize}
\item A vector field $\xi_{i,j}$ for each pair of phases $1\leq i,j\leq P$, $i\neq j$. 
Denoting by~${\bar{I}}_{i,j}$ the interface between the phases~$i$ and~$j$ in the strong 
solution~$\bar \chi$ and by~$\bar{ \vec{n}}_{i,j}$ its unit normal vector field pointing 
from phase~$i$ to phase~$j$, we require~$\xi_{i,j}$ to be an extension of~$\bar{\vec{n}}_{i,j}$ 
subject to the coercivity condition
\begin{align}
\label{LengthBound}
|\xi_{i,j}|\leq 1-c \min\{\dist^2(\cdot,{\bar I}_{i,j}),1\}
\end{align}
for some $c\in (0,1)$.
\item The extended normal vector fields $\xi_{i,j}$ must have the structure 
$\xi_{i,j}=\xi_i-\xi_j$ for some vector fields $\xi_i$, $1\leq i\leq P$. 
(This structure is reminiscent of the corresponding condition for classical 
calibrations and in fact serves a similar purpose, see the explanation 
preceding \eqref{RelativeEntropyRewritten} below.)
\item A \emph{single} velocity field $B$, which approximately transports \emph{all} extended normal vector fields $\xi_{i,j}$ in the sense
\begin{align}
\label{FirstEq}
\partial_t \xi_{i,j} + (B\cdot \nabla) \xi_{i,j} + (\nabla B)^{\mathsf{T}} \xi_{i,j} = O(\dist(\cdot,{\bar{I}}_{i,j})).
\end{align}
Furthermore, the length of the extended normal vector fields is transported to higher accuracy in the sense
\begin{align}
\partial_t |\xi_{i,j}|^2 + (B\cdot \nabla) |\xi_{i,j}|^2 = O(\dist^2(\cdot,{\bar{I}}_{i,j})).
\end{align}
\item Near the interfaces ${\bar{I}}_{i,j}$ of the strong solution, the normal velocity $\xi_{i,j}\cdot B$ is given by the mean curvature of ${\bar{I}}_{i,j}$ in the sense
\begin{align}
\label{ThirdEq}
\xi_{i,j} \cdot B = -\nabla \cdot \xi_{i,j} + O(\dist(\cdot,{\bar{I}}_{i,j})).
\end{align}
Note that on the interface ${\bar{I}}_{i,j}$, the expression $-\nabla \cdot \xi_{i,j}$ is exactly equal to its mean curvature.
\end{itemize}
\end{subequations}

If a gradient flow calibration exists, we may introduce a measure for the 
difference between any $\BV$ solution $\chi=(\chi_1,\ldots,\chi_P)$ to multiphase mean 
curvature flow and the strong solution $\bar \chi$ by defining
\begin{align}
\label{ErrorFunctional}
E[\chi|\xi]:=
\frac{1}{2}\sum_{i,j=1,i\neq j}^{P} \int_{I_{i,j}} 1 - \xi_{i,j} \cdot \vec{n}_{i,j} \,\mathrm{d}\mathcal{H}^{d-1},
\end{align}
with $I_{i,j}:=\partial^* \{\chi_i=1\} \cap \partial^* \{\chi_j=1\}$ denoting the 
interface between phases $i$ and $j$ and with $\vec{n}_{i,j}$ being its normal pointing from phase $i$ to phase $j$.
Note that the condition \eqref{LengthBound} then precisely ensures that $E[\chi|\xi]$ is a 
suitable notion of error between the $\BV$ solution $\chi$ and the strong solution $\bar \chi$: 
In addition to providing a tilt-excess-like control of the error, it also provides an estimate 
on the distance of the interfaces.

\begin{table}
\begin{tabular}{ll}
\emph{Calibrations}
&\emph{Gradient flow calibrations}
\vspace{2mm}
\\
Existence implies global minimality
&Existence implies uniqueness of
\\of surface energy among all partitions
&$\BV$ solutions to gradient flow
\vspace{2mm}
\\
Shortness condition
&Coercivity condition
\\
$|\xi_{i,j}|\leq 1$
&$|\xi_{i,j}|\leq 1-c\min\{\dist^2(\cdot,{\bar{I}}_{i,j}),1\}$
\vspace{2mm}
\\
Stationary situation
&Advection equation
\\
($\partial_t \xi_{i,j}\equiv 0$, $B\equiv 0$)
&$\partial_t \xi_{i,j}+(B\cdot \nabla)\xi_{i,j}+(\nabla B)^{\mathsf{T}} \xi_{i,j}$
\\&$~~~~~~~~~~~~~~~~~~~~~~~~= O(\dist(\cdot,{\bar{I}}_{i,j}))$
\vspace{1.5mm}
\\
Vector fields solenoidal
&Motion by mean curvature
\\
$\nabla \cdot \xi_i=0$
&
$\xi_{i,j}\cdot B=-\nabla \cdot \xi_{i,j}+O(\dist(\cdot,{\bar{I}}_{i,j}))$~~~~~
\vspace{2.5mm}
\end{tabular}
\caption{A comparison of the concept of calibrations for minimal partitions with the new concept of gradient flow calibrations.
\label{TableCalibrationsGradientFlow}
}
\end{table}

On the other hand, the calibration structure $\xi_{i,j}=\xi_i-\xi_j$ ensures that 
the error functional \eqref{ErrorFunctional} may be rewritten as an expression involving 
only two contributions: First, the total interface energy of the $\BV$ solution $E[\chi]$ 
and, second, a linear functional of the characteristic functions $\chi_i$ of the phases.
Indeed, we may compute
\begin{align}
\label{RelativeEntropyRewritten}
E[\chi|\xi]
&=
\frac{1}{2}\sum_{i,j=1,i\neq j}^{P} \int_{I_{i,j}} 1 - \xi_{i,j} \cdot \vec{n}_{i,j} \,\mathrm{d}\mathcal{H}^{d-1}
\\&
\nonumber
=E[\chi] - \frac{1}{2} \sum_{i,j=1,i\neq j}^P \int_{I_{i,j}} (\xi_i-\xi_j) \cdot \vec{n}_{i,j} \,\mathrm{d}\mathcal{H}^{d-1}
\\&
\nonumber
=E[\chi] - \sum_{i,j=1,i\neq j}^P \int_{I_{i,j}} \xi_i \cdot \vec{n}_{i,j} \,\mathrm{d}\mathcal{H}^{d-1}
\\&
\nonumber
=E[\chi] + \sum_{i=1}^P \int_{\Rd} \xi_i \cdot \,\mathrm{d}\nabla \chi_i
\\&
\nonumber
=E[\chi] - \sum_{i=1}^P \int_{\Rd} \chi_i \nabla \cdot \xi_i \,\mathrm{d}x.
\end{align}
This enables us to estimate the time evolution of the error functional $E[\chi| \xi]$ 
using only two ingredients, namely, first, the sharp energy dissipation 
estimate~\eqref{EnergyDissipationInequalityBVSolution} for the interface 
energy~$E[\chi]$ for $\BV$~solutions, and, second, the evolution 
equation~\eqref{EvolutionPhasesBVSolution} for the phase indicator 
functions $\chi_i$ from the $\BV$ formulation of mean curvature flow. 
The equations \eqref{FirstEq}--\eqref{ThirdEq} are crucial for deriving 
a Gronwall-type estimate for $E[\chi|\xi]$ in subsequent rearrangements. 
We remark that this approach may be regarded as an instance of the relative 
entropy method introduced independently by Dafermos~\cite{Dafermos} and Di~Perna~\cite{DiPerna}.

Note that locally at a two-phase interface or a triple junction of the strong 
solution, for any fixed time $t$ the blowups of our vector fields $\xi_{i}(\cdot,t)$ 
turn out to precisely be calibrations of the planar interface or the triple junction, 
respectively. However, on a global (not blown-up) scale, the vector fields $\xi_i$ may 
be thought of as deformed variants of classical calibrations which follow the (smooth 
but typically curved) interface of the strong solution. We refer to Figure~\ref{FigureCurveCalibrations} 
and Figure~\ref{fig:triple_rotation} for the illustration of a vector field $\xi_{i,j}=\xi_i-\xi_j$ 
at a two-phase interface and at a triple junction, respectively. 

\begin{figure}
\includegraphics[scale=0.55]{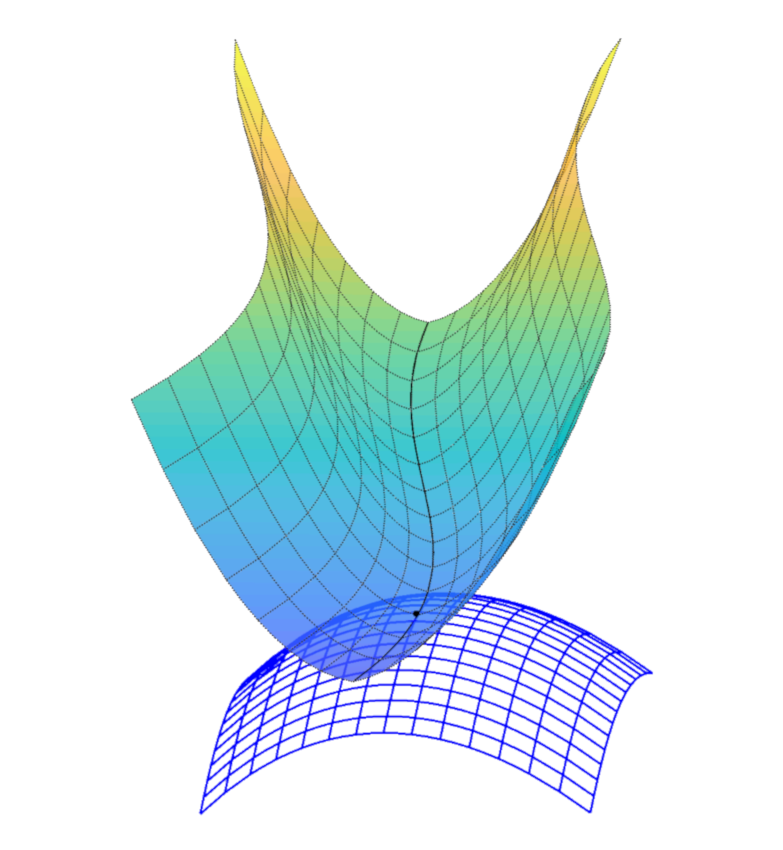}
\caption{An illustration of the energy landscape interpretation of our construction: The gradient flow 
calibration provides a smooth lower bound for the rough energy landscape of the interface energy 
functional $E[\chi]$, correctly capturing the energy and its subgradient at the current 
configuration. \label{EnergyLandscapeInterpretation}}
\bigskip
\end{figure}

Let us finally comment on the energy landscape interpretation of our approach, as 
illustrated in Figure~\ref{EnergyLandscapeInterpretation}. Let $\bar \chi(t)$ 
be a classical solution to the gradient flow of the interface energy functional, i.e., 
a classical solution to multiphase mean curvature flow. For each point in time, the 
``calibration for the gradient flow'' gives a smooth lower bound
\begin{align*}
\mathcal{F}_t:=\frac{1}{2}\sum_{i,j=1,i\neq j}^P \int_{I_{i,j}} \xi_{i,j}\cdot \vec{n}_{i,j} \,\mathrm{d} \mathcal{H}^{d-1} \overset{\eqref{RelativeEntropyRewritten}}{=} \sum_{i=1}^P \int_{\Rd} \chi_i \nabla \cdot \xi_i \,\mathrm{d} x
\end{align*}
(illustrated as the blue wireframe plot in Figure~\ref{EnergyLandscapeInterpretation}) for 
the rough landscape of the interface energy functional $E[\chi]$ (illustrated as the colored 
surface plot in Figure~\ref{EnergyLandscapeInterpretation}). This lower bound is sharp for 
the network described by $\bar \chi(t)$ in the sense $\mathcal{F}_t[\bar \chi(t)]=E[\bar \chi(t)]$, and 
it describes the local direction and speed of steepest descent of the energy functional $E[\chi]$ 
at $\bar \chi(t)$ correctly in the sense $D\mathcal{F}_t[\bar \chi(t)]\in DE[\bar \chi(t)]$ 
(where heuristically $DE$ denotes the subdifferential of $E$). Moreover, for each $\chi$ 
the difference $E[\chi]-\mathcal{F}_t[\chi] =E[\chi| \xi] $ provides an estimate for the error between 
the smooth solution $\bar \chi(t)$ and the configuration $\chi$, as measured in a tilt-excess-like quantity.

\section{Main results}

\subsection{Weak-strong uniqueness principle}
In the following, we present our weak-strong uniqueness principle for 
$\BV$~solutions of multiphase mean curvature flow in the plane. In addition, 
we provide a quantitative stability estimate, i.\,e., as long as a strong solution 
exists, any solution to the $\BV$~formulation of multiphase mean curvature
flow with slightly perturbed initial data remains close to it.  
Our results are valid under minimal assumptions on the surface tensions, 
see Definition~\ref{DefinitionAdmissibleSurfaceTensions}; in particular, the choice of 
equal surface tensions $\sigma_{i,j}=1$ for any pair $i\neq j$ is admissible.

\begin{theorem}[Weak-strong uniqueness and quantitative stability]
\label{MainResult}
Let $d=2$ and $P\in \mathbb{N}$, $P\geq 2$. Let $\sigma_{i,j}>0$, $1\leq i,j\leq P$, 
be an admissible set of surface tensions in the sense of Definition~\ref{DefinitionAdmissibleSurfaceTensions}.
Let $\chi=(\chi_1,\ldots,\chi_P)$ be a $\BV$~solution of multiphase mean curvature 
flow in the sense of Definition~\ref{DefinitionBVSolution} on some time interval $[0,\TBV)$. 
Let $\bar\chi=(\bar\chi_1,\ldots,\bar\chi_P)$ be a strong solution 
of multiphase mean curvature flow on $\Rd[d]$ in the sense of 
Definition~\ref{DefinitionStrongSolution} on some time interval $[0,\Tstrong)$ with $\Tstrong\leq\TBV$.

Then, the $BV$~solution $\chi$ must coincide with the strong solution 
$\bar \chi$ for almost all $0\leq t<\Tstrong$, provided that it starts from the same initial data.

Furthermore, the evolution by mean curvature is stable with respect to perturbations in the initial data in the following sense: Denote by $\xi_{i,j}$ the gradient flow calibration for the classical solution~$\bar\chi$ as constructed in Theorem~\ref{TheoremExistenceCalibration} and define
\begin{align*}
E[\chi|\xi](t):=
\sum_{i,j=1,i\neq j}^P \sigma_{i,j} \int_{I_{i,j}(t)} 
1-\xi_{i,j}(\cdot,t) \cdot \vec{n}_{i,j}(\cdot,t) \,\mathrm{d}\mathcal{H}^{d-1}.
\end{align*}
Then for every $T\in (0,\Tstrong)$ the stability estimates
\begin{align*}
E[\chi|\xi](t) & \leq e^{Ct}E[\chi|\xi](0),\\
E_{\mathrm{volume}}[\chi|{\bar \chi}](t)
& \leq e^{Ct} \big(E_{\mathrm{volume}}[\chi| \bar \chi ](0)+E[\chi|\xi](0)\big)
\end{align*}
hold true for almost every $t\in[0,T]$, where the bulk error 
functional $E_{\mathrm{volume}}[\chi| \bar \chi]$ is defined in \eqref{volumeErrorFunctional} and where the constant $C>0$ 
only depends on~$\bar \chi$ and~$T$ through certain higher derivatives of functions associated to~$\bar \chi$.
\end{theorem}

\begin{proof}
Given the assumptions of our theorem, Theorem~\ref{TheoremExistenceCalibration} ensures the existence of a gradient flow calibration, while Lemma~\ref{LemmaWeightedVolumeControl} yields the existence of a family of transport weights. Thus, the conditional weak-strong uniqueness principles of Theorem~\ref{TheoremUniqueness} and Proposition~\ref{PropositionUniquenessConditional} are applicable and provide our assertion.
\end{proof}

\subsection{Calibrations and inclusion principle}
The key ingredient for our uniqueness result prior to topology changes is the following gradient 
flow analogue of the notion of calibrations for minimal partitions. Our main result, 
Theorem~\ref{MainResult}, is then an immediate consequence of two implications: First, 
the existence of a gradient flow calibration guarantees uniqueness of the $\BV$ solution 
(see Theorem~\ref{TheoremUniqueness} and Proposition~\ref{PropositionUniquenessConditional})
in arbitrary ambient dimension $d\geq 2$; second, classical solutions to planar multiphase mean 
curvature flow are calibrated in the sense that a gradient flow calibration exists 
(see Theorem~\ref{TheoremExistenceCalibration} and 
Lemma~\ref{LemmaWeightedVolumeControl}).

\begin{definition}[Calibrations for the gradient flow and calibrated flows]
\label{DefinitionCalibrationGradientFlow}
Let $d\geq 2$, $P\geq 2$ be integers and let $\sigma\in\Rd[P\times P]$ be an admissible
matrix of surface tensions in the sense of Definition~\ref{DefinitionAdmissibleSurfaceTensions}.
Let $T>0$, and for all~$i\in\{1,\ldots,P\}$ let
$\smash{\bar\Omega_i:=\bigcup_{t\in [0,T]}\bar\Omega_i(t){\times}\{t\}}$
such that for all $t\in [0,T]$ the family $(\bar \Omega_1(t),\ldots,\bar \Omega_P(t))$ 
is a partition of finite surface energy of $\Rd$ in the sense of Definition~\ref{DefinitionPartition}.
For each $i,j\in \{1,\ldots,P\}$ with $i\neq j$ and all~$t\in [0,T]$, let 
${\bar{I}}_{i,j}(t):=\partial^* {\bar{\Omega}}_i(t)\cap \partial^* {\bar{\Omega}}_j(t)$
be the interface between the phases~$i$ and~$j$ at time~$t$.

A pair~$(\xi=(\xi_i)_{i\in\{1,\ldots,P\}},B)$ consisting of vector fields
\begin{align*}
\xi_i &\in C^1([0,T];C^0_{\mathrm{cpt}}(\Rd[d];\Rd[d]))
\cap C^0([0,T];C^1_{\mathrm{cpt}}(\Rd[d];\Rd[d])), \quad i\in\{1,\ldots,P\},
\\
B &\in C^0([0,T];C^1_{\mathrm{cpt}}(\Rd[d];\Rd[d]))
\end{align*}
is called a \emph{calibration for the gradient flow}
for the \emph{calibrated flow $\smash{(\bar\Omega_1,\ldots,\bar\Omega_P)}$ on~$[0,T]$}
if the following conditions are satisfied:
\begin{subequations}
\begin{itemize}
\item For each pair of phases $i,j\in\{1,\ldots, P\}$ and all~$t\in [0,T]$, the vector field
\begin{align}
\label{Calibrations}
\xi_{i,j}(\cdot,t):=\frac{1}{\sigma_{i,j}}(\xi_i-\xi_j)(\cdot,t)
\end{align}
coincides on ${\bar{I}}_{i,j}(t)$ with the associated unit normal vector field~$\bar{\vec{n}}_{i,j}(\cdot,t)$ 
(with the convention that $\bar{\vec{n}}_{i,j}(\cdot,t)$ points from phase~$i$ into phase~$j$), and it
satisfies an estimate of the form
\begin{align}
\label{LengthControlXi}
|\xi_{i,j}(x,t)|\leq 1-c \min\{\dist^2(x,{\bar{I}}_{i,j}(t)),\,1\}
\end{align}
for some $c\in (0,1)$ and all $(x,t)\in \Rd\times [0,T]$.
\item The evolution of the vector fields $\xi_{i,j}$ is approximately transported by the velocity field $B$ in the sense
\begin{align}
\label{TransportEquationXi}
\big|\partial_t \xi_{i,j} + (B\cdot \nabla) \xi_{i,j} + (\nabla B)^{\mathsf{T}} \xi_{i,j} \big|(x,t)
\leq C \big(\dist(x,{\bar{I}}_{i,j}(t))\wedge 1\big)
\end{align}
and
\begin{align}
\label{LengthConservation}
\big|\partial_t |\xi_{i,j}|^2 + (B\cdot \nabla) |\xi_{i,j}|^2 \big|(x,t)
\leq C \big(\dist^2(x,{\bar{I}}_{i,j}(t))\wedge 1\big)
\end{align}
for some $C>0$ and all $(x,t)\in \Rd\times [0,T]$.
\item For each~$t\in [0,T]$, the normal component of the velocity field~$B(\cdot,t)$ 
		 near the interface ${\bar{I}}_{i,j}(t)$ is approximately given by the mean curvature 
			of ${\bar{I}}_{i,j}(t)$ in the sense that
\begin{align}
\label{Dissip}
\big|\xi_{i,j}\cdot B+\nabla \cdot \xi_{i,j}\big| (x,t)
\leq
C \big(\dist(x,{\bar{I}}_{i,j}(t))\wedge 1\big)
\end{align}
for some $C>0$ and all $(x,t)\in \Rd\times [0,T]$.
\end{itemize}
\end{subequations}
\end{definition}
Note that, at least heuristically, such a calibrated flow is a solution to mean 
curvature flow as on $\bar I_{i,j}$ the normal velocity $ \vec{\bar{n}}_{i,j} \cdot B$ 
coincides with the mean curvature due to \eqref{Dissip}.

The next proposition states that for general $d\geq 2$ the existence of a gradient flow calibration for a 
given time-evolving partition of $\Rd$ into $P$ domains $(\bar \Omega_1,\ldots,\bar \Omega_P)$ 
constrains the possible locations of the interfaces in weak ($\BV$) solutions to mean curvature 
flow to the corresponding interfaces of the partition $(\bar \Omega_1,\ldots,\bar \Omega_P)$. 
This assertion may be seen as a multiphase analogue of the varifold comparison principle by 
Ilmanen \cite[Theorem 10.7]{Ilmanen}, which for two-phase mean curvature flow provides a 
corresponding inclusion given any Brakke solution and a level set solution. Note that such 
an inclusion does not yet yield uniqueness of $\BV$ solutions, as it does not exclude the sudden 
vanishing of all phases except one.

\begin{theorem}[Quantitative inclusion principle]
\label{TheoremUniqueness}
Let $d\geq 2$ and $P\geq 2$ be integers and let $\sigma \in \Rd[P\times P]$ be an 
admissible matrix of surface tensions, see Definition~\ref{DefinitionAdmissibleSurfaceTensions}.
Let $T>0$, and let $\smash{(\bar\Omega_1,\ldots,\bar\Omega_P)}$ 
be a calibrated flow on~$[0,T]$ in the sense of Definition~\ref{DefinitionCalibrationGradientFlow}.

Then the interfaces $I_{i,j}(t) := \partial^* \{\chi_i(t)=1\} \cap \partial^* \{\chi_j(t)=1\}$ 
of any $\BV$ solution $(\chi_1,\ldots,\chi_P)$ to mean curvature flow on~$[0,T]$ in the sense of 
Definition~\ref{DefinitionBVSolution} with the same initial data as the calibrated flow
must be contained in the corresponding interfaces 
${\bar{I}}_{i,j}(t):=\partial^* \bar \Omega_i(t)\cap \partial^* \bar \Omega_j(t)$ for a.\,e.\ $0<t<T$, i.e., 
it holds $I_{i,j}(t)\subset {\bar{I}}_{i,j}(t)$ for all $i,j$ with $i\neq j$ up to $\mathcal{H}^{d-1}$ null sets.

Furthermore, the existence of a gradient flow calibration also implies a stability 
estimate: Introducing the interface error functional
\begin{align}
\label{DefinitionRelativeEntropy}
E[\chi|\xi](t):=
\sum_{i,j=1,i\neq j}^P \sigma_{i,j} \int_{I_{i,j}(t)} 
1-\xi_{i,j}(\cdot,t) \cdot \vec{n}_{i,j}(\cdot,t) \,\mathrm{d}\mathcal{H}^{d-1},
\end{align}
there exist two constants $c,C>0$ depending on the calibrated flow such that we have the stability estimate
\begin{align}\label{StabilityEstimate}
&E[\chi|\xi](t)
+c\sum_{i,j=1,i\neq j}^{P}
\int_0^{t}\int_{I_{i,j}(\tilde t)}
|V_{i,j}{+}\nabla\cdot\xi_{i,j}|^2
+|V_{i,j}\vec{n}_{i,j}{-}(B\cdot\xi_{i,j})\xi_{i,j}|^2
\,\mathrm{d}\mathcal{H}^{d-1}\dtildet
\\&~~~~~~
\notag
\leq e^{Ct}E[\chi|\xi](0)
\end{align}
for general BV solutions $\chi= (\chi_1,\ldots \chi_P)$ and almost every $t\in[0,T]$.
\end{theorem}

As already discussed, the interface error control provided by the 
functional~\eqref{DefinitionRelativeEntropy}
suffers from a lack of coercivity concerning the vanishing of interface length in a $\BV$ solution.
For this reason, we also have to consider a lower-order term $E_{\mathrm{volume}}[\chi|\bar\chi]$, 
see~\eqref{volumeErrorFunctional} below, which controls bulk deviations from the grains 
of the strong solution~$\bar \Omega$. The main input for the bulk error functional
is captured in the following notion of transported weights.

\begin{definition}[Transported weights]
\label{def:transportedWeights} 
Let $d\geq 2$, $P\geq 2$ be integers and denote by $T\in (0,\infty)$ a finite time horizon. 
For all~$i\in\{1,\ldots,P\}$ let
$\smash{\bar\Omega_i:=\bigcup_{t\in [0,T]}\bar\Omega_i(t){\times}\{t\}}$
such that for all $t\in [0,T]$ the family $(\bar \Omega_1(t),\ldots,\bar \Omega_P(t))$ 
is a partition of finite surface energy of $\Rd$ in the sense of Definition~\ref{DefinitionPartition}.
Denote by $\bar\chi=(\bar\chi_1,\ldots,\bar\chi_P)$ 
the associated family of indicator functions for~$\bar\Omega=(\bar\Omega_1,\ldots,\bar\Omega_P)$. 
Assume that for all~$i\in\{1,\ldots,P\}$ the measure $\partial_t\bar\chi_i$ is absolutely continuous
with respect to the measure $|\nabla\bar\chi_i|$, and that the boundary $\partial\bar\Omega_i(\cdot,t)$
is Lipschitz at all times $t\in[0,T]$.
Let finally $B\in C^0([0,T];C^1_{\mathrm{cpt}}(\Rd;\Rd))$.

In this setting, a family of measurable maps
\begin{align*}
\vartheta_i\colon\Rd\times [0,T] \to [-1,1],\quad i\in\{1,\ldots,P\},
\end{align*} 
is called a \emph{family of transported weights with respect to $(\bar\Omega,B)$
on~[0,T]} if the following conditions are satisfied:
\begin{itemize}[leftmargin=0.7cm]
\item (Regularity) For all phases $i\in\{1,\ldots,P\}$ it holds $$\vartheta_i\in W^{1,1}(\Rd\times [0,T])
			\cap W^{1,\infty}(\Rd\times [0,T]).$$
\item (Coercivity) For all phases $i\in\{1,\ldots,P\}$ and all~$t\in [0,T]$, 
			we have $\vartheta_i(\cdot,t) < 0$ in the essential interior
			of $\bar\Omega_i(t)$, $\vartheta_i(\cdot,t) > 0$ in the essential exterior of $\bar\Omega_i(t)$, and
			$\vartheta_i(\cdot,t) = 0$ on $\partial\bar\Omega_i(t)$.
\item (Advection equation) The weights are transported by the vector field~$B$ in the sense that
			\begin{align}
			\label{AdvectionEquationVolumeControl}
			|\partial_t\vartheta_i + (B\cdot\nabla)\vartheta_i| \leq C|\vartheta_i|
			\end{align}
			holds true in $\Rd\times [0,T]$ for all phases $i\in\{1,\ldots,P\}$. 
\end{itemize}
\end{definition}

The merit of the previous definition is that it allows to sharpen
the quantitative inclusion principle of Theorem~\ref{TheoremUniqueness}
to a conditional weak-strong uniqueness principle (with an associated
conditional stability estimate) for BV solutions of multiphase mean curvature flow;
see Proposition~\ref{PropositionUniquenessConditional} below for the precise statement.
The result is conditional in the sense that in addition to the existence of a gradient flow calibration
(see Definition~\ref{DefinitionCalibrationGradientFlow}),
the existence of a family of transported weights
(see Definition~\ref{def:transportedWeights}) is assumed. 
However, the crucial point is that it already holds in arbitrary ambient dimension $d\geq 2$. 

\begin{proposition}[Conditional weak-strong uniqueness and quantitative stability]
\label{PropositionUniquenessConditional}
Let $d\geq 2$, $P\geq 2$ be integers and $\sigma\in\Rd[P\times P]$ be an admissible matrix of surface tensions
in the sense of Definition~\ref{DefinitionAdmissibleSurfaceTensions}. 
Let $\chi=(\chi_1,\ldots,\chi_P)$ be a BV solution of multiphase mean curvature flow in
the sense of Definition~\ref{DefinitionBVSolution} on $[0,T]$.
Let moreover $\bar\Omega=(\bar\Omega_1,\ldots,\bar\Omega_P)$ be as in Definition~\ref{def:transportedWeights}
on~$[0,T]$. The associated family of indicator functions is denoted by $\bar\chi=(\bar\chi_1,\ldots,\bar\chi_P)$.

Assume also that there exists a gradient flow calibration $((\xi_i)_{i\in\{1,\ldots,P\}},B)$
with respect to $\bar\Omega$ on $[0,T]$ in the sense of Definition~\ref{DefinitionCalibrationGradientFlow},
and that there exists a family of transported weights $(\vartheta_i)_{i\in\{1,\ldots,P\}}$
with respect to $(\bar\Omega,B)$ on~$[0,T]$ in the sense of Definition~\ref{def:transportedWeights}.
Recall the definition~\eqref{DefinitionRelativeEntropy} of the interface error functional,
and define a bulk error functional by means of
\begin{align}
\label{volumeErrorFunctional}
E_{\mathrm{volume}}[\chi|\bar\chi](t) :=
\sum_{i=1}^P \int_{\Rd}|\chi_i(\cdot,t){-}\bar\chi_i(\cdot,t)|
|\vartheta_i(\cdot,t)| \dx, \quad t\in [0,T].
\end{align}

Then it holds
\begin{align*}
\chi(\cdot,0) = \bar\chi(\cdot,0) \text{ a.e.\ in } \Rd
\Rightarrow \chi(\cdot,t) = \bar\chi(\cdot,t) \text{ a.e.\ in }  \Rd
\text{ for a.e.\ } t\in [0,T].
\end{align*}
Moreover, the interface error functional $E[\chi|\xi]$ from~\eqref{DefinitionRelativeEntropy}
and the bulk error functional $E_{\mathrm{volume}}[\chi|\bar\chi]$ from~\eqref{volumeErrorFunctional}
satisfy the quantitative stability estimate
\begin{align}
\label{eq:stabilityBulkError}
E_{\mathrm{volume}}[\chi|{\bar \chi}](t)
& \leq e^{Ct} \big(E_{\mathrm{volume}}[\chi| \bar \chi ](0)+E[\chi|\xi](0)\big)
\end{align}
for almost every $t\in [0,T]$ in addition to the stability estimate \eqref{StabilityEstimate}.
\end{proposition}

\subsection{Gradient flow calibrations for regular networks}
In view of Proposition~\ref{PropositionUniquenessConditional} above, the question of weak-strong uniqueness for BV solutions
of multiphase mean curvature flow is reduced to the task of constructing a gradient flow calibration
and a family of transported weights. As it turns out, in the planar case the existence of a classical solution to mean curvature 
flow --- in the sense of a smooth evolution of curves meeting at triple junctions with the correct contact angle, 
see Definition~\ref{DefinitionStrongSolution} --- entails the existence of a calibration for the gradient flow:
\begin{theorem}
\label{TheoremExistenceCalibration}
Let $d=2$ and $P \in \mathbb{N}$, $P\geq 2$. Let $(\bar{\Omega}_1,\ldots,\bar{\Omega}_P)$ 
be a strong solution to multiphase mean curvature flow on~$[0,T]$
in the sense of Definition~\ref{DefinitionStrongSolution}. 
Then there exists an associated gradient flow calibration on~$[0,T]$ 
in the sense of Definition~\ref{DefinitionCalibrationGradientFlow}.
\end{theorem}

In fact, our construction of gradient flow calibrations provides several additional properties.
\begin{remark}
\label{RemarkFurtherCalibrationProperties}
The gradient flow calibrations constructed in the proof of Theorem~\ref{TheoremExistenceCalibration} satisfy the following additional properties, which may be useful in the context of diffuse interface approximations:
\begin{itemize}
\item[i)] In the case of equal surface energies $\sigma_{i,j}=\sigma_{k,l}$ for all $i\neq j$ and all $k\neq l$, we have the estimates $|\xi_{i,j}\cdot \xi_k|(x,t)\leq C \dist(x,{\bar{I}}_{i,j}(t))$ for all $i\neq j$, all $k\notin\{i,j\}$
and all $(x,t) \in \Rd{\times}[0,T]$, as well as $|\xi_i|\leq \frac{1}{\sqrt{3}}$ for all $i$.
\item[ii)] It holds that $|\nabla B : \xi_{i,j} \otimes \xi_{i,j}|(x,t)\leq C 
\dist(x,{\bar{I}}_{i,j}(t))$ for all $i\neq j$ and all $(x,t) \in \Rd{\times}[0,T]$.
\item[iii)] Finally, we can achieve the estimate $\big|\nabla B : \big(\xi_{i,j} \otimes J\xi_{i,j}+J\xi_{i,j} \otimes \xi_{i,j}\big)\big|(x,t)\leq C \dist(x,{\bar{I}}_{i,j}(t))$ for all $i\neq j$
and all $(x,t) \in \Rd{\times}[0,T]$, where the matrix $J$ denotes the counter-clockwise rotation by $90^\circ$.
\end{itemize}
\end{remark}

In the same setting as above, one can in addition establish
the existence of a family of transported weights.

\begin{lemma}
\label{LemmaWeightedVolumeControl}
Let $d=2$ and $P \in \mathbb{N}$, $P\geq 2$. Let $(\bar{\Omega}_1,\ldots,\bar{\Omega}_P)$ 
be a strong solution to multiphase mean curvature flow on~$[0,T]$
in the sense of Definition~\ref{DefinitionStrongSolution}. 
Let~$B$ denote the velocity field from Theorem~\ref{TheoremExistenceCalibration}. 
Then there exists a family of transported weights on~$[0,T]$ with respect to~$(\bar\Omega,B)$ 
in the sense of Definition~\ref{def:transportedWeights}.
\end{lemma}

\subsection{Basic definitions}

In the following, we recall the precise definitions of the solution concepts for multiphase mean curvature flow which our main results are concerned with.
We begin with the notion of admissible surface tensions. 

\begin{definition}[Admissible matrix of surface tensions]\label{DefinitionAdmissibleSurfaceTensions}
Let $P\geq 2$ be an integer and $\sigma=(\sigma_{i,j})_{i,j=1,\ldots,P}\in \Rd[P\times P]$. The matrix $\sigma$
is called an \emph{admissible matrix of surface tensions} if the following conditions are satisfied:
\begin{itemize}[leftmargin=0.7cm]
\item[i)] (Symmetry) It holds that $\sigma_{i,j}=\sigma_{j,i}$ and $\sigma_{i,i}= 0$ for every $i,j\in\{1,\ldots,P\}$.
\item[ii)] (Positivity) We have $\sigma_{\mathrm{min}}:=\min\{\sigma_{i,j}\colon i,j\in\{1,\ldots,P\},i\neq j\}>0$.
\item[iii)] (Coercivity) 
The matrix of surface tensions $\sigma$ is non-degenerately $\ell^2$-embeddable into $\Rd[P-1]$, i.e., there exists a non-degenerate $(P-1)$-simplex $(q_1,\ldots, q_P)$ in $\Rd[P-1]$ such that $\sigma_{i,j} = | q_i-q_j|$ for all $i,j \in \{1,\ldots,P\}$,
	see Figure~\ref{fig:l2embedding}.
\end{itemize}
\end{definition}

We briefly comment on the previous definition.

\begin{remark}
	The above conditions on the matrix of surface tensions are natural, which is clear for the first two items, 
	while condition~{iii)} already appeared in \cite{Lawlor-Morgan} as being necessary for the existence 
	of calibrations in the static case. It implies another coercivity condition in the form of the strict triangle inequality
		\begin{align}\label{TriangleInequalitySurfaceTensions}
		\sigma_{i,j} < \sigma_{i,k} + \sigma_{k,j}
		\end{align}
		for all choices of pairwise distinct $i,j,k\in\{1,\ldots, P\}$.
		
	We call condition~{iii)} of Definition~\ref{DefinitionAdmissibleSurfaceTensions} 
	and condition~\eqref{TriangleInequalitySurfaceTensions} coercivity properties
	for the following reasons:	
		First, the strict triangle inequality \eqref{TriangleInequalitySurfaceTensions} will ensure that 
		our relative entropy functional provides control on wetting, i.e., the nucleation of 
		a thin layer of a third phase along the smooth part of an interface between two phases.
		Second, the embeddability condition~{iii)} will prevent 
		the nucleation of a fourth phase (or clusters of phases) at a triple junction.
	
	It is well known, see \cite[Section 3]{schoenberg1938metric}, that 
	condition~{iii)} of Definition~\ref{DefinitionAdmissibleSurfaceTensions} may be equivalently phrased as follows: 
	The symmetric $(P\times P)$-matrix $ Q=( \sigma_{i,j}^2)_{i,j=1,\dots,P}$ 
	is strictly conditionally negative definite in the sense that
	\begin{align}\label{eq:Q negative}
	 z\cdot Q z <0 \quad \text{for all }  z\in \Rd[P]\setminus\{0\} \text{ with } \sum_{i=1}^P  z_i =0.
	\end{align}
\end{remark}

Incidentally, it seems that the crucial coercivity property~\textit{iii)}
of Definition~\ref{DefinitionAdmissibleSurfaceTensions} has not yet been verified for 
commonly used classes of surface tensions. 
One can easily generate instances of surface tensions which satisfy the 
triangle inequality~\eqref{TriangleInequalitySurfaceTensions} but violate 
this property and indeed lead to nucleation at triple junctions~\cite{cahn1991stability}. 
In contrast, the following lemma shows that the coercivity condition~\textit{iii)}
of Definition~\ref{DefinitionAdmissibleSurfaceTensions} holds for a certain class of surface tensions arising in models for grain boundary motion in polycrystalline materials. In this context, different phases correspond to regions with different crystal lattice orientations. The surface tension of an interface between two phases $i$ and $j$ is then often approximated as a function of the misorientation angle $\theta_i-\theta_j$ between the grains (i.\,e., the angular mismatch between the crystal lattice orientations). The Read-Shockley low-angle grain boundary formulas with high-angle saturation \cite{read1950dislocation,holm2001misorientation} for the interface energies take the form
\begin{equation}\label{eq:read-shockley sigma}
	\sigma_{i,j} := f\left(\min_{k\in \mathbb{Z}} \left|\theta_i - \theta_j -k\frac{\pi}2\right|\right),
\end{equation}
where the profile $f$ is given by
\begin{align}\label{eq:read-shockley f}
	f(\theta) = 
	\begin{cases}
		\frac{\theta}{\theta_\ast} \big(1-\log \big(\frac{\theta}{\theta_\ast}\big)\big),& 0\leq \theta \leq \theta_\ast\\
		1,& \theta_\ast<\theta \leq \pi/4.
	\end{cases}
\end{align}
Here $\theta_\ast \in (0,\pi/4)$ (typically, $\theta_\ast $ lies between $10^\circ$ and $30^\circ$) and we assumed for simplicity that the crystal lattice has cubic symmetry, as can be seen in Figure~\ref{fig:read-shockley circle}.

\begin{figure}
	\subcaptionbox{\label{fig:read-shockley f}}{
		\begin{tikzpicture}[scale=4.1]
			\draw[->] (-0.1,0)--(1,0) node [below] {$\theta$};
			\draw (pi/4,-0.01)--(pi/4,0.01) node [below] {$\pi/4$};
			\draw (pi/10,-0.01)--(pi/10,0.01) node [below] {$\theta_\ast$};
			\draw[->] (0,-0.1)--(0,1.2) node [left] {$f(\theta)$};
			\draw (-0.01,1)--(0.01,1) node [left] {$1$};
			
			\draw [thick, color=red, domain=0.00001:pi/10, samples=200, smooth]
			plot (\x, {1*\x/(pi/10)*(1-ln(\x/(pi/10)))});
			
			\draw [thick, color=red, domain={pi/10-0.001}:pi/4, samples=200, smooth]
			plot (\x, {1});
			
		\end{tikzpicture}	
	}
	\subcaptionbox{\label{fig:read-shockley circle}}{
		\begin{tikzpicture}[scale=.6]
		
			\foreach \s in {0, 1, 2, 3} {
				\draw [lightgray] (0,0) circle (\s + 0.5);
				\draw (0,0) circle (\s);
			}
		
			\foreach \theta in {0,...,31} {
				\draw [lightgray] (0,0) -- (\theta * 180 / 16:4);
			}
			
			\foreach \theta/\label/\direction in {
				0/0/right,
				1/{\pi/4}/{above right},
				2/{\pi/2}/above,
				3/{3\pi/4}/{above left},
				4/{\pi}/left,
				5/{5\pi/4}/{below left},
				7/{7\pi/4}/{below right},
				6/{3\pi/2}/below} {
				\draw (0,0) -- (\theta * 180 / 4:4.1);
				\node [fill=white] at (\theta * 180 / 4:4.2) [\direction] { $\label$};
			}
			
			\draw [style=thick] (0,0) circle (4);
			
			\foreach \k in {0,...,3}
			{
				\draw [thick, color=red, domain=0.01:pi/10, samples=200, smooth]
				plot (xy polar cs:angle={\x r+pi/2*\k r}, radius={4*(\x )/(pi/10)*(1-ln((\x)/(pi/10)))});
				
				\draw [thick, color=red, domain=0.00001:pi/10, samples=200, smooth]
				plot (xy polar cs:angle={-\x r+pi/2*\k r}, radius={4*(\x )/(pi/10)*(1-ln((\x)/(pi/10)))});
				
				\draw [thick, color=red, domain={pi/10-0.001}:{pi/2-pi/10+0.001}, samples=200, smooth]
				plot (xy polar cs:angle={\x r + pi/2*\k r}, radius={4});
			}
		\end{tikzpicture} 
	}
	\caption{Surface tension $\sigma=f(\theta)$ depending on misorientation angle $\theta$ 
	according to the Read-Shockley formula for low-angle grain boundaries with high-angle 
	saturation, and cubic symmetry. a) Graph of $f$ for small, positive misorientation 
	angle $\theta$. b) Graph of $f$ for all misorientation angles $\theta$.}
\end{figure}
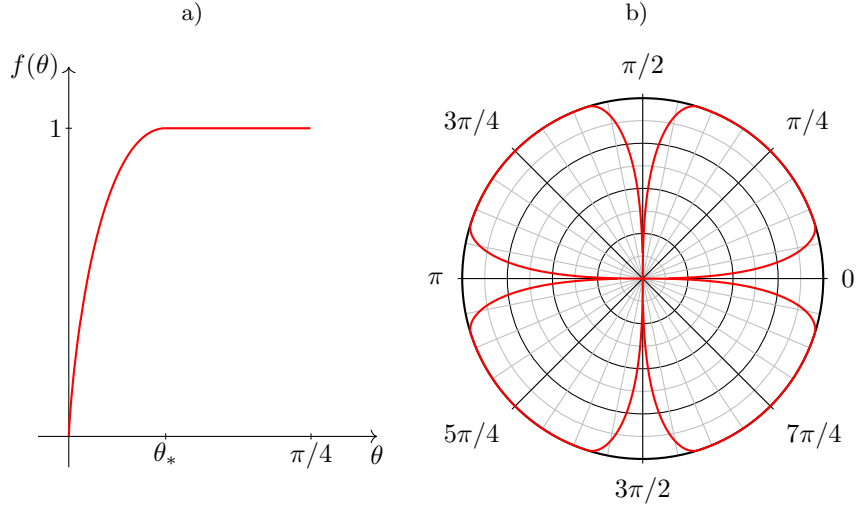

\begin{lemma}\label{lemma:read-shockley}
	Let $\theta_1,\ldots, \theta_P \in \Rd[]$ be given angles such that $\theta_i \neq \theta_j\mod \frac\pi2$ for $i\neq j$, and define the matrix of surface 
	tensions $\sigma=(\sigma_{i,j})$ by
	\eqref{eq:read-shockley sigma} with a function $f\colon[0,\frac\pi4] \to [0,1]$ such that the complex Fourier coefficients of the (evenly extended function) $f^2 \colon [-\frac\pi4, \frac\pi4] \to \Rd[]$ satisfy the negativity condition
	\begin{align}\label{eq:read shockley negative fourier}
		\widehat{(f^2)}_k \text{ is a negative real number for all } k \in \mathbb{Z}\setminus\{0\}.
	\end{align} 
	Then $\sigma$ is admissible in the sense of Definition \ref{DefinitionAdmissibleSurfaceTensions}.
	
	In particular, any matrix of surface tensions $\sigma$ given by the Read-Shockley formulas \emph{\eqref{eq:read-shockley sigma}--\eqref{eq:read-shockley f}} for any saturation angle $\theta_\ast \in (0,\pi/4)$ is admissible in the sense of Definition \ref{DefinitionAdmissibleSurfaceTensions}.
\end{lemma}

The simple proof of this lemma is inspired by the one of~\cite[Theorem 5.5]{EsedogluOtto15} 
where the triangle inequality \eqref{TriangleInequalitySurfaceTensions} and 
the $\ell^2$-embaddability of the matrix $\big(\sqrt{\sigma_{i,j}}\big)$ are derived.
Here, we prove the $\ell^2$-embeddability of $(\sigma_{i,j})$, which in particular 
implies their first conclusion, but appears to be unrelated to the latter.

\begin{definition}[Partitions with finite interface energy, cf.\ \cite{AmbrosioFuscoPallara}]
\label{DefinitionPartition}
Let $d\geq 2$, let $P\geq 2$ be an integer and let $\sigma\in\Rd[P\times P]$ be an admissible
matrix of surface tensions in the sense of Definition~\ref{DefinitionAdmissibleSurfaceTensions}. 
Let $(\Omega_1,\ldots,\Omega_P)$ be a partition of $\Rd$ in the sense that for 
$i,j =1, \ldots , P$ with $i\neq j$ we have $\Omega_i \subset \Rd$ and the sets 
$\Omega_i \cap \Omega_j$ and $\Rd \setminus \smash{\bigcup_{i=1}^P} \Omega_i$ 
have $\mathcal{L}^d$-measure zero. Let $\chi_i:=\chi_{\Omega_i}$ denote
the characteristic function of the $\mathcal{L}^d$-measurable set $\Omega_i$ for $i = 1,\ldots, P$.

We call $\chi=(\chi_1,\ldots,\chi_P)$, or equivalently $(\Omega_1,\dots,\Omega_P)$, 
a \emph{partition of $\Rd$ with finite interface energy} if the energy
\begin{align}\label{EnergyFunctionalBVPartition}
E[\chi] := \sum_{i,j=1,i\neq j}^P \sigma_{i,j}\int_{\Rd}\frac{1}{2}
\big(\mathrm{d}|\nabla\chi_i|+\mathrm{d}|\nabla\chi_j|-\mathrm{d}|\nabla(\chi_i{+}\chi_j)|\big)
\end{align}
is finite.
\end{definition}
Note that for a partition of $\Rd$ with finite interface energy, 
each $\Omega_i$ is a set of finite perimeter. By introducing the 
interfaces $I_{i,j}:=\partial^*\Omega_i\cap\partial^*\Omega_j$ as the intersection 
of the respective reduced boundaries, the energy of a partition $\chi$ can be 
rewritten in the equivalent form
\begin{align}
\label{eq:energy}
E[\chi] =  \sum_{i,j=1,i\neq j}^P \sigma_{i,j}\int_{I_{i,j}}1\,\mathrm{d}\mathcal{H}^{d-1}.
\end{align}
We next recall the notion of $\BV$ solutions to multiphase mean curvature flow as in \cite{LauxOtto, LauxOttoBrakke}.

\begin{definition}[BV solutions for multiphase mean curvature flow]\label{DefinitionBVSolution}
Let $d\geq 2$ and $P\geq 2$ be integers. Let $\sigma\in\Rd[P\times P]$ be an admissible matrix of surface tensions
in the sense of Definition~\ref{DefinitionAdmissibleSurfaceTensions}, and let $\TBV>0$ be a finite time horizon.
Let $\chi_0=(\chi_{0,1},\ldots,\chi_{0,P})$ be an initial partition of $\Rd$ with finite interface energy in the sense of Definition~\ref{DefinitionPartition}.

A measurable map
\begin{align*}
\chi=(\chi_1,\ldots,\chi_P)\colon \Rd \times [0,\TBV)\to\{0,1\}^P
\end{align*}
(respectively the corresponding tuple of sets $\smash{\Omega_i:=\bigcup_{t\in [0,\TBV)}\Omega_i(t){\times}\{t\}}$,
$\Omega_i(t):=\{\chi_i(t) {=}1\}$ for $i\in\{1,\ldots, P\}$ and~$t\in [0,\TBV)$)
is called a \emph{BV solution for multiphase mean curvature flow with initial data $\chi_0$}
if the following conditions are satisfied:
\begin{subequations}
\begin{itemize}[leftmargin=0.7cm]
\item[i)] (Partition with finite interface energy) For almost every $t\in [0,\TBV)$, $\chi(\cdot,t)$
is a partition of $\Rd$ with finite interface energy in the sense of Definition~\ref{DefinitionPartition} and
\begin{align}\label{GlobalEnergyBoundBVSolution}
\esssup_{t\in [0,\TBV)} E[\chi(\cdot,t)] 
= \esssup_{t\in [0,\TBV)} \sum_{i,j=1,i\neq j}^P \sigma_{i,j}\int_{I_{i,j}(t)}1\,\mathrm{d}\mathcal{H}^{d-1} < \infty,
\end{align}
where for all~$t\in [0,T]$ we denote by 
$I_{i,j}(t)=\partial^* \Omega_i(t) \cap \partial^* \Omega_j(t)$ for $i\neq j$ 
the interface between the phases~$\Omega_i(t)$ and~$\Omega_j(t)$.
\item[ii)] (Evolution equation) For all $i\in \{1,\ldots,P\}$, there exist normal velocities $V_i\in L^2(\Rd\times [0,\TBV),|\nabla\chi_i|\otimes\mathcal{L}^1)$
 in the sense that each $\chi_i$ satisfies the evolution equation
\begin{align}\label{EvolutionPhasesBVSolution}
\nonumber
&\int_{\Rd} \chi_i(\cdot,T) \varphi(\cdot,T) \dx
-\int_{\Rd} \chi_{0,i} \varphi(\cdot,0) \dx
\\&
=\int_0^T \int_{\Rd} V_i \varphi \dnablachii \dt
+ \int_0^T \int_{\Rd} \chi_i \partial_t \varphi \dx \dt
\end{align}
for almost every $T\in [0,\TBV)$ and all $\varphi\in C^\infty_{\mathrm{cpt}}(\Rd \times [0,\TBV))$.
Moreover, the (reflection) symmetry condition 
$V_i \smash{\frac{\nabla \chi_i}{|\nabla \chi_i|}}=V_j  \smash{\frac{\nabla \chi_j}{|\nabla \chi_j|}}$
holds $\mathcal{H}^{d-1}\otimes\mathcal{L}^1$-almost everywhere on 
$\bigcup_{t\in [0,\TBV)}I_{i,j}(t){\times}\{t\}$, $i\neq j$.
\item[iii)] (BV formulation of mean curvature) The normal velocities are given by the weak formulation of mean curvature in the sense that
\begin{align}\label{BVFormulationMeanCurvature}
\nonumber
&\sum_{i,j=1,i\neq j}^P \sigma_{i,j}\int_0^{T}\int_{I_{i,j}(t)}
V_i \frac{\nabla \chi_i}{|\nabla \chi_i|} \cdot \vec{B} 
\,\mathrm{d}\mathcal{H}^{d-1} \dt
\\&
= \sum_{i,j=1,i\neq j}^P \sigma_{i,j}\int_0^{T}\int_{I_{i,j}(t)} 
\bigg(\Id-\frac{\nabla \chi_i}{|\nabla \chi_i|}\otimes \frac{\nabla \chi_i}{|\nabla \chi_i|}\bigg) : \nabla \vec{B} 
\,\mathrm{d}\mathcal{H}^{d-1} \dt
\end{align}
holds for almost every $T\in [0,\TBV)$
and all $B\in C^\infty_{\mathrm{cpt}}(\Rd \times [0,\TBV);\Rd)$.
\item[iv)] (Energy dissipation inequality) The sharp energy dissipation inequality
\begin{align}\label{EnergyDissipationInequalityBVSolution}
E[\chi(\cdot, T)] + \sum_{i,j=1,i\neq j}^P \sigma_{i,j}\int_0^T\int_{I_{i,j}(t)} 
|V_i|^2 \,\mathrm{d}\mathcal{H}^{d-1}\dt
\leq E[\chi_0]
\end{align}
holds true for almost every $T\in [0,\TBV)$.
\end{itemize}
\end{subequations}

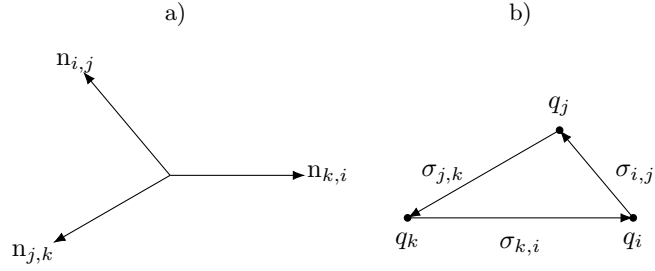
\begin{figure}
	\subcaptionbox{\label{fig:herring}}{
		\begin{tikzpicture}[scale=3]
		  	\tikzmath{\l = .6;};
		  	\draw[-Latex] (0,0) -- (0:\l);
		  	\node at ($(0:\l+.1)$) {$\vec{n}_{k,i}$};
		  	\draw[-Latex] (0,0) -- (130:\l);
		  	\node at ($(130:\l+.05)$) {$\vec{n}_{i,j}$};
		  	\draw[-Latex] (0,0) -- (210:\l);
		  	\node at ($(210:\l+.1)$) {$\vec{n}_{j,k}$};
		\end{tikzpicture}
	}
	\subcaptionbox{\label{fig:l2embedding}}{
		\begin{tikzpicture}[scale=5]
			\tikzmath{\l = .6;};
			\coordinate (A) at (0,0);
			\coordinate (B) at (30:\l);
			\draw [name path=A--B,color=white] (A) -- (B);
			\coordinate (C) at (0:\l);
			\coordinate (D) at ($(0:\l)+(130:\l)$);
			\draw [name path=C--D,color=white] (C) -- (D);
			\coordinate (E) at (intersection of A--B and C--D);
			
		  	\draw[-Latex] (0,0) -- node[label= below:{$\sigma_{k,i}$}] {} (0:\l);
		  	\draw[-Latex] (0:\l) --  node[label= right:{$\sigma_{i,j}$}] {}  (E);
		  	\draw[-Latex] (E) --node[label= left:{$\sigma_{j,k}$}] {}  (0,0);
		  	
		  	\filldraw (0,0) circle[radius=.25pt];
			\node[below, outer sep=2pt] (0,0) {$q_k$};
			\filldraw (\l,0) circle[radius=.25pt];
			\node[below, outer sep=2pt] at (\l,0) {$q_i$};
			\filldraw (E) circle[radius=.25pt];
			\node[above, outer sep=2pt] at (E) {$q_j$};
			
		\end{tikzpicture}
	}
	\caption{a) Normals $\vec{n}_{i,j}$, $\vec{n}_{j,k}$ and $\vec{n}_{k,i}$ satisfying 
	the balance-of-forces condition $ \sigma_{i,j}\vec{n}_{i,j} + \sigma_{j,k}\vec{n}_{j,k} + \sigma_{k,i}\vec{n}_{k,i}=0$. 
	b) Sketch of the points $q_i$, $q_j$ and $q_k$ of the $l^2$-embedding of $\sigma$.}
\end{figure}

The same definition can be used to define a BV solution for multiphase mean curvature flow on the closed time interval $[0,\TBV]$ 
for maps $\chi=(\chi_1,\ldots,\chi_P)\colon \Rd \times [0,\TBV]\to\{0,1\}^P$.
\end{definition}

Next, we give the definition of strong solutions to multiphase mean curvature flow. To this end, we first define a notion
of regular partitions and regular networks of interfaces (cf.\ \cite[Definitions 2.1, 2.7 and 4.7]{MantegazzaNovagaPludaSchulze}).

\begin{definition}[Regular partitions and networks of interfaces]
\label{DefinitionRegularPartition}
Let $d=2$, let $P\geq 2$ be an integer, 
and let $(\bar{\Omega}_1,\ldots,\bar{\Omega}_P)$ be a partition with finite interface energy of open subsets of $\Rd[2]$ such that $\partial^* \bar \Omega_i = \partial \bar \Omega_i$.
Moreover, let $\bar\chi_i:=\chi_{\bar{\Omega}_i}$ denote
the characteristic function of the $\mathcal{L}^d$-measurable set ${\bar{\Omega}}_i$, and
let ${\bar{I}}_{i,j}:=\partial \bar{\Omega}_i\cap\partial \bar{\Omega}_j$ denote the respective interfaces for $i\neq j$.

We call $\bar\chi=(\bar\chi_1,\ldots,\bar\chi_P)$, or equivalently $(\bar \Omega_1,\ldots, \bar \Omega_P)$,
a \emph{regular partition of~$\Rd[2]$} and $\mathcal{I}:=\smash{\bigcup_{i\neq j}{\bar{I}}_{i,j}}$ a 
\emph{regular network of interfaces in~$\Rd[2]$} if the
following properties are satisfied:
\begin{subequations}
\begin{itemize}[leftmargin=0.7cm]
\item[i)] (Regularity) Each interface~${\bar{I}}_{i,j}$ is a one-dimensional manifold 
with boundary of class $C^5$. The interior of each interface is embedded. Moreover,
each interface~${\bar{I}}_{i,j}$ is compact and consists of finitely many connected components.
\item[ii)] (Multi-points are triple junctions) Only different interfaces may intersect,
and if this is the case then only at their boundary. Moreover,
at each intersection point exactly three interfaces meet. In other words, all multi-points 
of the network of interfaces are triple junctions.
\item[iii)] (Balance-of-forces condition) Let $p\in \Rd[2]$ be
a triple junction present in the network. Assume for notational
concreteness that at the triple junction $p$, the three phases $\bar{\Omega}_i$, $\bar\Omega_j$ 
and $\bar{\Omega}_k$ meet. Then, the balance-of-forces condition.
\begin{align}\label{HerringAngleCondition}
\sigma_{i,j}\bar{\vec{n}}_{i,j}(p) + \sigma_{j,k}\bar{\vec{n}}_{j,k}(p) + \sigma_{k,i}\bar{\vec{n}}_{k,i}(p) = 0
\end{align}
has to be satisfied, see Figure~\ref{fig:herring}. Here, $\bar{\vec{n}}_{i,j}(x)$ 
denotes the unit normal vector of the interface
${\bar{I}}_{i,j}$ at $x\in \bar I_{i,j}$ pointing from phase $\bar{\Omega}_i$ towards phase $\bar{\Omega}_j$.
\item[iv)]  (Second- and third-order compatibility) We additionally have the second-order compatiblity condition
\begin{align}\label{SecondOrderCompDef}
	\sigma_{i,j} H_{i,j}(p) + \sigma_{j,k} H_{j,k}(p)+ \sigma_{k,i} H_{k,i}(p)=0 
\end{align}
for the scalar mean curvatures $H_{i,j} := -\nabla^{\mathrm{tan}} \cdot\bar {\vec{n}}_{i,j}$, 
which is equivalent to the existence of a ``velocity'' vector $B(p) \in \Rd[2]$ 
with $H_{l,m}(p) = \bar{\vec{n}}_{l,m}(p) \cdot B(p)$ for all distinct $l,m\in\{i,j,k\}$.
For the choice of tangent vectors $\bar{\tau}_{i,j}:= J^{-1}\bar{\vec{n}}_{i,j}$ 
with $\smash{J:= \big(\begin{smallmatrix} 0 & -1 \\ 1 & 0	\end{smallmatrix}\big)}$, 
we furthermore have the third-order condition
\begin{align}\label{ThirdOrderCompDef}
 \begin{split}
	\bar{\tau}_{i,j}(p)\cdot\left(H_{i,j} B +  \nabla H_{i,j}  \right)(p) 
	& = \bar{\tau}_{j,k}(p)\cdot\left(H_{j,k} B +  \nabla H_{j,k}  \right)(p)\\
	& = \bar{\tau}_{k,i}(p)\cdot\left(H_{k,i} B +  \nabla H_{k,i}  \right)(p).
 \end{split}
\end{align}
Here, we slightly abuse notation by denoting the tangential derivative of $H_{i,j}$ 
in direction $\bar{\tau}_{i,j}$ by $\bar{\tau}_{i,j} \cdot \nabla H_{i,j}$.

\end{itemize}
Let $\sigma\in\Rd[P\times P]$ be an admissible
matrix of surface tensions in the sense of Definition~\ref{DefinitionAdmissibleSurfaceTensions}.
We call $\bar\chi=(\bar\chi_1,\ldots,\bar\chi_P)$, or equivalently $(\bar \Omega_1,\ldots, \bar \Omega_P)$, 
a \emph{regular partition of $\Rd[2]$ with finite interface energy}
if $\bar\chi$ satisfies
\begin{align}\label{EnergyFunctionalRegularPartition}
E[\bar\chi] := \sum_{i,j=1,i\neq j}^P \sigma_{i,j}\int_{{\bar{I}}_{i,j}}1\,\mathrm{d}S < \infty
\end{align}
in addition to the previous requirements.
\end{subequations}
\end{definition}

Interpreting the triple junction as a free boundary of the interfaces, 
the identities~\eqref{SecondOrderCompDef} and~\eqref{ThirdOrderCompDef} can be 
viewed as parabolic compatibility conditions:
They arise from differentiating in time the zero-th order condition (that is, $p$ 
being the common endpoint of $\bar{I}_{i,j}$, $\bar{I}_{j,k}$, and $\bar{I}_{k,i}$) and the first-order condition~\eqref{HerringAngleCondition} (that is, the contact angle condition), respectively.
Keeping in mind parabolic scaling, the condition~\eqref{SecondOrderCompDef} is indeed second order, 
while~\eqref{ThirdOrderCompDef} is third order.

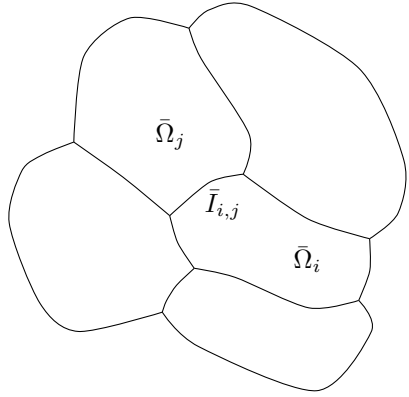
\begin{figure}
		\begin{tikzpicture}[scale=.8]

		\draw plot [smooth] coordinates { (0:3)(9.5:2)(50:1.4)};
		\draw plot [smooth] coordinates { (50:1.4)(60:2)(86.15:3)(90:3.5)};
		\draw plot [smooth] coordinates { (50:1.4)(70:1) (130:0.5)};
		\draw plot [smooth] coordinates { (280:.5) (200:.2) (130:0.5)};
		\draw plot [smooth] coordinates {(130:0.5)(137:1.5)(140:2.5)};

		\draw plot [smooth] coordinates {(140:2.5)(120:3.5)(105:3.8)(90:3.5)};
		\draw plot [smooth] coordinates {(0:3)(7:3.5)(20:3.8)(45:4)(75:4)(85:3.82)(90:3.5)};
		\draw plot [smooth] coordinates {(140:2.5)(150:2.8)(158:3)(175:3)(203.5:2.7)(220:2.4)(250:1.3)};
	  \draw plot [smooth] coordinates {(250:1.3)(250:1.1)(255:0.8)(280:.5)};
		\draw plot [smooth] coordinates {(250:1.3)(275:1.8)(310:3.3)(333:3.4)(340:3)};
		\draw plot [smooth] coordinates {(280:.5)(320:1)(330:2.3)(340:3)};
	  \draw plot [smooth] coordinates {(340:3)(350:3.05)(0:3)};
				
		\node at (350:2){${\bar{\Omega}}_i$};
		
		\node at (100:1.75){${\bar{\Omega}}_j$};
		
		\node at (45:.8){${\bar{I}}_{i,j}$};
		
	\end{tikzpicture}
	\caption{Sketch of a regular partition of the plane and the corresponding regular network. \label{fig:partition} }
\end{figure}

We say that a regular partition along with its associated regular network of interfaces evolves smoothly if no topological changes occur in the sense of the following definition:

\begin{definition}[Smoothly evolving partitions and smoothly evolving networks of interfaces]
\label{DefinitionSmoothlyEvolvingPartition}
Let $d=2$, let $P\geq 2$ be an integer and let
$\bar\chi_0=(\bar\chi_1^0,\ldots,\bar\chi_P^0)$ be a regular partition of~$\Rd[2]$ with
a regular network of interfaces $\mathcal{I}_0=\smash{\bigcup_{i\neq j}{\bar{I}}_{i,j}^{\,0}}$ in the sense of 
Definition~\ref{DefinitionRegularPartition}. Let $T > 0$,
and consider
$\smash{\bar\Omega_i:=\bigcup_{t\in [0,T]}\bar\Omega_i(t){\times}\{t\}}$, $i\in\{1,\ldots,P\}$,
so that for all $t\in [0,T]$ the family $(\bar \Omega_1(t),\ldots,\bar \Omega_P(t))$ 
is a regular partition of $\Rd[2]$ in the sense of Definition~\ref{DefinitionRegularPartition}.
For each~$i\in\{1,\ldots,P\}$ let $\bar\chi_i\colon \Rd[2]{\times}[0,T]\to\Rd[2]$
be the characteristic function of~$\bar\Omega_i$, and for each pair $i\neq j$ with $i,j\in\{1,\ldots,P\}$
and all~$t\in [0,T]$ define the interfaces ${\bar{I}}_{i,j}(t):=\partial \bar{\Omega}_i(t)\cap\partial \bar{\Omega}_j(t)$.

We say that $\bar\chi=(\bar\chi_1,\ldots,\bar\chi_P)$, 
or equivalently $(\bar{\Omega}_1,\ldots,\bar{\Omega}_P)$, 
is a \emph{smoothly evolving regular partition of $\Rd[2]{\times}[0,T]$}
and $\mathcal{I}:=\bigcup_{i,j\in\{1,\ldots,P\},i\neq j}\bar{I}_{i,j}$ 
is a \emph{smoothly evolving regular network 
of interfaces in $\Rd[2]{\times}[0,T]$}, where $\bar{I}_{i,j}:=\bigcup_{t\in[0,T]}\bar{I}_{i,j}(t){\times}\{t\}$
for all $i,j\in\{1,\ldots,P\}$ with~$i\neq j$, if there exists a time-dependent family of diffeomorphisms
\begin{align*}
\psi^t\colon \Rd[2] \to \Rd[2],\quad t\in [0,T],
\end{align*}
with the following properties:
\begin{itemize}[leftmargin=0.7cm]
	\item[i)] $\psi^0=\mathrm{Id}$, $\bar\chi_i(t) = \bar\chi_i^0\circ\left(\psi^t\right)^{-1}$
						and ${\bar{I}}_{i,j}(t)=\psi^t({\bar{I}}_{i,j}^{\,0})$ for all~$i,j\in\{1,\ldots,P\}$
						with~$i\neq j$ and all $t\in [0,T]$,
	\item[ii)] for all~$i,j\in\{1,\ldots,P\}$ with $i\neq j$, the map
						 $$\psi_{i,j}\colon \bar I_{i,j}^{\,0}{\times} [0,T] \to \bar I_{i,j},\quad
						 (x,t) \to (\psi^t(x),t)$$ is a diffeomorphism of 
						 class $(C^0_tC^5_x \cap C^1_tC^3_x)(\bar I_{i,j}^{\,0}{\times} [0,T])$.
\end{itemize}
\end{definition}

We have everything in place to proceed with the definition of strong solutions for multiphase
mean curvature flow.

\begin{definition}[Strong solution for multiphase mean curvature flow]
\label{DefinitionStrongSolution}
Let $d=2$, $P\geq 2$ be an integer, $\sigma\in\Rd[P\times P]$ be an admissible matrix of surface tensions
in the sense of Definition~\ref{DefinitionAdmissibleSurfaceTensions}, and let $\Tstrong>0$ be a finite time horizon.
Let $\bar\chi_0=(\bar\chi_1^0,\ldots,\bar\chi_P^0)$ be an initial regular partition of $\Rd[2]$ with finite interface energy
in the sense of Definition~\ref{DefinitionRegularPartition}.

A measurable map
\begin{align*}
\bar\chi=(\bar\chi_1,\ldots,\bar\chi_P)\colon \Rd\times [0,\Tstrong)\to\{0,1\}^P,
\end{align*}
(respectively, the corresponding tuple of sets $\smash{\bar\Omega_i:=
\bigcup_{t\in [0,\Tstrong)}\bar\Omega_i(t){\times}\{t\}}$,
$\bar\Omega_i(t):=\{\bar\chi_i(t) {=}1\}$ for $i\in\{1,\ldots, P\}$ and $t \in [0,\Tstrong)$)
is called a \emph{strong solution for multiphase mean curvature flow with initial data $\bar\chi_0$}
if for all $T\in [0,\Tstrong)$ it is a \emph{strong solution for multiphase mean curvature flow 
on $[0,T]$} in the following sense:
\begin{subequations}
\begin{itemize}[leftmargin=0.7cm]
\item[i)] (Smoothly evolving regular partition with finite interface energy) 
The map~$\bar\chi$ is a smoothly evolving regular partition of $\Rd[2]{\times}[0,T]$
and $\mathcal{I}:=\bigcup_{i,j\in\{1,\ldots,P\},i\neq j}\bar{I}_{i,j}$ 
is a smoothly evolving regular network of interfaces in 
$\Rd[2]{\times}[0,T]$ in the sense of Definition~\ref{DefinitionSmoothlyEvolvingPartition}. In particular, 
for every ${t}\in [0,T]$, $\bar\chi(\cdot,t)$ is a regular partition of $\Rd[2]$
and $\smash{\bigcup_{i\neq j}\bar{I}_{i,j}(t)}$ is a regular network of interfaces in $\Rd[2]$
in the sense of Definition~\ref{DefinitionRegularPartition} such that
\begin{align}\label{GlobalEnergyBoundStrongSolution}
\sup_{t\in [0,T]} E[\bar\chi(\cdot,{t})] 
= \sup_{t\in [0,T]} \sum_{i,j=1,i\neq j}^P 
\sigma_{i,j}\int_{{\bar{I}}_{i,j}(t)}1\,\mathrm{d}S < \infty.
\end{align}
\item[ii)] (Evolution by mean curvature) For $i,j =1,\ldots, P$ with 
$i\neq j$ and $(x,t)\in {\bar{I}}_{i,j}$ let $\bar V_{i,j}(x,t)$ denote 
the normal speed of the interface at the point $x\in {\bar{I}}_{i,j}(t)$, i.e., 
$\bar V_{i,j}(x,t):=(\bar{\vec{n}}_{i,j}(x,t),0)\cdot \partial_t\psi_{i,j}(y,t)$
at $y=(\psi^t)^{-1}(x)\in {\bar{I}}_{i,j}(0)$, where $\psi_{i,j}$ and $\psi^t$ 
are the maps from Definition~\ref{DefinitionSmoothlyEvolvingPartition}.
Denoting by $\vec{H}_{i,j}(x,t)$ the mean curvature vector of ${\bar{I}}_{i,j}(t)$
at~$x\in\bar I_{i,j}(t)$, we then assume that the interfaces ${\bar{I}}_{i,j}$ evolve by mean curvature in the sense
\begin{align}\label{MotionByMeanCurvature}
\bar V_{i,j}(x,t)\bar{\vec{n}}_{i,j}(x,t) = \vec{H}_{i,j}(x,t),
\quad\text{for all } t\in [0,T],\, x\in {\bar{I}}_{i,j}(t).
\end{align}
	\item[iii)] (Initial conditions) We have $\bar \chi_i(x,0) = \bar \chi_{0,i}(x)$ for all points $x\in \Rd$ and 
	each phase $i\in\{1,\ldots, P\}$.
\end{itemize}
\end{subequations}
\end{definition}

\subsection{Relative entropy inequality}
The key ingredient for the proof of Theorem~\ref{TheoremUniqueness} is the derivation of a Gronwall-type inequality for the tilt-excess-like error functional \eqref{DefinitionRelativeEntropy}: We aim to derive an estimate of the form
\begin{align}\label{eq:Gronwall}
E[\chi|\xi](T) \leq E[\chi|\xi](0) +C(\xi)\int_0^TE[\chi|\xi](t)\dt
\end{align}
for almost all admissible times $T \geq 0$
from which one may infer the desired stability estimate \eqref{StabilityEstimate} 
by an application of Gronwall's lemma. The weak-strong uniqueness principle then 
follows by means of the coercivity properties of the relative entropy error functional 
\eqref{DefinitionRelativeEntropy} and a subsequent estimate for $E_{\mathrm{volume}}[\chi|\bar \chi]$, 
see Proposition~\ref{PropositionUniquenessConditional}. The following result contains the first key step in the 
derivation of the Gronwall-type inequality \eqref{eq:Gronwall}; it is valid for general vector fields 
$\xi_i$ and $B$ with sufficient smoothness (not just for gradient flow calibrations).

\begin{proposition}[Relative entropy inequality]
\label{PropositionRelativeEntropyInequality}
Let $d\geq 2$, $P\geq 2$ be integers, and
let $\sigma\in\Rd[P\times P]$ be an admissible matrix of surface tensions
in the sense of Definition~\ref{DefinitionAdmissibleSurfaceTensions}. 
Let $\chi=(\chi_1,\ldots,\chi_P)$ be a BV solution of multiphase mean curvature flow in
the sense of Definition~\ref{DefinitionBVSolution} on some time interval $[0,T']$ with $T'>0$.
For $i,j = 1,\ldots,P$ with $i\neq j$ we denote by
\begin{align}\label{UnitNormalsBVSolution}
\vec{n}_{i,j}&:=\frac{\nabla \chi_{j}}{|\nabla \chi_{j}|}
=-\frac{\nabla \chi_i}{|\nabla \chi_i|},
\quad
\mathcal{H}^{d-1}\text{-a.e.\ on }I_{i,j},
\end{align}
the (measure-theoretic) unit normal vector of the interface $I_{i,j}$ pointing 
from the $i$-th to the $j$-th phase of the $\BV$ solution. Moreover, let
\begin{align}\label{NormalVelocitiesBVSolution}
V_{i,j} := V_{i} = -V_{j}, \quad
\mathcal{H}^{d-1}\text{-a.e.\ on }I_{i,j}.
\end{align}

Let $(\xi_{i,j})_{i\neq j\in\{1,\ldots,P\}}$ and $(\xi_i)_{i=1,\ldots,P}$ 
be families of compactly supported vector fields such that
\begin{align*}
\xi_{i,j},\xi_i \in C^1([0,T'];C^0_{\mathrm{cpt}}(\Rd[d];\Rd[d]))
\cap C^0([0,T'];C^1_{\mathrm{cpt}}(\Rd[d];\Rd[d]))
\end{align*}
as well as $\sigma_{i,j}\xi_{i,j}=\xi_i-\xi_j$ for all $i\neq j$.
Let
\begin{align*}
B\in C^0([0,T'];C^1_{\mathrm{cpt}}(\Rd[d];\Rd[d]))
\end{align*}
be an arbitrary compactly supported vector field.
Consistently with \eqref{DefinitionRelativeEntropy}, define the interface error control
\begin{align}
\label{DefinitionRelativeEntropyFunctional}
E[\chi|\xi](t)
:=
\sum_{i,j=1,i\neq j}^P \sigma_{i,j} \int_{I_{i,j}(t)} 
1-\xi_{i,j}(\cdot,t) \cdot \vec{n}_{i,j}(\cdot,t) \,\mathrm{d}\mathcal{H}^{d-1}.
\end{align}

Then the interface error control
is subject to the estimate
\begin{align}
\nonumber
&E[\chi|\xi](T)
\\&\nonumber
+\sum_{i,j=1,i\neq j}^{P} \frac{\sigma_{i,j}}{2}
\int_0^{T}\int_{I_{i,j}(t)}
|V_{i,j}{+}\nabla\cdot\xi_{i,j}|^2
+|V_{i,j}\vec{n}_{i,j}{-}(B\cdot\xi_{i,j})\xi_{i,j}|^2
\,\mathrm{d}\mathcal{H}^{d-1}\dt
\\&\label{RelativeEntropyInequality}
\leq E[\chi|\xi](0) + R_{\mathrm{dt}} + R_{\mathrm{dissip}}
\end{align}
for almost every $T \in [0,T']$. Here, we made use of the abbreviations
\begin{align*}
R_{\mathrm{dt}}
&:= -\sum_{i,j=1,i\neq j}^{P}\sigma_{i,j}
\int_0^T\int_{I_{i,j}(t)}
\frac{1}{2}\big(\partial_t|\xi_{i,j}|^2{+}(B\cdot\nabla)|\xi_{i,j}|^2\big)
\,\mathrm{d}\mathcal{H}^{d-1}\dt \notag
\\&~~~~
-\sum_{i,j=1,i\neq j}^{P}\sigma_{i,j}
\int_0^T\int_{I_{i,j}(t)}
\big(\partial_t\xi_{i,j}{+}(B\cdot\nabla)\xi_{i,j}{+}(\nabla B)^\mathsf{T}\xi_{i,j}\big)
\cdot(\vec{n}_{i,j}{-}\xi_{i,j})\,\mathrm{d}\mathcal{H}^{d-1}\dt,\\
R_{\mathrm{dissip}} &:=
\sum_{i,j=1,i\neq j}^{P}\sigma_{i,j}
\int_0^T\int_{I_{i,j}(t)}
\frac{1}{2}|(\nabla\cdot\xi_{i,j})+B\cdot\xi_{i,j}|^2
\,\mathrm{d}\mathcal{H}^{d-1}\dt
\\&~~~~
-\sum_{i,j=1,i\neq j}^{P}\sigma_{i,j}
\int_0^T\int_{I_{i,j}(t)}
\frac{1}{2}|B\cdot\xi_{i,j}|^2(1-|\xi_{i,j}|^2)
\,\mathrm{d}\mathcal{H}^{d-1}\dt
\\&~~~~
-\sum_{i,j=1,i\neq j}^{P}\sigma_{i,j}
\int_0^T\int_{I_{i,j}(t)}
(1-\vec{n}_{i,j}\cdot\xi_{i,j})(\nabla\cdot\xi_{i,j})(B\cdot\xi_{i,j})
\,\mathrm{d}\mathcal{H}^{d-1}\dt
\\&~~~~
+\sum_{i,j=1,i\neq j}^{P}\sigma_{i,j}
\int_0^T\int_{I_{i,j}(t)}
\Big((\mathrm{Id}{-}\xi_{i,j}\otimes\xi_{i,j})B\Big)\cdot(V_{i,j}{+}\nabla\cdot\xi_{i,j})\vec{n}_{i,j}
\,\mathrm{d}\mathcal{H}^{d-1}\dt
\\&~~~~
+\sum_{i,j=1,i\neq j}^{P}\sigma_{i,j}
\int_0^T\int_{I_{i,j}(t)}
(1-\vec{n}_{i,j}\cdot\xi_{i,j})(\nabla\cdot B)
\,\mathrm{d}\mathcal{H}^{d-1}\dt
\\&~~~~
-\sum_{i,j=1,i\neq j}^{P}\sigma_{i,j}
\int_0^T\int_{I_{i,j}(t)}
(\vec{n}_{i,j}{-}\xi_{i,j})\otimes(\vec{n}_{i,j}{-}\xi_{i,j}):\nabla B
\,\mathrm{d}\mathcal{H}^{d-1}\dt.
\end{align*}
\end{proposition}

\subsection{Weak-strong uniqueness and stability of varifold-$\BV$ solutions}

A very recent solution concept by Stuvard and Tonegawa \cite{StuvardTonegawaMCF} combines the concept of Brakke's notion of varifold solutions with ideas from the notion of $\BV$ solutions. We shall refer to this new solution concept by the name ``varifold-$\BV$ solutions''; as we shall see in Theorem~\ref{TheoremStabilityVarifoldSolution} below, our arguments from the case of $\BV$ solutions can readily be extended to also prove weak-strong uniqueness and stability for varifold-$\BV$ solutions.
\begin{definition}
\label{DefinitionVarifoldBVsolution}
Let $\mathcal{V}=(\mathcal{V}_t)_{t\in [0,\infty)}$ be a measurable family of integral and rectifiable $(d{-}1)$-varifolds; denote by $(\mu_t)_{t \in [0,\infty)}$ the associated family of weight measures. Let $(\chi_1,\ldots,\chi_P):\Rd\times [0,\infty)\rightarrow \{0,1\}^P$ denote a family of indicator functions of sets with bounded perimeter subject to the properties in item~\textit{i)} of Definition~\ref{DefinitionBVSolution}.

We then call the tuple $(\mathcal{V},\chi)$ a \emph{varifold-BV solution} to multiphase mean curvature flow if the following conditions are satisfied:
\begin{subequations}
\begin{itemize}
\item[i)] For a.e.\ $t\in[0,\infty)$, there exists a generalized mean curvature vector $h(\cdot,t) \in L^2(\Rd,\mu_t)$ of $\mathcal{V}_t$ in the sense that
\begin{align}
\label{eq:meanCurvatureTonegawa}
- \int_{\Rd} h \cdot B \,\mathrm{d}\mu_t = \int_{\Rd{\times}\mathbf{G}(d,d{-}1)} 
\mathrm{Id}_{\mathbf{G}(d,d{-}1)} \colon \nabla B \,\mathrm{d}\mathcal{V}_t 
\end{align} 
for all $B\in C^\infty_{\mathrm{cpt}}(\Rd;\Rd)$. Here, as usual $\mathbf{G}(d,d{-}1)$ denotes the space of all $(d{-}1)$-dimensional linear subspaces of~$\Rd$.
\item[ii)] The family of varifolds $\mathcal{V}$ is a Brakke solution to multiphase mean curvature flow.
In particular, the global energy dissipation estimate
\begin{align}
\label{eq:dissipationTonegawa}
\mu_T(\Rd) + \int_{0}^{T} \int_{\Rd} |h|^2 \,\mathrm{d}\mu_t
\leq \mu_0(\Rd)
\end{align}
holds true for a.e.\ $T \in (0,\infty)$.

\item[iii)] For a.e.\ $t \in (0,\infty)$, the varifold $\mathcal{V}_t$ describes the interfaces $\partial^* \{\chi_i(\cdot,t)=1\}$ in the sense that
\begin{align}
\label{CompatibilityVarifoldChi}
\frac{1}{2}\sum_{i=1}^P |\nabla \chi_i(\cdot,t)| \leq \mu_t.
\end{align}

\item[iv)] The indicator functions $\chi_i$ evolve according to the mean curvature of $\mathcal{V}$ in the sense that
\begin{align}
\label{eq:evolutionPhaseTonegawa}
\partial_t \chi_i + h \cdot \nabla\chi_i = 0
\end{align}
holds distributionally for all $i\in\{1,\ldots,P\}$.
\end{itemize}
\end{subequations}
\end{definition}

Consider a calibrated flow with time horizon $T\in(0,\infty)$ with associated gradient-flow calibration $(\xi=(\xi_{i})_{i=1,\ldots,P},B)$. Let~$(\mathcal{V},\chi)$ be a varifold-$\BV$ solution to multiphase mean curvature flow in the sense of Stuvard and Tonegawa~\cite{StuvardTonegawaMCF}. The natural varifold solution analogue of our relative entropy functional \eqref{ErrorFunctional} is then given by
\begin{align}
\label{eq:relEntropyTonegawa}
E[\mathcal{V},\chi|\xi](t) := 2\mu_t(\Rd) - 
\sum_{i,j=1,\,i\neq j}^{P} \int_{I_{i,j}(t)} n_{i,j}\cdot\xi_{i,j}\dH.
\end{align}
Note that the varifold relative entropy controls the relative entropy for $\BV$ solutions: Denoting the Radon-Nikodym derivatives $\tfrac{d|\nabla \chi_i(\cdot,t)|}{d\mu_t}$ by $\rho_i(\cdot,t)$, we may write
\begin{align*}
E[\mathcal{V},\chi|\xi](t) = 2\int_{\mathbb{R}^d} 1-\frac{1}{2}\sum_{i=1}^P \rho_i(\cdot,t) \,\mathrm{d}\mu_t + E[\chi|\xi](t).
\end{align*}
Note that the first term on the right-hand side is nonnegative by \eqref{CompatibilityVarifoldChi} and provides control of the multiplicity of the varifold whenever it exceeds the multiplicity of the $\BV$ interfaces $\tfrac{1}{2}\sum_{i=1}^P |\nabla \chi_i(\cdot,t)|$.

By arguments mostly analogous to the case of $\BV$ solutions, we derive the following weak-strong uniqueness and stability result for varifold-$\BV$ solutions.
\begin{theorem}
\label{TheoremStabilityVarifoldSolution}
Let $\bar \chi$ be a strong solution to planar multiphase mean curvature flow and let $(\xi,B)$ be an associated gradient flow calibration.

Let $(\mathcal{V},\chi)$ be a varifold-$\BV$ solution to multiphase mean curvature flow in the sense of Definition~\ref{DefinitionVarifoldBVsolution}. The varifold relative entropy \eqref{eq:relEntropyTonegawa} then satisfies the stability estimate
\begin{align}
\label{StabilityVarifold}
E[\mathcal{V},\chi|\xi](T) \leq e^{Ct} E[\mathcal{V},\chi|\xi](0).
\end{align}
Furthermore, the stability estimate \eqref{eq:stabilityBulkError} for the bulk error holds
(with the $\BV$ relative entropy replaced by the varifold relative entropy).

In particular, if the initial data of the varifold-$\BV$ solution coincides with the strong solution in the sense that $\chi(\cdot,0)=\bar \chi(\cdot,0)$ and in the sense that $\mu_0=\tfrac{1}{2}\sum_{i=1}^P |\nabla \bar \chi(\cdot,0)|$, we have $\chi=\bar \chi$ and $\mu_t = \tfrac{1}{2}\sum_{i=1}^P |\nabla \bar \chi(\cdot,t)|$ for all $t$ prior to the first topology change.
\end{theorem}
Just like in the case of $\BV$ solutions, the stability estimate \eqref{StabilityVarifold} is valid in general ambient dimension, assuming that a gradient flow calibration exists.

\subsection{Structure of the paper}
The remaining part of the paper is organized as follows. 
Section~\ref{SectionOutlineStrategy} illustrates our strategy at the 
two most important examples, a smooth interface and a triple junction.

In Section~\ref{SectionWeakStrongUniqueness}, we prove the stability of evolving partitions for which a gradient flow calibration exists. To this aim, we exploit the properties of our gradient flow calibrations and the \emph{weak solution}: 
In Subsection~\ref{SectionDerivationRelativeEntropyInequality} we derive the relative entropy 
inequality in its full generality of Proposition~\ref{PropositionRelativeEntropyInequality}; 
and in Subsection~\ref{SectionProofMainResult}, we prove the quantitative inclusion principle, 
Theorem~\ref{TheoremUniqueness}.
The latter is upgraded to the conditional weak-strong uniqueness principle of Proposition~\ref{PropositionUniquenessConditional}
in Subsection~\ref{SectionProofConditionalWeakStrong}.

The subsequent three sections of the manuscript are devoted to the construction of our gradient flow calibrations given a \emph{strong solution}. First, we provide explicit constructions at a smooth interface (Section~\ref{SectionLocalConstructionsTwoPhase}) 
and at a triple junction (Section~\ref{SectionLocalConstructionsTriod}). These cases do not 
only serve as model examples but they also form the building blocks for our general construction 
in Section~\ref{sec:networkConstruction}. Therein, we glue together these local constructions 
to obtain a gradient flow calibration for regular networks, which establishes 
Theorem~\ref{TheoremExistenceCalibration}.

Section~\ref{SectionVolumeError} provides the construction of a family of transported weights given a strong solution. 
Finally, we prove in the last section Lemma~\ref{lemma:read-shockley}, which states that 
the Read-Shockley type surface tensions given by~\eqref{eq:read-shockley sigma} 
and~\eqref{eq:read-shockley f} are admissible.

\section{Outline of the strategy}
\label{SectionOutlineStrategy}

	\subsection{Idea of proof for a smooth interface}
	Let us give a brief idea of the proof, ignoring technical difficulties in the 
	simplest case of two phases sharing one single interface with $\sigma=1$.
	In that case, it is sufficient to describe the weak solution and the calibrated flow by a single phase $\Omega(t) \subset \Rd$, resp.\ $\bar \Omega(t) \subset \Rd$ for $t\in [0,T]$, the first being a set of finite perimeter and the second being a simply connected, smooth set.
	The relative entropy is then simply given by
	\begin{align*}
		E[\chi | \xi ](t)
		= \int_{\partial^* \Omega(t)} (1-\vec{n}\cdot \xi) \dH,
	\end{align*}
	which has the interpretation of an oriented excess of the weak solution with respect to the strong one. 
	Here $\chi=\chi(x,t)$ denotes the characteristic function of
	$\Omega=\Omega(t)$ and $\vec{n} = -\frac{\mathrm{d} \nabla \chi}{\mathrm{d} |\nabla \chi |}$ 
	denotes the (measure theoretic) exterior unit normal of $\partial^\ast \Omega(t)$.
	Furthermore, the vector field $\xi(\cdot,t)$ is an extension of the exterior unit normal 
	$ \vec{\bar n}(\cdot,t)$ of the calibrated, smooth interface $\bar{I}(t) := \partial \bar \Omega$
	(note that it is necessary to extend the vector field due to the fact that we evaluate it on the weak solution).
	
	In order to relate the extension $\xi$ to the evolution, we require it to be 
	transported along an extension $B$ of the velocity field of $\bar{I}$ in the sense that
	\begin{equation}\label{eq:xitransport2phase}
		\partial_t \xi =-\left( B\cdot \nabla \right)\xi -(\nabla B)^\mathsf{T} \xi + O\big(\dist(\cdot,\bar{I}\,)\big),
	\end{equation}
	which will help make the second term of $R_{\mathrm{dt}}$ small (see Proposition~\ref{PropositionRelativeEntropyInequality} for the definition).
	The extension for $B$ will be done such that it is constant in the ``normal'' $\xi$-direction, 
	meaning we have $(\xi \cdot \nabla ) B=0$, and such that the motion law 
	$\vec{\bar n} \cdot B = \bar V= H = - \nabla^{\tan} \cdot \vec{\bar n}$ is still approximately 
	true away from the interface in the sense that
	\begin{equation}\label{eq:Banddivxi2phase}
		\xi \cdot B = -\nabla \cdot \xi + O\big(\dist(\cdot,\bar{I}\,)\big),
	\end{equation}
	helping with the first term of $R_{\mathrm{dissip}}$.
	
	As we also want the functional $E[\chi | \xi]$ to ensure that 
	$\chi$ cannot be too far away from $\bar \chi$, we allow for $\xi$ to 
	be short, i.e., we have $|\xi|\leq 1$, and we ask this effect to be 
	transported by $B$ up to quadratic error
	\begin{equation}\label{eq:modxiquadratic}
		\partial_t |\xi|^2 + (B\cdot \nabla )	|\xi|^2 = O\big(\dist^2(\cdot,\bar{I}\,)\big),
	\end{equation}
	keeping the first term of $R_{\mathrm{dt}}$ small.

	In the present case of a single interface, the construction of these vector fields 
	is straightforward using the signed distance function $s = s(x,t)$ 
	to the smooth interface $\bar{I}$:
	We set
	\begin{align*}
		\xi := \zeta(s) \nabla s \quad \text{and} \quad  B := - (\Delta s )\xi,
	\end{align*}
	where $\zeta$ is a suitable cut-off function such that $\zeta(\tilde s) = 1-\tilde s^2$ close to $\tilde s=0$.
	Note that since $|\nabla s | =1$, this implies
	\begin{align}\label{EntropyControlsDistance}
		s^2 = 1-\zeta(s) \leq 1-\zeta(s)\,  \vec{n} \cdot \nabla s = 1- \vec{n} \cdot \xi
	\end{align} 
	in the region where $s$ is small, so that the relative entropy controls the (truncated) $L^2$ distance of the 
	weak solution and the calibrated flow.
	
	In the following heuristic derivation of the relative entropy inequality (from Proposition~\ref{PropositionRelativeEntropyInequality}) in the case of a single interface, we will use the abbreviation $\int_{\partial^\ast \Omega} \cdot := \int_{\partial^\ast \Omega(t)} \cdot \dH$ for the integral along 
	a time slice $\partial^* \Omega(t)$, $t\in [0,T]$, of the weak solution.
	Recall that $V$ denotes the 
	normal velocity of the weak solution characterized by the distributional equation 
	$\partial_t \chi = V\left|\nabla \chi\right|$, see \eqref{EvolutionPhasesBVSolution}, so that the sign convention is $V>0$ 
	for expanding $\Omega$.
	
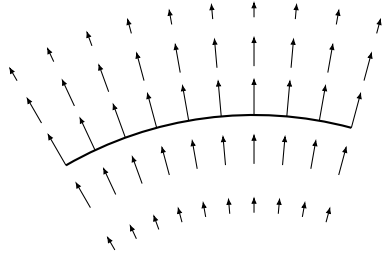
\begin{figure}
\begin{tikzpicture}[scale=5]
\draw[color=white] (0,0.45)--(0,-0.22);
		\begin{scope}
			\draw[thick,opacity=1,color=black] (0,0) arc (120:75:1);
		\end{scope}
			\foreach \i in {0,...,9}
			{
				\foreach \j in {-2,...,2}
				{
					\begin{scope}[shift={($(300:1)+(120 - 5*\i :1)+(120-5*\i : .13*\j)$) }]
						\draw[-{Latex[length=1mm]},opacity=1,color=black]  (0,0) -- ($(120-5*\i:.1-\j*\j*0.02+\j*\j*\j*\j*0.0015)$);
					\end{scope}
				}
			}
\end{tikzpicture}
\caption{Illustration of the vector field $\xi$ at a smooth interface~${\bar{I}}(t)$. 
The vector field~$\xi$ extends the unit normal vector field of~${\bar{I}}(t)$ 
by projection onto~${\bar{I}}(t)$ and multiplication with a cutoff function.\label{FigureCurveCalibrations}}
\end{figure}

	The optimal energy dissipation rate \eqref{EnergyDissipationInequalityBVSolution} and the definition \eqref{EvolutionPhasesBVSolution} of $V$ imply
	\begin{align*}
		\ddt E[\chi | \xi]
		= \ddt |\partial^*\Omega| -  \ddt \int_\Omega (\nabla \cdot \xi ) \,\mathrm{d}x
		\leq -\int_{\partial^\ast \Omega} V^2  - \int_{\partial^\ast \Omega} V\left(\nabla \cdot \xi \right) - \int_{\partial^\ast \Omega} \partial_t \xi \cdot \vec{n}.
	\end{align*}
	Testing the distributional mean curvature flow equation \eqref{BVFormulationMeanCurvature} with the extended velocity field $B$ gives
	\begin{align*}
		0=\int_{\partial^\ast \Omega} V\left(\vec{n}\cdot B \right) + \int_{\partial^\ast \Omega} \left(\Id - \vec{n}\otimes \vec{n}\right) \colon \nabla B.
	\end{align*}
	Adding these terms to the right-hand side of the previous inequality yields
	\begin{align*}
		\ddt E[\chi | \xi]
		& \leq -\int_{\partial^\ast \Omega}
		\left( 
			V^2 +V \left(\nabla \cdot \xi\right) -V\left(\vec{n}\cdot B\right)
		\right)
		+ \int_{\partial^\ast \Omega} (\nabla \cdot B)
		-\int_{\partial^\ast \Omega} \vec{n}\otimes \vec{n}  \colon \nabla  B\\
		& \quad - \int_{\partial^\ast \Omega} \partial_t \xi \cdot \vec{n}.
	\end{align*}
	We now write $B=\left(\xi\cdot B\right)\xi + \left(\Id - \xi \otimes \xi \right) B$, which we interpret as a decomposition of $B$ into ``normal'' and ``tangential'' parts.
	 Then we complete the squares, and add and subtract $\left( B\cdot \nabla \right)\xi + (\nabla B)^\mathsf{T} \xi$ to make the transport equation for $\xi$ appear. We obtain
	\begin{align}
	\label{eq:firstcompletionsquare}
	\notag
		\ddt E[\chi | \xi]
		\leq& -\frac12 \int_{\partial^\ast \Omega}
		\left( 
			\left( 
				V+\nabla \cdot \xi
			\right)^2
			+
			\left|
				V\vec{n} - \left(\xi\cdot B\right)\xi 
			\right|^2
		\right)
		\\\notag
		&+\frac12 \int_{\partial^\ast \Omega}
		\left(
			\left( 
				\nabla \cdot \xi
			\right)^2
			+
			|\xi|^2\left(
				\xi \cdot B
			\right)^2
		\right)
		+\int_{\partial^\ast \Omega} V\vec{n} \cdot  \left(\Id - \xi \otimes \xi \right) B 
		\\\notag
		&+\int_{\partial^\ast \Omega} \left(\nabla \cdot B\right)
		-\int_{\partial^\ast \Omega}   \vec{n}\otimes \vec{n} \colon  \nabla  B
		\\&\notag +\int_{\partial^\ast \Omega} \vec{n} \cdot \left( B \cdot \nabla \right) \xi
		+\int_{\partial^\ast \Omega}  \xi \cdot \left(\vec{n}\cdot \nabla\right) B
		\\
		&- \int_{\partial^\ast \Omega} \left(\partial_t \xi +\left( B\cdot \nabla \right)\xi + (\nabla B)^\mathsf{T} \xi\right)\cdot \vec{n},
	\end{align}
	where the second line collects precisely the terms left after completing the squares.
	
	By symmetry considerations, we have
	\begin{align*}
		0 
		&=  \int_\Omega \nabla \cdot 
		\left[
			\nabla \cdot
			\left(
				B\otimes \xi -\xi \otimes B
			\right)
		\right]
		\dx 
		= \int_{\partial^\ast \Omega} \left[ \nabla \cdot 
		\left(
			B\otimes \xi -\xi \otimes B
		\right) \right] \cdot \vec{n} \\
		& = \int_{\partial^\ast \Omega}
		\left[
			\left( \nabla \cdot \xi	\right) \vec{n} \cdot B
			- \left(\nabla \cdot B\right) \vec{n} \cdot \xi
			-\vec{n} \cdot \left( B\cdot \nabla \right)\xi
		\right],
	\end{align*}
	where for the second line we used $(\xi \cdot \nabla)B=0$.
	Now we use $|\xi | \leq 1$ to drop the prefactor $|\xi|^2$ of $(\xi\cdot B)^2$ 
	in the second right-hand side integral in inequality \eqref{eq:firstcompletionsquare}, 
	complete the square, add the above identity, and collect all terms involving $\nabla B$ to deduce
	\begin{align*}
		\ddt E[\chi | \xi ]
		& \leq -\frac12 \int_{\partial^\ast \Omega} 
		\left( 
		\left( 
		V+\nabla \cdot \xi
		\right)^2
		+
		\left|
		V\vec{n} - \left(\xi\cdot B\right)\xi 
		\right|^2
		\right)\\
		&\quad +\frac12 \int_{\partial^\ast \Omega}
		\left(
		\nabla \cdot \xi
		+
		\xi \cdot B
		\right)^2
		+ \int_{\partial^\ast \Omega} \left(\nabla \cdot \xi \right) \left(\vec{n}-\xi\right) \cdot B\\
		& \quad +\int_{\partial^\ast \Omega} V\vec{n} \cdot  \left(\Id - \xi \otimes \xi \right) B 
		 +\int_{\partial^\ast \Omega} (1-\vec{n}\cdot \xi ) \left(\nabla\cdot  B\right) \\
		& \quad -\int_{\partial^\ast \Omega}  \left(\vec{n}-\xi\right) \otimes  \left(\vec{n}-\xi\right)\colon  \nabla  B + \int_{\partial^\ast \Omega} \xi \otimes \xi \colon \nabla B
		\\&\quad- \int_{\partial^\ast \Omega} \left(\partial_t \xi +\left( B\cdot \nabla \right)\xi + (\nabla B)^\mathsf{T} \xi\right)\cdot \vec{n}.
	\end{align*}
	Once more, we decompose $B$ into ``tangential'' and ``normal'' components with respect to $\xi$ and manipulate the last integral to  finally arrive at the entropy dissipation inequality
	\begin{align*}
		\ddt E[\chi | \xi ]
		\leq& -\frac12 \int_{\partial^\ast \Omega}
		\left( 
		\left( 
		V+\nabla \cdot \xi
		\right)^2
		+
		\left|
		V\vec{n} - \left(\xi\cdot B\right)\xi 
		\right|^2
		\right)\\
		&+\frac12 \int_{\partial^\ast \Omega}
		\left(
		\nabla \cdot \xi
		+
		\xi \cdot B
		\right)^2
		+ \int_{\partial^\ast \Omega} \left(\nabla \cdot \xi \right) \left(\vec{n}\cdot \xi - 1\right) \left(\xi \cdot B\right)\\
		&+\int_{\partial^\ast \Omega} \left(\nabla \cdot \xi + V \right) \vec{n} \cdot  \left(\Id - \xi \otimes \xi \right) B \\
		&+\int_{\partial^\ast \Omega} (1-\vec{n}\cdot \xi ) \left(\nabla\cdot  B\right) 
		-\int_{\partial^\ast \Omega}   \left(\vec{n}-\xi\right) \otimes  \left(\vec{n}-\xi\right) \colon  \nabla  B\\
		&- \int_{\partial^\ast \Omega} \left(\partial_t \xi +\left( B\cdot \nabla \right)\xi + (\nabla B)^\mathsf{T} \xi\right)\cdot (\vec{n}-\xi)\\
		& - \int_{\partial^\ast \Omega} \left(\partial_t \xi +\left( B\cdot \nabla \right)\xi \right)\cdot \xi.
	\end{align*}
	
	 Now let us briefly argue term-by-term that the right-hand side can be controlled by the 
	relative entropy $E[\chi|\xi]$. Combining the resulting estimate
	\begin{align*}
	\ddt E[\chi | \xi ] \leq C E[\chi|\xi]
	\end{align*}
	with a Gronwall argument and a subsequent bound~\eqref{eq:stabilityBulkError} for the bulk error, this would yield 
	Theorem~\ref{MainResult} for $P=2$.
	
	Indeed, thanks to \eqref{eq:Banddivxi2phase}, 
	the first integrand in the second line is quadratic in $\dist(\cdot,{\bar{I}})$; thus, this integral is controlled by the relative entropy due to \eqref{EntropyControlsDistance}. The second integral of the second line is controlled by the relative entropy since $\nabla \xi$ and $B$ are bounded. To handle the third line, we use Cauchy-Schwarz and Young, and absorb $\int(\nabla \cdot \xi + V)^2$ in the first integral. The remaining 
	integral of $ | (\Id - \xi \otimes \xi) \vec{n}|^2 = |\vec{n} - (\xi \cdot \vec{n}) \xi|^2 
	\lesssim |\vec{n}-\xi|^2 + (1-\vec{n}\cdot \xi)^2$ is controlled by the relative entropy.
	Clearly, both terms in the fourth line are controlled by the relative entropy.
	 Finally, the integrals 
	in the fifth and sixth lines are of the order $\int_{\partial^\ast \Omega} \big(|\vec{n}-\xi|^2 + \dist^2(\cdot,\bar{I})\wedge 1\big)$ due to \eqref{eq:xitransport2phase} and the 
	factor $\vec{n}-\xi$, and \eqref{eq:modxiquadratic}, respectively.

\subsection{Idea of proof for a triple junction}

The second model case is given by a triple junction, say, with equal surface tensions. To illustrate the additional difficulties, we also present the idea of our proof in this case. However, we restrict ourselves to the case $d=2$.

We denote the phases of the \emph{weak} solution by 
$\Omega_1$, $\Omega_2$, and $\Omega_3$ with characteristic functions $\chi_1,$ $\chi_2$, and $\chi_3$.
To simplify notation, we identify indices if they are equivalent mod $3$, i.\,e., we define 
$\chi_{4}:=\chi_1$, $\chi_5:=\chi_2$, $\chi_0:=\chi_3$, and so on. 
Following the notation of Proposition \ref{PropositionRelativeEntropyInequality}, we denote the normal 
vector of the interface $I_{i,i+1}=\partial^* \Omega_i\cap \partial^* \Omega_{i+1}$ between phases $i$ and $i+1$ for $i=1,2,3$ in the weak solution by
\begin{align*}
\vec{n}_{i,i+1}&:=\frac{\mathrm{d}\nabla \chi_{i+1}}{\mathrm{d}|\nabla \chi_{i+1}|}
=-\frac{\mathrm{d}\nabla \chi_i}{\mathrm{d}|\nabla \chi_i|}\quad\quad\mathcal{H}^{1}\text{-a.\,e.\ on }\partial^* \Omega_i \cap \partial^* \Omega_{i+1}.
\end{align*}
The normal velocity of $I_{i,i+1}$, denoted by $V_i$, is characterized by the distributional 
identity $\partial_t \chi_i = V_i |\nabla \chi_i|$.
Furthermore, we will consider its restriction $V_{i,i+1}:= V_i|_{I_{i,i+1}}$ to the interface $I_{i,i+1}$ together with the symmetry condition $V_{i+1,i} := - V_{i,i+1}$.
As before, the corresponding quantities in the \emph{calibrated}
solution will be indicated by an additional bar on top of the quantity, i.e., for example 
$\bar \chi_i$ for the indicator function of the corresponding phases, $\bar{\vec{n}}_{i,i+1}$ 
for the corresponding normal, and so on.

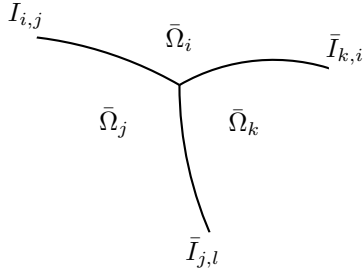
\begin{figure}
	\begin{tikzpicture}[scale=2.5]
		\begin{scope}
			\clip (0,0) circle [radius=.8];
			
			\draw[thick] (0,0) arc (60:90:2);
			\draw[thick] (0,0) arc (180:205:2);
			\draw[thick] (0,0) arc (120:70:1);
			
		\end{scope}
		
			\node at (90:.25){${\bar{\Omega}}_i$};
			\node at (210:.4){${\bar{\Omega}}_j$};
			\node at (330:.4){${\bar{\Omega}}_k$};
			
			\node at (12:.9){${\bar{I}}_{k,i}$};
			\node at (155:.9){${\bar{I}}_{i,j}$};
			\node at (278:.9){${\bar{I}}_{j,l}$};
	\end{tikzpicture}
	\caption{Sketch of a triple junction. \label{fig:triple}}
\end{figure}

The first key step is to construct extensions $\xi_{i,i+1}$, $i=1,2,3$, 
of the unit normal vector field $\bar{\vec{n}}_{i,i+1}$ of the \emph{calibrated} interfaces ${\bar{I}}_{i,i+1}$. As in the case of a single interface, the extensions 
$\xi_{i,i+1}$ and the velocity field $B$ are constructed to have the following properties:
\begin{subequations}
\label{CalibrationIdentitities}
\begin{itemize}[leftmargin=0.7cm]
\item The time evolution of the vector fields $\xi_{i,i+1}$ is approximately described 
by transport along the flow of the velocity field $B$. More precisely, for the vector field $B$ we have for $i=1,2,3$ that
\begin{align}
\quad\quad\quad~
\partial_t \xi_{i,i+1} = -(B\cdot \nabla)\xi_{i,i+1} - (\nabla B)^\mathsf{T}\xi_{i,i+1} 
+ O(\dist(\cdot,{\bar{I}}_{i,i+1})).
\end{align}
\item On each interface ${\bar{I}}_{i,i+1}$, $i=1,2,3$, of the calibrated solution, the normal part of the velocity field $B$ must satisfy $\bar{\vec{n}}_{i,i+1}\cdot B = \bar H_{i,i+1} := -\nabla^{\tan} \cdot \bar{\vec{n}}_{i,i+1}$, where $\bar H_{i,i+1}$ is the scalar mean curvature of ${\bar{I}}_{i,i+1}$. We strengthen this identity to approximately
hold even away from the interface, in form of
\begin{align}
\xi_{i,i+1}\cdot B = -\nabla \cdot \xi_{i,i+1} + O(\dist(\cdot,{\bar{I}}_{i,i+1}))
\quad\quad\text{for }i=1,2,3.
\end{align}
\item The vector fields $\xi_{i,i+1}$ have at most unit length $|\xi_{i,i+1}|\leq 1$.
\item The length of the vector fields $\xi_{i,i+1}$ is advected with the flow of $B$ to higher order
\begin{align}
\quad\quad\quad~
\partial_t |\xi_{i,i+1}|^2 = -(B\cdot \nabla)|\xi_{i,i+1}|^2 
+ O\big(\dist^2(\cdot,{\bar{I}}_{i,i+1})\big)\quad\quad\text{for }i=1,2,3.
\end{align}
\end{itemize}
\end{subequations}

The new aspect of a triple junction as opposed to a single interface is that one also has to extend the normal of an interface to locations where a different interface may be closer.
To this end, we turn to Herring's angle condition \eqref{HerringAngleCondition}, which in our case of equal surface tensions says 
that the three interfaces must meet at the triple junction to form equal angles of $120^\circ$ each, and 
require it to hold throughout the domain in the sense that
\begin{align}\label{eq:herringthrougout}
\sum_{i=1}^3 \xi_{i,i+1}(x,t)=0\quad\quad\text{for all }x,t.
\end{align}

Furthermore, note carefully that we only define a single extension $B$ of the velocity field, and that 
$B$ is not necessarily a normal vector field on each interface ${\bar{I}}_{i,i+1}$:
Indeed, we expect the triple junction $\tj(t)$ to move according to $\ddt p = B(p(t),t)$, so that not 
allowing for tangential components would pin the triple junction in space.
It turns out that in addition to Herring's angle condition, which we take to be of first order, we require higher-order compatibility conditions of the interfaces at the triple junction.
For instance, in part iv) of Definition~\ref{DefinitionRegularPartition} we have already seen that the second-order condition $H_{1,2}(p(t),t)+H_{2,3}(p(t),t)+H_{3,1}(p(t),t)=0$ is equivalent to the existence of the vector $B(p(t),t)$.

To construct the extensions $\xi_{i,i+1}$ of the normal vector fields 
$\vec{\bar n}_{i,i+1}$, $i=1,2,3$, we first partition space into six wedge-shaped sets around the triple junction:
Three contain one strong interface each, while the remaining three wedges lie entirely within a single phase, see Figure~\ref{fig:triple_wedges}.
On the mixed phase wedges, we first extend the corresponding normal by an expansion ansatz, see Figure~\ref{fig:triple_extension}, and then define the remaining vector 
fields to satisfy the identity \eqref{eq:herringthrougout} by $120^\circ$ rotations of the ansatz, see Figure~\ref{fig:triple_rotation}.
On the single phase wedges, we will interpolate between the competing definitions of the two adjacent mixed phase wedges.

\begin{figure}
	\centering
     \subcaptionbox{\label{fig:triple_wedges}}{
	  \centering
	  \begin{tikzpicture}[scale=3]
	  	\begin{scope}
			\clip (0,0) circle [radius=.8];
			\fill[pattern=north west lines,pattern color=blue,opacity=.2] (120:.8) -- (1,1.3) -- (1,-1)--  (300:.95) -- (120:.8);
			\fill[pattern=north east lines,pattern color=red,opacity=.2] (60:.8) -- (-1,1.3) -- (-1,-1)--  (240:.95) -- (60:.8);
			
			\foreach \i in {0,...,35}{
				\draw[opacity=.05] (-1,-\i*0.025) -- (1,-\i*0.025);
			};

		\draw[very thick,opacity=.6,color=red] (0,0) arc (60:90:2);
		\draw[very thick,opacity=.3] (0,0) arc (180:205:2);
		\draw[very thick,color=blue] (0,0) arc (120:70:1);

		\draw[color=red,opacity=.4] (0,0) -- (60:.8);
		\draw[color=blue,opacity=0.5] (0,0) -- (120:.8);
		\draw[opacity=.1]  (0,0) -- (180:1);
		\draw[color=red,opacity=.4] (0,0) -- (240:.95);
		\draw[color=blue,opacity=0.5] (0,0) -- (300:.95);
		\draw[opacity=.1] (0,0) -- (0:.85);
		
		\end{scope};
		
		\node at (10:.9){${\bar{I}}_{k,i}$};
		\node at (168:.9){${\bar{I}}_{i,j}$};
		\node at (290:.9){${\bar{I}}_{j,k}$};

		\node at (260:.9){$W_{j,k}$};
		\node at (330:.9){$W_{k}$};
		\node at (30:.9){$W_{k,i}$};
		\node at (90:.9){$W_{i}$};
		\node at (145:.9){$W_{i,j}$};
		\node at (210:.9){$W_{j}$};
	\end{tikzpicture}
     }
     \subcaptionbox{\label{fig:triple_extension}}{
	  \centering
	 	\begin{tikzpicture}[scale=3]
		\begin{scope}
			\clip (0,0) circle [radius=.8];
			\draw[thick,opacity=.6,color=red] (0,0) arc (60:90:2);
			\draw[thick,opacity=.3] (0,0) arc (180:205:2);
			\draw[thick,color=blue] (0,0) arc (120:70:1);
		
			\draw[color=red,opacity=.6] (0,0) -- (60:.8);
			\draw[color=blue] (0,0) -- (120:.8);
			\draw[opacity=.3]  (0,0) -- (180:1);
			\draw[color=red,opacity=.6] (0,0) -- (240:.95);
			\draw[color=blue] (0,0) -- (300:.95);
			\draw[opacity=.3] (0,0) -- (0:.85);
			
		\end{scope}
		
		
			\node at (40:.9){$W_{k,i}$};
			\node at (90:.9){$W_{i}$};
			\node at (135:.9){$W_{i,j}$};

			\foreach \i in {0,...,9}
			{
				\foreach \j in {-2,...,2}
				{
					\begin{scope}[shift={($(300:1)+(120 - 5*\i :1)+(120-5*\i : .13*\j)$) }]
						\draw[-{Latex[length=1mm]},color=blue]  (0,0) -- ($(120-5*\i - 10*\j - 1.7\j*\i:.1)$);
					\end{scope}
				}
			}

			\foreach \i in {0,...,8}
			{
				\foreach \j in {-2,...,2}
				{
					\begin{scope}[shift={($(240:2)+(60 + 2.7*\i :2) + (60+2.7*\i: .13*\j)$) }]
						\draw[-{Latex[length=1mm]},opacity=.6,color=red]  (0,0) -- ($(240+2.7*\i - 10*\j + 1.7*\j*\i  :.1)$);
					\end{scope}
				}
			}
	\end{tikzpicture}
     }
     \subcaptionbox{\label{fig:triple_rotation}}{
	  \centering
	 	\begin{tikzpicture}[scale=3]
	 	\begin{scope}
	 		\clip (0,0)  circle [radius=.8];
		
			\draw[thick,opacity=.6,color=red] (0,0) arc (60:90:2);
			\draw[thick,opacity=.3] (0,0) arc (180:205:2);
			\draw[thick,color=blue] (0,0) arc (120:70:1);
		
			\draw[color=red,opacity=.6] (0,0) -- (60:.8);
			\draw[color=blue] (0,0) -- (120:.8);
			\draw[opacity=.3]  (0,0) -- (180:1);
			\draw[color=red,opacity=.6] (0,0) -- (240:.95);
			\draw[color=blue] (0,0) -- (300:.95);
			\draw[opacity=.3] (0,0) -- (0:.85);
		\end{scope}
		
		
			\node at (40:.9){$W_{k,i}$};
			\node at (90:.9){$W_{i}$};
			\node at (135:.9){$W_{i,j}$};

			\foreach \i in {0,...,9}
			{
				\foreach \j in {-2,...,2}
				{
					\begin{scope}[shift={($(300:1)+(120 - 5*\i :1)+(120-5*\i : .13*\j)$) }]
						\draw[-{Latex[length=1mm]},color=blue]  (0,0) -- ($(120-5*\i - 10*\j - 1.7\j*\i:.1)$);
					\end{scope}
				}
			}

			\foreach \i in {0,...,8}
			{
				\foreach \j in {-2,...,2}
				{
					\begin{scope}[shift={($(240:2)+(60 + 2.7*\i :2) + (60+2.7*\i: .13*\j)$) }]
						\draw[-{Latex[length=1mm]},opacity=.6,color=red]  (0,0) -- ($(120+2.7*\i - 10*\j + 1.7*\j*\i  :.1)$);
					\end{scope}
				}
			}
	\end{tikzpicture}
     }
     \caption{a) The gray, horizontally hatched domain is $\mathbb{H}_{j,k}$, the region 
		hatched in red from the bottom left to the top right is $\mathbb{H}_{i,j}$, and 
		$\mathbb{H}_{k,i}$ is shown hatched in blue from the top left to the bottom right. 
		The simply hatched regions indicate the wedges $W_{i,j}$, $W_{j,k}$ and $W_{k,i}$ 
		containing the interfaces ${\bar{I}}_{i,j}$, ${\bar{I}}_{j,k}$ and ${\bar{I}}_{k,i}$. 
		The interpolation wedges $W_i$, $W_j$ and $W_k$ are shown as doubly hatched regions. 
		b) Sketch of the initial extensions of $\bar{\vec{n}}_{k,i}$ in blue on the right and 
		$\bar{\vec{n}}_{i,j}$ in red on the left, defined on $W_{k,i}$ and $W_{i,j}$, as well 
		as the two respective neighboring interpolation wedges. c) The image shows the vector 
		field $\bar{\vec{n}}_{k,i}$ (in blue on the right) and the rotated vector field 
		$R \bar{\vec{n}}_{i,j}$ (in red on the left), where $R$ is the clockwise rotation by $120^\circ$.}
        \label{fig:triple_xi}
\end{figure}
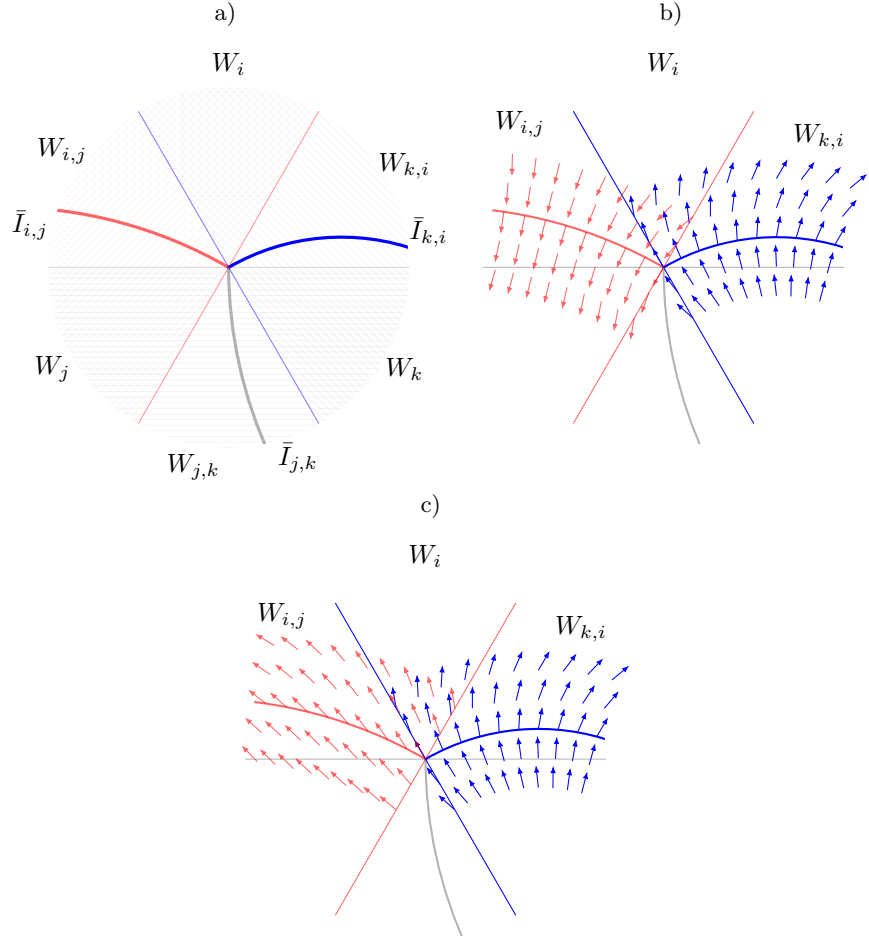

All rigorous discussions of compatibility will be deferred to Section \ref{SectionLocalConstructionsTriod}, and we will only describe the initial extension procedure here.
Let us fix $i=1,2,3$. In fact, it is more instructive to first extend the velocity field $B$ 
in the wedge-shaped neighborhood of the interface ${\bar{I}}_{i,i+1}$.
To this end, we recall $\bar \tau_{i,i+1} = J^{-1} \bar{\vec{n}}_{i,i+1}$ on $\bar{I}_{i,i+1}$ with $\smash{J = \big(\begin{smallmatrix} 0 & -1 \\ 1 & 0	\end{smallmatrix}\big)}$ from Definition~\ref{DefinitionRegularPartition} and use the extension ansatz
\begin{align*}
	B := \bar{H}_{i,i+1} \vec{\bar n}_{i,i+1} + \alpha_{i,i+1} \bar \tau_{i,i+1} +\beta_{i,i+1} s_{i,i+1} \bar \tau_{i,i+1},
\end{align*}
where $\vec{\bar{n}}_{i,i+1}$ and $\bar \tau_{i,i+1}$ are extended to be constant in the $\vec{\bar{n}}_{i,i+1}$-direction, $s_{i,i+1}$ is the signed distance function to ${\bar{I}}_{i,i+1}$ with the sign convention $\nabla s_{i,i+1} = \bar{\vec{n}}_{i,i+1}$, and $ \alpha_{i,i+1}$ and $\beta_{i,i+1}$ are still to be determined.
As $\ddt p (t)= B(p(t),t)$, it is reasonable that $\alpha_{i,i+1}(p(t),t) := \bar{\tau}_{i,i+1}(p(t),t)\cdot \ddt p(t)$ should be the tangential 
velocity of $p$ at the triple junction.
It turns out to be convenient to extend $\alpha_{i,i+1}$ along the interface ${\bar{I}}_{i,i+1}$
by means of the \emph{ordinary} differential equation $(\bar{\tau}_{i,i+1}\cdot\nabla)\alpha_{i,i+1}=H^2_{i,i+1}$.
In view of the third-order compatibility condition \ref{ThirdOrderCompDef}, the choice
$\beta_{i,i+1}(x,t): = (\bar \tau_{i,i+1} \cdot \nabla ) H_{i,i+1} + \alpha_{i,i+1} H_{i,i+1}$ 
for $x\in {\bar{I}}_{i,i+1}(t)$ is a good candidate to make $B$ independent of $i$.
To define $\alpha_{i,i+1}$ and $\beta_{i,i+1}$ away from the interface, we once again require them to be constant in $\bar{\vec{n}}_{i,i+1}$-direction.

To achieve the desired identitities \eqref{CalibrationIdentitities}, it turns out that one should construct the extension $\xi=\xi_{i,i+1}(x,t)$ of $\vec{\bar n}_{i,i+1}$ by an expansion ansatz of the form
\begin{equation}\label{xi3j_overview}
	\xi= \vec{\bar{n}} + \alpha s \bar \tau - \tfrac12 \alpha^2 s^2 \vec{\bar{n}} 
\end{equation}
where the functions $\alpha=\alpha_{i,i+1}(x,t)$ 
are as above and we dropped the indices $i,i+1$ for ease of notation.
Note that in particular $\xi_{i,i+1}=\bar{\vec{n}}_{i,i+1}$ on the interface ${\bar{I}}_{i,i+1}$ 
and that we allow for linear corrections of the tangential component as we move away from the interface, 
but only for quadratic corrections of the normal component of $\xi$.
In particular, this expansion ansatz will allow for zeroth and first order compatibility of the constructions of $\xi_{i,i+1}$ in the various wedges around the triple junction, facilitating a glueing procedure near the triple junction that preserves the identities \eqref{CalibrationIdentitities}.

We then measure the error between the weak solution $\chi$ and the calibrated solution $\bar \chi$ 
by means of the relative entropy functional
\begin{align*}
E[\chi|\xi](t)
:=\sum_{i=1}^3 \int_{I_{i,i+1}(t)} (1 - \vec{n}_{i,i+1} \cdot \xi_{i,i+1}) \,\mathrm{d}\mathcal{H}^{1}.
\end{align*}
Let us use the abbreviation 
$\sum_i = \sum_{i=1}^3$ for the summation over the three relevant indices.

As in the two-phase case, we only use two ingredients to evaluate the time evolution of the relative entropy: the energy dissipation inequality for the weak solution in the sharp form
\begin{align*}
\ddt \sum_i \int_{I_{i,i+1}} 1 \,\mathrm{d}\mathcal{H}^{1} \leq -\sum_{i=1}^3 \int_{I_{i,i+1}} V_{i,i+1}^2 \,\mathrm{d}\mathcal{H}^{1},
\end{align*}
and the weak formulation of the evolution equation of the indicator functions $\chi_i$
\begin{align*}
\ddt \int_{\Rd} \chi_i \varphi \dx = \int_{\partial^* \Omega_i} V_i \varphi \,\,\mathrm{d}\mathcal{H}^{1}  +\int_{\Rd} \chi_i \partial_t \varphi \dx
\end{align*}
for compactly supported, smooth $\varphi$.
In order to make use of the latter equation, we have to rewrite the contributions $\int_{I_{i,i+1}} \vec{n}_{i,i+1} \cdot \xi_{i,i+1}(x,t) $ as a volume integral. It turns out that the annihilation condition $\sum_i \xi_{i,i+1}(x,t)=0$ enables us to rewrite $\xi_{i,i+1}$ as
\begin{align}\label{overview frame xi}
\xi_{i,i+1} = \xi_i-\xi_{i+1}
\end{align}
by defining the vector field $\xi_{i}$ as $\xi_{i} :=\frac{1}{3} (\xi_{i,i+1}-\xi_{i-1,i})$.
Combining \eqref{overview frame xi} with the symmetry $\vec{n}_{i,i+1}=-\frac{\mathrm{d}\nabla \chi_i}{\mathrm{d}|\nabla \chi_i|}=\frac{\mathrm{d}\nabla \chi_{i+1}}{\mathrm{d}|\nabla \chi_{i+1}|}$ and the decomposition $\partial^* \Omega_{i} = I_{i-1,i} \cup I_{i,i+1}$, we rewrite the second term in the relative entropy as
\begin{align*}
	-\sum_i \int_{I_{i,i+1}}   \vec{n}_{i,i+1} \cdot \xi_{i,i+1} \,\mathrm{d}\mathcal{H}^{1}
	&=
	\sum_i \bigg(\int_{I_{i,i+1}} \xi_i\cdot \mathrm{d}\nabla \chi_i+\int_{I_{i,i+1}} \xi_{i+1}\cdot \mathrm{d} \nabla \chi_{i+1}\bigg)
	\\&
	=
	\sum_i \int_{\partial^* \Omega_i} \xi_i\cdot \mathrm{d}\nabla \chi_i
	\\&
	= -\sum_i \int_{\Rd} \chi_i (\nabla \cdot \xi_i) \dx.
\end{align*}
This indeed enables us to evaluate the time evolution of the relative entropy as
\begin{align*}
\ddt E[\chi|\xi]
\leq&
-\sum_i \int_{I_{i,i+1}} V_{i,i+1}^2\,\mathrm{d}\mathcal{H}^{1}
\\&
-\sum_i  \int_{\partial^* \Omega_i} V_i (\nabla \cdot \xi_i)\,\mathrm{d}\mathcal{H}^{1}
+\sum_i \int_{\partial^* \Omega_i} \partial_t \xi_i \cdot \mathrm{d} \nabla \chi_i \,\mathrm{d}\mathcal{H}^{1}.
\end{align*}
Arguing analogously to the previous computation in reverse order---that is, 
splitting the integrals into contributions $\partial^* \Omega_i\cap \partial^* \Omega_{i+1} = I_{i,i+1}$, 
using \eqref{overview frame xi} and the definitions of $\vec{n}_{i,i+1}$ and $V_{i,i+1}$---we obtain
\begin{align*}
\ddt E[\chi|\xi]
\leq
-\sum_i \int_{I_{i,i+1}} V_{i,i+1}^2\,\mathrm{d}\mathcal{H}^{1}
&-\sum_i  \int_{I_{i,i+1}} V_{i,i+1} (\nabla \cdot \xi_{i,i+1})\,\mathrm{d}\mathcal{H}^{1}
\\&
-\sum_i \int_{I_{i,i+1}} \partial_t \xi_{i,i+1} \cdot \vec{n}_{i,i+1}\,\mathrm{d}\mathcal{H}^{1}.
\end{align*}

Now we proceed as in the two-phase case in the previous section: The $\BV$ formulation of 
mean curvature flow in this three-phase setting reads
\begin{align*}
&\sum_i \int_{I_{i,i+1}} V_{i,i+1} \vec{n}_{i,i+1} \cdot B \,\mathrm{d}\mathcal{H}^{1}
=- \sum_i \int_{I_{i,i+1}} (\Id-\vec{n}_{i,i+1}\otimes \vec{n}_{i,i+1}) : \nabla B\,\mathrm{d}\mathcal{H}^{1}.
\end{align*}
Following precisely the same algebraic manipulations as in the two-phase case we obtain
\begin{align*}
		\ddt &E[\chi | \xi]
	\\
	\leq& -\frac12 \sum_i \int_{I_{i,i+1}} 
	\left( 
	\left( 
	V_{i,i+1}+\nabla \cdot \xi_{i,i+1}
	\right)^2
	+
	\left|
	V_{i,i+1}\vec{n}_{i,i+1} - \left(\xi_{i,i+1}\cdot B\right)\xi_{i,i+1}
	\right|^2
	\right)
	\,\mathrm{d}\mathcal{H}^{1}
	\\
	&+\frac12 \sum_i \int_{I_{i,i+1}}
	\left(
	\nabla \cdot \xi_{i,i+1}
	+
	\xi_{i,i+1} \cdot B
	\right)^2 \,\mathrm{d}\mathcal{H}^{1}
	\\&+ \sum_i \int_{I_{i,i+1}} \left(\nabla \cdot \xi_{i,i+1} \right) 
	\left(\vec{n}_{i,i+1}\cdot \xi_{i,i+1} - 1 \right) \left(\xi_{i,i+1} \cdot B\right) \,\mathrm{d}\mathcal{H}^{1}
	\\
	&+\sum_i\int_{I_{i,i+1}} \left(\nabla \cdot \xi_{i,i+1} + V_{i,i+1} \right) 
	\vec{n}_{i,i+1} \cdot  \left(\Id - \xi_{i,i+1} \otimes \xi_{i,i+1} \right) B\,\mathrm{d}\mathcal{H}^{1}
	 \\
	&+\sum_i\int_{I_{i,i+1}} (1-\vec{n}_{i,i+1}\cdot \xi_{i,i+1} ) \left(\nabla\cdot  B\right) \,\mathrm{d}\mathcal{H}^{1}
	\\&-\sum_i \int_{I_{i,i+1}}   \left(\vec{n}_{i,i+1}-\xi_{i,i+1}\right) \otimes  
	\left(\vec{n}_{i,i+1}-\xi_{i,i+1}\right) \colon  \nabla  B \,\mathrm{d}\mathcal{H}^{1}
	\\
	&-\sum_i \int_{I_{i,i+1}} \left(\partial_t \xi_{i,i+1} 
	+\left( B\cdot \nabla \right)\xi_{i,i+1} 
	+ (\nabla B)^\mathsf{T} \xi_{i,i+1}\right)\cdot (\vec{n}_{i,i+1}-\xi_{i,i+1})\,\mathrm{d}\mathcal{H}^{1}
	\\
	& - \sum_i\int_{I_{i,i+1}} \left(\partial_t \xi_{i,i+1} +\left( B\cdot \nabla \right)\xi_{i,i+1} \right)\cdot \xi_{i,i+1}\,\mathrm{d}\mathcal{H}^{1}.
\end{align*}
With this inequality at our disposal we can conclude as in the two-phase case.

\section{Stability of calibrated flows}
\label{SectionWeakStrongUniqueness}
This section is devoted to the proof of the stability properties of calibrated flows. 
In the next three subsections, we derive the relative entropy inequality 
Proposition~\ref{PropositionRelativeEntropyInequality} and the quantitative inclusion 
principle Theorem~\ref{TheoremUniqueness}.

\subsection{Relative entropy inequality: Proof of  Proposition~\ref{PropositionRelativeEntropyInequality}}
\label{SectionDerivationRelativeEntropyInequality}
We start with the proof of the relative entropy inequality for a BV solution 
$\chi=(\chi_1,\ldots,\chi_P)$ of multiphase mean curvature flow in the sense of Definition~\ref{DefinitionBVSolution}.
Recall the definition of the relative entropy functional $E[\chi | \xi]$ in \eqref{DefinitionRelativeEntropyFunctional}.

\begin{proof}[Proof of Proposition~\ref{PropositionRelativeEntropyInequality}]
In order to make use of the evolution equations \eqref{EvolutionPhasesBVSolution}
for the indicator functions $\chi_i$ of the BV solution,
we start by rewriting the interface error control of our relative entropy. 
Using the relation $\sigma_{i,j}\xi_{i,j}=\xi_i-\xi_j$ from Definition~\ref{DefinitionCalibrationGradientFlow}
of a gradient flow calibration, the symmetry relation $\vec{n}_{i,j}=-\vec{n}_{j,i}$,
the definition \eqref{UnitNormalsBVSolution} of the measure theoretic normal,
as well as the representation of the energy \eqref{eq:energy},
we obtain by an application of the generalized divergence theorem
\begin{align}\label{eq:SurfaceToVolume}
\nonumber
E[\chi | \xi](T)&=
\sum_{i,j=1,i\neq j}^P \sigma_{i,j} \int_{I_{i,j}(T)} 
1-\xi_{i,j}(\cdot,T) \cdot \vec{n}_{i,j}(\cdot,T) 
\,\mathrm{d}\mathcal{H}^{d-1}
\\&\nonumber
= E[\chi(\cdot,T)] - \sum_{i,j=1,i\neq j}^P \int_{I_{i,j}(T)} 
(\xi_i(\cdot,T){-}\xi_j(\cdot,T))\cdot \vec{n}_{i,j}(\cdot,T) 
\,\mathrm{d}\mathcal{H}^{d-1}
\\&\nonumber
= E[\chi(\cdot,T)] 
+ \sum_{i=1}^P\sum_{j=1,j\neq i}^P\int_{I_{i,j}(T)} 
\xi_i(\cdot,T)\cdot \frac{\nabla\chi_i(\cdot,T)}{|\nabla\chi_i(\cdot,T)|} 
\,\mathrm{d}\mathcal{H}^{d-1}
\\&~~~\nonumber
+ \sum_{j=1}^P\sum_{i=1,i\neq j}^P\int_{I_{i,j}(T)}
\xi_j(\cdot,T)\cdot \frac{\nabla\chi_j(\cdot,T)}{|\nabla\chi_j(\cdot,T)|} 
\,\mathrm{d}\mathcal{H}^{d-1}
\\&\nonumber
= E[\chi(\cdot,T)] + 2\sum_{i=1}^P \int_{\Rd} 
\xi_i(\cdot,T)\cdot \frac{\nabla\chi_i(\cdot,T)}{|\nabla\chi_i(\cdot,T)|} 
\,\mathrm{d}|\nabla\chi_i(\cdot,T)|
\\&
= E[\chi(\cdot,T)] - 2\sum_{i=1}^P \int_{\Rd} \chi_i(\cdot,T)
(\nabla\cdot\xi_i(\cdot,T))\dx.
\end{align}
This enables us to compute by the sharp energy dissipation inequality \eqref{EnergyDissipationInequalityBVSolution},
the evolution equations \eqref{EvolutionPhasesBVSolution} for the indicator functions $\chi_i$
of the BV solution, and the definition \eqref{NormalVelocitiesBVSolution} of the velocities $V_{i,j}$
for almost every $T\in [0,T']$
\begin{align*}
&E_{\interface}[\chi|\xi](T)
\\&
\leq E[\chi(\cdot,0)] - 2\sum_{i=1}^P \int_{\Rd} \chi_{0,i}
(\nabla\cdot\xi_i(\cdot,0))\dx
-\sum_{i,j=1,i\neq j}^P\sigma_{i,j}\int_0^T\int_{I_{i,j}(t)} 
|V_{i,j}|^2\,\mathrm{d}\mathcal{H}^{d-1}\dt
\\&~~~~
-2\sum_{i=1}^P \int_0^T\int_{\Rd}
\chi_i\partial_t(\nabla\cdot\xi_i)\dx\dt
-2\sum_{i=1}^P \int_0^T\int_{\Rd}
V_i(\nabla\cdot\xi_i)\,\mathrm{d}|\nabla\chi_i|\dt.
\end{align*}
The first two terms combine to $E_{\mathrm{interface}}[\chi|\bar\chi](0)$ using \eqref{eq:SurfaceToVolume}
in reverse order. We aim to rewrite the latter two terms back to surface integrals over the interfaces as well.
To this end, we argue analogously to the computation in \eqref{eq:SurfaceToVolume} but now in reverse order.
Using first the generalized divergence theorem, then splitting the integrals over the reduced boundaries of the phases 
into contributions over the interfaces $I_{i,j}=\partial^* \Omega_i\cap \partial^* \Omega_{j}$ 
by means of $\sigma_{i,j}\xi_{i,j}=\xi_{i}-\xi_{j}$ from Definition~\ref{DefinitionCalibrationGradientFlow}
of a gradient flow calibration, we obtain
\begin{align*}
-2\sum_{i=1}^P \int_0^T\int_{\Rd}\chi_i\partial_t(\nabla\cdot\xi_i)\dx\dt
&=2\sum_{i=1}^P \int_0^T\int_{\Rd}\frac{\nabla\chi_i}{|\nabla\chi_i|}\cdot\partial_t\xi_i
\,\mathrm{d}|\nabla\chi_i|\dt
\\&
=\sum_{i=1}^P\sum_{j=1,j\neq i}^P\int_0^T\int_{I_{i,j}(t)}
\frac{\nabla\chi_i}{|\nabla\chi_i|}\cdot\partial_t\xi_i
\,\mathrm{d}\mathcal{H}^{d-1}\dt
\\&~~~
+\sum_{j=1}^P\sum_{i=1,i\neq j}^P\int_0^T\int_{I_{i,j}(t)}
\frac{\nabla\chi_j}{|\nabla\chi_j|}\cdot\partial_t\xi_j
\,\mathrm{d}\mathcal{H}^{d-1}\dt
\\&
\stackrel{\eqref{UnitNormalsBVSolution}}{=}
-\sum_{i,j=1,i\neq j}^P\int_0^T\int_{I_{i,j}(t)}
\vec{n}_{i,j}\cdot\partial_t(\xi_i-\xi_j)\,\mathrm{d}\mathcal{H}^{d-1}\dt
\\&
=-\sum_{i,j=1,i\neq j}^P\sigma_{i,j}\int_0^T\int_{I_{i,j}(t)} 
\vec{n}_{i,j}\cdot\partial_t\xi_{i,j}\,\mathrm{d}\mathcal{H}^{d-1}\dt.
\end{align*}
The term incorporating the normal velocities is treated similarly.
In addition to the above ingredients, i.e., $\sigma_{i,j}\xi_{i,j}=\xi_{i}-\xi_{j}$ 
from Definition~\ref{DefinitionCalibrationGradientFlow} of a gradient flow calibration and
splitting the integrals over the reduced boundaries of the phases 
into contributions over the interfaces $I_{i,j}=\partial^* \Omega_i\cap \partial^* \Omega_{j}$, 
we also use that $V_{i,j}=-V_{j,i}$ on ${\bar{I}}_{i,j}$ together with definition 
\eqref{NormalVelocitiesBVSolution} to compute
\begin{align*}
-2\sum_{i=1}^P \int_0^T\int_{\Rd}
V_i(\nabla\cdot\xi_i)\,\mathrm{d}|\nabla\chi_i|\dt
&=-\sum_{i=1}^P\sum_{j=1,j\neq i}^P\int_0^T\int_{I_{i,j}(t)}
V_{i,j}(\nabla\cdot\xi_i)\,\mathrm{d}\mathcal{H}^{d-1}\dt
\\&~~~
+\sum_{j=1}^P\sum_{i=1,i\neq j}^P\int_0^T\int_{I_{i,j}(t)}
V_{i,j}(\nabla\cdot\xi_j)\,\mathrm{d}\mathcal{H}^{d-1}\dt
\\&
=-\sum_{i,j=1,i\neq j}^P\sigma_{i,j}\int_0^T\int_{I_{i,j}(t)} 
V_{i,j}(\nabla\cdot\xi_{i,j})\,\mathrm{d}\mathcal{H}^{d-1}\dt.
\end{align*}
Combining the last two identities, we obtain for almost every $T\in [0,T']$
\begin{align*}
&E_{\interface}[\chi|\xi](T)
\\&
\leq E_{\interface}[\chi|\xi](0)
-\sum_{i,j=1,i\neq j}^P\sigma_{i,j}\int_0^T\int_{I_{i,j}(t)} 
|V_{i,j}|^2\,\mathrm{d}\mathcal{H}^{d-1}\dt
\\&~~~~
-\sum_{i,j=1,i\neq j}^P\sigma_{i,j}\int_0^T\int_{I_{i,j}(t)} 
\vec{n}_{i,j}\cdot\partial_t\xi_{i,j}\,\mathrm{d}\mathcal{H}^{d-1}\dt
\\&~~~~
-\sum_{i,j=1,i\neq j}^P\sigma_{i,j}\int_0^T\int_{I_{i,j}(t)} 
V_{i,j}(\nabla\cdot\xi_{i,j})\,\mathrm{d}\mathcal{H}^{d-1}\dt.
\end{align*}
For the next step, we use the vector field $B$
as a test function in the BV formulation of mean curvature flow \eqref{BVFormulationMeanCurvature}. Adding the
resulting equation to the previous inequality, observing in the process that 
$V_i\frac{\nabla\chi_i}{|\nabla\chi_i|}=-V_{i,j}\vec{n}_{i,j}$ on $I_{i,j}$ 
due to \eqref{UnitNormalsBVSolution} and \eqref{NormalVelocitiesBVSolution},
as well as decomposing $B=(\mathrm{Id}{-}\xi_{i,j}\otimes\xi_{i,j})B
+(B\cdot\xi_{i,j})\xi_{i,j}$, we obtain
\begin{align}\label{DerivationRelEntropyInequ1}
\nonumber
&E_{\interface}[\chi|\xi](T)
\\&
\leq E_{\interface}[\chi|\xi](0)
-\sum_{i,j=1,i\neq j}^P\sigma_{i,j}\int_0^T\int_{I_{i,j}(t)} 
|V_{i,j}|^2\,\mathrm{d}\mathcal{H}^{d-1}\dt
\\&~~~~\nonumber
+\sum_{i,j=1,i\neq j}^P\sigma_{i,j}\int_0^T\int_{I_{i,j}(t)} 
(B\cdot\xi_{i,j})\xi_{i,j}\cdot V_{i,j}\vec{n}_{i,j}
\,\mathrm{d}\mathcal{H}^{d-1}\dt
\\&~~~~\nonumber
-\sum_{i,j=1,i\neq j}^P\sigma_{i,j}\int_0^T\int_{I_{i,j}(t)} 
V_{i,j}(\nabla\cdot\xi_{i,j})\,\mathrm{d}\mathcal{H}^{d-1}\dt
\\&~~~~\nonumber
+\sum_{i,j=1,i\neq j}^P\sigma_{i,j}\int_0^T\int_{I_{i,j}(t)} 
(\mathrm{Id}{-}\xi_{i,j}\otimes\xi_{i,j})B\cdot V_{i,j}\vec{n}_{i,j}
\,\mathrm{d}\mathcal{H}^{d-1}\dt
\\&~~~~\nonumber
+\sum_{i,j=1,i\neq j}^P\sigma_{i,j}\int_0^T\int_{I_{i,j}(t)} 
(\nabla\cdot B)\,\mathrm{d}\mathcal{H}^{d-1}\dt
\\&~~~~\nonumber
-\sum_{i,j=1,i\neq j}^P\sigma_{i,j}\int_0^T\int_{I_{i,j}(t)} 
\vec{n}_{i,j}\otimes\vec{n}_{i,j}:\nabla B\,\mathrm{d}\mathcal{H}^{d-1}\dt
\\&~~~~\nonumber
-\sum_{i,j=1,i\neq j}^P\sigma_{i,j}\int_0^T\int_{I_{i,j}(t)} 
\vec{n}_{i,j}\cdot\partial_t\xi_{i,j}\,\mathrm{d}\mathcal{H}^{d-1}\dt,
\end{align}
which holds for almost every $T\in [0,T']$.
In order to obtain the dissipation term on the left hand side of the
relative entropy inequality \eqref{RelativeEntropyInequality}, we complete
the squares yielding for almost every $T\in [0,T']$
\begin{align}\label{DerivationRelEntropyInequ2}
\nonumber
&-\sum_{i,j=1,i\neq j}^P\sigma_{i,j}\int_0^T\int_{I_{i,j}(t)} 
|V_{i,j}|^2\,\mathrm{d}\mathcal{H}^{d-1}\dt
\\&\nonumber
+\sum_{i,j=1,i\neq j}^P\sigma_{i,j}\int_0^T\int_{I_{i,j}(t)} 
(B\cdot\xi_{i,j})\xi_{i,j}\cdot V_{i,j}\vec{n}_{i,j}
\,\mathrm{d}\mathcal{H}^{d-1}\dt
\\&
-\sum_{i,j=1,i\neq j}^P\sigma_{i,j}\int_0^T\int_{I_{i,j}(t)} 
V_{i,j}(\nabla\cdot\xi_{i,j})\,\mathrm{d}\mathcal{H}^{d-1}\dt
\\&\nonumber
= -\sum_{i,j=1,i\neq j}^{P}\sigma_{i,j}
\int_0^T\int_{I_{i,j}(t)}\Big(
\frac{1}{2}|V_{i,j}{+}\nabla\cdot\xi_{i,j}|^2
+\frac{1}{2}|V_{i,j}\vec{n}_{i,j}{-}(B\cdot\xi_{i,j})\xi_{i,j}|^2\Big)
\,\mathrm{d}\mathcal{H}^{d-1}\dt
\\&~~~~\nonumber
+ \sum_{i,j=1,i\neq j}^{P}\sigma_{i,j}\int_0^T\int_{I_{i,j}(t)}
\Big(\frac{1}{2}|\nabla\cdot\xi_{i,j}|^2+\frac{1}{2}|(B\cdot\xi_{i,j})\xi_{i,j}|^2\Big)
\,\mathrm{d}\mathcal{H}^{d-1}\dt.
\end{align}
Furthermore, on the one hand, adding and subtracting $(B\cdot\nabla)\xi_{i,j}+(\nabla B)^\mathsf{T}\xi_{i,j}$ yields
\begin{align}\label{DerivationRelEntropyInequ3}
\nonumber
&\sum_{i,j=1,i\neq j}^P\sigma_{i,j}\int_0^T\int_{I_{i,j}(t)} 
(\nabla\cdot B)\,\mathrm{d}\mathcal{H}^{d-1}\dt
\\&\nonumber
-\sum_{i,j=1,i\neq j}^P\sigma_{i,j}\int_0^T\int_{I_{i,j}(t)} 
\vec{n}_{i,j}\otimes\vec{n}_{i,j}:\nabla B\,\mathrm{d}\mathcal{H}^{d-1}\dt
\\&\nonumber
-\sum_{i,j=1,i\neq j}^P\sigma_{i,j}\int_0^T\int_{I_{i,j}(t)} 
\vec{n}_{i,j}\cdot\partial_t\xi_{i,j}\,\mathrm{d}\mathcal{H}^{d-1}\dt
\\&
= \sum_{i,j=1,i\neq j}^P\sigma_{i,j}\int_0^T\int_{I_{i,j}(t)} 
(\nabla\cdot B)\,\mathrm{d}\mathcal{H}^{d-1}\dt
\\&~~~~\nonumber
-\sum_{i,j=1,i\neq j}^P\sigma_{i,j}\int_0^T\int_{I_{i,j}(t)} 
(\vec{n}_{i,j}-\xi_{i,j})\cdot(\vec{n}_{i,j}\cdot\nabla)B\,\mathrm{d}\mathcal{H}^{d-1}\dt
\\&~~~~\nonumber
+\sum_{i,j=1,i\neq j}^P\sigma_{i,j}\int_0^T\int_{I_{i,j}(t)}
\big((B\cdot\nabla)\xi_{i,j}\big)\cdot\vec{n}_{i,j}\,\mathrm{d}\mathcal{H}^{d-1}\dt
\\&~~~~\nonumber
-\sum_{i,j=1,i\neq j}^{P}\sigma_{i,j}\int_0^T\int_{I_{i,j}(t)}
\big(\partial_t\xi_{i,j}{+}(B\cdot\nabla)\xi_{i,j}{+}(\nabla B)^\mathsf{T}\xi_{i,j}\big)\cdot\vec{n}_{i,j}
\,\mathrm{d}\mathcal{H}^{d-1}\dt
\end{align}
for almost every $T\in [0,T']$. On the other hand, we may exploit 
symmetry to obtain (relying again on the by now routine fact that one can
switch back and forth between certain volume integrals and surface integrals over the individual interfaces
by means of $\sigma_{i,j}\xi_{i,j}=\xi_i-\xi_j$ from Definition~\ref{DefinitionCalibrationGradientFlow}
of a gradient flow calibration, the symmetry relation $\vec{n}_{i,j}=-\vec{n}_{j,i}$ and
the definition \eqref{UnitNormalsBVSolution})
\begin{align*}
&\sum_{i,j=1,i\neq j}^{P}\sigma_{i,j}\int_0^T\int_{I_{i,j}(t)}
\vec{n}_{i,j}\cdot\big(\nabla\cdot(B\otimes\xi_{i,j})\big)
\,\mathrm{d}\mathcal{H}^{d-1}\dt
\\&
=\sum_{i,j=1,i\neq j}^{P}\int_0^T\int_{I_{i,j}(t)}
\vec{n}_{i,j}\cdot\big(\nabla\cdot(B\otimes(\xi_{i}{-}\xi_j))\big)
\,\mathrm{d}\mathcal{H}^{d-1}\dt
\\&
=-2\sum_{i=1}^P\int_0^T\int_{\Rd}\frac{\nabla\chi_i}{|\nabla\chi_i|}
\cdot\big(\nabla\cdot(B\otimes\xi_{i})\big)\,\mathrm{d}\mathcal{H}^{d-1}\dt
\\&
=2\sum_{i=1}^P\int_0^T\int_{\Rd}\chi_i\nabla\cdot
\big(\nabla\cdot(B\otimes\xi_{i})\big)\dx\dt
\\&
=2\sum_{i=1}^P\int_0^T\int_{\Rd}\chi_i\nabla\cdot
\big(\nabla\cdot(\xi_{i}\otimes B)\big)\dx\dt
\\&
=\sum_{i,j=1,i\neq j}^{P}\sigma_{i,j}\int_0^T\int_{I_{i,j}(t)}
\vec{n}_{i,j}\cdot\big(\nabla\cdot(\xi_{i,j}\otimes B)\big)
\,\mathrm{d}\mathcal{H}^{d-1}\dt.
\end{align*}
Because of this identity, we can now compute
\begin{align}
\label{DerivationRelEntropyInequ4}
\nonumber
0 &= \sum_{i,j=1,i\neq j}^{P}\sigma_{i,j}\int_0^T\int_{I_{i,j}(t)}
\vec{n}_{i,j}\cdot\big(\nabla\cdot(B\otimes\xi_{i,j}-\xi_{i,j}\otimes B)\big)
\,\mathrm{d}\mathcal{H}^{d-1}\dt
\\&
= \sum_{i,j=1,i\neq j}^{P}\sigma_{i,j}\int_0^T\int_{I_{i,j}(t)}
(\nabla\cdot\xi_{i,j})B\cdot\vec{n}_{i,j}
\,\mathrm{d}\mathcal{H}^{d-1}\dt
\\&~~~~\nonumber
+ \sum_{i,j=1,i\neq j}^{P}\sigma_{i,j}\int_0^T\int_{I_{i,j}(t)}
\vec{n}_{i,j}\cdot(\xi_{i,j}\cdot\nabla) B
\,\mathrm{d}\mathcal{H}^{d-1}\dt
\\&~~~~\nonumber
- \sum_{i,j=1,i\neq j}^{P}\sigma_{i,j}\int_0^T\int_{I_{i,j}(t)}
\vec{n}_{i,j}\cdot(B\cdot\nabla)\xi_{i,j}
\,\mathrm{d}\mathcal{H}^{d-1}\dt
\\&~~~~\nonumber
- \sum_{i,j=1,i\neq j}^{P}\sigma_{i,j}\int_0^T\int_{I_{i,j}(t)}
(\nabla\cdot B)\xi_{i,j}\cdot\vec{n}_{i,j}	
\,\mathrm{d}\mathcal{H}^{d-1}\dt.
\end{align}
Making use of the identities \eqref{DerivationRelEntropyInequ2} and \eqref{DerivationRelEntropyInequ3}
in the inequality \eqref{DerivationRelEntropyInequ1} as well as adding \eqref{DerivationRelEntropyInequ4}
to the right hand side of \eqref{DerivationRelEntropyInequ1}, we arrive at the following bound for
the time evolution of the interface error control of our relative entropy functional
\begin{align}
\nonumber
&E_{\interface}[\chi|\xi](T)
\\&\nonumber
+\sum_{i,j=1,i\neq j}^{P}\sigma_{i,j}
\int_0^T\int_{I_{i,j}(t)}\Big(
\frac{1}{2}|V_{i,j}{+}\nabla\cdot\xi_{i,j}|^2
+\frac{1}{2}|V_{i,j}\vec{n}_{i,j}{-}(B\cdot\xi_{i,j})\xi_{i,j}|^2\Big)
\,\mathrm{d}\mathcal{H}^{d-1}\dt
\\&\label{eq:prefinalRelEntropyInequality}
\leq E_{\interface}[\chi|\xi](0)
\\&~~~\nonumber
+ \sum_{i,j=1,i\neq j}^{P}\sigma_{i,j}\int_0^T\int_{I_{i,j}(t)}
\Big(\frac{1}{2}|\nabla\cdot\xi_{i,j}|^2+\frac{1}{2}|(B\cdot\xi_{i,j})\xi_{i,j}|^2\Big)
\,\mathrm{d}\mathcal{H}^{d-1}\dt
\\&~~~\nonumber
+\sum_{i,j=1,i\neq j}^{P}\sigma_{i,j}\int_0^T\int_{I_{i,j}(t)}
(\nabla\cdot\xi_{i,j})B\cdot\vec{n}_{i,j}
\,\mathrm{d}\mathcal{H}^{d-1}\dt
\\&~~~\nonumber
+\sum_{i,j=1,i\neq j}^P\sigma_{i,j}\int_0^T\int_{I_{i,j}(t)} 
(\mathrm{Id}{-}\xi_{i,j}\otimes\xi_{i,j})B\cdot V_{i,j}\vec{n}_{i,j}
\,\mathrm{d}\mathcal{H}^{d-1}\dt
\\&~~~\nonumber
+\sum_{i,j=1,i\neq j}^{P}\sigma_{i,j}\int_0^T\int_{I_{i,j}(t)}
(\nabla\cdot B)(1-\xi_{i,j}\cdot\vec{n}_{i,j})	
\,\mathrm{d}\mathcal{H}^{d-1}\dt
\\&~~~\nonumber
-\sum_{i,j=1,i\neq j}^P\sigma_{i,j}\int_0^T\int_{I_{i,j}(t)} 
(\vec{n}_{i,j}-\xi_{i,j})\otimes \vec{n}_{i,j}:\nabla B\,\mathrm{d}\mathcal{H}^{d-1}\dt
\\&~~~\nonumber
+\sum_{i,j=1,i\neq j}^P\sigma_{i,j}\int_0^T\int_{I_{i,j}(t)} 
\vec{n}_{i,j}\otimes \xi_{i,j}:\nabla B\,\mathrm{d}\mathcal{H}^{d-1}\dt
\\&~~~\nonumber
-\sum_{i,j=1,i\neq j}^{P}\sigma_{i,j}\int_0^T\int_{I_{i,j}(t)}
\big(\partial_t\xi_{i,j}{+}(B\cdot\nabla)\xi_{i,j}{+}(\nabla B)^\mathsf{T}\xi_{i,j}\big)\cdot\vec{n}_{i,j}
\,\mathrm{d}\mathcal{H}^{d-1}\dt,
\end{align}
which is valid for almost every $T\in [0,T']$. Completing squares and adding zero
yields for the second, third and fourth term on the right hand side of \eqref{eq:prefinalRelEntropyInequality}
\begin{align}
\nonumber
&\sum_{i,j=1,i\neq j}^{P}\sigma_{i,j}\int_0^T\int_{I_{i,j}(t)}
\Big(\frac{1}{2}|\nabla\cdot\xi_{i,j}|^2+\frac{1}{2}|(B\cdot\xi_{i,j})\xi_{i,j}|^2\Big)
\,\mathrm{d}\mathcal{H}^{d-1}\dt
\\&\nonumber
+\sum_{i,j=1,i\neq j}^{P}\sigma_{i,j}\int_0^T\int_{I_{i,j}(t)}
(\nabla\cdot\xi_{i,j})B\cdot\vec{n}_{i,j}
\,\mathrm{d}\mathcal{H}^{d-1}\dt
\\&\nonumber
+\sum_{i,j=1,i\neq j}^P\sigma_{i,j}\int_0^T\int_{I_{i,j}(t)} 
(\mathrm{Id}{-}\xi_{i,j}\otimes\xi_{i,j})B\cdot V_{i,j}\vec{n}_{i,j}
\,\mathrm{d}\mathcal{H}^{d-1}\dt
\\&\label{eq:prefinalRelEntropyInequality2}
=\sum_{i,j=1,i\neq j}^{P}\sigma_{i,j}\int_0^T\int_{I_{i,j}(t)}
\frac{1}{2}|(\nabla\cdot\xi_{i,j})+B\cdot\xi_{i,j}|^2
\,\mathrm{d}\mathcal{H}^{d-1}\dt 
\\&~~~\nonumber
-\sum_{i,j=1,i\neq j}^{P}\sigma_{i,j}\int_0^T\int_{I_{i,j}(t)}
\frac{1}{2}|B\cdot\xi_{i,j}|^2(1-|\xi_{i,j}|^2)
\,\mathrm{d}\mathcal{H}^{d-1}\dt
\\&~~~\nonumber
+\sum_{i,j=1,i\neq j}^P\sigma_{i,j}\int_0^T\int_{I_{i,j}(t)} 
(\mathrm{Id}{-}\xi_{i,j}\otimes\xi_{i,j})B\cdot 
(V_{i,j}+\nabla\cdot\xi_{i,j})\vec{n}_{i,j}
\,\mathrm{d}\mathcal{H}^{d-1}\dt
\\&~~~\nonumber
-\sum_{i,j=1,i\neq j}^{P}\sigma_{i,j}\int_0^T\int_{I_{i,j}(t)}
(1-\vec{n}_{i,j}\cdot\xi_{i,j})(\nabla\cdot\xi_{i,j})(B\cdot\xi_{i,j})
\,\mathrm{d}\mathcal{H}^{d-1}\dt.
\end{align}
Adding zero in the last term on the right hand side of \eqref{eq:prefinalRelEntropyInequality}
in order to generate the transport equation for the length of the vector fields $\xi_{i,j}$, we
observe that the last three terms on the right hand side of \eqref{eq:prefinalRelEntropyInequality}
combine to
\begin{align}
\nonumber
&-\sum_{i,j=1,i\neq j}^P\sigma_{i,j}\int_0^T\int_{I_{i,j}(t)} 
(\vec{n}_{i,j}-\xi_{i,j})\otimes \vec{n}_{i,j}:\nabla B\,\mathrm{d}\mathcal{H}^{d-1}\dt
\\&\nonumber
+\sum_{i,j=1,i\neq j}^P\sigma_{i,j}\int_0^T\int_{I_{i,j}(t)} 
\vec{n}_{i,j}\otimes \xi_{i,j}:\nabla B\,\mathrm{d}\mathcal{H}^{d-1}\dt
\\&\nonumber
-\sum_{i,j=1,i\neq j}^{P}\sigma_{i,j}\int_0^T\int_{I_{i,j}(t)}
\big(\partial_t\xi_{i,j}{+}(B\cdot\nabla)\xi_{i,j}{+}(\nabla B)^\mathsf{T}\xi_{i,j}\big)\cdot\vec{n}_{i,j}
\,\mathrm{d}\mathcal{H}^{d-1}\dt
\\&\label{eq:prefinalRelEntropyInequality3}
=-\sum_{i,j=1,i\neq j}^P\sigma_{i,j}\int_0^T\int_{I_{i,j}(t)} 
(\vec{n}_{i,j}-\xi_{i,j})\otimes (\vec{n}_{i,j}-\xi_{i,j}):\nabla B\,\mathrm{d}\mathcal{H}^{d-1}\dt
\\&~~~\nonumber
-\sum_{i,j=1,i\neq j}^{P}\sigma_{i,j}\int_0^T\int_{I_{i,j}(t)}
\big(\partial_t\xi_{i,j}{+}(B\cdot\nabla)\xi_{i,j}{+}(\nabla B)^\mathsf{T}\xi_{i,j}\big)
\cdot(\vec{n}_{i,j}-\xi_{i,j})\,\mathrm{d}\mathcal{H}^{d-1}\dt
\\&~~~\nonumber
-\sum_{i,j=1,i\neq j}^{P}\sigma_{i,j}\int_0^T\int_{I_{i,j}(t)}
\frac{1}{2}\big(\partial_t|\xi_{i,j}|^2{+}(B\cdot\nabla)|\xi_{i,j}|^2\big)
\,\mathrm{d}\mathcal{H}^{d-1}\dt. 
\end{align}
Employing the notation of Proposition~\ref{PropositionRelativeEntropyInequality}
as well as using \eqref{eq:prefinalRelEntropyInequality2} and \eqref{eq:prefinalRelEntropyInequality3} 
in \eqref{eq:prefinalRelEntropyInequality}, we deduce that the right hand side
of \eqref{eq:prefinalRelEntropyInequality} indeed reduces to
\begin{align*}
&E_{\interface}[\chi|\xi](T)
\\&\nonumber
+\sum_{i,j=1,i\neq j}^{P}\sigma_{i,j}
\int_0^T\int_{I_{i,j}(t)}\Big(
\frac{1}{2}|V_{i,j}{+}\nabla\cdot\xi_{i,j}|^2
+\frac{1}{2}|V_{i,j}\vec{n}_{i,j}{-}(B\cdot\xi_{i,j})\xi_{i,j}|^2\Big)
\,\mathrm{d}\mathcal{H}^{d-1}\dt
\\&
\leq E_{\interface}[\chi|\xi](0)
+R_{\mathrm{dt}}+R_{\mathrm{dissip}},
\end{align*}
which is valid for almost every $T\in [0,T']$.
This concludes the proof of \eqref{RelativeEntropyInequality}.
\end{proof}

\subsection{Quantitative inclusion principle: Proof of Theorem~\ref{TheoremUniqueness}}\label{SectionProofMainResult}

We now prove the inclusion principle stating that interfaces of any $\BV$ solution 
must be contained in the corresponding interfaces of a calibrated flow, provided both 
start with the same initial data.

\begin{proof}[Proof of Theorem~\ref{TheoremUniqueness}]

\emph{Step 1: The stability estimate \eqref{StabilityEstimate}.}
The starting point is the estimate on the evolution of the interface 
error functional \eqref{DefinitionRelativeEntropy} from Proposition~\ref{PropositionRelativeEntropyInequality}.
In the following, we estimate the terms appearing on
the right hand side one-by-one. Let $T \in [0,T']$.

Due to \eqref{TransportEquationXi}, \eqref{LengthConservation}, as well as \eqref{LengthControlXi} and the trivial relation
\begin{align}\label{eq:L2excess}
|\vec{n}_{i,j}{-}\xi_{i,j}|^2\leq 2(1-\vec{n}_{i,j}\cdot\xi_{i,j})
\end{align}
(which follows by $|\xi_{i,j}|\leq 1$),
we immediately deduce using Young's inequality
\begin{align}\label{BoundRdt}
|R_{\mathrm{dt}}|&
\leq 
\sum_{i,j=1,i\neq j}^{P}
C\int_0^T\int_{I_{i,j}(t)}
|\vec{n}_{i,j}-\xi_{i,j}|^2+\dist^2(\cdot,{\bar{I}}_{i,j}(t))\wedge 1
\,\mathrm{d}\mathcal{H}^{d-1}\dt
\\&
\leq C\int_0^T E_{\interface}[\chi|\xi](t)\dt.
\notag
\end{align}
Making use of the simple estimate $1{-}|\xi_{i,j}|^2\leq 2(1-|\xi_{i,j}|)\leq 2(1-\vec{n}_{i,j}\cdot\xi_{i,j})$
and again the bound \eqref{eq:L2excess}, we obtain by similar arguments
\begin{align*}
|R_{\mathrm{dissip}}|
&\leq 
\sum_{i,j=1,i\neq j}^{P}\sigma_{i,j}
\int_0^T\int_{I_{i,j}(t)}
\frac{1}{2}|(\nabla\cdot\xi_{i,j})+B\cdot\xi_{i,j}|^2
\,\mathrm{d}\mathcal{H}^{d-1}\dt
\\&~~~~
+\sum_{i,j=1,i\neq j}^{P}\sigma_{i,j}
\int_0^T\int_{I_{i,j}(t)}
(\mathrm{Id}{-}\xi_{i,j}\otimes\xi_{i,j})B\cdot(V_{i,j}{+}\nabla\cdot\xi_{i,j})\vec{n}_{i,j}
\,\mathrm{d}\mathcal{H}^{d-1}\dt
\\&~~~~
+C\int_0^T E_{\interface}[\chi|\xi](t)\dt
\\&
=: I + II + C\int_0^T E_{\interface}[\chi|\xi](t)\dt.
\end{align*}
By means of \eqref{Dissip}, we may directly estimate
\begin{align*}
|I| \leq C\int_0^T E_{\interface}[\chi|\xi](t)\dt.
\end{align*}
Furthermore, by an application of H\"older's and Young's inequality we deduce
\begin{align*}
|II| &= \bigg|\sum_{i,j=1,i\neq j}^{P}\sigma_{i,j}
\int_0^T\int_{I_{i,j}(t)}
(\mathrm{Id}{-}\xi_{i,j}\otimes\xi_{i,j})B\cdot(V_{i,j}{+}\nabla\cdot\xi_{i,j})
(\vec{n}_{i,j}{-}\xi_{i,j})\,\mathrm{d}\mathcal{H}^{d-1}\dt\bigg|
\\&
\leq \delta\sum_{i,j=1,i\neq j}^{P}\sigma_{i,j}\int_0^T\int_{I_{i,j}(t)}
\frac{1}{2}(V_{i,j}{+}\nabla\cdot\xi_{i,j})^2\,\mathrm{d}\mathcal{H}^{d-1}\dt
\\&~~~~
+C\delta^{-1}\int_0^T E_{\interface}[\chi|\xi](t)\dt,
\end{align*}
uniformly over all $\delta\in (0,1)$. Hence, we get the bound
\begin{align}\label{BoundRdissip}
|R_{\mathrm{dissip}}| 
&\leq 
\delta\sum_{i,j=1,i\neq j}^{P}\sigma_{i,j}\int_0^T\int_{I_{i,j}(t)}
\frac{1}{2}(V_{i,j}{+}\nabla\cdot\xi_{i,j})^2\,\mathrm{d}\mathcal{H}^{d-1}\dt
\\&~~~~\nonumber
+C\delta^{-1}\int_0^T E_{\interface}[\chi|\xi](t)\dt.
\end{align}

Plugging in the bounds from \eqref{BoundRdt} and \eqref{BoundRdissip} 
into the relative entropy inequality from Proposition~\ref{PropositionRelativeEntropyInequality},
and then choosing $\delta\in (0,1)$ sufficiently small in order to absorb the first right-hand side term, we therefore get constants $C_1, C_2>0$
such that the estimate
\begin{align}
\label{eq:postProcessedInterfaceError}
&E[\chi|\xi](T)
\\&\nonumber
+ C_1\sum_{i,j=1,i\neq j}^{P}
\int_0^T\int_{I_{i,j}(t)}\Big(\frac{1}{2}(V_{i,j}{+}\nabla\cdot\xi_{i,j})^2
+\frac{1}{2}|V_{i,j}\vec{n}_{i,j}{-}(B\cdot\xi_{i,j})\xi_{i,j}|^2\Big)
\,\mathrm{d}\mathcal{H}^{d-1}\dt
\\&\nonumber
\leq C_2 \int_0^T E[\chi|\xi](t)\dt
\end{align}
holds true for almost every $T\in [0,T']$. By an application of the Gronwall lemma, 
the asserted stability estimate \eqref{StabilityEstimate} from Theorem~\ref{TheoremUniqueness} follows. 

\emph{Step 3: Weak-strong comparison.}
In the case of coinciding initial conditions, i.e.\ $E[\chi|\xi](0)=0$, the stability 
estimate \eqref{StabilityEstimate} entails $E[\chi|\xi]=0$ for almost every $t\in [0,T']$. From this and \eqref{LengthControlXi}, it immediately follows by the definition of the relative entropy \eqref{DefinitionRelativeEntropy} that $I_{i,j}(t)\subset {\bar{I}}_{i,j}(t)$ 
holds up to an $\mathcal{H}^{d-1}$-null set for almost every $t\in [0,T']$.
This proves the quantitative inclusion principle for $\BV$~solutions of multiphase mean curvature flow.
\end{proof}

\subsection{Conditional weak-strong uniqueness: Proof of Proposition~\ref{PropositionUniquenessConditional}}
\label{SectionProofConditionalWeakStrong}
We start with an analogue of the relative entropy inequality
of Proposition~\ref{PropositionRelativeEntropyInequality}
in terms of the bulk error functional $E_{\mathrm{volume}}[\chi|\bar\chi]$ from~\eqref{volumeErrorFunctional}.

\begin{lemma}
\label{relativeEntropyVolume}
Let $d\geq 2$, $P\geq 2$ be integers and $\sigma\in\Rd[P\times P]$ be an admissible matrix of surface tensions
in the sense of Definition~\ref{DefinitionAdmissibleSurfaceTensions}. 
Let $\chi=(\chi_1,\ldots,\chi_P)$ be a BV solution of multiphase mean curvature flow in
the sense of Definition~\ref{DefinitionBVSolution} on some time interval $[0,T']$.
Recall from \eqref{UnitNormalsBVSolution} resp.\ \eqref{NormalVelocitiesBVSolution} the definitions of the 
(measure-theoretic) unit normal vectors $\vec{n}_{i,j}$ resp.\ of the normal velocities $V_{i,j}$ of a BV solution.
Let moreover $\bar\Omega=(\bar\Omega_1,\ldots,\bar\Omega_P)$ be a time-dependent partition of $\Rd$
with finite interface energy on $[0,T']$ as in Definition~\ref{def:transportedWeights},
and assume that there exists an associated family of transported weights $(\vartheta_i)_{i\in\{1,\ldots,P\}}$
with velocity field $B$. Finally, let $(\xi_{i,j})_{i\neq j\in\{1,\ldots,P\}}$ be a family 
of compactly supported vector fields such that
\begin{align*}
\xi_{i,j} \in C^0([0,T'];C^1_{\mathrm{cpt}}(\Rd[d];\Rd[d])).
\end{align*}

Then, the bulk error functional
$E_{\mathrm{volume}}[\chi|\bar\chi]$ from~\eqref{volumeErrorFunctional}
is subject to the identity
\begin{align}
\label{RelativeEntropyInequalityVolumeControl}
E_{\mathrm{volume}}[\chi|\bar\chi](T) 
= E_{\mathrm{volume}}[\chi|\bar\chi](0) + R_{\mathrm{volume}}
\end{align}
for almost every $T\in [0,T']$. Here, we made use of the abbreviation
\begin{align*}
R_{\mathrm{volume}} &:= 
-\sum_{i,j=1,i\neq j}^{P}\int_0^T\int_{I_{i,j}(t)}
\vartheta_i(B\cdot\xi_{i,j}-V_{i,j})\,\mathrm{d}\mathcal{H}^{d-1}\dt
\\&~~~
-\sum_{i,j=1,i\neq j}^{P}\int_0^T\int_{I_{i,j}(t)}
\vartheta_iB\cdot(\vec{n}_{i,j}-\xi_{i,j})\,\mathrm{d}\mathcal{H}^{d-1}\dt
\\&~~~
+\sum_{i=1}^{P}\int_0^T\int_{\Rd}(\chi_i-\bar\chi_i)
\vartheta_i(\nabla\cdot B)\dx\dt
\\&~~~
+\sum_{i=1}^{P}\int_0^T\int_{\Rd}(\chi_i-\bar\chi_i)
(\partial_t\vartheta_i+(B\cdot\nabla)\vartheta_i)\dx\dt.
\end{align*}
Denote for $i,j\in\{1,\ldots,P\}$ with $i\neq j$ and~$t\in [0,T']$
by $\bar I_{i,j}(t):=\partial\bar\Omega_i(t)\cap\partial\bar\Omega_j(t)$
the interfaces associated with~$\bar\Omega$.
Then, the identity~\eqref{RelativeEntropyInequalityVolumeControl}
may be upgraded to the estimate
\begin{align}
\nonumber
&E_{\mathrm{volume}}[\chi|\bar\chi](T) 
\\& \label{eq:postProcessedRelEntropyWeightedVolume}
\leq E_{\mathrm{volume}}[\chi|\bar\chi](0)
+ \delta \sum_{i,j=1,i\neq j}^{P}\int_0^T\int_{I_{i,j}(t)}
|B\cdot\xi_{i,j}-V_{i,j}|^2 \,\mathrm{d}\mathcal{H}^{d-1}\dt
\\&~~~\nonumber
+ \frac{C}{\delta}\sum_{i,j=1,i\neq j}^{P}\int_0^T\int_{I_{i,j}(t)}
\dist^2(\cdot,\bar I_{i,j}) \wedge 1 \,\mathrm{d}\mathcal{H}^{d-1}\dt
\\&~~~\nonumber
+ \frac{C}{\delta}\sum_{i,j=1,i\neq j}^{P}\int_0^T\int_{I_{i,j}(t)}
1 - \vec{n}_{i,j}\cdot\xi_{i,j} \,\mathrm{d}\mathcal{H}^{d-1}\dt
\\&~~~\nonumber
+ C\sum_{i=1}^{P}\int_0^T E_{\mathrm{volume}}[\chi|\bar\chi](t) \dt
\end{align}
valid for almost every $T\in [0,T']$, all $\delta\in (0,1]$ and a constant $C>0$
that is independent of $\delta$.
\end{lemma}

\begin{proof}
We split the proof into two steps.

\textit{Proof of~\eqref{RelativeEntropyInequalityVolumeControl}.}
To compute the time evolution, 
note that the sign conditions on~$\vartheta_i$
from Definition~\ref{def:transportedWeights} 
of a family of transported weights
is precisely what is needed to have
\begin{align*}
E_{\mathrm{volume}}[\chi|\bar\chi](T) = 
\sum_{i=1}^{P}\int_{\Rd}(\chi_i(\cdot,T){-}\bar\chi_i(\cdot,T))\vartheta_i(\cdot,T)\dx. 
\end{align*}
Hence, we may make use of the evolution equations \eqref{EvolutionPhasesBVSolution}
for the indicator functions~$\chi_i$ of the BV solution 
which together with $\partial_t\bar\chi_i \ll |\nabla\bar\chi_i|$ and $\vartheta_i = 0$
on $\supp |\nabla\bar\chi_i|$ (see Definition~\ref{def:transportedWeights})  
yields for almost every $T\in [0,T']$
\begin{align*}
&E_{\mathrm{volume}}[\chi|\bar\chi](T) 
\\&
= E_{\mathrm{volume}}[\chi|\bar\chi](0)
+\sum_{i=1}^{P}\int_0^T\int_{\Rd}(\chi_i{-}\bar\chi_i)\partial_t\vartheta_i\dx\dt
+\sum_{i=1}^{P}\int_0^T\int_{\Rd} V_i\vartheta_i\,\mathrm{d}|\nabla\chi_i|\dt.
\end{align*}
We next use the convention~\eqref{NormalVelocitiesBVSolution} and rewrite
\begin{align*}
\sum_{i=1}^{P}\int_0^T\int_{\Rd} V_i\vartheta_i\,\mathrm{d}|\nabla\chi_i|\dt
=\sum_{i,j=1,i\neq j}^{P}\int_0^T\int_{I_{i,j}(t)} \vartheta_iV_{i,j}\,\mathrm{d}\mathcal{H}^1\dt.
\end{align*}
Furthermore, by adding and subtracting $(B\cdot\nabla)\vartheta_i$, 
an integration by parts, the fact that $\vartheta_i = 0$ on $\supp |\nabla\bar\chi_i|$
(see Definition~\ref{def:transportedWeights}), and the definition~\eqref{UnitNormalsBVSolution} 
of the measure theoretic unit normal, we obtain
\begin{align*}
&\sum_{i=1}^{P}\int_0^T\int_{\Rd}(\chi_i{-}\bar\chi_i)\partial_t\vartheta_i\dx\dt
\\&
= -\sum_{i=1}^{P}\int_0^T\int_{\Rd}(\chi_i{-}\bar\chi_i)(B\cdot\nabla)\vartheta_i\dx\dt
+\sum_{i=1}^{P}\int_0^T\int_{\Rd}(\chi_i{-}\bar\chi_i)(\partial_t\vartheta_i{+}(B\cdot\nabla)\vartheta_i)\dx\dt
\\&
=-\sum_{i=1}^{P}\int_0^T\int_{\Rd}(\chi_i{-}\bar\chi_i)\nabla\cdot(\vartheta_iB)\dx\dt
+\sum_{i=1}^{P}\int_0^T\int_{\Rd}(\chi_i{-}\bar\chi_i)\vartheta_i(\nabla\cdot B)\dx\dt
\\&~~~
+\sum_{i=1}^{P}\int_0^T\int_{\Rd}(\chi_i{-}\bar\chi_i)(\partial_t\vartheta_i{+}(B\cdot\nabla)\vartheta_i)\dx\dt
\\&
=\sum_{i=1}^{P}\int_0^T\int_{\Rd}
\frac{\nabla\chi_i}{|\nabla\chi_i|}\cdot\vartheta_iB
\,\mathrm{d}|\nabla\chi_i|\dt
+\sum_{i=1}^{P}\int_0^T\int_{\Rd}(\chi_i{-}\bar\chi_i)\vartheta_i(\nabla\cdot B)\dx\dt
\\&~~~
+\sum_{i=1}^{P}\int_0^T\int_{\Rd}(\chi_i{-}\bar\chi_i)(\partial_t\vartheta_i{+}(B\cdot\nabla)\vartheta_i)\dx\dt
\\&
=-\sum_{i,j=1,i\neq j}^{P}\int_0^T\int_{I_{i,j}(t)}
\vartheta_iB\cdot\xi_{i,j}\,\mathrm{d}\mathcal{H}^1\dt
\\&~~~
-\sum_{i,j=1,i\neq j}^{P}\int_0^T\int_{I_{i,j}(t)}
\vartheta_iB\cdot(\vec{n}_{i,j}{-}\xi_{i,j})\,\mathrm{d}\mathcal{H}^1\dt
\\&~~~
+\sum_{i=1}^{P}\int_0^T\int_{\Rd}(\chi_i{-}\bar\chi_i)\vartheta_i(\nabla\cdot B)\dx\dt
\\&~~~
+\sum_{i=1}^{P}\int_0^T\int_{\Rd}(\chi_i{-}\bar\chi_i)(\partial_t\vartheta_i{+}(B\cdot\nabla)\vartheta_i)\dx\dt
\end{align*}
for almost every $T\in [0,T']$. The combination of the previous
three displays thus proves~\eqref{RelativeEntropyInequalityVolumeControl} 
as asserted.

\textit{Step 2: Proof of~\eqref{eq:postProcessedRelEntropyWeightedVolume}.}
Starting point is of course \eqref{RelativeEntropyInequalityVolumeControl}
meaning that we need to estimate the term $R_{\mathrm{volume}}$.
First, we may infer based on the bound~\eqref{AdvectionEquationVolumeControl}
on the advective derivative of the transported weights $\vartheta_i$, 
the bound $|B| \leq C$ (see Definition~\ref{def:transportedWeights}),
H\"older's and Young's inequality as well as the bound~\eqref{eq:L2excess} that the estimate
\begin{align*}
|R_{\mathrm{volume}}| 
&\leq
\delta \sum_{i,j=1,i\neq j}^{P}\int_0^T\int_{I_{i,j}(t)}
|B\cdot\xi_{i,j}-V_{i,j}|^2 \,\mathrm{d}\mathcal{H}^{d-1}\dt
\\&~~~\nonumber
+ \frac{C}{\delta}\sum_{i,j=1,i\neq j}^{P}\int_0^T\int_{I_{i,j}(t)}
\vartheta_i^2 
\,\mathrm{d}\mathcal{H}^{d-1}\dt
\\&~~~\nonumber
+ \frac{C}{\delta}\sum_{i,j=1,i\neq j}^{P}\int_0^T\int_{I_{i,j}(t)}
1 - \vec{n}_{i,j}\cdot\xi_{i,j} \,\mathrm{d}\mathcal{H}^{d-1}\dt
\\&~~~\nonumber
+ C\sum_{i=1}^{P}\int_0^T E_{\mathrm{volume}}[\chi|\bar\chi](t) \dt
\end{align*}
holds true, uniformly over all $\delta\in (0,1)$. 
As $\vartheta_i = 0$ on $\supp|\nabla\bar\chi_i|$, 
$\vartheta_i\in W^{1,\infty}_{x,t}(\Rd\times [0,T'];[-1,1])$ and 
$\partial\bar\Omega_i$ is Lipschitz (see Definition~\ref{def:transportedWeights}),
we may further estimate $$\vartheta_i^2 \leq C(\dist^2(\cdot,\partial\bar\Omega_i)\wedge 1)
\leq C(\dist^2(\cdot,\bar I_{i,j})\wedge 1)$$ for all phases $i,j\in\{1,\ldots,P\}$ with $i\neq j$.
This, however, concludes the proof.
\end{proof}

We have everything in place to lift the quantitative inclusion
principle from Theorem~\ref{TheoremUniqueness} to
the conditional weak-strong uniqueness principle of
Proposition~\ref{PropositionUniquenessConditional}
(with an associated conditional stability estimate).

\begin{proof}[Proof of Proposition~\ref{PropositionUniquenessConditional}]
As our assumptions entail the applicability of Theorem~\ref{TheoremUniqueness} (which only requires the existence of a gradient flow calibration $((\xi_i)_{i\in\{1,\ldots,P\}},B)$ with respect to $\bar\Omega$), the stability estimate~\eqref{StabilityEstimate} concerning the interface error applies.
Recall from~\eqref{LengthControlXi} that $\dist^2(\cdot,\bar I_{i,j})\wedge 1
\leq C(1-|\xi_{i,j}|)$ for all $i,j\in\{1,\ldots,P\}$ with $i\neq j$.
Inserting these bounds into
the corresponding right hand side terms of~\eqref{eq:postProcessedRelEntropyWeightedVolume}, we obtain
\begin{align*}
E_{\mathrm{volume}}[\chi|{\bar \chi}](T')
\leq E_{\mathrm{volume}}[\chi|{\bar \chi}](0) + C e^{CT'} E[\chi|\xi](0) + C\int_0^{T'} E_{\mathrm{volume}}[\chi|{\bar \chi}](t) \dt
\end{align*}
for almost every $T'\in [0,T]$. The stability estimate~\eqref{eq:stabilityBulkError} 
for the bulk error is now a direct consequence of Gronwall's lemma.

It remains to prove the conditional weak-strong uniqueness statement.
To this end, note first that $\chi(\cdot,0) = \bar\chi(\cdot,0)$
almost everywhere in $\Rd$ entails $E[\chi|\xi](0)=0$ 
and $E_{\mathrm{volume}}[\chi|{\bar \chi}](0)=0$ 
as a consequence of the respective definitions~\eqref{DefinitionRelativeEntropy}
and~\eqref{volumeErrorFunctional}. In view of the stability estimate~\eqref{eq:stabilityBulkError}, 
this directly implies $E_{\mathrm{volume}}[\chi|{\bar \chi}](T')=0$ for almost every $T'\in [0,T]$.
It then follows from the coercivity properties of a family of transported weights 
(see Definition~\ref{def:transportedWeights}) that $\chi(\cdot,T') = \bar\chi(\cdot,T')$ almost everywhere in $\Rd$
for almost every $T'\in [0,T]$. This, however, is the desired weak-strong
uniqueness principle.
\end{proof}

\subsection{Weak-strong uniqueness and stability for varifold-$\BV$ solutions}

Before proceeding with the proof of Theorem~\ref{TheoremStabilityVarifoldSolution}, let us collect some additional compatibility properties of the varifold $\mathcal{V}_t$ and the indicator functions $\chi_i$ that may be inferred from Definition~\ref{DefinitionVarifoldBVsolution}. First, observe that given a varifold-$\BV$ solution $(\mathcal{V},\chi)$, for each $i\in\{1,\ldots,P\}$ and a.e.\ $t\in(0,\infty)$ the Radon--Nikodym derivative
\begin{align}
\label{eq:inverseMultiplicityTonegawa}
\rho_i(\cdot,t):=\frac{\mathrm{d}|\nabla\chi_i(\cdot,t)|}{\mathrm{d}\mu_t} \in [0,1] 
\end{align}
exists. Since $\mathcal{V}$ is a family of integral varifolds and since $\sum_{i=1}^P |\nabla \chi_i| \leq 2 \mathcal{H}^{d-1}$, for a.e.\ $t\in(0,\infty)$ it holds that for $\mu_t$-a.e.\ $x\in\mathbb{R}^d$
\begin{align}
\label{eq:propertiesInverseMultiplicityTonegawa}
\text{either}\quad \frac{1}{2} \sum_{i=1}^P \rho_i(x,t) = 1
\quad\text{or}\quad \frac{1}{2} \sum_{i=1}^P \rho_i(x,t) \leq \frac{1}{2}
\quad\text{holds true}.
\end{align}
Finally, since $\mu_t\llcorner\{\frac{1}{2}\sum_{i=1}^P\rho_i(\cdot,t)=1\}
= \mathcal{H}^{d-1}\llcorner\big(\{\frac{1}{2}\sum_{i=1}^P\rho_i(\cdot,t)=1\}\cap\bigcup_{i\neq j} I_{i,j}(t)\big)$
and $\mathcal{V}_t$ is rectifiable for a.e.\ $t\in(0,\infty)$, it follows that 
\begin{equation}
\begin{aligned}
\label{eq:approxTangentSpaceTonegawa}
&\mathcal{V}_t \llcorner \Big\{\frac{1}{2}\sum_{i=1}^P\rho_i(\cdot,t)=1\Big\}
\\&
= \frac{1}{2} \sum_{i=1}^P \bigg(\supp|\nabla\chi_i(\cdot,t)| \llcorner \Big\{\frac{1}{2}\sum_{i=1}^P\rho_i(\cdot,t){=}1\Big\}
\otimes \big(\delta_{\mathrm{Tan}^{d-1}_x(\supp|\nabla\chi_i(\cdot,t)|)}\big)_{x\in\supp|\nabla\chi_i(\cdot,t)|}\bigg)
\end{aligned}
\end{equation}
for a.e.\ $t\in(0,\infty)$. In particular, from Brakke's perpendicularity theorem it follows for a.e.\ $t\in(0,\infty)$ that
\begin{align}
\label{eq:brakkeOrthogonalityTonegawa}
h(\cdot,t) = \Big(h(\cdot,t)\cdot\frac{\nabla\chi_i(\cdot,t)}{|\nabla\chi_i(\cdot,t)|}\Big)
\frac{\nabla\chi_i(\cdot,t)}{|\nabla\chi_i(\cdot,t)|}
\end{align}
$\mathcal{H}^{d-1}$-a.e.\ on $\{\frac{1}{2}\sum_{i=1}^P\rho_i(\cdot,t)=1\}\cap\supp|\nabla\chi_i(\cdot,t)|$
for all $i\in\{1,\ldots,P\}$.

We now have all the ingredients for the proof of Theorem~\ref{TheoremStabilityVarifoldSolution}.
\begin{proof}[Proof of Theorem~\ref{TheoremStabilityVarifoldSolution}]
We first prove the estimate
\begin{align*}
E[\mathcal{V},\chi|\xi](T) \leq e^{Ct} E[\mathcal{V},\chi|\xi](0).
\end{align*}
As usual, the Gronwall inequality reduces this task to establishing the bound
\begin{align}
\label{eq:stabilityRelEntropyTonegawa}
E[\mathcal{V},\chi|\xi](T') \leq E[\mathcal{V},\chi|\xi](0)
+ C \int_{0}^{T'} E[\mathcal{V},\chi|\xi](t) \dt  
\end{align}
for a.e.\ $T'\in (0,T)$ and for some constant $C=C(\xi,B)>0$. The proof of this estimate can be reduced to the computation in the case of $\BV$ solutions as follows.
\textit{Step 1: Error control by relative entropy.}
By~\eqref{eq:inverseMultiplicityTonegawa} and~\eqref{eq:propertiesInverseMultiplicityTonegawa},
we obtain
\begin{align*}
E[\mathcal{V},\chi|\xi](t)
&= 2\int_{\Rd} 1 - \frac{1}{2}\sum_{i=1}^P \rho_i(\cdot,t) \,\mathrm{d}\mu_t
+ \sum_{i,j=1,\,i\neq j}^{P} \int_{I_{i,j}(t)} 1 - n_{i,j}\cdot\xi_{i,j}\dH
\\&
\geq \int_{\Rd \cap \{\frac{1}{2}\sum_{i=1}^P\rho_i(\cdot,t) \leq \frac{1}{2}\}} 1 \,\mathrm{d}\mu_t
+ \sum_{i,j=1,\,i\neq j}^{P} \int_{I_{i,j}(t)} 1 - n_{i,j}\cdot\xi_{i,j}\dH
\end{align*}
for a.e.\ $t\in (0,T)$. Hence, $E[\mathcal{V},\chi|\xi]$ inherits all of
the coercivity properties from the BV~case and in addition controls
the measure of higher multiplicity areas.

\textit{Step 2: Dissipation control.} Define $V_i(\cdot,t):=-h(\cdot,t)\cdot
\frac{\nabla\chi_i(\cdot,t)}{|\nabla\chi_i(\cdot,t)|}$ for all $i\in\{1,\ldots,P\}$
and a.e.\ $t\in(0,\infty)$. Again by~\eqref{eq:inverseMultiplicityTonegawa} 
and~\eqref{eq:propertiesInverseMultiplicityTonegawa}, and this time
also relying on~\eqref{eq:brakkeOrthogonalityTonegawa}, we then estimate
\begin{align*}
&- 2\int_{\Rd} |h(\cdot,t)|^2 \,\mathrm{d}\mu_t
\\&
= - 2\int_{\Rd} |h(\cdot,t)|^2\Big(1 - \frac{1}{2}\sum_{i=1}^P \rho_i(\cdot,t)\Big) \,\mathrm{d}\mu_t
- \sum_{i=1}^P \int_{\Rd} |V_i(\cdot,t)|^2 \,\mathrm{d}|\nabla\chi_i(\cdot,t)|
\\&
\leq - \int_{\Rd \cap \{\frac{1}{2}\sum_{i=1}^P\rho_i(\cdot,t) \leq \frac{1}{2}\}} |h(\cdot,t)|^2 \,\mathrm{d}\mu_t
- \sum_{i=1}^P \int_{\Rd} |V_i(\cdot,t)|^2 \,\mathrm{d}|\nabla\chi_i(\cdot,t)|
\end{align*}
for a.e.\ $t\in(0,T)$. 

\textit{Step 3: From BV to varifold mean curvature.}
We first bound by means of~\eqref{eq:propertiesInverseMultiplicityTonegawa}, \eqref{eq:approxTangentSpaceTonegawa}
and \textit{Step~1}
\begin{align*}
&- \sum_{i=1}^P \int_{\Rd} \Big(\mathrm{Id} - \frac{\nabla\chi_i(\cdot,t)}{|\nabla\chi_i(\cdot,t)|}
\otimes \frac{\nabla\chi_i(\cdot,t)}{|\nabla\chi_i(\cdot,t)|}\Big) : \nabla B \,\mathrm{d}|\nabla\chi_i(\cdot,t)|
\\&
\leq -2\int_{\Rd{\times}\mathbf{G}(d,d{-}1)} 
\mathrm{Id}_{\mathbf{G}(d,d{-}1)} \colon \nabla B \,\mathrm{d}\mathcal{V}_t 
+ C E[\mathcal{V},\chi|\xi](t)
\end{align*}
for a.e.\ $t\in(0,T)$. We then proceed using~\eqref{eq:meanCurvatureTonegawa},
H\"older's and Young's inequality, \textit{Step~1}, the identity~\eqref{eq:brakkeOrthogonalityTonegawa},
and finally the definition of~$V_i$ from \textit{Step~2} to obtain
\begin{align*}
&- \sum_{i=1}^P \int_{\Rd} \Big(\mathrm{Id} - \frac{\nabla\chi_i(\cdot,t)}{|\nabla\chi_i(\cdot,t)|}
\otimes \frac{\nabla\chi_i(\cdot,t)}{|\nabla\chi_i(\cdot,t)|}\Big) : \nabla B \,\mathrm{d}|\nabla\chi_i(\cdot,t)|
\\&
\leq - \sum_{i=1}^P \int_{\Rd} V_i(\cdot,t)\frac{\nabla\chi_i(\cdot,t)}{|\nabla\chi_i(\cdot,t)|}
\cdot B \,\mathrm{d}|\nabla\chi_i(\cdot,t)|
+ \frac{1}{2}\int_{\Rd \cap \{\frac{1}{2}\sum_{i=1}^P\rho_i(\cdot,t) \leq \frac{1}{2}\}} |h(\cdot,t)|^2 \,\mathrm{d}\mu_t
\\&~~~
+ C E[\mathcal{V},\chi|\xi](t)
\end{align*}
for a.e.\ $t\in(0,T)$.

\textit{Step 4: Conclusion.} In summary, it follows from the dissipation estimate~\eqref{eq:dissipationTonegawa},
the transport equation~\eqref{eq:evolutionPhaseTonegawa} in form of $\partial_t\chi_i=V_i|\nabla\chi_i|$,
and the previous three steps that
\begin{align*}
E[\mathcal{V},\chi|\xi](T') &\leq E[\mathcal{V},\chi|\xi](0)
- \frac{1}{2}\int_{0}^{T'} \int_{\Rd \cap \{\frac{1}{2}\sum_{i=1}^P\rho_i \leq \frac{1}{2}\}} 
|h|^2 \,\mathrm{d}\mu_t\dt
\\&~~~
+ C \int_{0}^{T'} E[\mathcal{V},\chi|\xi](t) \dt 
\\&~~~
- \sum_{i=1}^P \int_{0}^{T'} \int_{\Rd} |V_i|^2 \,\mathrm{d}|\nabla\chi_i|\dt
\\&~~~
- 2\sum_{i=1}^P \int_{0}^{T'} \int_{\Rd} V_i (\nabla\cdot\xi_i) \,\mathrm{d}|\nabla\chi_i| \dt
\\&~~~
- \sum_{i=1}^P \int_{0}^{T'} \int_{\Rd} V_i\frac{\nabla\chi_i}{|\nabla\chi_i|}
\cdot B \,\mathrm{d}|\nabla\chi_i| \dt
\\&~~~
+ \sum_{i=1}^P \int_{0}^{T'} \int_{\Rd} \Big(\mathrm{Id} {-} 
\frac{\nabla\chi_i}{|\nabla\chi_i|} \otimes \frac{\nabla\chi_i}{|\nabla\chi_i|}\Big) : \nabla B \,\mathrm{d}|\nabla\chi_i| \dt
\\&~~~
+ 2\sum_{i=1}^P \int_{0}^{T'} \int_{\Rd} \frac{\nabla\chi_i}{|\nabla\chi_i|} \cdot \partial_t\xi_i \,\mathrm{d}|\nabla\chi_i| \dt \dt 
\end{align*}
for a.e.\ $T'\in (0,T)$. From here on, one may estimate the last five terms by following the corresponding computations in the case of $\BV$ solutions line by line.

By similar arguments, one may lift the BV~computation for the bulk error 
functional to the case of a varifold solution in the
sense of Stuvard and Tonegawa \cite{StuvardTonegawaMCF} to establish the bound \eqref{eq:stabilityBulkError}. Details 
in this direction are left to the interested reader. We only
remark that the additional dissipation control on higher multiplicity areas
as provided by the second right hand side term of the last display is crucial.
\end{proof}

\section{Gradient flow calibrations at a smooth manifold}
\label{SectionLocalConstructionsTwoPhase}
The aim of this section is to construct a gradient flow calibration 
in the simple situation of one single connected manifold (with or without boundary) evolving by mean curvature,
see Lemma~\ref{LemmaBoundsLocalConstructionsTwoPhase} for the main result of this section. For the sake of simplicity, we stick to the case $d=2$, but the construction in this section immediately carries over to arbitrary dimensions.

In terms of a gradient flow calibration for a whole network of interfaces in the sense of
Definition~\ref{DefinitionCalibrationGradientFlow}, the vector fields constructed in
Lemma~\ref{LemmaBoundsLocalConstructionsTwoPhase} provide the local building block at a smooth two-phase interface 
of the network. These vector fields therefore only live in a small tubular neighborhood of the evolving interface, so that in the case of general networks a suitable localization in terms of a family of cutoff functions will be necessary.
We defer these considerations to Section~\ref{SectionPartitionOfUnity}.

First, we provide the precise setting of this section by giving a suitable
notion of neighborhood for a single space-time connected component of
the evolving network of interfaces.

\begin{definition}
\label{DefinitionStrongSolutionTwoPhase}
Let $d=2$ and $P \in \mathbb{N}$, $P\geq 2$. Let $(\bar{\Omega}_1,\ldots,\bar{\Omega}_P)$ 
be a strong solution to multiphase mean curvature flow in the sense of Definition~\ref{DefinitionStrongSolution}. 
Fix phases $i,j\in\{1,\ldots,P\}$ with $i\neq j$ such that~$\bar I_{i,j}=\bigcup_{t\in [0,T]}\bar{I}_{i,j}(t){\times}\{t\}$
is a non-trivial interface (possibly with boundary). A scale~$r_{i,j}\in (0,1]$ is called an 
\emph{admissible localization radius for the interface~$\bar I_{i,j}$} if for all~$t\in [0,T]$
the following two ball conditions are satisfied:
\begin{itemize}[leftmargin=0.7cm]
\item[i)] For each interior point~$x\in \bar I_{i,j}(t)$ it 
					holds~$\overline{B_{2r_{i,j}}(x{\pm}2r_{i,j}\bar{\vec{n}}_{i,j}(x,t))}\cap\bar I_{i,j}(t)=\{x\}$.
\item[ii)] In addition, for a boundary point~$x\in\partial\bar I_{i,j}(t)$ 
					 (i.e., a triple junction) denote by~$\bar{\vec{t}}_{i,j}(x,t)$
					 the tangent at~$x$ pointing away from the curve~$\bar I_{i,j}(t)$, and by~$\mathbb{H}_{\bar{\vec{t}}_{i,j}}(x,t)$
					 the half-space~$\{y\in\Rd[2]\colon (y - x)\cdot \bar{\vec{t}}_{i,j}(x,t) > 0\}$.
					 We then require that $\overline{B_{2r_{i,j}}(y)}\cap\bar I_{i,j}(t)=\{x\}$
					 for all $y\in\partial B_{2r_{i,j}}(x) \cap \overline{\mathbb{H}_{\bar{\vec{t}}_{i,j}}(x,t)}$.
\end{itemize}
\end{definition}

It follows from our regularity requirements in Definition~\ref{DefinitionStrongSolution}
that an admissible localization radius always exists. Moreover,
\begin{align}
\label{DiffeoTubularNeighborhood}
\Psi_{i,j}\colon \bar I_{i,j}\times (-r_{i,j},r_{i,j})
\to \Rd[2]\times [0,T],\quad
(x,t,s) \mapsto (x + s\bar{\vec{n}}_{i,j}(x,t),t)
\end{align}
defines a bijective map onto its image 
\begin{align}
\nonumber
\mathrm{im}(\Psi_{i,j}) &:=
\Psi_{i,j}(\bar I_{i,j}{\times}(-r_{i,j},r_{i,j}))
\\&~\label{eq:imageDiffeo}
= \bigcup_{t\in [0,T]}\bigg(\big\{\dist(\cdot,\bar I_{i,j}(t)) < r_{i,j}\big\} 
\setminus\bigcup_{x\in\partial\bar I_{i,j}(t)}\big(\mathbb{H}_{\bar{\vec{t}}_{i,j}}(x,t)\cap B_{r_{i,j}}(x)\big)\bigg)
\times \{t\},
\end{align}
and the inverse map is a diffeomorphism of class $(C^0_tC^4_x\cap C^1_tC^2_x)(\overline{\mathrm{im}(\Psi_{i,j})})$.
We may further split the inverse of the diffeomorphism~\eqref{DiffeoTubularNeighborhood}
as follows:
\begin{align*}
\Psi_{i,j}^{-1}\colon \mathrm{im}(\Psi_{i,j}) 
\to {\bar I_{i,j}}\times (-r_{i,j},r_{i,j}),
\quad (x,t) \mapsto \big(P_{i,j}(x,t),t,s_{i,j}(x,t)\big)
\end{align*}
where the map $s_{i,j}\colon \mathrm{im}(\Psi_{i,j}) 
\to (-r_{i,j},r_{i,j})$ represents a signed distance function
\begin{align}\label{SignedDistanceTwoPhase}
s_{i,j}(x,t) := 
\begin{cases}
					\dist(x,{\bar{I}_{i,j}}(t)), & (x,t)\in 
					\Psi_{i,j}\big({\bar I_{i,j}}{\times} [0,r_{i,j})\big), \\
					-\dist(x,{\bar{I}_{i,j}}(t)), & (x,t)\in 
					\Psi_{i,j}\big({\bar I_{i,j}}{\times} (-r_{i,j},0)\big),
					\end{cases}
\end{align}
and the map $P_{i,j}\colon \mathrm{im}(\Psi_{i,j}) 
\to\bigcup_{t\in [0,T]}{\bar I_{i,j}}(t)$ represents in each time slice
the projection onto the nearest point on the interface in the sense that
\begin{align}
P_{i,j}(x,t) := P_{{\bar I_{i,j}}(t)}(x) 
= \argmin_{y\in\bar I_{i,j}(t)} |y-x|,
\quad  (x,t)\in \mathrm{im}(\Psi_{\bar I_{i,j}}). 
\end{align}
Note that we have the identity
\begin{align}\label{defPI}
P_{i,j}(x,t) 
= x - s_{i,j}(x,t)\bar{\vec{n}}_{i,j}\big(P_{i,j}(x,t),t\big)
\in{\bar I_{i,j}}(t),
\quad  (x,t)\in \mathrm{im}(\Psi_{i,j}). 
\end{align}

As a consequence of our regularity assumptions on $\bar I_{i,j}$,
see again Definition~\ref{DefinitionStrongSolution},
we also know that (for the former, one may consult
Lemma~\ref{LemmaPropertiesSignedDistanceTwoPhase} below)
\begin{align}
\label{eq:regSignedDistanceProjection}
s_{i,j} \in (C^0_tC^5_x\cap C^1_tC^3_x)(\overline{\mathrm{im}(\Psi_{i,j})}),
\quad
P_{i,j} \in (C^0_tC^4_x\cap C^1_tC^2_x)(\overline{\mathrm{im}(\Psi_{i,j})}).
\end{align}
We may now introduce extensions of the unit normal~$\bar{\vec{n}}_{i,j}$ and
the scalar mean curvature~$H_{i,j}$ (oriented with respect to~$\bar{\vec{n}}_{i,j}$)
of the interface~$\bar I_{i,j}$ to the space-time domain~$\mathrm{im}(\Psi_{i,j})$.
Slightly abusing notation, we define
\begin{align}
\label{def:extensionNormal}
\bar{\vec{n}}_{i,j} \colon \mathrm{im}(\Psi_{i,j})
\to \mathbb{S}^1, \quad &(x,t) \mapsto  \nabla s_{i,j}(x,t),
\\
\label{def:extensionCurvature}
H_{i,j} \colon \mathrm{im}(\Psi_{i,j})
\to \Rd[], \quad &(x,t) \mapsto  (-\Delta s_{i,j})(P_{i,j}(x,t),t).
\end{align}
We note as a consequence of the definitions that
\begin{align}
\label{eq:regNormalCurvature}
\bar{\vec{n}}_{i,j} \in (C^0_tC^4_x\cap C^1_tC^2_x)(\overline{\mathrm{im}(\Psi_{i,j})}), \quad
H_{i,j} \in (C^0_tC^3_x\cap C^1_tC^1_x)(\overline{\mathrm{im}(\Psi_{i,j})}).
\end{align}

The following result provides a (two-phase version of a) gradient flow calibration for a single \textit{connected} interface.
Note that the velocity field $B$ can accomodate arbitrary tangential components, a fact we will exploit when constructing a velocity field for general networks in Section~\ref{sec:networkConstruction}.

\begin{lemma}\label{LemmaBoundsLocalConstructionsTwoPhase}
Let $d=2$ and $P \in \mathbb{N}$, $P\geq 2$. Let $(\bar{\Omega}_1,\ldots,\bar{\Omega}_P)$  
be a strong solution to multiphase mean curvature flow in the sense of Definition~\ref{DefinitionStrongSolution}. 
Fix~$i,j\in\{1,\ldots,P\}$ with $i\neq j$ such that~$\bar I_{i,j}=\bigcup_{t\in [0,T]}\bar{I}_{i,j}(t){\times}\{t\}$
is a non-trivial interface. Let $r_{i,j}\in (0,1]$ be an admissible localization radius for~$\bar I_{i,j}$
in the sense of Definition~\ref{DefinitionStrongSolutionTwoPhase}. Fix a space-time connected component
(of which there are finitely many) $\mathcal{T}=\bigcup_{t\in [0,T]}\mathcal{T}(t){\times}\{t\} \subset \bar I_{i,j}$
of the interface~$\bar I_{i,j}$. Denote by~$\Psi_{\mathcal{T}}$ the restriction of the
diffeomorphism~\eqref{DiffeoTubularNeighborhood} to~$\mathcal{T}{\times} (-r_{i,j},r_{i,j})$,
and its image by $\mathrm{im}(\Psi_{\mathcal{T}}):=\Psi_{\mathcal{T}}(\mathcal{T}{\times} (-r_{i,j},r_{i,j}))$.
	
	Let $\gamma \in C^0_tC^2_x(\overline{\mathrm{im}(\Psi_{\mathcal{T}}}))$ be an arbitrary map, and define the tangent vector field 
	\begin{align}
	\label{def:tangent}
	\bar\tau_{i,j} := J^\mathsf{T} \vec{\bar n}_{i,j}\colon\mathrm{im}(\Psi_{i,j})\to\mathbb{S}^1, \quad \bar\tau_{i,j}
	\in (C^0_tC^4_x\cap C^1_tC^2_x)(\overline{\mathrm{im}(\Psi_{i,j})}),
	\end{align}
	where $J$ denotes the counter-clockwise rotation by $90^\circ$.
	Then the vector fields $\xi_{i,j}\colon \mathrm{im}(\Psi_{\mathcal{T}})\to \mathbb{S}^1$
	and $B\colon \mathrm{im}(\Psi_{\mathcal{T}})\to \Rd[2]$ given by
	\begin{align}
	\label{DefinitionXiTwoPhase}
	\xiTwoPh_{i,j} &:= \bar{\vec{n}}_{i,j},
	\\
	\label{DefinitionVelocityTwoPhase}
	\BTwoPh &:= H_{i,j}\bar{\vec{n}}_{i,j} + 
	\gamma \bar\tau_{i,j}
	\end{align}
	satisfy $\xi_{i,j} \in (C^0_tC^4_x\cap C^1_tC^2_x)(\overline{\mathrm{im}(\Psi_{\mathcal{T}}}))$,
	$B\in C^0_tC^2_x(\overline{\mathrm{im}(\Psi_{\mathcal{T}}}))$, with corresponding quantitative estimates
	\begin{align}
	\label{eq:estimatesTwoPhaseXi}
	r_{i,j}^k|\nabla^k\xi_{i,j}| &\leq C, && k\in \{0,1,\ldots,4\},
	\\
	\label{eq:estimatesTwoPhaseXiTimeDeriv}
	r_{i,j}^{k+2}|\partial_t\nabla^k\xi_{i,j}| &\leq C,
	&& k\in \{0,1,2\},
	\\
	\label{eq:estimatesTwoPhaseVel}
	r_{i,j}^k|\nabla^k B| &\leq Cr_{i,j}^{-1} 
	+ C\sum_{l=0}^k r_{i,j}^l|\nabla^{l}\gamma|,
	&& k\in\{0,1,2\},
	\end{align}
	throughout the space-time domain~$\mathrm{im}(\Psi_{\mathcal{T}})$.	Moreover, it holds
	\begin{align}
	\label{eq:transportSignedDistanceB}
	\partial_t s_{i,j} + (B\cdot\nabla) s_{i,j} &= 0,
	\\
	\label{TransportEquationXiTwoPhase}
	\partial_t\xiTwoPh_{i,j} + (\BTwoPh\cdot\nabla)\xiTwoPh_{i,j} 
	+ (\nabla\BTwoPh)^\mathsf{T}\xiTwoPh_{i,j} &= 0, \\
	\label{LengthConservationTwoPhase}
	\xiTwoPh_{i,j}\cdot\partial_t\xiTwoPh_{i,j} 
	+ \xiTwoPh_{i,j}\cdot(\BTwoPh\cdot\nabla)\xiTwoPh_{i,j} &= 0, \\
	\label{DissipTwoPhase}
	|\BTwoPh\cdot\xiTwoPh_{i,j} + \nabla\cdot\xiTwoPh_{i,j}| &\leq Cr^{-2}_{i,j}\dist(\cdot,\bar I_{i,j})
	\end{align}
	throughout the space-time domain~$\mathrm{im}(\Psi_{\mathcal{T}})$.
	The constant in the estimates~\eqref{eq:estimatesTwoPhaseXi}, \eqref{eq:estimatesTwoPhaseXiTimeDeriv}, \eqref{eq:estimatesTwoPhaseVel} 
	and~\eqref{DissipTwoPhase} is independent of~$r_{i,j}$.
	
Furthermore, if we choose $\gamma$ to satisfy
	\begin{align}
	\label{RequirementForGamma}
	\bar{\vec{n}}_{i,j} \cdot \nabla \gamma=\gamma H_{i,j} - ({\bar{\tau}}_{i,j}\cdot \nabla) H_{i,j}
	\end{align}
	on the interface $\bar I_{i,j}$, we have the additional property
	\begin{align}
	\label{AdditionalPropertyTwoPhase}
	\big|\nabla B : \big(\xi_{i,j} \otimes J\xi_{i,j}+J\xi_{i,j} \otimes \xi_{i,j}\big)\big| + |\nabla B : \xi_{i,j} \otimes \xi_{i,j}| \leq C \dist(x,{\bar{I}}_{i,j}(t)).
	\end{align}
\end{lemma}

\begin{proof}
For ease of notation, we omit all indices, superscripts, 
	and arguments for the rest of the proof unless specifically required otherwise.
	Since~$\Psi$ represents in each time slice a tubular neighborhood diffeomorphism
	on scale~$r\in (0,1]$, we have $\max_{k=0,\ldots,5} r^k|\nabla^k s| \leq Cr$
	throughout~$\mathrm{im}(\Psi)$. From the definitions~\eqref{def:extensionNormal},
	\eqref{def:tangent}, \eqref{def:extensionCurvature} and~\eqref{defPI}, we then deduce
	$\max_{k=0,\ldots,4} r^k(|\nabla^k \bar{\vec{n}}| {+} |\nabla^k \bar\tau|
	{+}|\nabla^k P|)\leq C$ and $\max_{k=0,\ldots,3} r^k|\nabla^k H| \leq Cr^{-1}$.
	Due to \eqref{TransportSignedDistanceTwoPhase} and  \eqref{SignedDistanceNormal},
	it holds $\partial_t s = -H$. Hence,
	$\max_{k=0,\ldots,3}r^{k+2}|\partial_t \nabla^k s|\leq Cr$,
	$\max_{k=0,1,2} r^k(|\partial_t\nabla^k \bar{\vec{n}}| {+} |\partial_t\nabla^k \bar\tau|
	{+}|\partial_t\nabla^k P|)\leq C$ and finally
	$\max_{k=0,1} r^{k+2}|\partial_t\nabla^k H| \leq Cr^{-1}$.
	The estimates~\eqref{eq:estimatesTwoPhaseXi}--\eqref{eq:estimatesTwoPhaseVel}
	now directly follow from the definitions~\eqref{DefinitionXiTwoPhase}--\eqref{DefinitionVelocityTwoPhase}.
	
	It follows from \eqref{TransportSignedDistanceTwoPhase} and  \eqref{SignedDistanceNormal} below, 
	as well as from the orthogonality $\bar\tau \cdot \bar{\vec{n}}=0$  that the tangential term in the 
	definition of $\BTwoPh$ does not have an effect on the transport equation~\eqref{TransportSignedDistanceTwoPhase} for 
	the signed distance $s$, i.e., we have
	\begin{align*}
	\partial_t s = -\big(H\bar {\vec{n}}\cdot\nabla\big) s
	=-\big(\BTwoPh\cdot\nabla\big) s.
	\end{align*}
	We may take the gradient of this identity so that by definition of $\xiTwoPh$ we have
	\begin{align*}
	\partial_t\xiTwoPh = \nabla\partial_t s = 
	-\big(\BTwoPh\cdot\nabla\big)\xiTwoPh
	-\big(\nabla\BTwoPh\big)^\mathsf{T}\xiTwoPh,
	\end{align*}
	which proves \eqref{TransportEquationXiTwoPhase}. The validity of \eqref{LengthConservationTwoPhase}
	is evident from the fact that $|\xiTwoPh|^2 \equiv 1$.
	For the identity \eqref{DissipTwoPhase}, note first that
	$\BTwoPh\cdot\xiTwoPh = \bar {\vec{n}}\cdot\xiTwoPh = H$
	as a consequence of the orthogonality $\bar\tau \cdot \bar{\vec{n}}=0$. 
	By definition~\eqref{def:extensionNormal} and definition~\eqref{DefinitionXiTwoPhase},
	it holds $\nabla\cdot\xi = \Delta s$. Hence, $B\cdot\xi=H = -\nabla\cdot\xi + O(r^{-2}\dist(\cdot,\bar I))$
	in view of the definition~\eqref{eq:regNormalCurvature} and the regularity estimates for the signed distance.	
	This concludes the proof upon noticing that \eqref{AdditionalPropertyTwoPhase} follows by a straightforward computation.
\end{proof}

The preceding result relies on a number of well-known properties of the signed distance and the nearest point projection.
For further reference, we present them here in a separate statement.

\begin{lemma}\label{LemmaPropertiesSignedDistanceTwoPhase}
Let $d=2$ and $P \in \mathbb{N}$, $P\geq 2$. Let $(\bar{\Omega}_1,\ldots,\bar{\Omega}_P)$  
be a strong solution to multiphase mean curvature flow in the sense of Definition~\ref{DefinitionStrongSolution}. 
Fix~$i,j\in\{1,\ldots,P\}$ with $i\neq j$ such that~$\bar I_{i,j}=\bigcup_{t\in [0,T]}\bar{I}_{i,j}(t){\times}\{t\}$
is a non-trivial interface. Let $r_{i,j}\in (0,1]$ be an admissible localization radius for~$\bar I_{i,j}$
in the sense of Definition~\ref{DefinitionStrongSolutionTwoPhase}.

Then $s_{i,j} \in (C^0_tC^5_x\cap C^1_tC^3_x)(\overline{\mathrm{im}(\Psi_{i,j})})$.
The time evolution of the signed distance~$s_{i,j}$
is moreover given by transport along the flow of the mean curvature vector field 
in the sense that we have
\begin{align}\label{TransportSignedDistanceTwoPhase}
\partial_t s_{i,j} = -\big(H_{i,j}\bar{\vec{n}}_{i,j}\cdot\nabla\big) s_{i,j}
\quad \text{throughout }\mathrm{im}(\Psi_{i,j}).
\end{align}
The gradient of the projection map~\eqref{defPI} is given by
\begin{align}\label{GradientProjectionInterface}
\nabla P_{i,j} = \bar{\tau}_{i,j}\otimes
\bar{\tau}_{i,j} - s_{i,j}\nabla\vec{\bar{n}}_{i,j}
\quad \text{throughout }\mathrm{im}(\Psi_{i,j}).
\end{align}
Finally, for all $(x,t)\in\mathrm{im}(\Psi_{i,j})$
the derivatives of the signed distance $s_{i,j}$ are subject to the relations
\begin{align}
\label{SignedDistanceNormal}
\nabla s_{i,j}(x,t)  = \nabla s_{i,j}(y,t)|_{y=P_{i,j}(x,t)} & = \bar{\vec{n}}_{i,j}(x,t), \\
\label{LengthConservation1}
\nabla s_{i,j}(x,t)\cdot\partial_t\nabla s_{i,j}(x,t) &= 0, \\
\label{LengthConservation2}
\big(\nabla s_{i,j}(x,t)\cdot\nabla\big) \nabla s_{i,j}(x,t) &= 0,
\\
\label{TimeDerivativeSignedDistance}
\partial_t s_{i,j}(x,t) &=\partial_t s_{i,j}(y,t)|_{y=P_{i,j}(x,t)}.
\end{align}
\end{lemma} 

\begin{proof}
The representation of $s_{i,j}$ as a component of the inverse of $\Psi_{i,j}$ initially gives the regularity 
$s_{i,j} \in (C^0_tC^4_x\cap C^1_tC^2_x)(\overline{\mathrm{im}(\Psi_{i,j})})$.
A proof of the well-known identities~\eqref{TransportSignedDistanceTwoPhase}--\eqref{TimeDerivativeSignedDistance}
was given for instance in \cite[Lemma 10]{FischerHensel} with
the only difference being the precise form of the normal velocity of the evolving family of interfaces. Note that for instance \eqref{LengthConservation1} and \eqref{LengthConservation2} follow immediately from differentiating the constraint $|\nabla s_{i,j}|^2=1$ with respect to time and space, respectively.
The higher regularity for the signed distance $s_{i,j}$ and its time derivative $\partial_t s_{i,j}$ 
finally follows from~\eqref{eq:regNormalCurvature} and the identity~\eqref{SignedDistanceNormal}.
\end{proof}

\section{Gradient flow calibrations at a triple junction}
\label{SectionLocalConstructionsTriod}
The aim of this section is to construct a gradient flow calibration
in the model case of three regular interfaces meeting at a single triple junction.
The space-time trajectory of such a triple junction will be denoted
by $\mathcal{T}=\bigcup_{t\in [0,T]}\mathcal{T}(t){\times}\{t\}$
where $\mathcal{T}(t)\subset\Rd[2]$ is a singleton for all~$t\in [0,T]$.
For simplicity, we assume throughout the section 
that the triple junction consists of interfaces between the phases $1$, $2$, and $3$.
We will also use cyclical indices $i=1,2,3$ throughout the section, i.\,e.\ for simplicity we identify $i=0$ with $i=3$, $i=4$ with $i=1$, and so on; for instance, we may write $\xi_{0,1}$ instead of $\xi_{3,1}$ etc.

Similar to the previous one, the constructions provided in this
section are local in the sense that they are restricted to a sufficiently small space-time
neighborhood of the evolving triple junction~$\mathcal{T}$.
We first formalize this by introducing the notion of an \textit{admissible localization radius
$r=r_{\mathcal{T}}\in (0,1]$ for the triple junction~$\mathcal{T}$} in Definition~\ref{def:locRadiusTripleJunction}. 
We then state the main result of this section,
Proposition~\ref{prop:xi_triple_junction}, which provides all relevant properties of the 
constructed calibrations.

The construction of a calibration $\xi_{i,j}$ for $i,j \in \{1,2,3\}$ with $i\neq j$ 
along with an associated velocity field $B$ proceeds in three steps.
First, we extend the normal of the interface ${\bar{I}}_{i,j}$ of the strong solution to 
auxiliary vector fields $\widetilde\xi_{i,j}$ defined on the natural domain 
$\mathbb{H}_{i,j}:=\mathrm{im}(\Psi_{i,j}) \cap \bigcup_{t\in [0,T]} B_{r}(\mathcal{T}(t)) {\times} \{t\}$, 
see Figure~\ref{fig:triod}, on which the nearest point-projection onto ${\bar{I}}_{i,j}$ is well-defined and regular; 
see Definition~\ref{DefinitionStrongSolutionTwoPhase} and the subsequent discussion.
One should think of $\widetilde\xi_{i,j}$ as the main building block for the vector field
$\xi_{i,j}$ on the domain $\mathbb{H}_{i,j}$ containing the corresponding interface~${\bar{I}}_{i,j}$.
Similarly, we also construct auxiliary velocity fields $B_{i,j}$ on $\mathbb{H}_{i,j}$ by 
choosing its normal component as an extension of the scalar mean curvature $H_{i,j}$ of the interface ${\bar{I}}_{i,j}$. However, let us emphasize that we do not define $\widetilde\xi_{i,j}$ by just extending the unit normal vector field of ${\bar{I}}_{i,j}$ using the nearest point projection; indeed, to satisfy certain compatibility conditions, a more refined choice becomes necessary, see below for a more detailed discussion.

In the second step, we aim to identify a candidate vector field for the definition of $\xi_{i,j}$
outside of its natural domain of definition $\mathbb{H}_{i,j}$. 
The guiding principle is to make sure that the Herring angle condition at the triple junction
\begin{align}
\label{eq:auxHerring}
\sigma_{1,2}\bar{\vec{n}}_{1,2} + \sigma_{2,3}\bar{\vec{n}}_{2,3} + \sigma_{3,1}\bar{\vec{n}}_{3,1} = 0,
\end{align}
is satisfied by the calibrations $(\xi_{1,2},\xi_{2,3},\xi_{3,1})$
in the whole neighborhood of the triple junction:
\begin{align}
\label{eq:auxHerringAngleConditionExtensions}
	\sigma_{1,2}\xi_{1,2} + \sigma_{2,3} \xi_{2,3} + \sigma_{3,1} \xi_{3,1} = 0.
\end{align}
This allows us to define vector fields $(\xi_1,\xi_2,\xi_3)$ such that $\sigma_{i,i+1}\xi_{i,i+1}=\xi_i-\xi_{i+1}$
holds true for all cyclical indices $i=1,2,3$. The latter identity in turn is precisely the property of gradient flow calibrations necessary to differentiate the relative entropy functional in time, see for example equation~\eqref{RelativeEntropyRewritten}.

In order to achieve \eqref{eq:auxHerringAngleConditionExtensions} we note that it represents an angle condition.
As the union of the domains $\mathbb{H}_{i,i+1}$ for $i=1,2,3$ covers a neighborhood of the triple junction, 
see Figure~\ref{fig:triple_wedges}, we would like to define $\xi_{i+1,i-1}$ and 
$\xi_{i-1,i}$ on $\mathbb{H}_{i,i+1}$ by simply rotating $\widetilde \xi_{i,i+1}$, see Figure~\ref{fig:triple_rotation}.

However, as these domains overlap, see Figure~\ref{fig:triod}, we 
will have to interpolate between the competing definitions of the calibrations and velocities.
To this end, we partition the neighborhood of the triple junction into six wedges
centered at the triple junction as indicated in Figure~\ref{fig:triod_wedges}, 
three of which are denoted by $W_{i,j}=W_{j,i}$
and the remaining three by $W_i$. We require that
$B_{r}(\mathcal{T}(t))\cap{\bar{I}}_{i,j}\subset W_{i,j} \cup \mathcal{T}(t)
\subset\mathbb{H}_{i,j}$,
see Figure~\ref{fig:triod_wedges}, the first inclusion corresponding 
to a geometric smallness condition for the interfaces away from the triple junction. 
For the remaining three wedges it is required that 
$W_i\subset\bigcap_{j\neq i}\mathbb{H}_{i,j}$, see again Figure~\ref{fig:triod_wedges}. 
We will refer to these wedges as \textit{interpolation wedges}
since on them we will interpolate between the two competing definitions for $\xi_{i,i+1}$ and $B$.

\begin{figure}
	\centering
     \subcaptionbox{\label{fig:triod}}{
	  \centering
 	  \vspace{-\baselineskip}
	  \begin{tikzpicture}[scale=2.9]
	   	\node at (85:.9) {\phantom{$W_{1}$}};
		\node at (185:.9) {\phantom{$W_{2}$}};
		\node at (302:.9) {\phantom{$W_{2}$}};
		
		\node at (145:.9) {\phantom{$W_{1,2}$}};
		\node at (255:.9) {\phantom{$W_{2,3}$}};
		\node at (17:.9) {\phantom{$W_{3,1}$}};
	  \begin{scope}
	  	\clip (0,0) circle [radius=.8];
	 	\fill[pattern=north west lines,pattern color=blue,opacity=.2] (110:.8) -- (1,1.3) -- (1,-1)--  (290:.95) -- (110:.8);
		\fill[pattern=north east lines,pattern color=red,opacity=.2] (60:.8) -- (-1,1.3) -- (-1,-1)--  (240:.95) -- (60:.8);

	  	\draw[thick,color=red,opacity=.6] (0,0) arc (60:90:2);
		\draw[thick,opacity=.4] (0,0) arc (130:155:2);
		\draw[thick,color=blue] (0,0) arc (110:60:1);
		
		\draw[color=red,opacity=.6] (60:-.9)--(60:.8);
		\draw[opacity=.4] (130:-.8)--(130:.9);
		\draw[color=blue] (110:-.8)--(110:.8);
		
		\begin{scope}
			\clip (130:-1.5)--(130:1.5) -- (-2,2) -- (-2,-2) -- (130:-1.5);
		\foreach \i in {-35,...,35}{
				\draw[opacity=.2] (-1,-\i*0.025) -- (1,-\i*0.025);
			};
		\end{scope}

		\end{scope}
		
			\node[fill=white,inner sep=1pt] at (85:.25){$\bar{\Omega}_1$};
			\node[fill=white,inner sep=1pt] at (190:.35){$\bar{\Omega}_2$};
			\node[fill=white,inner sep=1pt] at (310:.3){$\bar{\Omega}_3$};
			
			\node at (3:.9){$\bar{I}_{3,1}$};
			\node at (157:.9){$\bar{I}_{1,2}$};
			\node at (230:.9){$\bar{I}_{2,3}$};
			
			
	\end{tikzpicture}
     }
     \subcaptionbox{\label{fig:triod_wedges}}{
	  \centering
	 \begin{tikzpicture}[scale=2.9]
	 \begin{scope}
	 	\clip (0,0) circle [radius=0.8];
	 	
	 	\begin{scope}
	 		\clip (0,0) -- (70:1.5) --(100:1.5) -- (0,0);
	 		\foreach \i in {0,...,15}
			{
				\draw[color=blue] (-.5,\i*.05) -- (.5,\i*.05);
				\draw[color=red,opacity=.6] (-.5,\i*.05+.025 ) -- (.5,\i*.05 +.025);
			};
	 	\end{scope}
	 	
	 		\begin{scope}
	 			\clip (0,0) -- (170:1.5) --(200:1.5) -- (0,0);
	 			\foreach \i in {-15,...,15}
				{
					\draw[opacity=.4] ($(120:1.5) + (210:\i*.05)$) -- ($(120:-1.5) + (210:\i*.05)$);
					\draw[color=red, opacity=.6] ($(120:1.5) + (210:\i*.05+.025)$) -- ($(120:-1.5) + (210:\i*.05+ .025)$);
				};
	 		\end{scope}
	 		
	 		\begin{scope}
	 			\clip (0,0) -- (295:1.5) --(305:1.5) -- (0,0);
	 			\foreach \i in {-15,...,15}
				{
					\draw[opacity=.3] ($(240:1.5) + (330:\i*.05)$) -- ($(240:-1.5) + (330:\i*.05)$);
					\draw[color=blue] ($(240:1.5) + (330:\i*.05+.025)$) -- ($(240:-1.5) + (330:\i*.05+ .025)$);
				};
	 		\end{scope}
	 	
		\draw[thick,color=red,opacity=.6] (0,0) arc (60:90:2);
		\draw[thick,opacity=.4] (0,0) arc (130:155:2);
		\draw[thick,color=blue] (0,0) arc (110:60:1);
		
		\draw[dashed,color=red,opacity=.6] (60:-.9)--(60:.8);
		\draw[dashed,opacity=.4] (130:-.8)--(130:.9);
		\draw[dashed,color=blue] (110:-.8)--(110:.8);

		\draw (0,0) -- (70:.8);
		\draw (0,0) -- (100:.8);

		\draw (0,0) -- (170:1);
		\draw (0,0) -- (200:1);

		\draw (0,0) -- (295:.8);
		\draw (0,0) -- (305:.8);

	\end{scope}
		
		\node at (85:.9) {$W_{1}$};
		\node at (185:.9) {$W_{2}$};
		\node at (302:.9) {$W_{3}$};
		
		\node at (145:.9) {$W_{1,2}$};
		\node at (255:.9) {$W_{2,3}$};
		\node at (17:.9) {$W_{3,1}$};
	\end{tikzpicture}
     }
     \caption{a) Sketch of a triple junction with phases $\bar{\Omega}_1$, $\bar{\Omega}_2$, and $\bar{\Omega}_3$; 
		and the corresponding interfaces. The bottom left to top right hatched region is the 
		domain $\mathbb{H}_{1,2}$, the horizontally hatched region is $\mathbb{H}_{2,3}$, and 
		the top left to bottom right hatching represents $\mathbb{H}_{3,1}$. b) The interpolation 
		wedges, shown as hatched, are given by $W_1$, $W_2$ and $W_3$. The remaining wedges 
		$W_{1,2}$, $W_{2,3}$ and $W_{3,1}$ contain the corresponding interfaces. }
        \label{fig:triod_and_wedges}
\end{figure}
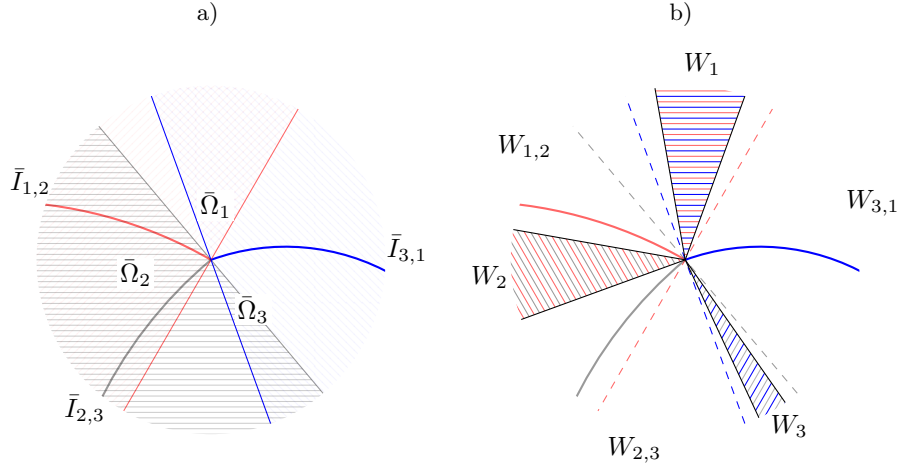

It turns out that in order to preserve our gradient flow calibration properties \eqref{TransportEquationXiTriod} and \eqref{DissipTriod} during the gluing construction, we need $C^1$ compatibility of these three constructions of the vector field $\xi_{i,i+1}$ and the velocity field $B$ at the triple junction. As we shall see, this necessitates a Taylor expansion ansatz for the construction of the vector fields $\xi_{i,i+1}$: While for a two-phase interface we were able to define $\xi_{i,i+1}$ by extending the unit normal vector field $\vec{n}_{i,i+1}$ of the interface ${\bar{I}}_{i,i+1}$ by orthogonal projection, even in the wedge $W_{i,i+1}$ (where the projection onto ${\bar{I}}_{i,i+1}$ is still well-defined) we now need to employ a more general ansatz of the form
\begin{align*}
\widetilde\xi_{i,i+1}(x,t) & := 
\Big(1-\frac{1}{2} \widehat \alpha^2_{i,i+1}(P_{\bar I_{i,i+1}}x,t)
	s_{i,i+1}^2(x,t)\Big) \bar{\vec{n}}_{i,i+1}(P_{\bar I_{i,i+1}}x,t)
\\&~~~~
+\widehat \alpha_{i,i+1}(P_{\bar I_{i,i+1}}x,t)
s_{i,i+1}(x,t)\bar{\tau}_{i,i+1}(P_{\bar I_{i,i+1}}x,t)
\end{align*}
for a suitably chosen function $\widehat \alpha_{i,i+1}$ (and with $s_{i,i+1}$ denoting the signed distance to the interface ${\bar{I}}_{i,i+1}$ and $\bar{\tau}_{i,i+1}:=J^\mathsf{T}\bar{\vec{n}}_{i,i+1}$). Note that the ansatz in particular features first-order terms in the (signed) distance to the interface (the role of the term involving $s_{i,i+1}^2$ being just that of an approximate normalization of the overall vector). It will turn out that for a suitable choice of the (a priori arbitrary) values of $\widehat \alpha_{i,i+1}$ at the triple junction, our ansatz ensures that the three values of $\nabla \widetilde\xi_{i,i+1}$ according to the three definitions of $\xi_{i,i+1}$ in the wedges $W_{i,i+1}$, $W_{i-1,i}$, and $W_{i+1,i+2}$ coincide at the triple junction. To see this, we also exploit the compatibility conditions satisfied by a strong solution to multiphase mean curvature flow at the triple junction.

Concerning the transport velocity field $B$, we observe that only its normal component $B\cdot \bar{\vec{n}}_{i,i+1}$ is defined naturally on the interface ${\bar{I}}_{i,i+1}$ of the strong solution, being given there by the mean curvature $H_{i,i+1}$ of the interface. Again, we shall require $C^1$ compatibility at the triple junction for the three different constructions. This motivates the ansatz in the wedge $W_{i,i+1}$
\begin{align*}
\widetilde B_{(i,i+1)}(x,t) 
		&:= H_{i,i+1}(P_{\bar I_{i,i+1}}x,t) \bar{\vec{n}}_{i,i+1}(P_{\bar I_{i,i+1}}x,t)
		\\&~~~~
		+ \widehat\alpha_{i,i+1}(P_{\bar I_{i,i+1}}x,t) \bar{\tau}_{i,i+1}(P_{\bar I_{i,i+1}}x,t) 
		\\&~~~~
		+ \beta_{i,i+1}(P_{\bar I_{i,i+1}}x,t) s_{i,i+1}(x,t) \bar{\tau}_{i,i+1}(P_{\bar I_{i,i+1}}x,t)
\end{align*}
(it turns out that a term of the form $s_{i,i+1} \bar{\vec{n}}_{i,i+1}$ is not needed). By a suitable choice of the $\bar\tau_{i,i+1} \cdot \nabla \widehat\alpha_{i,i+1}$ and the $\beta_{i,i+1}$ at the triple junction, using also again the compatibility conditions for the strong solution at the triple junction, we can again achieve compatibility of the three definitions of $B$ and $\nabla B$ in the three wedges $W_{i,i+1}$, $W_{i-1,i}$, $W_{i+1,i+2}$.

\begin{definition}
\label{def:locRadiusTripleJunction}
Let $d=2$ and $P \in \mathbb{N}$, $P\geq 2$. Let $(\bar{\Omega}_1,\ldots,\bar{\Omega}_P)$ 
be a strong solution to multiphase mean curvature flow in the sense of Definition~\ref{DefinitionStrongSolution}. 
Let $\mathcal{T}=\bigcup_{t\in [0,T]}\mathcal{T}(t){\times}\{t\}$ be an evolving triple junction present
in the network of interfaces of~$\bar\Omega$, 
and assume for simplicity that it is formed by the phases~$1,2$ and~$3$.
For each~$i\in\{1,2,3\}$, denote by $\mathcal{T}_{i,i+1}=\bigcup_{t\in [0,T]}\mathcal{T}_{i,i+1}(t){\times}\{t\}$
the unique space-time connected component of~$\bar I_{i,i+1}$ with an endpoint at the triple junction,
and let~$r_{i,i+1}\in (0,1]$ be an admissible localization
radius for the interface~$\bar I_{i,i+1}$ in the sense of Definition~\ref{DefinitionStrongSolutionTwoPhase}.

We call a scale~$r=r_{\mathcal{T}}\in (0,r_{1,2}\wedge r_{2,3} \wedge r_{3,1}]$
an \emph{admissible localization radius for the triple junction~$\mathcal{T}$} if 
there exists a wedge decomposition of the space-time neighborhood
$\mathcal{U}_r := \bigcup_{t\in [0,T]} B_r(\mathcal{T}(t)){\times}\{t\}$
of the triple junction in the following precise sense:
\begin{itemize}[leftmargin=0.7cm]
\item[i)] For each~$i\in\{1,2,3\}$ there exist sets $W_{i,i+1} := 
\bigcup_{t\in [0,T]} W_{i,i+1}(t){\times}\{t\}$ and
$W_{i}:=\bigcup_{t\in [0,T]} W_i(t){\times}\{t\}$
(in order to not rely on cyclical notation in later sections, 
we also define $W_{i+1,i}:=W_{i,i+1}$ for all $i\in\{1,2,3\}$)
subject to the following requirements:

First, for each~$t\in [0,T]$
the six sets~$(W_{i,i+1}(t))_{i\in\{1,2,3\}}$ and~$(W_{i}(t))_{i\in\{1,2,3\}}$
are pairwise disjoint, non-empty open subsets of~$B_{r}(\mathcal{T}(t))$ such that
\begin{align}
\label{eq:decompByWedges}
\bigcup_{i\in\{1,2,3\}} \overline{W_{i,i+1}(t)} \cup \overline{W_{i}(t)} 
= \overline{B_r(\mathcal{T}(t))}.
\end{align} 

Second, there exist six time-dependent unit vectors~$(X^{i}_{i,i+1},X^{i+1}_{i,i+1})_{i\in\{1,2,3\}}$
of class $C^1([0,T])$ such that for all~$i\in\{1,2,3\}$
and all~$t\in [0,T]$ we have
\begin{align}
\label{def:interfaceWedge}
W_{i,i+1}(t) &= \big(\mathcal{T}(t){+}\big\{\gamma_1 X^{i}_{i,i+1}(t) {+} 
\gamma_2 X^{i+1}_{i,i+1}(t)\colon \gamma_1,\gamma_2\in (0,\infty)\big\}\big) \cap B_{r}(\mathcal{T}(t)),
\\
\label{def:interpolWedge}
W_{i}(t) &= \big(\mathcal{T}(t){+}\big\{\gamma_1 X^{i}_{i,i+1}(t) {+} 
\gamma_2 X^{i}_{i-1,i}(t)\colon \gamma_1,\gamma_2\in (0,\infty)\big\}\big) \cap B_{r}(\mathcal{T}(t)).
\end{align}
For all~$i\in\{1,2,3\}$, the scalar products $X^{i}_{i,i+1}\cdot X^{i+1}_{i,i+1}\in (0,1)$  
and $X^{i}_{i,i+1}\cdot X^{i}_{i-1,i}$ are constant in time, and their values
only depend on the surface tensions.

Third, we require that for all~$i\in\{1,2,3\}$ and all~$t\in [0,T]$ it holds
\begin{align}
\label{eq:inlcusionInterfaceWedge}
B_{r}(\mathcal{T}(t)) \cap\mathcal{T}_{i,i+1}(t) &\subset W_{i,i+1}(t) \cup \mathcal{T}(t) 
\subset \mathbb{H}_{i,i+1}(t),
\\
\label{eq:inclusionInterpolWedge}
W_{i}(t) &\subset 
\mathbb{H}_{i,i+1}(t) \cap \mathbb{H}_{i,i-1}(t),
\end{align}
with the space-time domains~$\mathbb{H}_{i,i+1}:=\bigcup_{t\in [0,T]} \mathbb{H}_{i,i+1}(t){\times}\{t\}$
being defined by $\mathbb{H}_{i,i+1}(t) := \{x\in\Rd[2]\colon (x,t)\in \mathrm{im}(\Psi_{i,i+1})\}
\cap B_{r}(\mathcal{T}(t))$, $t\in [0,T]$.
\item[ii)] There exists a constant $C=C(\sigma)>0$ depending only on
					 the surface tensions such that for all~$i\in\{1,2,3\}$
\begin{align}
\label{eq:compDistances1}
&\max\{\dist(\cdot,\mathcal{T}), \dist(\cdot,\bar I_{i,i+1}), \dist(\cdot,\bar I_{i-1,i}) \} 
\leq C\min_{j=1,2,3}\dist(\cdot,\bar I_{j,j+1}) \quad\text{in } W_{i},
\\ \label{eq:compDistances3}
&\dist(\cdot,\bar I_{i,i+1})
\leq C\min_{j=1,2,3}\dist(\cdot,\bar I_{j,j+1}) \quad \text{in } W_{i,i+1}, 
\\
\label{eq:compDistances2}
&\dist(x,\mathcal{T}) \leq C\dist(\cdot,\bar I_{i,i+1})
\quad \text{in } W_{i-1,i} \cup W_{i+1,i-1}.
\end{align}
\end{itemize}
In view of the properties~\emph{\eqref{def:interfaceWedge}--\eqref{eq:inclusionInterpolWedge}},
we call each $W_{i,i+1}$ an \emph{interface wedge}, and each $W_{i}$ an \emph{interpolation wedge}. 
\end{definition}

The following lemma ensures the existence of an admissible localization radius
for a triple junction of the strong solution; in particular, that we can indeed find wedges with the desired properties.
Its proof is deferred to the end of Subsection~\ref{subsection:step_2}.

\begin{lemma}
\label{lem:existenceLocRadius}
Let the assumptions of Definition~\ref{def:locRadiusTripleJunction} be in place.
Then there exists an admissible localization radius for the triple junction~$\mathcal{T}$.
In fact, one may choose $r=\frac{1}{C}(r_{1,2}\wedge r_{2,3} \wedge r_{3,1})$ for a constant
$C=C(\sigma)\geq 1$ depending only on the surface tensions at the triple junction.
\end{lemma}

As a final remark concerning the construction of the calibrations 
and the velocity, one has to make sure that they have sufficiently 
high regularity at the triple junction. Naively, one might choose
 the auxiliary vector fields $\widetilde\xi_{i,j}$ as in the case of a 
single connected interface from the previous section, i.e., 
$\widetilde\xi_{i,j}:=\bar{\vec{n}}_{i,j}$
on $\mathbb{H}_{i,j}$. However, this ansatz after the rotation and interpolation steps only provides 
continuous vector fields $\xi_{i,j}$ which in general already fail to be 
Lipschitz at the triple junction, as we will see later.
Hence, in the first step we will employ a more careful expansion ansatz 
in terms of the signed distance function to ${\bar{I}}_{i,j}$,
see \eqref{AnsatzAuxiliaryXiHalfSpace}.

We are now in a position to state the main result of this section,
namely the existence of a gradient flow calibration in the
vicinity of an evolving triple junction.

\begin{proposition}
\label{prop:xi_triple_junction}
Let $d=2$ and $P \in \mathbb{N}$, $P\geq 2$. Let $(\bar{\Omega}_1,\ldots,\bar{\Omega}_P)$
be a strong solution to multiphase mean curvature flow in the sense of Definition~\ref{DefinitionStrongSolution}. 
Let $\mathcal{T}=\bigcup_{t\in [0,T]}\mathcal{T}(t){\times}\{t\}$ be an evolving triple junction present
in the network of interfaces of the strong solution, and assume for simplicity that it is formed by the phases~$1,2$ and~$3$.
Let $r=r_{\mathcal{T}}\in (0,1]$ be an associated admissible localization radius for the triple junction~$\mathcal{T}$
as given by Lemma~\ref{lem:existenceLocRadius}. In particular, for all distinct~$i,j\in\{1,2,3\}$,
let~$r_{i,j}$ be an admissible localization radius for~$\bar I_{i,j}$ in the sense of 
Definition~\ref{DefinitionStrongSolutionTwoPhase}.

Then there exists a constant~$\hat C=\hat C(\bar\Omega)\geq 1$, depending only on~$\bar\Omega$
but independent of~$(r_{i,j})_{i,j\in\{1,2,3\},i\neq j}$, so that
the radius~$\hat r:=\hat C^{-1} r$ has the following properties: 
Define $\mathcal{U}_{\hat r}:=\bigcup_{t \in [0,T]} B_{\hat r}(\mathcal{T}(t))\times\{t\}$. 
For all $i,j \in \{1,2,3\}$ with $i\neq j$, there exist continuous extensions
of the unit-normal vector fields and a continuous velocity field
\begin{align*}
	\xiTrJ_{i,j}\colon \mathcal{U}_{\hat r} \to \Rd[2], 
	\quad \BTrJ \colon \mathcal{U}_{\hat r} \to \Rd[2],
\end{align*}
which are of regularity $\xiTrJ_{i,j} \in (C^0_tC^2_x\cap C^1_tC^0_x)
(\overline{\mathcal{U}_{\hat r}}\setminus\mathcal{T})$
respectively $\BTrJ \in C^0_tC^2_x(\overline{\mathcal{U}_{\hat r}}\setminus\mathcal{T})$, and which are
furthermore subject to the following properties:
\begin{itemize}[leftmargin=0.7cm]
		\item[i)] It holds $\xi_{i,j}(x,t) = \bar{\vec{n}}_{i,j}(x,t)$ for all $t\in [0,T]$ and
				for all $x \in \mathcal{T}_{i,j}(t)\cap B_{\hat r}(\mathcal{T}(t))$,
				where~$\mathcal{T}_{i,j}$ is the unique space-time connected component of~$\bar I_{i,j}$
				with an endpoint at the triple junction~$\mathcal{T}$.
				We also have $|\xi_{i,j}(x,t)|=1$ for all $(x,t) \in \mathcal{U}_{\hat r}$.
				Expressing the triple junction in form of~$\mathcal{T}(t)=\{p(t)\}$,
				it holds~$B(p(t),t)=\frac{\mathrm{d}}{\mathrm{d}t}p(t)$ for all~$t\in [0,T]$.
		\item[ii)] We have the skew-symmetry relation $\xi_{i,j}=-\xi_{j,i}$.
		\item[iii)] The family of vector fields $(\xi_{i,j})_{i\neq j}$ satisfies the Herring angle condition \eqref{eq:auxHerring}
		            in the entire neighborhood of the triple junction, i.e., it holds for all $(x,t) \in \mathcal{U}_{\hat r}$
			\begin{align}
			\label{xi_compatibility}
				\sigma_{1,2}\xi_{1,2}(x,t) + \sigma_{2,3}\xi_{2,3}(x,t) + \sigma_{3,1}\xi_{3,1}(x,t) = 0.
			\end{align}
		\item[iv)] There exists a constant~$C=C(\bar\Omega)>0$, depending only on the strong solution~$\bar\Omega$
							 but independent of~$\hat r$, such that
							 throughout $\mathcal{U}_{\hat r}\setminus\mathcal{T}$ and for all $i,j\in\{1,2,3\}$ with $i\neq j$, 
							 we have the bounds
			\begin{align}
				\label{TransportEquationXiTriod}
				|\partial_t\xiTrJ_{i,j} + (\BTrJ\cdot\nabla)\xiTrJ_{i,j} + (\nabla\BTrJ)^\mathsf{T}\xiTrJ_{i,j}| 
				&\leq C\hat r^{-3}\dist(\cdot,{\bar{I}_{i,j}}),  
				\\
				\label{DissipTriod}
				|(\nabla\cdot\xiTrJ_{i,j}) + \BTrJ\cdot\xiTrJ_{i,j}| &\leq 
				C\hat r^{-2}\dist(\cdot,{\bar{I}_{i,j}}),
				\\
				\label{LengthConservationXiTriodNormalized}
				\xiTrJ_{i,j}\cdot\partial_t\xiTrJ_{i,j} + \xiTrJ_{i,j}\cdot(\BTrJ\cdot\nabla)\xiTrJ_{i,j} &= 0,
			\end{align}
			as well as
			\begin{align}
			\label{AdditionalPropertyThreePhase}
			\big|\nabla B : \big(\xi_{i,j} \otimes J\xi_{i,j}+J\xi_{i,j} \otimes \xi_{i,j}\big)\big| + |\nabla B : \xi_{i,j} \otimes \xi_{i,j}| \leq C \dist(x,{\bar{I}}_{i,j}(t)).
			\end{align}
\item[v)] Finally, there exists a constant~$C=C(\bar\Omega)>0$, depending only on the strong solution~$\bar\Omega$
							 but independent of~$\hat r$,
such that
\begin{align}
\label{boundDerivativesXi}
\hat r^2|\partial_t\xi_{i,j}| \leq C, \quad
\hat r^k|\nabla^k\xi_{i,j}| &\leq C, 
&& k\in\{0,1,2\},
\\\label{boundDerivativeB}
\hat r^k|\nabla^k B| &\leq C\hat r^{-1}, && k\in\{0,1,2\}
\end{align}
throughout the space-time domain $\mathcal{U}_{\hat r}\setminus\mathcal{T}$.
\end{itemize}
\end{proposition}

\subsection{Construction close to individual interfaces}\label{subsection:step_1}
For all what follows in this subsection,
let the assumptions of Proposition~\ref{prop:xi_triple_junction}
and the notation of Section~\ref{SectionLocalConstructionsTwoPhase} 
and Definition~\ref{def:locRadiusTripleJunction} be in place.
In this subsection, we for $i=1,2,3$ first introduce the previously discussed auxiliary vector fields $\widetilde\xi_{i,i+1}$ 
as extensions of the normals $\bar{\vec{n}}_{i,i+1}$ of the interfaces $\bar I_{i,j}$ to the domains $\mathbb{H}_{i,i+1}$.

We would like to define $\widetilde \xi_{i,i+1}$, and later also a candidate for the velocity field $B$, 
by an expansion ansatz in terms of the signed distance function $s_{i,i+1}$ to the interface ${\bar{I}}_{i,i+1}$,
see~\eqref{SignedDistanceTwoPhase}. To this end, we introduce two sets of coefficient functions $\alpha_{i,i+1}$ and $\beta_{i,i+1}$.
Recalling the definitions~\eqref{defPI}, 
\eqref{def:extensionNormal}, \eqref{def:extensionCurvature}, \eqref{def:tangent},
the ansatz for the extension $\widetilde\xi_{i,i+1}$ of the normal 
vector field $\bar{\vec{n}}_{i,i+1}|_{{\bar{I}}_{i,i+1}}$ then is
\begin{equation}
\label{AnsatzAuxiliaryXiHalfSpace}
\begin{aligned}
	\widetilde\xi_{i,i+1}(x,t)  
	& := 
	\bar{\vec{n}}_{i,i+1}(x,t)
	\\&~~~~
	+\alpha_{i,i+1}(x,t)
	s_{i,i+1}(x,t)\bar{\tau}_{i,i+1}(x,t) 
	\\&~~~~ 
	-\frac{1}{2} \alpha^2_{i,i+1}(x,t)
	s_{i,i+1}^2(x,t)\bar{\vec{n}}_{i,i+1}(x,t) .
\end{aligned}
\end{equation}
Furthermore, we set $\widetilde \xi_{i+1,i} := - \widetilde \xi_{i,i+1}$ for $t\in [0,T]$, $x \in \mathbb{H}_{i,i+1}(t)$, and $i\in\{1,2,3\}$.

Apart from the family of vector fields $(\xi_{i,j})_{i\neq j}$, the notion of gradient flow calibrations also requires a suitably defined velocity field $B$. For its construction in the vicinity
of a triple junction, we introduce in a first step certain auxiliary symmetric 
velocity fields $\widetilde B_{(i,j)}=\widetilde B_{(j,i)}$.
To this end, we employ the expansion ansatz
\begin{equation}\label{AnsatzAuxiliaryVelocityHalfSpace}
\begin{aligned}
		\widetilde B_{(i,i+1)}(x,t) 
		&:= H_{i,i+1}(x,t) \bar{\vec{n}}_{i,i+1}(x,t)
		\\&~~~~
		+ \alpha_{i,i+1}(x,t) \bar{\tau}_{i,i+1}(x,t) 
		\\&~~~~
		+ \beta_{i,i+1}(x,t) s_{i,i+1}(x,t) \bar{\tau}_{i,i+1}(x,t)
\end{aligned}
\end{equation}
for every $i\in\{1,2,3\}$, $t\in [0,T]$ and $x \in \mathbb{H}_{i,i+1}(t)$. We also set $\widetilde B_{(i+1,i)} := \widetilde{B}_{(i,i+1)}$.

To complete the definition of $\widetilde \xi_{i,i+1}$ and $\widetilde B_{(i,i+1)}$, it remains to specify $\alpha_{i,i+1}$ and $\beta_{i,i+1}$. We construct $\alpha_{i,i+1}$ as
\begin{align}
\label{eq:defAlpha}
\alpha_{i,i+1}\colon \mathbb{H}_{i,i+1} \to \Rd[],\quad
(x,t) \mapsto \widehat\alpha_{i,i+1}(P_{i,i+1}(x,t),t),
\end{align}
being defined by projection onto~$\bar I_{i,i+1}$ in terms of the solution
\begin{align}
\widehat\alpha_{i,i+1}\colon\bigcup_{t\in[0,T]}
\mathcal{T}_{i,i+1}(t)\times\{t\} \to \mathbb{R}
\end{align}
to the following ODE posed on the space-time connected component~$\mathcal{T}_{i,i+1}$ 
of the interface~${\bar{I}}_{i,i+1}$ with initial
condition at the triple junction $\mathcal{T}(t)=\{p(t)\}$:
\begin{align}
\label{ODE_alpha}
  \begin{cases}
	\widehat\alpha_{i,i+1}(p(t),t) & = \bar{\tau}_{i,i+1}(p(t),t)\cdot\frac{\mathrm{d}}{\mathrm{d}t}p(t) \\
	\left(\bar{\tau}_{i,i+1}(x,t) \cdot \nabla \right) \widehat\alpha_{i,i+1}(x,t) & = H_{i,i+1}^2(x,t),
	\qquad\qquad\quad x\in  \mathcal{T}_{i,i+1}(t).
  \end{cases}
\end{align}
Second, we define for each~$i\in\{1,2,3\}$
the function $\beta_{i,i+1}\colon\mathbb{H}_{i,i+1} \to \mathbb{R}$ 
by means of
\begin{align}\label{def_beta}
	\beta_{i,i+1} := - \alpha_{i,i+1} H_{i,i+1}
	- (\bar{\tau}_{i,i+1} \cdot \nabla) H_{i,i+1}.
\end{align}
We next briefly present the regularity properties of $\widetilde \xi_{i,i+1}$.

\begin{lemma}
\label{lemma:regularity_tilde_xi}
Let the assumptions of Proposition~\ref{prop:xi_triple_junction}
be in place, in particular the notation of Definition~\ref{def:locRadiusTripleJunction}.
For all phases $i\in\{1,2,3\}$, the auxiliary vector field
$\widetilde\xi_{i,i+1}$ is of class $(C^0_tC^2_x\cap C^1_tC^0_x)(\overline{\mathbb{H}_{i,i+1}})$.
More precisely, we have the estimates
\begin{align}
\label{eq:boundsDerivativesTildeXi}
|\widetilde\xi_{i,i+1}| +
r_{i,i+1}|\nabla \widetilde\xi_{i,i+1}|
+ r_{i,i+1}^2\big(|\nabla^2\widetilde\xi_{i,i+1}|{+}|\partial_t\widetilde\xi_{i,i+1}|\big)
\leq C
\end{align}
for some~$C=C(\bar\Omega)>0$ only depending on~$\bar\Omega$
but independent of~$(r_{i,j})_{i,j\in\{1,2,3\},i\neq j}$.
\end{lemma}

\begin{proof}
\textit{Step 1 (Qualitative differentiability):}
In view of the expansion ansatz~\eqref{AnsatzAuxiliaryXiHalfSpace},
the regularity~\eqref{eq:regSignedDistanceProjection} of the signed distance~$s_{i,i+1}$,
the regularity~\eqref{eq:regNormalCurvature} of the normal~$\bar {\vec{n}}_{i,i+1}$, 
and the regularity~\eqref{def:tangent} of the tangent~$\bar\tau_{i,i+1}$,
it suffices to prove that $\alpha_{i,i+1}\in (C^0_tC^2_x\cap C^1_tC^0_x)(\overline{\mathbb{H}_{i,i+1}})$
to conclude $\widetilde\xi_{i,i+1} \in (C^0_tC^2_x\cap C^1_tC^0_x)(\overline{\mathbb{H}_{i,i+1}})$.

We start with the time regularity of the initial value of the ODE~\eqref{ODE_alpha}.
Using the evolution equation 
$\frac{\mathrm{d}}{\mathrm{d}t} p(t) \cdot \bar{\vec{n}}_{i,i+1}(p(t),t) = H_{i,i+1}(p(t),t)$ 
at the triple junction we get
\begin{align}\label{eq:evolutionTripleJunction}
	\frac{\mathrm{d}}{\mathrm{d}t} \tj(t) = H_{i,i+1}(p(t),t) \bar{\vec{n}}_{i,i+1}(p(t),t)  
	+ \Big( \bar{\tau}_{i,i+1}(p(t),t) \cdot \frac{\mathrm{d}}{\mathrm{d}t} \tj(t) \Big) \bar{\tau}_{i,i+1}(p(t),t)
\end{align}
for $i\in\{1,2,3\}$.
Note that this identity is equivalent to the second-order compatibility condition~\eqref{SecondOrderCompDef}.
We can now identify the term in the parenthesis as
$\alpha_{i,i+1}(p(t),t)$ due to the intial value of the ODE~\eqref{ODE_alpha} and multiply the 
above equation with the rotation matrix~$J$ in order to deduce
\begin{align}\label{eq:compatiblity_xi0}
	- H_{1,2} \,  \bar{\tau}_{1,2}  + \alpha_{1,2} \,  \bar{\vec{n}}_{1,2} &  = 
	- H_{2,3} \,  \bar{\tau}_{2,3}  + \alpha_{2,3} \,  \bar{\vec{n}}_{2,3}  = 
	- H_{3,1} \,  \bar{\tau}_{3,1}  + \alpha_{3,1} \,  \bar{\vec{n}}_{3,1}
\end{align}
at the triple junction.

For $i\neq j$, we then define $c_{i,j}:=\vec{\bar{n}}_{i,i+1}(p(t),t)\cdot\vec{\bar{n}}_{j,j+1}(p(t),t)$ and
$d_{i,j}:=\vec{\bar{n}}_{i,i+1}(p(t),t)\cdot\bar{\tau}_{j,j+1}(p(t),t)$ and notice that they are indeed constant in time 
due to only depending on the angles between interfaces determined by the surface tensions.
Furthermore, note $|c_{i,j}| <1$ as the surface tensions satisfy the triangle inequality.
Multiplying~\eqref{eq:compatiblity_xi0} with the normal $\vec{\bar{n}}_{i,i+1}(p(t),t)$ thus yields
\begin{align*}
\alpha_{i,i+1}(p(t),t)=-H_{j,j+1}(p(t),t)d_{i,j}+\alpha_{j,j+1}(p(t),t)c_{i,j}
\end{align*}
for all $i\neq j$ and all~$t\in [0,T]$. Switching the roles of $i$ and $j$ in the previous formula entails 
\begin{align}
\label{eq:repInitialValueAlpha}
	\alpha_{i,i+1}(p(t),t)
=-(1{-}c_{i,j}^2)^{-1} \big(H_{j,j+1}(p(t),t)d_{i,j}+H_{i,i+1}(p(t),t)d_{i,j}c_{i,j}\big)
\end{align}
for all $i\neq j$ and all~$t\in [0,T]$. Hence, we deduce~$t\mapsto \alpha_{i,i+1}(p(t),t) \in C^1([0,T])$.

We proceed by explicitly integrating the ODE~\eqref{ODE_alpha}, and exploiting
the regularity~\eqref{eq:regNormalCurvature} of the extended scalar mean curvature~$H_{i,i+1}$,
as well as the regularity of the space-time curve~$\mathcal{T}_{i,i+1}$.
Let us make this argument explicit.
To this end, we first choose a $C^5$ diffeomorphic parametrization
$\gamma_0\colon [0,1]\to \mathcal{T}_{i,i+1}(0)$ of the initial curve $\mathcal{T}_{i,i+1}(0)$
such that $\gamma_0(0)=p(0)$,
and then define $\gamma_t(s):=\psi^t(\gamma_0(s))$ for all $(s,t)\in [0,1]{\times}[0,T]$
by means of the flow maps from Definition~\ref{DefinitionSmoothlyEvolvingPartition}.
Capturing orientation by means of the constant $c_{\pm}=\bar\tau_{i,i+1}(\gamma_t(s),t)
\cdot\frac{\partial_s\gamma_t(s)}{|\partial_s\gamma_t(s)|}\in\{\pm 1\}$,
we set
\begin{align}
\label{eq:ansatzCoefficient0}
\widetilde\alpha_{i,i+1}(s,t) := 
\bar{\tau}_{i,i+1}(p(t),t)\cdot\frac{\mathrm{d}}{\mathrm{d}t}p(t)
+ c_{\pm}\int_0^s H^2_{i,i+1}(\gamma_t(\ell),t)
|\partial_s\gamma_t(\ell)|\,\mathrm{d}\ell
\end{align}
for all $(s,t) \in [0,1]{\times}[0,T]$, and then have
\begin{align}
\label{eq:ansatzCoefficient1}
\widehat\alpha_{i,i+1}(x,t) = \widetilde\alpha_{i,i+1}\big((\gamma_t)^{-1}(x),t\big) 
\end{align}
for all $t\in [0,T]$ and all $x\in \mathcal{T}_{i,i+1}(t)$. 
The validity of~\eqref{ODE_alpha} is indeed a simple consequence
of the ansatz~\eqref{eq:ansatzCoefficient0}, the definition~\eqref{eq:ansatzCoefficient1} 
and the chain rule. The required regularity
$\alpha_{i,i+1}\in (C^0_tC^2_x\cap C^1_tC^0_x)(\overline{\mathbb{H}_{i,i+1}})$ in turn
follows from the regularity~\eqref{eq:regSignedDistanceProjection} of the projection,
the regularity~\eqref{def:tangent} of the tangent, the regularity~\eqref{eq:regNormalCurvature} 
of the curvature, and the regularity condition~\textit{ii)} of Definition~\ref{DefinitionSmoothlyEvolvingPartition}.

\textit{Step 2 (Quantitative estimates):}
Since in each time slice the map~$\Psi_{i,i+1}$ from~\eqref{DiffeoTubularNeighborhood}
represents a tubular neighborhood diffeomorphism on scale~$r_{i,i+1}\in (0,1]$, we deduce
\begin{align}
\label{Bound2ndDerivativeSignedDistance}
r_{i,i+1}^k|\nabla^{k}s_{i,i+1}|\leq Cr_{i,i+1},\quad k\in\{0,1,2,3,4,5\},
\end{align}
and thus from the definitions~\eqref{def:extensionNormal} 
and~\eqref{def:tangent} that
\begin{align}
\label{Bound2ndDerivativeNormalTangent}
r_{i,i+1}^k|\nabla^{k}\bar{\vec{n}}_{i,i+1}| + r^k|\nabla^{k}\bar\tau_{i,i+1}|
\leq C, \quad k\in\{0,1,2,3,4\}.
\end{align}
The previous estimates in addition entail the following bounds for the nearest-point projections
due to~\eqref{defPI} (in form of $P_{i,i+1}(x,t)=x{-}s_{i,i+1}(x,t)\nabla s_{i,i+1}(x,t)$) 
and the (extensions of the) scalar mean curvatures  
due to~\eqref{def:extensionCurvature}
\begin{align}
\label{BoundProjection}
r_{i,i+1}^k|\nabla^k P_{i,i+1}| &\leq Cr_{i,i+1}, &&\quad k\in\{1,2,3,4\},
\\ \label{BoundMeanCurvature}
r_{i,i+1}^k|\nabla^{k} H_{i,i+1}| &\leq Cr_{i,i+1}^{-1}, &&\quad k\in\{0,1,2,3\}.
\end{align}
As a consequence of the evolution equation~\eqref{TransportSignedDistanceTwoPhase}
for the signed distance, we also obtain the following estimate
on the time derivatives
\begin{equation}
\begin{aligned}
\label{boundTimeDerivatives}
r_{i,i+1}|\partial_t s_{i,i+1}| &+ r_{i,i+1}^2|\partial_t \bar{\vec{n}}_{i,i+1}|
+ r_{i,i+1}^2|\partial_t \bar\tau_{i,i+1}| 
\\&
+ r_{i,i+1}|\partial_t P_{i,i+1}|
+ r_{i,i+1}^3|\partial_t H_{i,i+1}|
\leq C.
\end{aligned}
\end{equation}
It then follows from the representations~\eqref{eq:repInitialValueAlpha}
and~\eqref{eq:evolutionTripleJunction} that
\begin{align}
\label{eq:boundsInitialValues}
r_{i,i+1}|\alpha_{i,i+1}(p(t),t)| + r_{i,i+1}\Big|\frac{\mathrm{d}}{\mathrm{d}t}p(t)\Big|
+ r_{i,i+1}^3\Big|\frac{\mathrm{d}}{\mathrm{d}t}\alpha_{i,i+1}(p(t),t)\Big|
\leq C 
\end{align}
for all~$t\in [0,T]$.

We next claim that
\begin{align}
\label{BoundCoefficientsAlpha}
\max_{k=0,1,2} r_{i,i+1}^{k}|\nabla^k\alpha_{i,i+1}|
+ r_{i,i+1}^2|\partial_t\alpha_{i,i+1}| &\leq Cr_{i,i+1}^{-1}.
\end{align}
Once this is established, the asserted bound~\eqref{eq:boundsDerivativesTildeXi} 
for the derivatives of the vector fields $\widetilde\xi_{i,i+1}$
can then be directly inferred from the ansatz~\eqref{AnsatzAuxiliaryXiHalfSpace}
and the above regularity estimates.
The estimate~\eqref{BoundCoefficientsAlpha}, however, is a consequence of 
the regularity estimates~\eqref{Bound2ndDerivativeNormalTangent}--\eqref{eq:boundsInitialValues}
and the representations~\eqref{eq:defAlpha}--\eqref{ODE_alpha}
in form of~$\nabla\alpha_{i,i+1}=H^2_{i,i+1}(\bar\tau_{i,i+1}\otimes\bar\tau_{i,i+1} : \nabla P_{i,i+1})\bar\tau_{i,i+1}$.
For later reference, we note that
\begin{align}
\label{eq:errorDerivAlpha}
(\bar\tau_{i,i+1}\cdot\nabla)\alpha_{i,i+1}
= H^2_{i,i+1} + O(r_{i,i+1}^{-3}|s_{i,i+1}|)
\end{align}
due to~\eqref{GradientProjectionInterface}, \eqref{Bound2ndDerivativeSignedDistance}
and~\eqref{BoundMeanCurvature}.
\end{proof}

Ultimately, the point of the ansatz \eqref{AnsatzAuxiliaryXiHalfSpace}
is to ensure both \eqref{xi_compatibility} throughout $B_r(\mathcal{T}(t))$ and sufficiently high regularity of $\xi_{i,j}$ 
at the triple junction. Moreover, the relations \eqref{ODE_alpha} and \eqref{def_beta} also
holding true on the interface away from the triple junction turns out to be crucial to obtain the 
estimates \eqref{TransportEquationXiTriod} and \eqref{DissipTriod} 
on the whole space-time domain. The first step towards these goals are the following relations, 
which in particular yield that---after rotation $R_{(i,j)}$---the vector fields 
are compatible to second order at the triple junction:

\begin{lemma}
\label{lemma:compatibility_xi}
  Let the assumptions of Proposition~\ref{prop:xi_triple_junction} be in place.
  For each pair $i,j\in\{1,2,3\}$ there exist uniquely determined rotations $R_{(i,j)} \in SO(2)$,
	only depending on the restriction~$(\sigma_{i,j})_{i,j = 1,2,3}$ of the admissible matrix of surface tensions
	for the given strong solution~$\bar\Omega$, such that 
	\begin{align}\label{def_rotation}
		\bar{\vec{n}}_{i,i+1}(\cdot,t) & = R_{(i,j)} \bar{\vec{n}}_{j,j+1}(\cdot,t)
		\quad\text{at } \mathcal{T}(t)
	\end{align}
	for all $t\in [0,T]$, and
	\begin{align}
		R_{(i,j)} R_{(j,i)}& = \operatorname{Id},\label{rotation_inverse}\\
		R_{(i,i-1)}R_{(i-1,i+1)}R_{(i+1,i)} &= \operatorname{Id}.\label{rotation_full_circle}
	\end{align}
	Furthermore, the ansatz \eqref{AnsatzAuxiliaryXiHalfSpace} satisfies
	the first-order compatibility conditions at the triple junction:
	\begin{align}
	  \label{compatibility_xi_zero}
		\widetilde \xi_{i,i+1}(\cdot,t) &= R_{(i,j)} \widetilde \xi_{j,j+1}(\cdot,t)
		&&\text{at } \mathcal{T}(t),
		\\
		\label{compatibility_xi_first}
		\nabla \widetilde \xi_{i,i+1}(\cdot,t) & = 
		\nabla \big( R_{(i,j)} \widetilde \xi_{j,j+1} \big)(\cdot,t)
		&&\text{at } \mathcal{T}(t),
	\end{align}
	for all $t\in [0,T]$.
\end{lemma}

\begin{proof}
	Identity \eqref{def_rotation} uniquely defines $R_{(i,j)}$.
	It is immediate from the ansatz \eqref{AnsatzAuxiliaryXiHalfSpace} and \eqref{def_rotation} that 
	the zero-order condition \eqref{compatibility_xi_zero} is satisfied.	
	The two properties \eqref{rotation_inverse} and \eqref{rotation_full_circle} follow from
	\begin{align}
		R_{(i,j)} R_{(j,i)} \bar{\vec{n}}_{i,i+1} & = \bar{\vec{n}}_{i,i+1},\\
		R_{(i,i-1)}R_{(i-1,i+1)}R_{(i+1,i)}\bar{\vec{n}}_{i,i+1} & = \bar{\vec{n}}_{i,i+1},
	\end{align}
	which follow straightforwardly from iterating \eqref{def_rotation}.
	Therefore, it is sufficient to prove the remaining statement~\eqref{compatibility_xi_first}
	for $j=i+1$, as it then follows automatically for $j=i-1$
	by \eqref{rotation_inverse} and	\eqref{rotation_full_circle} that
	at~$\mathcal{T}(t)$ it holds
	\begin{align*}
	\nabla \big( R_{(i,i-1)} \widetilde \xi_{i-1,i} \big)(\cdot,t)
	&= R_{(i,i+1)}\nabla \big( R_{(i+1,i-1)} \widetilde \xi_{i-1,i} \big)(\cdot,t)
	\\&
	= R_{(i,i+1)}\nabla \big( \widetilde \xi_{i+1,i-1} \big)(\cdot,t)
	\\&
	=  \nabla \big( R_{(i,i+1)} \widetilde \xi_{i+1,i-1} \big)(\cdot,t)
	= \nabla\widetilde\xi_{i,i+1}(\cdot,t).
	\end{align*}

	For ease of notation, we also fix the index $i$ and omit all indices, superscripts, 
	and arguments for the rest of the proof unless specifically required otherwise.
	The ansatz \eqref{AnsatzAuxiliaryXiHalfSpace} then reads
	\begin{equation}\label{AnsatzXiWedge}
		\widetilde \xi = 
		\bar{\vec{n}} + \alpha s \bar{\tau} 
		- \frac{1}{2} \alpha^2 s^2 \bar{\vec{n}}. 
	\end{equation}
By definition~\eqref{def:extensionNormal}, $\nabla^2 s$ being symmetric,
the identity~\eqref{LengthConservation2}, and the orthogonality relation~$\bar\tau\cdot\bar{\vec{n}}=0$
we have~$\nabla\bar{\vec{n}}=\Delta s \, \bar\tau\otimes\bar\tau$.
Hence, by the definitions~\eqref{def:tangent} and~\eqref{def:extensionCurvature}
as well as the estimate~\eqref{Bound2ndDerivativeSignedDistance}, we then get
\begin{align}
  \label{eq:derivativeNormal}
	\nabla \bar{\vec{n}} &= - H
	\bar{\tau} \otimes \bar{\tau} + O(r^{-2}|s|),
	\\
	\label{eq:derivativeTangent}
	\nabla \bar{\tau} &= \phantom{+} H
	\bar{\vec{n}}\otimes\bar{\tau}+ O(r^{-2}|s|).
\end{align}
As a result we infer from this and~\eqref{BoundCoefficientsAlpha}
\begin{equation}\label{eq:xi_first_derivative}
\begin{aligned}
		\nabla \widetilde \xi & =  - H
		\, \bar{\tau} \otimes \bar{\tau} 
		+ 
		\alpha 
		\, \bar{\tau} \otimes \bar{\vec{n}} 
		+ O(r^{-2}|s|).
\end{aligned}
\end{equation}
This in turn yields
\begin{align}\label{eq:xi_first_derivative_tj}
		\nabla \widetilde \xi & = \bar{\tau}\otimes \left(- H \,  \bar{\tau}  + \alpha \, \bar{\vec{n}}\right)
		\quad\text{at the triple junction } \mathcal{T}.
\end{align}

Now we are in a position to prove the compatibility condition 
\eqref{compatibility_xi_first}.
By~\eqref{def_rotation} and $J\bar\tau=\bar{\vec{n}}$, see \eqref{def:tangent},
we obtain
\begin{align}\label{rotated_tangent}
	\bar\tau_{i,i+1} = R_{(i,j)} \bar\tau_{j,j+1}
	\quad\text{at the triple junction } \mathcal{T}.
\end{align}
Moreover, expressing the evolving triple junction in form of $\mathcal{T}(t)=\{p(t)\}$ for all $t\in [0,T]$,
it follows from the evolution equation $\frac{\mathrm{d}}{\mathrm{d}t}p\cdot\bar{\vec{n}}_{i,i+1}=H_{i,i+1}$
and the choice of the initial value in the ODE~\eqref{ODE_alpha} that
\begin{align}
\label{eq:compatiblity_xi}
\frac{\mathrm{d}}{\mathrm{d}t} p
= \phantom{+}H_{1,2}\bar{\vec{n}}_{1,2} + \alpha_{1,2}\bar\tau_{1,2}
&= \phantom{+}H_{2,3}\bar{\vec{n}}_{2,3} + \alpha_{2,3}\bar\tau_{2,3}
= \phantom{+}H_{3,1}\bar{\vec{n}}_{3,1} + \alpha_{3,1}\bar\tau_{3,1},
\\
\label{eq:compatiblity_xi2}
-H_{1,2}\bar\tau_{1,2} + \alpha_{1,2}\bar{\vec{n}}_{1,2}
&= -H_{2,3}\bar\tau_{1,2} + \alpha_{2,3}\bar{\vec{n}}_{2,3}
= -H_{3,1}\bar\tau_{1,2} + \alpha_{3,1}\bar{\vec{n}}_{3,1}
\end{align} 
at the triple junction~$\mathcal{T}$ (the latter follows from the former by multiplication with~$J$).
Therefore, by~\eqref{eq:xi_first_derivative_tj}, \eqref{rotated_tangent} 
and~\eqref{eq:compatiblity_xi2} we indeed at~$\mathcal{T}$ get
\begin{align*}
	\nabla \big( R_{(i,j)} \widetilde \xi_{j,j+1} \big)	 
	&= R_{(i,j)} \bar\tau_{j,j+1}\otimes \big(-H_{j,j+1}\bar\tau_{j,j+1} + \alpha_{j,j+1}\bar{\vec{n}}_{j,j+1}\big) \\
	& =\bar\tau_{i,i+1}\otimes\big(-H_{i,i+1}\bar\tau_{i,i+1}  + \alpha_{i,i+1}\bar{\vec{n}}_{i,i+1}\big)\\
	& =\nabla \widetilde \xi_{i,i+1}.
\end{align*}
This concludes the proof of Lemma~\ref{lemma:compatibility_xi}.
\end{proof}

We next discuss the regularity properties of our construction for $\widetilde B_{(i,i+1)}$.
\begin{lemma}
\label{lemma:regularity_tilde_B}
Let the assumptions of Proposition~\ref{prop:xi_triple_junction} be in place,
in particular the notation of Definition~\ref{def:locRadiusTripleJunction}.
For all phases $i\in\{1,2,3\}$, the auxiliary velocity field
$\widetilde B_{(i,i+1)}$ is of class $C^0_tC^2_x(\overline{\mathbb{H}_{i,i+1}})$.
More precisely, we have the estimates
\begin{align}
\label{eq:boundsDerivativesTildeB}
|\widetilde B_{(i,i+1)}|
+ r_{i,i+1}|\nabla\widetilde B_{(i,i+1)}|
+ r_{i,i+1}^2|\nabla^2\widetilde B_{(i,i+1)}| \leq Cr_{i,i+1}^{-1}
\end{align}
for some~$C=C(\bar\Omega)>0$, depending only on~$\bar\Omega$
but independent of~$(r_{i,j})_{i,j\in\{1,2,3\},i\neq j}$.
\end{lemma}

\begin{proof}
In view of the expansion ansatz~\eqref{AnsatzAuxiliaryVelocityHalfSpace}
and the ingredients of the proof of Lemma~\ref{lemma:regularity_tilde_xi},
it suffices to prove that $\beta_{i,i+1}\in C^0_tC^2_x(\overline{\mathbb{H}_{i,i+1}})$
with corresponding estimate
\begin{align}
\label{BoundCoefficientsBeta}
|\nabla^k\beta_{i,i+1}| &\leq C r_{i,i+1}^{-k-2}, && k\in\{0,1,2\}.
\end{align}
Recalling the definition~\eqref{def_beta} of the coefficients~$\beta_{i,i+1}$,
the bound~\eqref{BoundCoefficientsBeta} is immediate 
from~\eqref{BoundCoefficientsAlpha}, \eqref{Bound2ndDerivativeNormalTangent},
\eqref{BoundMeanCurvature}, and~\eqref{boundTimeDerivatives}.
\end{proof}

We again have to make sure that our ansatz~\eqref{AnsatzAuxiliaryVelocityHalfSpace}
for the auxiliary velocity fields satisfies a first-order compatibility condition at the triple junction.

\begin{lemma}
\label{lemma:compatiblity_B}
Let the assumptions of Proposition~\ref{prop:xi_triple_junction} be in place.
Expressing the evolving triple junction in form of $\mathcal{T}(t)=\{p(t)\}$ for all $t\in [0,T]$,
for every $i,j\in\{1,2,3\}$ the ansatz \eqref{AnsatzAuxiliaryVelocityHalfSpace} then satisfies
	\begin{align}
	\label{compatibility_B_zero}
	\widetilde B_{(i,i+1)}(\tj(t),t) &= \widetilde B_{(j,j+1)}(\tj(t),t)
	= \frac{\mathrm{d}}{\mathrm{d}t} \tj(t), 
	\\
	\label{compatibility_B_first}
	\nabla\widetilde B_{(i,i+1)}(p(t),t) &= \nabla\widetilde B_{(j,j+1)}(\tj(t),t),	
	\\
	\label{short_nabla_B_tilde}
	\nabla \widetilde B_{(i,i+1)} & = -\beta_{i,i+1} J + O(r^{-3}|s_{i,i+1}|),
	\end{align}
	for all $t\in [0,T]$.
\end{lemma}

\begin{proof}
  We again fix the index $i$ and omit all indices, superscripts, 
	and function arguments unless specifically required.	At the triple junction, we have
	\begin{align}
		\widetilde B(\tj(t),t) = \frac{\mathrm{d}}{\mathrm{d}t} \tj(t)
	\end{align}
	by~\eqref{eq:compatiblity_xi} and the ansatz~\eqref{AnsatzAuxiliaryVelocityHalfSpace}. 
	This of course proves \eqref{compatibility_B_zero}. 
	
	An explicit computation making use of the ansatz \eqref{AnsatzAuxiliaryVelocityHalfSpace}, 
	the estimates~\eqref{eq:derivativeNormal} and~\eqref{eq:derivativeTangent},
	the choices of the coefficients~\eqref{ODE_alpha} and~\eqref{def_beta}---in 
	particular~\eqref{eq:errorDerivAlpha}---as well as the 
	estimates~\eqref{BoundCoefficientsAlpha} and~\eqref{BoundCoefficientsBeta}
	moreover gives
	\begin{equation}\label{nabla_B_tilde}
	\begin{aligned}
		\nabla \widetilde B & = \left(- H^2 
		+(\bar{\tau} \cdot \nabla  \alpha) \right)\bar{\tau}\otimes \bar{\tau}
		\\
		& \quad  +( (\bar{\tau}\cdot \nabla) H + \alpha H ) 
		\bar{\vec{n}}\otimes \bar{\tau}   
		\\
		& \quad  + \beta \bar{\tau} \otimes \bar{\vec{n}} + O(r^{-3}|s|)
		\\
		& = \beta \left(\bar{\tau} \otimes \bar{\vec{n}}
		- \bar{\vec{n}} \otimes \bar{\tau} \right)+ O(r^{-3}|s|).
	\end{aligned}
	\end{equation}
	As we have $\left(\bar{\tau} \otimes \bar{\vec{n}} - \bar{\vec{n}} \otimes \bar{\tau}  \right) \bar{\vec{n}} 
	= \bar{\tau} = - J \bar{\vec{n}}$ and 
	$\left(\bar{\tau} \otimes \bar{\vec{n}} - \bar{\vec{n}} \otimes \bar{\tau} \right) \bar{\tau}
	= - \bar{\vec{n}} = -J \bar{\tau}$	it follows that	$\left(\bar{\tau} \otimes \bar{\vec{n}}
	- \bar{\vec{n}} \otimes \bar{\tau} \right) = -J$, where 
	we recall that $J$ denotes the counter-clockwise rotation by $90^\circ$. Therefore we get \eqref{short_nabla_B_tilde}.
	Hence, \eqref{compatibility_B_first} holds true once we established that
	$\beta_{1,2}=\beta_{2,3}=\beta_{3,1}$ at the triple junction.
	This, however, follows from a combination of the definition~\eqref{def_beta},
	the choice of the initial value in the ODE~\eqref{ODE_alpha}, and the
	third-order compatibility condition~\eqref{ThirdOrderCompDef}.
\end{proof}

In a preparatory step towards the proof of \eqref{TransportEquationXiTriod} and \eqref{DissipTriod}, 
we now present the corresponding estimates for the (rotated) auxiliary vector fields $\widetilde \xi_{i,i+1}$
and the auxiliary velocity fields $\widetilde B_{(i,i+1)}$ on their respective domains of definition.

\begin{lemma}\label{EquationsTildeConstructions}
  Let the assumptions of Proposition~\ref{prop:xi_triple_junction} be in place,
	in particular the notation of Definition~\ref{def:locRadiusTripleJunction}.
	Then there exists a constant~$C=C(\bar\Omega)>0$, depending only on~$\bar\Omega$
but independent of~$(r_{i,j})_{i,j\in\{1,2,3\},i\neq j}$, such that the following holds:
	 For every $i,j\in\{1,2,3\}$ and throughout the space-time 
	domain~$\mathbb{H}_{j,j+1}$ we have 
	\begin{align}
		\label{TransportEquationXiTildeTriod}
		\begin{split}
		& \big|\partial_tR_{(i,j)}\widetilde\xi_{j,j+1} 
		+ (\widetilde B_{(j,j+1)}\cdot\nabla)R_{(i,j)}\widetilde\xi_{j,j+1}
		+(\nabla\widetilde B_{(j,j+1)})^\mathsf{T}R_{(i,j)}\widetilde\xi_{j,j+1}\big|
		\\&
		 \leq Cr_{j,j+1}^{-3}\dist(\cdot,{\bar{I}}_{j,j+1}),
	\end{split}
\end{align}
as well as	
\begin{align}
		\label{DissipXiTildeTriod}
		\big|\nabla\cdot R_{(i,j)}\widetilde\xi_{j,j+1}
		+\widetilde B_{(j,j+1)}\cdot R_{(i,j)}\widetilde\xi_{j,j+1}\big|
		& \leq Cr_{j,j+1}^{-2}\dist(\cdot,{\bar{I}}_{j,j+1}),
		\\
	  \label{LengthXiTildeTriod}
		\big|1-|R_{(i,j)}\widetilde \xi_{j,j+1}|^2\big| 
		&\leq  Cr_{j,j+1}^{-4}\dist^4(\cdot,{\bar{I}}_{j,j+1}),
		\\
		\label{DerivativesLengthXiTilde1}
		\big|\nabla|R_{(i,j)}\widetilde \xi_{j,j+1}|^2\big| 
		&\leq Cr_{j,j+1}^{-4}\dist^3(\cdot,{\bar{I}}_{j,j+1}),
				\\
		\label{DerivativesLengthXiTilde2}
		\big|\partial_t|R_{(i,j)}\widetilde \xi_{j,j+1}|^2\big| 
		&\leq Cr_{j,j+1}^{-5}\dist^3(\cdot,{\bar{I}}_{j,j+1}),
		\\
		\label{LengthConservationXiTildeTriod}
		\big|\partial_t|R_{(i,j)}\widetilde\xi_{j,j+1}|^2
		+ (\widetilde B_{(j,j+1)}\cdot\nabla)|R_{(i,j)}\widetilde\xi_{j,j+1}|^2 \big| 
		&\leq  Cr_{j,j+1}^{-6}\dist^3(\cdot,{\bar{I}}_{j,j+1}).
	\end{align}
	We also have for all pairs $i,j\in\{1,2,3\}$ with $i\neq j$ throughout the 
	intersection $\mathbb{H}_{i,i+1}\cap \mathbb{H}_{j,j+1}$ that
	(with $r_{\mathrm{min}}:=r_{1,2}\wedge r_{2,3} \wedge r_{3,1}$)
	\begin{align}
	\label{boundCompatibilityXi}
	|R_{(i,j)}\widetilde\xi_{j,j+1}-R_{(i,j-1)}\widetilde\xi_{j-1,j}| &\leq
	Cr_{\min}^{-2}\dist^2(\cdot,\mathcal{T}),
	\\\label{boundCompatibilityGradientXi}
	|\nabla R_{(i,j)}\widetilde\xi_{j,j+1}-\nabla R_{(i,j-1)}\widetilde\xi_{j-1,j}| &\leq
	Cr_{\min}^{-2}\dist(\cdot,\mathcal{T}),
	\\
	\label{boundCompatibilityVelocity}
	|\widetilde B_{(i,i+1)} - \widetilde B_{(j,j+1)}| &\leq Cr_{\min}^{-3}\dist^{2}(\cdot,\mathcal{T}),
		\\\label{boundCompatibilityGradientVelocity}
	|\nabla \widetilde B_{(i,i+1)}-\nabla \widetilde B_{(j,j+1)}| &\leq
	Cr_{\min}^{-3}\dist(\cdot,\mathcal{T}).
	\end{align}
\end{lemma}

\begin{proof}
	By the ansatz~\eqref{AnsatzAuxiliaryXiHalfSpace}
	and $R_{(i,j)} \in SO(2)$ we have
	\begin{equation}
	\begin{aligned}
	\label{eq:ComputeLengthOfXiTilde}
		|R_{(i,j)}\widetilde\xi_{j,j+1}|^2 &= \Big(1-\frac{1}{2}\alpha^2_{j,j+1}s_{j,j+1}^2\Big)^2 
		+ \alpha^2_{j,j+1}s_{j,j+1}^2
		\\
		&=1 + \frac{1}{4}\alpha^4_{j,j+1}s_{j,j+1}^4
	\end{aligned}
	\end{equation}
	from which together with~\eqref{Bound2ndDerivativeSignedDistance}, \eqref{BoundCoefficientsAlpha}, 
	\eqref{boundTimeDerivatives}, and~\eqref{eq:boundsDerivativesTildeB}
	the estimates~\eqref{LengthXiTildeTriod}--\eqref{LengthConservationXiTildeTriod} 
	immediately follow.
	
	To prove the estimates~\eqref{TransportEquationXiTildeTriod}--\eqref{DissipXiTildeTriod}, 
	let $i,j\in\{1,2,3\}$ be fixed.
	For what follows, we omit all indices and function arguments unless specifically required.
	Plugging in the ansatz~\eqref{AnsatzAuxiliaryXiHalfSpace} for~$\widetilde\xi$
	and introducing the commutator~$[C,D]:=CD-DC$ for matrices~$C,D\in\Rd[d{\times}d]$, we obtain
	\begin{align*}
	\partial_t R\widetilde\xi {+} (\widetilde B\cdot\nabla)R\widetilde\xi 
	{+} (\nabla\widetilde B)^{\mathsf{T}}R\widetilde\xi
	&= \Big(1{-}\frac{1}{2}\alpha^2s^2\Big)R\big(\partial_t \bar{\vec{n}} 
	+ (\widetilde B\cdot\nabla)\bar{\vec{n}} 
	+ (\nabla\widetilde B)^{\mathsf{T}}\bar{\vec{n}}\big)
	\\&~~~
	+ \alpha sR\big(\partial_t \bar\tau
	+ (\widetilde B\cdot\nabla)\bar\tau 
	+ (\nabla\widetilde B)^{\mathsf{T}}\bar\tau\big)
	\\&~~~
	+ \alpha\big(\partial_t s + (\widetilde B\cdot\nabla)s\big)
	\big(R\bar\tau - \alpha s R\bar{\vec{n}}\big)
	\\&~~~
	+ [(\nabla\widetilde B)^{\mathsf{T}},R]\tilde\xi
	\\&~~~
	+ \big(\partial_t\alpha +\tilde B \cdot \nabla \alpha\big) s\big(R\bar\tau - \alpha s R\bar{\vec{n}}\big).
	\end{align*}
	By the ansatz~\eqref{AnsatzAuxiliaryVelocityHalfSpace}, the auxiliary
	velocity~$\widetilde B$ only corrects~$H\bar{\vec{n}}$ in tangential direction.
	Hence, the identities~\eqref{TransportEquationXiTwoPhase} and~\eqref{eq:transportSignedDistanceB}
	are applicable and we obtain
	\begin{align*}
	\partial_t \bar{\vec{n}} 
	+ (\widetilde B\cdot\nabla)\bar{\vec{n}} 
	+ (\nabla\widetilde B)^{\mathsf{T}}\bar{\vec{n}} = 0,
	\quad \partial_t s + (\widetilde B\cdot\nabla)s = 0
	\end{align*}
	throughout~$\mathbb{H}_{j,j+1}$. Recalling the definition~$\bar\tau=J\bar{\vec{n}}$, 
	cf.\ \eqref{def:tangent}, we deduce from the previous display
	\begin{align*}
	\partial_t \bar\tau
	+ (\widetilde B\cdot\nabla)\bar\tau 
	+ (\nabla\widetilde B)^{\mathsf{T}}\bar\tau
	= [(\nabla\widetilde B)^{\mathsf{T}},J]\bar{\vec{n}}
	\end{align*}
	throughout~$\mathbb{H}_{j,j+1}$. Hence, recalling~\eqref{short_nabla_B_tilde}
	and using the fact that $[J^{\mathsf{T}},R]=0$ on account of both matrices being rotations in the plane we get
	\begin{align*}
	[(\nabla\widetilde B)^{\mathsf{T}},R] = O(r^{-3}|s|),
	\quad [(\nabla\widetilde B)^{\mathsf{T}},J] = O(r^{-3}|s|)
	\end{align*}
	throughout~$\mathbb{H}_{j,j+1}$. Together with the estimate~\eqref{BoundCoefficientsAlpha},
	the previous four displays in combination imply~\eqref{TransportEquationXiTildeTriod}.

We turn to the proof of~\eqref{DissipXiTildeTriod}. 
Due to the computation \eqref{eq:xi_first_derivative} of $\nabla \widetilde \xi$
we have on the one hand
\begin{align}\label{divxi_1}
\begin{split}
\nabla\cdot R\widetilde\xi
&= - H( R\bar\tau \cdot \bar\tau )
+\alpha
(R\bar\tau \cdot  \vec{\bar{n}} )
+O(r^{-2}|s|).
\end{split}
\end{align}
On the other hand, making use of the
definitions \eqref{AnsatzAuxiliaryXiHalfSpace} 
and \eqref{AnsatzAuxiliaryVelocityHalfSpace} of $\widetilde\xi$
and $\widetilde B$ we obtain
\begin{align}\label{divxi_2}
\begin{split}
\widetilde B\cdot R\widetilde\xi 
&= H\vec{\bar{n}}\cdot R\vec{\bar{n}}
+\alpha
(\bar\tau\cdot R\vec{\bar{n}})
+O(r^{-2}|s|).
\end{split}
\end{align}
Furthermore, recalling $J\bar\tau=\bar{\vec{n}}$, $J^\mathsf{T}=J^{-1}=-J$, and $[J^\mathsf{T},R]=0$ gives
\begin{align*}
	R\bar{\tau} \cdot \bar{\tau} & = R J^{-1} \bar{\vec{n}} \cdot \bar{\tau} 
	= R \bar{\vec{n}} \cdot J\bar{\tau} = R \bar{\vec{n}} \cdot  \bar{\vec{n}},
	\\
	R\bar{\tau} \cdot \bar{\vec{n}}  & =  R J^{-1} \bar{\vec{n}} \cdot \bar{\vec{n}} 
	= R \bar{\vec{n}} \cdot J\bar{\vec{n}} = - R \bar{\vec{n}} \cdot \bar{\tau}.
\end{align*}
Therefore, we can combine \eqref{divxi_1} and \eqref{divxi_2} to yield the estimate \eqref{DissipXiTildeTriod}.

We proceed with the verification of the bounds~\eqref{boundCompatibilityXi} 
and~\eqref{boundCompatibilityGradientXi}.
As by~\eqref{compatibility_xi_zero} and~\eqref{compatibility_xi_first} the Taylor 
polynomials at the triple junction of the functions $R_{(i,j)}\widetilde\xi_{j,j+1}$ 
and $R_{(i,j-1)}\widetilde\xi_{j-1,j}$ agree up to first order, 
the estimate \eqref{boundCompatibilityXi} follows by bounding 
the remainders using~\eqref{eq:boundsDerivativesTildeXi}.
One can argue similarly for the estimate \eqref{boundCompatibilityGradientXi}.
On the basis of \eqref{compatibility_B_zero},  
\eqref{compatibility_B_first} and \eqref{eq:boundsDerivativesTildeB}, 
the estimates \eqref{boundCompatibilityVelocity} and \eqref{boundCompatibilityGradientVelocity} 
follow by the same argument.
\end{proof}

\subsection{Gluing construction by interpolation}\label{subsection:step_2}
Throughout this subsection,
let again the assumptions of Proposition~\ref{prop:xi_triple_junction}
and the notation of Section~\ref{SectionLocalConstructionsTwoPhase} 
and Definition~\ref{def:locRadiusTripleJunction} be in place.
As we discussed in the previous subsection, the auxiliary vector fields $\widetilde\xi_{i,i+1}$
and the auxiliary velocity fields $\widetilde B_{(i,i+1)}$ serve as
the definition of the vector fields $\xi_{i,i+1}$ and the velocity field $B$ on
the interface wedge $W_{i,i+1}$, see Figure~\ref{fig:triod_wedges} for the partition
of the neighborhood of the triple junction.

The next step is to extend $\xi_{i,i+1}$ and $B$ to the entirety of 
the space-time domain. As we want Herring's angle condition \eqref{eq:auxHerringAngleConditionExtensions}
to hold throughout the ball $B_r(\mathcal{T}(t))$ we are essentially forced to set 
$\xi_{i,i+1} = R_{(i,j)} \xi_{j,j+1}$ for all $i,j\in\{1,2,3\}$ 
wherever the latter is defined, and where $R_{(i,j)}$ is given in Lemma \ref{lemma:compatibility_xi}.
As their domains of definition $\mathbb{H}_{i,i+1}$ overlap,
we resort to an interpolation procedure on the 
interpolation wedges $W_{i}$, see again Figure~\ref{fig:triod_wedges}.
We similarly deal with the issue of combining the velocity fields $\widetilde B_{(i,i+1)}$ into a single field.
To this end, we first define suitable interpolation functions which move and rotate with the evolving triple junction.

\begin{lemma}
\label{lemma:interpolation_functions}
  Let the assumptions of Proposition~\ref{prop:xi_triple_junction} be in place,
	in particular the notation of Definition~\ref{def:locRadiusTripleJunction}.
	Then there exists a constant~$C=C(\bar\Omega)>0$, depending only on~$\bar\Omega$
	but independent of~$(r_{i,j})_{i,j\in\{1,2,3\},i\neq j}$, 
	and interpolation functions
	\begin{align*}
		\lambda_i\colon \bigcup_{t\in [0,T]}
		\big(B_{r}(\mathcal{T}(t)) \cap \overline{W}_{i}(t)\big)\setminus\mathcal{T}(t)\times\{t\} &\to [0,1]
	\end{align*}
	for every $i\in\{1,2,3\}$	which satisfy the following properties:
	\begin{itemize}[leftmargin=0.7cm]
		\item[i)] It holds for all~$t\in [0,T]$ that
							\begin{align}
							\label{LambdaLeftBoundary}
							\lambda_i(x,t) &= 0 \qquad\text{for}\quad 
							x\in\big(\partial W_{i}(t)\cap\partial W_{i,i+1}(t)\big)\setminus\mathcal{T}(t),
							\\
							\label{LambdaRightBoundary}
							\lambda_i(x,t) &= 1 \qquad\text{for}\quad 
							x\in\big(\partial W_{i}(t)\cap\partial W_{i-1,i}(t)\big)\setminus\mathcal{T}(t).
							\end{align}
		\item[ii)] We have the estimates ($r_{\mathrm{min}}:=r_{1,2}\wedge r_{2,3}\wedge r_{3,1}$)
			\begin{align}
				\label{boundslambda1}
				|\nabla\lambda_i(x,t)| &\leq  C\dist(x,\mathcal{T}(t))^{-1}, \quad
				|\partial_t \lambda_i(x,t)| \leq Cr_{\mathrm{min}}^{-1}\dist(x,\mathcal{T}(t))^{-1},
				\\
				\label{boundslambda2}
				|\nabla^2 \lambda_i (x,t)| &\leq  C\dist(x,\mathcal{T}(t))^{-2}
			\end{align}
			for all $t\in [0,T]$ and all $x \in \big(B_{r}(\mathcal{T}(t)) \cap \overline{W}_{i}(t)\big) 
			\setminus\mathcal{T}(t)$.
			Furthermore, it holds
			\begin{align}
			  \label{first_derivative_lambda_vanishes_on_wedge}
				\nabla \lambda_i(x,t) & = 0,\quad \partial_t \lambda_i (x,t) = 0, 
				\\
				\label{second_derivative_lambda_vanishes_on_wedge}
				\nabla^2 \lambda_i(x,t) & = 0
			\end{align}
			for all $t\in [0,T]$ and all $x \in \big(B_r(\mathcal{T}(t))\cap\partial W_i(t)\big) \setminus \mathcal{T}(t)$.
		\item[iii)] Expressing the evolving triple junction via~$\mathcal{T}(t)=\{p(t)\}$ for all~$t\in [0,T]$,
		we have a bound on the advective derivative 
			\begin{align}\label{advectionlambda}
				\Big| \partial_t  \lambda_i(x,t) + \Big(\frac{\mathrm{d}}{\mathrm{d}t} 
				\tj(t) \cdot \nabla\Big)\lambda_i(x,t) \Big| & \leq Cr_{\mathrm{min}}^{-2}
			\end{align}
			for all $t \in [0,T]$ and all $x \in  \big(B_{r}(\mathcal{T}(t)) 
			\cap \overline{W}_{i}(t)\big)\setminus\mathcal{T}(t)$.
	\end{itemize}
\end{lemma}

\begin{proof}
Due to~\eqref{def:interpolWedge}, the interpolation wedge $W_i(t)$ is the
restriction to~$B_r(\mathcal{T}(t))$ of the interior of the conical hull
spanned by two unit vectors $X^{i}_{i,i+1}(t)$ and $X^{i}_{i-1,i}(t)$, whereas $W_{i,i+1}(t)$ 
is the restriction to~$B_r(\mathcal{T}(t))$ of the interior
of the conical hull spanned by unit vectors $X^{i}_{i,i+1}(t)$ and $X^{i+1}_{i,i+1}(t)$ due to~\eqref{def:interfaceWedge}.   
In particular, we can represent
$\partial W_i(t) \cap \partial W_{i,i+1}(t) =  \{\gamma X^{i}_{i,i+1}(t)\colon \gamma\geq 0\}$ 
and $\partial W_i(t) \cap \partial W_{i-1,i}(t) = \{\gamma X^{i}_{i-1,i}(t)\colon \gamma\geq 0\}$.
As the vectors $X^{i}_{i,i+1}(t)$ and $X^{i}_{i-1,i}(t)$ 
can be expressed as a (fixed-in-time) linear combination of the unit-normals 
$\bar{\vec{n}}_{i,j}(p(t),t)$ at the triple junction, we have
due to~\eqref{eq:boundsInitialValues}, \eqref{Bound2ndDerivativeNormalTangent} 
and~\eqref{boundTimeDerivatives} the bounds 
\begin{align}
\label{eq:boundsOpeningVectors}
\Big|\frac{\mathrm{d}}{\mathrm{d}t} X^{i}_{i,i+1}(t)\Big|
+ \Big|\frac{\mathrm{d}}{\mathrm{d}t} X^{i}_{i-1,i}(t)\Big|
\leq Cr^{-2}_{\mathrm{min}} \leq Cr^{-1}_{\mathrm{min}}\dist(x,\mathcal{T}(t))^{-1}
\end{align}
for all~$t\in [0,T]$, all~$x\in B_r(\mathcal{T}(t))$, and all~$i\in\{1,2,3\}$.

By Definition~\ref{def:locRadiusTripleJunction}, the opening angle~$\theta_i$ of the 
interpolation wedge~$W_i$, defined by $\cos(\theta_i) =X^{i}_{i,i+1}(t)\cdot X^{i}_{i-1,i}(t) \in (0,1)$, 
is time-independent and satisfies $\theta_i\in (0,\frac{\pi}{2})$.
(The angles only depend on~$\bar\Omega$ through the surface tensions.)
Let $\widetilde\lambda\colon\Rd[]\to [0,1]$ be any smooth function such that
$\widetilde\lambda\equiv 0$ on $(-\infty,\frac{1}{3}]$ and 
$\widetilde\lambda\equiv 1$ on $[\frac{2}{3},\infty)$.
We define
\begin{align*}
\lambda_i(x,t):=\widetilde\lambda\Bigg(\frac{1{-}X^{i}_{i,i+1}(t)\cdot\frac{x{-}p(t)}{|x{-}p(t)|}}{1{-}\cos\theta_i}\Bigg).
\end{align*}
Then the properties \eqref{LambdaLeftBoundary}--\eqref{second_derivative_lambda_vanishes_on_wedge}
are immediate consequences of the definitions and the bounds~\eqref{eq:boundsOpeningVectors}
and~\eqref{eq:boundsInitialValues}; cf.\ also the subsequent computation.

It remains to check the bound \eqref{advectionlambda} on the advective derivative.
To this end, we abbreviate $\lambda_i(x,t)=\widehat\lambda_i\big(X^{i}_{i,i+1}(t)\cdot\frac{x{-}p(t)}{|x{-}p(t)|}\big)$ 
with $\widehat\lambda_i(a):=\widetilde\lambda(\tfrac{1-a}{1-\cos \theta_i})$
and simply compute
\begin{align*}
&\partial_t\lambda_i(x,t) 
\\
&= -\widehat\lambda'_i\frac{X^{i}_{i,i+1}(t)}{|x{-}p(t)|}\cdot\Big( \Id {-} \frac{x{-}p(t)}{|x{-}p(t)|} 
\otimes \frac{x{-}p(t)}{|x{-}p(t)|}  \Big)
\frac{\mathrm{d}}{\mathrm{d}t}p(t)
+\widehat\lambda'_i\frac{x{-}p(t)}{|x{-}p(t)|}\cdot\frac{\mathrm{d}}{\mathrm{d}t}X^{i}_{i,i+1}(t)
\\&
=-\Big(\frac{\mathrm{d}}{\mathrm{d}t}\tj(t) \cdot \nabla\Big)\lambda_i(x,t)
+\widehat\lambda'_i\frac{x{-}p(t)}{|x{-}p(t)|}\cdot\frac{\mathrm{d}}{\mathrm{d}t}X^{i}_{i,i+1}(t)
\end{align*}
where $\widehat\lambda'_i$ is evaluated at $X^{i}_{i,i+1}(t)\cdot\frac{x{-}p(t)}{|x{-}p(t)|}$.
From this, the last remaining claim~\eqref{advectionlambda} immediately follows due to
the estimate~\eqref{eq:boundsOpeningVectors}.
\end{proof}

Equipped with these interpolating functions we are finally in the position 
to prove the main result of this section.

\begin{proof}[Proof of Proposition~\ref{prop:xi_triple_junction}]
\textit{Step 1: Interpolation of the vector fields.}
We define (not yet normalized) extensions of the normal vector fields~$\bar{\vec{n}}_{i,j}|_{\bar I_{i,j}}$
on the space-time neighborhood of the triple junction 
$\mathcal{U}_r:=\bigcup_{t\in [0,T]}B_r(\mathcal{T}(t))\times\{t\}$ as follows:
\begin{align}
\label{def:xi_triple_junction}
  \widehat\xi_{i,i+1}(x,t) :=
  	\begin{cases}
  		 R_{(i,j)} \widetilde \xi_{j,j+1}(x,t) & \text{ if } x \in W_{j,j+1}(t),\\
  		\!\begin{aligned}
	  		 & (1{-}\lambda_j(x,t)) R_{(i,j)} \widetilde \xi_{j,j+1}(x,t)\\
  			  &\quad +  \lambda_j(x,t)R_{(i,j-1)} \widetilde \xi_{j-1,j} (x,t)
  		 \end{aligned} 
  		 & \text{ if } x \in \overline{W}_{j}(t),
  	\end{cases}
\end{align}
and $\widehat\xi_{i+1,i} := -\widehat\xi_{i,i+1}$ for $i\in\{1,2,3\}$.
The velocity field is given by
\begin{align}
\label{def:velocity_triple_junction}
	B(x,t) := \begin{cases}
		\widetilde B_{(j,j+1)}(x,t) & \text{ if } x \in W_{j,j+1}(t),\\
		\begin{aligned}
		& (1{-}\lambda_{j}(x,t))  \widetilde B_{(j,j+1)}(x,t) \\
		&\quad + \lambda_j(x,t) \widetilde B_{(j-1,j)}(x,t) 
		\end{aligned}
		& \text{ if } x \in \overline{W}_{j}(t).
	\end{cases}
\end{align}
In the subsequent steps of the proof, we first establish all required properties
in terms of the vector fields $\widehat\xi_{i,j}$ and $B$. Only in the penultimate step we
will choose the radius $\hat r = \hat r(\bar\chi) \leq r$ and define
unit-length vector fields $\xi_{i,j}$ by normalization of the vector fields
$\widehat\xi_{i,j}$ defined in \eqref{def:xi_triple_junction} above. The last step
is then devoted to verify the required properties for the normalized vector
fields $\xi_{i,j}$.

\textit{Step 2: Regularity of $\widehat\xi_{i,j}$ and $B$, the estimates \eqref{boundDerivativesXi}
and \eqref{boundDerivativeB}, and properties i)--iii).}
We first remark that the above definitions make sense due to the 
second inclusion in~\eqref{eq:inlcusionInterfaceWedge}
and the inclusion in~\eqref{eq:inclusionInterpolWedge}. 
Indeed, these inclusions are precisely what is needed so
that the building blocks $\widetilde\xi_{i,i+1}$ and $\widetilde B_{(i,i+1)}$ are only
evaluated on their domains of definition. 

For every $i\in\{1,2,3\}$, we obtain $\widehat\xi_{i,i+1}(x,t) = \widetilde \xi_{i,i+1}(x,t) = \bar{\vec{n}}_{i,i+1}(x,t)$ 
for all $t \in [0,T]$ and all $x \in \mathcal{T}_{i,i+1}(t)\cap B_{r}(\mathcal{T}(t))$ from the first inclusion 
in~\eqref{eq:inlcusionInterfaceWedge} and the ansatz~\eqref{AnsatzAuxiliaryXiHalfSpace},
taking care of property~i); obviously except for the normalization condition away from the interfaces. The second property
$\widehat\xi_{i,j} = - \widehat\xi_{j,i}$ for $i,j \in \{1,2,3\}$ with $i \neq j$ holds by definition.
For every $j\in\{1,2,3\}$ we moreover have
\begin{align*}
	\sigma_{1,2} \widehat\xi_{1,2} + \sigma_{2,3} \widehat\xi_{2,3} + \sigma_{3,1}  \widehat\xi_{3,1} \equiv 
	\left( \sigma_{1,2} R_{(1,j)}  + \sigma_{2,3} R_{(2,j)}   + \sigma_{3,1} R_{(3,j)}    
	\right) \widetilde \xi_{j,j+1} =0
\end{align*}
on $W_{j,j+1}(t)$ by the defining property \eqref{def_rotation} of the rotations $R_{(i,j)}$. 
A similar argument ensures validity of \eqref{xi_compatibility} on the interpolation wedges $\overline{W}_{j}(t)$.

By the compatibility condition \eqref{compatibility_xi_zero} 
for the auxiliary vector fields $\widetilde\xi_{j,j+1}$ at the triple junction,
as well as the conditions \eqref{LambdaLeftBoundary} and \eqref{LambdaRightBoundary} 
 for the interpolation functions, 
the vector fields $\widehat\xi_{i,j}$ are continuous. Similarly, their first and second derivatives are continuous across
the boundaries of the interpolation wedges 
$\bigcup_{t\in [0,T]}\big(\big(B_r(\mathcal{T}(t))\cap \partial W_i(t)\big)\setminus\mathcal{T}(t)\big)\times\{t\}$ 
by the properties \eqref{first_derivative_lambda_vanishes_on_wedge} and \eqref{second_derivative_lambda_vanishes_on_wedge}
of the interpolation functions.

Moreover, all spatial derivatives up to second order are bounded in $\mathcal{U}_r\setminus\mathcal{T}$
with the asserted estimate given by \eqref{boundDerivativesXi}. Indeed, in the interface wedges~$W_{j,j+1}$
this follows from the estimates~\eqref{eq:boundsDerivativesTildeXi}
and the definition~\eqref{def:xi_triple_junction}. On the closure of the interpolation wedges~$W_j$,
we first compute using the definition~\eqref{def:xi_triple_junction}
\begin{align}
\label{derivative_interpolated_Xi}
\nabla\widehat\xi_{i,i+1} 
&= (1{-}\lambda_j) \nabla R_{(i,j)} \widetilde \xi_{j,j+1}
+ \lambda_j \nabla R_{(i,j-1)} \widetilde \xi_{j-1,j}
\\& \nonumber
~~~ - (R_{(i,j)} \widetilde \xi_{j,j+1}{-}R_{(i,j-1)} \widetilde \xi_{j-1,j})\nabla\lambda_j,
\\ \label{second_derivative_interpolated_Xi}
\nabla^2\widehat\xi_{i,i+1} &= (1{-}\lambda_j) \nabla^2 R_{(i,j)} \widetilde \xi_{j,j+1}
+ \lambda_j \nabla^2 R_{(i,j-1)} \widetilde \xi_{j-1,j}
\\&~~~ \nonumber
- 2 (\nabla R_{(i,j)} \widetilde \xi_{j,j+1}{-}\nabla R_{(i,j-1)} \widetilde \xi_{j-1,j})\nabla\lambda_j
\\&~~~ \nonumber
- (R_{(i,j)} \widetilde \xi_{j,j+1}{-}R_{(i,j-1)} \widetilde \xi_{j-1,j})\nabla^2\lambda_j.
\end{align}
Now, the bound~\eqref{boundDerivativesXi} with respect to spatial derivatives 
follows from the controlled blowup~\eqref{boundslambda1} and~\eqref{boundslambda2}
of the interpolation functions, the estimates~\eqref{eq:boundsDerivativesTildeXi}, \eqref{boundCompatibilityXi}
and~\eqref{boundCompatibilityGradientXi} for the auxiliary vector fields~$\widetilde\xi_{j,j+1}$,
as well as the estimate~\eqref{eq:compDistances1}.
In total, this proves~$\widehat\xi_{i,j}\in C^0_tC^2_x(\overline{\mathcal{U}_r}\setminus\mathcal{T})$. 
The other property $\widehat\xi_{i,j}\in C^1_tC^0_x(\overline{\mathcal{U}_r}\setminus\mathcal{T})$
together with the asserted bound~\eqref{boundDerivativesXi} in terms of the
time derivative follows similarly making use of Lemma~\ref{lemma:regularity_tilde_xi}, 
\eqref{boundCompatibilityXi}, \eqref{boundslambda1}, \eqref{eq:compDistances1} 
and the computation on the closure of $W_j$
\begin{align*}
\partial_t\widehat\xi_{i,i+1} &= (1{-}\lambda_j) \partial_t R_{(i,j)} \widetilde \xi_{j,j+1}
+ \lambda_j \partial_t R_{(i,j-1)} \widetilde \xi_{j-1,j} 
\\&~~~
- (R_{(i,j)} \widetilde \xi_{j,j+1}{-}R_{(i,j-1)} \widetilde \xi_{j-1,j})\partial_t\lambda_j.
\end{align*}

We proceed with the regularity of the velocity field $B$. First, 
by the compatibility condition~\eqref{compatibility_B_zero}  
for the auxiliary velocity fields~$\widetilde B_{(j,j+1)}$ at the triple junction,
as well as the conditions~\eqref{LambdaLeftBoundary} and~\eqref{LambdaRightBoundary} 
for the interpolation functions, the velocity field~$B$ is continuous. 
The asserted bound~\eqref{boundDerivativeB}
is a consequence of the definition~\eqref{def:velocity_triple_junction},
the estimates~\eqref{eq:boundsDerivativesTildeB}, \eqref{boundCompatibilityVelocity}
and~\eqref{boundCompatibilityGradientVelocity} for the auxiliary velocity fields, 
the controlled blowup~\eqref{boundslambda1} 
of the interpolation functions, the estimate~\eqref{eq:compDistances1}
as well as the computation
\begin{align}
\label{derivative_interpolated_B}
\nabla B &= (1{-}\lambda_{j})  \nabla\widetilde B_{(j,j+1)}
		+ \lambda_j \nabla\widetilde B_{(j-1,j)} 
		+ (\widetilde B_{(j-1,j)}{-}\widetilde B_{(j,j+1)})\nabla\lambda_j,
\\\label{second_derivative_interpolated_B}
\nabla^2 B &= (1{-}\lambda_{j})  \nabla^2\widetilde B_{(j,j+1)}
		+ \lambda_j \nabla^2\widetilde B_{(j-1,j)} 
		\\&~~~\nonumber
		+ 2 (\nabla\widetilde B_{(j-1,j)}{-}\nabla\widetilde B_{(j,j+1)})\nabla\lambda_j
		+ (\widetilde B_{(j-1,j)}{-}\widetilde B_{(j,j+1)})\nabla^2\lambda_j
\end{align}
on the closure of~$W_j$. This proves~$\BTrJ \in C^0_tC^2_x(\overline{\mathcal{U}_r}\setminus\mathcal{T})$.

\textit{Step 3: Proof of the estimate ($r_{\mathrm{min}}:=r_{1,2}\wedge r_{2,3}\wedge r_{3,1}$)}
\begin{align}
\label{eq:prelimEstimateTimeDerivative}
|\partial_t \widehat\xiTrJ_{i,j} + (\BTrJ\cdot\nabla)\widehat\xiTrJ_{i,j} 
+ (\nabla\BTrJ)^\mathsf{T}\widehat\xiTrJ_{i,j}| 
				&\leq Cr_{\mathrm{min}}^{-3}\dist(\cdot,{\bar{I}_{i,j}})
\quad\text{in } \mathcal{U}_r.
\end{align}

By the skew-symmetry $\widehat\xi_{i,j}=-\widehat\xi_{j,i}$, we only have 
to prove~\eqref{eq:prelimEstimateTimeDerivative} for $j=i+1$.
Let $i\in\{1,2,3\}$. First, we remark that 
the validity of~\eqref{TransportEquationXiTriod} for the vector field~$\widehat\xiTrJ_{i,i+1}$ 
on the interface wedges~$W_{j,j+1}$ for all $j=1,2,3$ follows from the estimate~\eqref{TransportEquationXiTildeTriod},
the definitions~\eqref{def:xi_triple_junction} and~\eqref{def:velocity_triple_junction},
and the estimate~\eqref{eq:compDistances3}.
Hence, it remains to prove the bound~\eqref{eq:prelimEstimateTimeDerivative} for~$\widehat\xiTrJ_{i,i+1}$
on each interpolation wedge $W_{j}$, $j\in\{1,2,3\}$.
In the interpolation wedge $W_j$, it is our goal to show simply 
	\begin{align*}
	|\left(\partial_t+ (B\cdot \nabla) + (\nabla B)^\mathsf{T} \right)\widehat \xi_{i,i+1}|\leq C \dist(\cdot,\mathcal{T}),
	\end{align*}
	as we may then use the equivalence $\dist(x,\mathcal{T})\leq C \dist(x,\bar I_{i,i+1})$ valid for all $i$ in the interpolation wedges $W_j$.

To this end, let us fix $j\in\{1,2,3\}$.
For the sake of readability, let us introduce the abbreviations, $\lambda=\lambda_j$, $R=R_{(i,j)}$, $R'=R_{(i,j-1)}$,
$\widetilde\xi=\widetilde\xi_{j,j+1}$, $\widetilde\xi'=\widetilde\xi_{j-1,j}$,
$\widetilde B=\widetilde B_{(j,j+1)}$ and $\widetilde B'=\widetilde B_{(j-1,j)}$.
Using the product rule and the definition \eqref{def:xi_triple_junction} of~$\widehat\xi_{i,i+1}$ 
on the closure of the interpolation wedge~$W_j$, we have
	\begin{align}\label{advection_interpolation}
	  \begin{split}
		\left(\partial_t+ (B\cdot \nabla) + (\nabla B)^\mathsf{T} \right)\widehat \xi_{i,i+1}
		&= 
		 (1-\lambda) \left( \partial_t + (B\cdot \nabla) + (\nabla B)^\mathsf{T}  \right) R \tilde \xi \\
		&~~~+ \lambda  \left( \partial_t + (B\cdot \nabla) + (\nabla B)^\mathsf{T}  \right) R'\tilde \xi'
		\\&~~~ +\big(\partial_t \lambda + (B\cdot \nabla) \lambda \,\big)(R'\widetilde \xi' - R \widetilde \xi).
	  \end{split}
	\end{align}
	
	We want to manipulate the first two right-hand side terms to make the advection equations \eqref{TransportEquationXiTildeTriod} appear. To this end, we write $B=\widetilde B + \lambda (\widetilde B' - \widetilde B)$ and obtain
	\begin{align*}
		 \left( \partial_t + (B\cdot \nabla) + (\nabla B)^\mathsf{T}  \right) R \tilde \xi 
		 =
		 &  \big( \partial_t + (\widetilde B\cdot \nabla) + (\nabla \widetilde B)^\mathsf{T}  \big) R \tilde \xi 
		 \\&+ \big(\lambda (\widetilde B' - \widetilde B) \cdot \nabla\big) R\widetilde \xi
		 + \lambda \big(\nabla\widetilde B' - \nabla\widetilde B\big)^\mathsf{T} R\widetilde \xi
		 \\& + \big((\widetilde B' - \widetilde B)\cdot R\widetilde \xi \,\big) \nabla \lambda .
	\end{align*}
	Using the compatibility conditions~\eqref{boundCompatibilityVelocity}--\eqref{boundCompatibilityGradientVelocity} 
	for the auxiliary velocity fields
	alongside with the bounds~\eqref{eq:boundsDerivativesTildeXi}, 
	\eqref{boundslambda1}, and the estimate~\eqref{eq:compDistances1} 
	one shows that the last three right-hand side terms are 
	of order $O(r_{\mathrm{min}}^{-3}\dist(\cdot,\bar I_{i,i+1}))$.
	By~\eqref{TransportEquationXiTildeTriod} and~\eqref{eq:compDistances1} the first term on the right-hand 
	side is also of order $O(r_{\mathrm{min}}^{-3}\dist(\cdot,\bar I_{i,i+1}))$.
	
	Consequently, the first term on the right-hand side of equation~\eqref{advection_interpolation} is of 
	required order.	A similar argument shows that the second one is, too.
	Finally, also the third term is of the desired order by the bounds~\eqref{boundslambda1} 
	on~$\lambda$, the second-order compatibility~\eqref{boundCompatibilityXi},
	and the estimate~\eqref{eq:compDistances1}, concluding the	proof of~\eqref{eq:prelimEstimateTimeDerivative}.

\textit{Step 4: Proof of the estimate ($r_{\mathrm{min}}:=r_{1,2}\wedge r_{2,3}\wedge r_{3,1}$)} 
\begin{align}
\label{eq:prelimEstimateMotionByCurvature}
|\nabla\cdot\widehat\xiTrJ_{i,j} + \BTrJ\cdot\widehat\xiTrJ_{i,j}| 
				&\leq Cr_{\mathrm{min}}^{-2}\dist(\cdot,{\bar{I}_{i,j}})
\quad\text{in } \mathcal{U}_r.
\end{align}

Let $i\in\{1,2,3\}$, and by the skew-symmetry $\widehat\xi_{i,j}=-\xi_{j,i}$, 
it again suffices to prove~\eqref{eq:prelimEstimateMotionByCurvature} in terms of~$\widehat\xi_{i,i+1}$.
Note that because of~\eqref{def:xi_triple_junction}--\eqref{def:velocity_triple_junction}, 
\eqref{DissipXiTildeTriod}, and~\eqref{eq:compDistances3}
it only remains to prove~\eqref{eq:prelimEstimateMotionByCurvature} for the vector
field~$\widehat\xiTrJ_{i,i+1}$ in the closure of the interpolation wedges~$W_{j}$, $j\in\{1,2,3\}$.
We again fix $j\in\{1,2,3\}$ and use the same abbreviations as in the previous step.

We proceed similarly as in the proof of~\eqref{eq:prelimEstimateTimeDerivative}.
Making use of the definition \eqref{def:xi_triple_junction} we get 
\begin{align*}
\nabla\cdot\widehat\xiTrJ_{i,i+1} &= (1{-}\lambda)\nabla\cdot R\widetilde\xi+\lambda\nabla\cdot R'\widetilde\xi' + 
\big((R'\widetilde\xi'{-}R\widetilde\xi)\cdot\nabla\big)\lambda.
\end{align*}
By the controlled blowup~\eqref{boundslambda1} of the interpolation functions,  
the compatibility estimate~\eqref{boundCompatibilityXi}, 
the approximate mean curvature flow equation~\eqref{DissipXiTildeTriod}
and the estimate~\eqref{eq:compDistances1} it then follows
\begin{align*}
\nabla\cdot\widehat\xiTrJ_{i,i+1} &=-(1-\lambda)\widetilde B\cdot R\widetilde\xi
-\lambda \widetilde B'\cdot R'\widetilde\xi'
+O(r_{\mathrm{min}}^{-2}\dist(\cdot,{\bar{I}_{i,i+1}})).
\end{align*}
Finally, the compatibility estimates~\eqref{boundCompatibilityXi}
and~\eqref{boundCompatibilityVelocity} in conjunction with 
definitions~\eqref{def:xi_triple_junction}--\eqref{def:velocity_triple_junction}
and the estimate~\eqref{eq:compDistances1} imply the desired bound~\eqref{eq:prelimEstimateMotionByCurvature}.

\textit{Step 5: Proof of the estimates ($r_{\mathrm{min}}:=r_{1,2}\wedge r_{2,3}\wedge r_{3,1}$)}
\begin{align}
\label{LengthXiTriod}
\big|1-|\widehat\xiTrJ_{i,j}|^2\big| &\leq Cr^{-2}_{\mathrm{min}}\dist^2(\cdot,\bar {I}_{i,j}) 
&&\text{in } \mathcal{U}_r,
\\
\label{DerivativesLengthXi}
r_{\mathrm{min}}^2\big|\partial_t|\widehat\xiTrJ_{i,j}|^2\big|
+ r_{\mathrm{min}}\big|\nabla|\widehat\xiTrJ_{i,j}|^2\big| 
&\leq Cr_{\mathrm{min}}^{-1}\dist(\cdot,\bar I_{i,j})
&&\text{in } \mathcal{U}_r.
\end{align}

Let $i\in\{1,2,3\}$. The validity of~\eqref{LengthXiTriod} resp.\ \eqref{DerivativesLengthXi}
for the vector field~$\widehat\xiTrJ_{i,i+1}$ in interface wedges~$W_{j,j+1}$, $j\in\{1,2,3\}$,
is directly implied by the definition~\eqref{def:xi_triple_junction}, the bound~\eqref{eq:compDistances3},
as well as the estimates~\eqref{LengthXiTildeTriod}
resp.\ \eqref{DerivativesLengthXiTilde1}--\eqref{DerivativesLengthXiTilde2}. 

For all $j\in\{1,2,3\}$, we then may compute on the closure of the interpolation wedge $W_{j}$
by~\eqref{def:xi_triple_junction} and adding zero several times
\begin{align}
\nonumber
|\widehat\xiTrJ_{i,i+1}|^2 
&= \lambda^2|R\widetilde\xi|^2
+ (1{-}\lambda)^2|R'\widetilde\xi'|^2
+ 2\lambda(1{-}\lambda)(R\widetilde\xi\cdot R'\widetilde\xi')
\\\label{length_nonconvexity}
&= 1 - \lambda(1{-}\lambda)|R\widetilde\xi - R'\widetilde\xi'|^2
+\lambda(|R\widetilde\xi|^2{-}1) + (1{-}\lambda)(|R'\widetilde\xi'|^2{-}1).
\end{align}
Hence, the estimates~\eqref{LengthXiTriod} and~\eqref{DerivativesLengthXi}
are the result of the estimates~\eqref{eq:boundsDerivativesTildeXi}, \eqref{boundCompatibilityXi},
\eqref{LengthXiTildeTriod}--\eqref{DerivativesLengthXiTilde2}, \eqref{boundslambda1}
and~\eqref{eq:compDistances1}.

\textit{Step 6: Choice of $\hat r = \hat r(\bar\Omega) \leq r$ and definition
of normalized vector fields $\xi_{i,j}$.}
We first define $\hat r := r\wedge \frac{1}{\sqrt{2C}}(r_{1,2}\wedge r_{2,3} \wedge r_{3,1})$
with~$C>0$ being the constant of~\eqref{LengthXiTriod}. Note then that~\eqref{LengthXiTriod}
implies
\begin{align}
\label{eq:nondegenerateLength}
\frac{1}{2} \leq |\widehat\xi_{i,j}|^2 \leq \frac{3}{2} \qquad
\text{in } \mathcal{U}_{\hat r}=\bigcup_{t\in[0,T]}B_{\hat r}(\mathcal{T}(t)){\times}\{t\}
\end{align}
for all $i,j\in\{1,2,3\}$ with $i\neq j$.
We may then define
\begin{align}
\label{def:normalizedXi}
\xi_{i,j}(x,t) := \frac{\widehat\xi_{i,j}(x,t)}{|\widehat\xi_{i,j}(x,t)|} \qquad
\text{for all } (x,t)\in \mathcal{U}_{\hat r}
\end{align}
and all $i,j\in\{1,2,3\}$ with $i\neq j$.
It remains to verify the asserted properties in terms of the vector fields~$\xi_{i,j}$ and~$B$
on the restricted space-time domain~$\mathcal{U}_{\hat r}$.

\textit{Step 7: Conclusion.}
Since $\xi_{i,j}(x,t)=\widehat\xi_{i,j}(x,t)$ for all $t\in [0,T]$
and all $x\in \mathcal{T}_{i,j}(t)\cap B_{\hat r}(\mathcal{T}(t))$, property~\textit{i)} is an immediate
consequence of definition~\eqref{def:normalizedXi}. 
Note that~\eqref{LengthConservationXiTriodNormalized} trivially follows. Obviously, the
skew-symmetry relation in property~\textit{ii)} carries over from
$\widehat\xi_{i,j}$ to $\xi_{i,j}$. Validity of the Herring angle 
condition~\eqref{xi_compatibility} in terms of the vector fields $\xi_{i,j}$
also follows immediately from their definition~\eqref{def:normalizedXi}, the fact that the vector fields~$\widehat\xi_{i,j}$ already
satisfy~\eqref{xi_compatibility}, and the fact that $|\widehat \xi_{1,2}|=|\widehat \xi_{2,3}|=|\widehat \xi_{3,1}|$. Indeed, recall that the vector fields
$\widehat\xi_{1,2},\,\widehat\xi_{2,3}$ resp.\ $\widehat\xi_{3,1}$ can be obtained from
each of the other ones by a rotation, see \eqref{def:xi_triple_junction} and Lemma~\ref{lemma:compatibility_xi}.

For a proof of~\eqref{boundDerivativesXi} (recall that the estimate~\eqref{boundDerivativeB}
is already part of~\textit{Step~2}), we simply compute
\begin{align}
\label{eq:derivativeNormalizedXi}
(\partial_t,\nabla)\xi_{i,j} = \frac{1}{|\widehat\xi_{i,j}|}
\Big(\mathrm{Id}{-}\frac{\widehat\xi_{i,j}}{|\widehat\xi_{i,j}|}\otimes
\frac{\widehat\xi_{i,j}}{|\widehat\xi_{i,j}|}\Big)(\partial_t,\nabla)\widehat\xi_{i,j}. 
\end{align} 
Because of~\eqref{eq:nondegenerateLength},
the estimate $\hat r|\nabla\xi_{i,j}|
+ \hat r^2|\partial_t\xi_{i,j}|\leq C$
throughout~$\mathcal{U}_{\hat r}\setminus\mathcal{T}$ thus follows
from the corresponding estimate in terms of~$\widehat\xi_{i,j}$
from~\textit{Step~2} of this proof. One proceeds similarly
for the required estimate on the second-order spatial derivative.

It therefore remains to argue that the estimates~\eqref{TransportEquationXiTriod}
and~\eqref{DissipTriod} hold true.
Using the product rule and the choice of~$\hat r$ in the previous step, we may on~$\mathcal{U}_{\hat r}$ compute
\begin{align*}
&\left( \partial_t + (B\cdot \nabla) +(\nabla B)^\mathsf{T}\right) \frac{\widehat\xi_{i,j}}{|\widehat\xi_{i,j}|}\\
&~~=\frac1{|\widehat\xi_{i,j}|} \left( \partial_t + (B\cdot \nabla) +(\nabla B)^\mathsf{T}\right)\widehat\xi_{i,j}
-\frac1{2|\widehat \xi_{i,j}|^3} \widehat\xi_{i,j} \left( \partial_t + (B\cdot \nabla) \right) |\widehat \xi_{i,j}|^2
\end{align*}
By~\eqref{eq:prelimEstimateTimeDerivative} and~\eqref{eq:nondegenerateLength}, 
the first right-hand side term is of the order $O(\hat r^{-3}\dist(\cdot,\bar I_{i,j}))$.
To handle the second term, it suffices to apply the estimate~\eqref{DerivativesLengthXi},
the estimate on the magnitude of the velocity~$|B|\leq C\hat r^{-1}$ from \textit{Step~2}, 
and the estimate~\eqref{eq:nondegenerateLength}. 
This proves the estimate \eqref{TransportEquationXiTriod}.

We now turn to the proof of \eqref{DissipTriod}. 
Here, we compute on~$\mathcal{U}_{\hat r}$ by means of the choice of~$\hat r$ in the
previous step
\begin{align*}
\nabla\cdot\frac{\widehat\xiTrJ_{i,j}}{|\widehat\xiTrJ_{i,j}|} 
&=\frac{\nabla\cdot\widehat\xiTrJ_{i,j}}{|\widehat\xiTrJ_{i,j}|}
-\frac{(\widehat\xiTrJ_{i,j}\cdot\nabla)|\widehat\xiTrJ_{i,j}|^2}{2|\widehat\xiTrJ_{i,j}|^3}.
\end{align*}
It is immediate from the estimates~\eqref{eq:nondegenerateLength} and~\eqref{DerivativesLengthXi}
to estimate the second term as being of order $O(\hat r^{-2}\dist(\cdot,{\bar{I}_{i,j}}))$. 
Using the approximate mean curvature flow equation~\eqref{eq:prelimEstimateMotionByCurvature} 
for the first term and the definition~\eqref{def:normalizedXi} of~$\xi_{i,j}$ then yields
\begin{align*}
	\nabla\cdot\frac{\widehat\xiTrJ_{i,j}}{|\widehat\xiTrJ_{i,j}|}  
	= -B\cdot\frac{\widehat\xiTrJ_{i,j}}{|\widehat\xiTrJ_{i,j}|}  
	+ O\big(\hat r^{-2}\dist(\cdot,{\bar{I}_{i,j}}) \big) 
	= -B \cdot \xi_{i,j} + O\big(\hat r^{-2}\dist(\cdot,{\bar{I}_{i,j}})\big).
\end{align*}
In total, this gives~\eqref{DissipTriod}.
\end{proof}

Finally, we provide the elementary-geometric proof for the existence of wedges
with the desired properties.

\begin{proof}[Proof of Lemma~\ref{lem:existenceLocRadius}]
We recall some notation in conjunction with Definition~\ref{DefinitionStrongSolutionTwoPhase}.
For each (cyclic)~$i\in\{1,2,3\}$ and all~$t\in [0,T]$, the unit vector~$\bar{\vec{t}}_{i,i+1}(p(t),t)$
denotes the tangent of~$\bar I_{i,i+1}(t)$ at the triple junction~$\mathcal{T}(t)=\{p(t)\}$,
with the orientation chosen such that it ``points away'' from the curve~$\bar I_{i,i+1}(t)$.
Define then~$\bar\tau_{i,i+1}(t) := -\bar{\vec{t}}_{i,i+1}(p(t),t)$ and
$\mathbb{H}_{\bar\tau_{i,i+1}}(t)=\{x\in\Rd[2]\colon (x{-}p(t))\cdot\bar\tau_{i,i+1}(t) > 0\}$.
Note that
	\begin{align}\label{wedges_balance_of_forces}
		\sigma_{1,2}\bar\tau_{1,2}(t) + \sigma_{2,3} \bar\tau_{2,3}(t) 
		+ \sigma_{3,1} \bar\tau_{3,1}(t) & = 0, \quad t\in [0,T].
	\end{align}
	Using the balance of forces condition~\eqref{wedges_balance_of_forces} 
	together with the strict triangle inequality~\eqref{TriangleInequalitySurfaceTensions} we see 
	that there exist constant-in-time angles~$\theta_i\in (0,\pi)$ 
	such that~$\cos(\theta_i) = \bar\tau_{i,i+1}(t) \cdot \bar\tau_{i-1,i}(t)$ for~$i=1,2,3$
	and~$t\in [0,T]$. For the following argument, see also Figure~\ref{fig:wedges_proof}.
	
	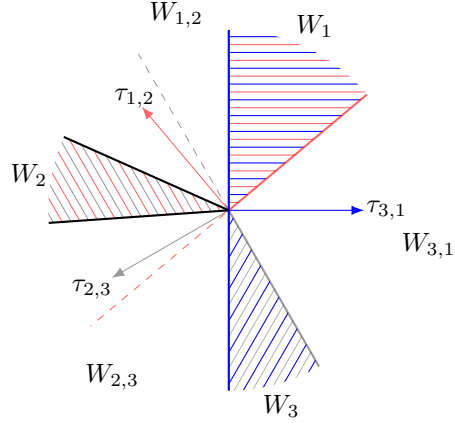
\begin{figure}
	  \begin{tikzpicture}[scale=3]
	  	\tikzmath{\l = .6;\L=.9;};
	  	\begin{scope}
	  		 \clip (0,0) circle [radius = .8];
	  		 \begin{scope}
	 			\clip (0,0) -- (40:1.5) --(90:1.5) -- (0,0);
	 			\foreach \i in {0,...,15}
				{
					\draw[color=blue] (-.8,\i*.07) -- (.8,\i*.07);
					\draw[color=red, opacity=.6] (-.8,\i*.07+.035 ) -- (.8,\i*.07 +.035);
				};
	 		\end{scope}
	 		
	 		\begin{scope}
	 			\clip (0,0) -- (156:1.5) --(184:1.5) -- (0,0);
	 			\foreach \i in {0,...,15}
				{
					\draw[opacity=.4] ($(120:1.5) + (210:\i*.07)$) -- ($(120:-1.5) + (210:\i*.07)$);
					\draw[color=red, opacity=.6] ($(120:1.5) + (210:\i*.07+.035)$) -- ($(120:-1.5) + (210:\i*.07+ .035)$);
				};
	 		\end{scope}
	 		
	 		\begin{scope}
	 			\clip (0,0) -- (270:1.5) --(300:1.5) -- (0,0);
	 			\foreach \i in {0,...,15}
				{
					\draw[opacity=.3] ($(240:1.5) + (330:\i*.07)$) -- ($(240:-1.5) + (330:\i*.07)$);
					\draw[color=blue] ($(240:1.5) + (330:\i*.07+.035)$) -- ($(240:-1.5) + (330:\i*.07+ .035)$);
				};
	 		\end{scope}
	 	
	  	\draw[-Latex,color=blue] (0,0) -- (0:\l);
	  	\node at ($(0:\l+.1)$) {$\tau_{3,1}$};
	  	\draw[-Latex,color=red,opacity=.6] (0,0) -- (130:\l);
	  	\node at ($(130:\l+.05)$) {$\tau_{1,2}$};
	  	\draw[-Latex,opacity=.4] (0,0) -- (210:\l);
	  	\node at ($(210:\l+.1)$) {$\tau_{2,3}$};
		
		\draw[dashed,color=red,opacity=.6] ($(40:-\L)$)--(40:0);
		
		\draw[dashed,opacity=.4] (120:0)--(120:\L);

%
		\draw[color=red,opacity=.6,thick] (0,0) -- (40:\L);
		\draw[color=blue,thick] (0,0) -- (90:\L);

		\draw[thick] (0,0) -- (156:1);
		\draw[thick] (0,0) -- (184:1);

		\draw[opacity=.4,thick] (0,0) -- (300:\L);
		\draw[color=blue,thick] (0,0) -- (270:\L);
		
		\end{scope}
		
				
		\node at  ($(285:\L)$)  {$W_{3}$};
		\node at ($(65:\L)$) {$W_{1}$};
		\node at ($(170:\L)$) {$W_{2}$};
		
		\node at (105:.9) {$W_{1,2}$};
		\node at (235:.9) {$W_{2,3}$};
		\node at (350:.9) {$W_{3,1}$};
	  \end{tikzpicture}
	  \caption{If the angle between two tangent vectors is less than~$90^\circ$, 
		we trisect it to obtain the desired interpolation wedge, see for example~$W_2$. 
		Otherwise, we take the corresponding intersection of the half-spaces, as is done 
		for $W_1$ and $W_3$. The wedges $W_{1,2}$, $W_{2,3}$ and $W_{3,1}$ lie inbetween. 
		\label{fig:wedges_proof}}
	\end{figure}
	
	If~$\theta_i > \frac{\pi}{2}$ we may define~$X^i_{i,i+1}(t), X^{i}_{i-1,i}(t) \in \mathbb{S}^1$ 
	such that the cone~$C_i(t):=\mathcal{T}(t) + \{\gamma_1 X^i_{i,i+1}(t) + \gamma_2 X^{i}_{i-1,i}(t) 
	\colon \gamma_1,\gamma_2 \in (0,\infty) \}$
	satisfies~$C_i(t)=\mathbb{H}_{\bar\tau_{i,i+1}}(t)\cap \mathbb{H}_{\bar\tau_{i-1,i}}(t)$.
	Otherwise, we choose~$X^i_{i,i+1}(t), X^{i}_{i-1,i}(t) \in \mathbb{S}^1$ 
	such that the cone~$C_i(t):=\mathcal{T}(t) + \{\gamma_1 X^i_{i,i+1}(t) + \gamma_2 X^{i}_{i-1,i}(t) 
	\colon \gamma_1,\gamma_2 \in (0,\infty) \}$	is the middle third of the 
	cone~$\{\gamma_1 \bar\tau_{i,i+1}(t) + \gamma_2 \bar\tau_{i-1,i}(t) \colon \gamma_1,\gamma_2 \in (0,\infty)\}$.
	In both cases, defining for~$i\in\{1,2,3\}$ and~$t\in [0,T]$ the cone
	$C_{i,i+1}(t) := \mathcal{T}(t) + \{\gamma_1 X^{i}_{i,i+1}(t) {+} 
	\gamma_2 X^{i+1}_{i,i+1}(t)\colon \gamma_1,\gamma_2\in (0,\infty)\}$ we then have
	\begin{align}
	\label{eq:inclusionInterpolCone}
	C_{i}(t)	&\subset \mathbb{H}_{\bar\tau_{i,i+1}}(t)\cap \mathbb{H}_{\bar\tau_{i-1,i}}(t),
	\\
	\label{eq:inclusionInterfaceCone}
	C_{i,i+1}(t) &\subset \mathbb{H}_{\bar\tau_{i,i+1}}(t), 
	\\
	\label{eq:decompFullSpace}
	\bigcup_{i=1,2,3} \overline{C_i(t)} \cup \overline{C_{i,i+1}(t)} &= \Rd[2],
	\\
	\label{eq:inclusionTangent}
	p(t) + \tau_{i,i+1}(t) &\in C_{i,i+1}(t) 
	\end{align}
	for all~$i\in \{1,2,3\}$ and all~$t\in [0,T]$.
	
	Let~$r\in (0, r_{1,2}\wedge r_{2,3}\wedge r_{3,1}]$, and for~$i\in\{1,2,3\}$
	and~$t\in [0,T]$ define~$W_{i}(t):=C_i(t)\cap B_r(\mathcal{T}(t))$
	and~$W_{i,i+1}(t):=C_{i,i+1}(t)\cap B_{r}(\mathcal{T}(t))$. As~\eqref{eq:decompByWedges}
	follows immediately from~\eqref{eq:decompFullSpace} it suffices to argue that
	there exists a constant~$C=C(\sigma)\geq 1$, depending only on the surface tensions at the triple junction,
	such that~$r:=\frac{1}{C}(r_{1,2}\wedge r_{2,3}\wedge r_{3,1})$ gives rise to the 
	inclusions~\eqref{eq:inlcusionInterfaceWedge}--\eqref{eq:inclusionInterpolWedge}
	and the comparability of distances in form of~\eqref{eq:compDistances1}--\eqref{eq:compDistances3}.
	
	First, \eqref{eq:inclusionInterpolWedge} follows from~\eqref{eq:inclusionInterpolCone} and
	the fact that~$\mathbb{H}_{\bar\tau_{i,i+1}}(t)\cap B_{r}(\mathcal{T}(t))$
	is included in the $t$-time slice of the image of the diffeomorphism from~\eqref{DiffeoTubularNeighborhood},
	see~\eqref{eq:imageDiffeo}. Analogously, one derives the second inclusion of~\eqref{eq:inlcusionInterfaceWedge}
	from~\eqref{eq:inclusionInterfaceCone}. For the first inclusion of~\eqref{eq:inlcusionInterfaceWedge},
	i.e., the curve trapping condition, one may argue as follows. On one side, it follows
	from the endpoint ball condition~\textit{ii)} of Definition~\ref{DefinitionStrongSolutionTwoPhase}
	and~$r\leq r_{1,2}\wedge r_{2,3}\wedge r_{3,1}$
	that~$\mathcal{T}_{i,i+1}(t)\cap B_{r}(\mathcal{T}(t))\subset 
	\overline{\mathbb{H}_{\bar\tau_{i,i+1}}(t)} \cap B_{r}(\mathcal{T}(t))$.
	On the other side, based on the ball condition~\textit{i)} of Definition~\ref{DefinitionStrongSolutionTwoPhase}
	at the triple junction~$\mathcal{T}(t)=\{p(t)\}$, we may sharpen this inclusion to
	\begin{align*}
	&\mathcal{T}_{i,i+1}(t)\cap B_{r}(\mathcal{T}(t))
	\\&
	\subset \Big(\overline{\mathbb{H}_{\bar\tau_{i,i+1}}(t)} \cap B_{r}(\mathcal{T}(t))\Big)
	\setminus \Big(B_{r}\big(p(t){+}r\bar{\vec{n}}_{i,i+1}(p(t),t)\big)
	\cup B_{r}\big(p(t){-}r\bar{\vec{n}}_{i,i+1}(p(t),t)\big)\Big).
	\end{align*}
	Hence, the first inclusion of~\eqref{eq:inlcusionInterfaceWedge} follows
	after choosing~$r\in (0, r_{1,2}\wedge r_{2,3}\wedge r_{3,1}]$ sufficiently small,
	with a proportionality constant depending only on the opening angles of the interface cones~$C_{i,i+1}$.
	
	We turn to the proof of the estimates~\eqref{eq:compDistances1}--\eqref{eq:compDistances2}.
	The estimate~\eqref{eq:compDistances3} is a consequence of the first inclusion of~\eqref{eq:inlcusionInterfaceWedge},
	the fact that the interface wedges~$W_{i,i+1}$, $i\in\{1,2,3\}$, are separated from each
	other by the interpolation wedges~$W_{i}$, $i\in\{1,2,3,\}$,
	and that within~$B_{r}(\mathcal{T}(t))$ the distance to~$\mathcal{T}_{i,i+1}$ equals
	the distance to~$\bar I_{i,i+1}$ by Definition~\ref{DefinitionStrongSolutionTwoPhase}
	and~$r\in (0, r_{1,2}\wedge r_{2,3}\wedge r_{3,1}]$. The estimate~\eqref{eq:compDistances2}
	follows from similar considerations, exploiting again that the interface wedges are
	separated from each other by the interpolation wedges.
	Also the argument for the proof
	of~\eqref{eq:compDistances1} is analogous; at least once we improved the curve trapping
	condition~\eqref{eq:inlcusionInterfaceWedge} to a wedge which is strictly included in~$W_{i,i+1}$. 
	A possible choice for such a wedge is to simply bisect
	the angles formed by $\bar\tau_{i,i+1},X^{i}_{i,i+1}$ and~$\bar\tau_{i,i+1},X^{i+1}_{i,i+1}$,
	respectively. The improvement of~\eqref{eq:inlcusionInterfaceWedge} then follows
	from possibly reducing~$r\in (0, r_{1,2}\wedge r_{2,3}\wedge r_{3,1}]$ even further.
	This in turn can be done again at the cost of a proportionality constant depending only on the
	surface tensions at the triple junction.
\end{proof}

\subsection{Local compatibility estimates}
We conclude this section with a result verifying
that the local constructions at a triple junction
from Proposition~\ref{prop:xi_triple_junction} 
are (in a certain sense) suitable perturbations of
the respective local constructions
from Lemma~\ref{LemmaBoundsLocalConstructionsTwoPhase} 
with respect to interfaces meeting at the triple junction.
It is precisely at this stage where we rely on the 
freedom to choose a tangential component for the
local velocity field from Lemma~\ref{LemmaBoundsLocalConstructionsTwoPhase}.

\begin{proposition}
\label{prop:localCompatibilityEstimates}
Let $d=2$ and $P \in \mathbb{N}$, $P\geq 2$. Let $\bar\Omega=(\bar{\Omega}_1,\ldots,\bar{\Omega}_P)$ 
be a strong solution to multiphase mean curvature flow in the sense of Definition~\ref{DefinitionStrongSolution}.
Let~$i,j\in\{1,\ldots,P\}$ such that~$i\neq j$ and~$\bar I_{i,j}$ is a non-trivial interface.
Denote by~$\mathcal{T}_c$ a space-time connected component of~$\bar I_{i,j}$, and assume that
$\mathcal{T}_c$ connects two evolving triple junctions~$\mathcal{T}_{p_{+}}$ and~$\mathcal{T}_{p_{-}}$, respectively.
Let~$\hat r_{p_+},\hat r_{p_{-}}\in (0,1]$ denote the associated localization scales 
from Proposition~\ref{prop:xi_triple_junction},
respectively. Finally, denote by~$(\xi_{i,j}^c,B^c)$ the local vector fields from 
Lemma~\ref{LemmaBoundsLocalConstructionsTwoPhase}.

Then there exists a choice of the tangential component~$\gamma_c$ of~$B^c$ satisfying
\begin{align}
\label{eq:estimateDerivTangentialComp}
\max_{k=0,1,2} (\hat r_{p_+} \wedge \hat r_{p_-} \wedge \ell)^{k+1} |\nabla^k \gamma_c| \leq C,
\quad 3\ell := \min_{t\in [0,T]} \dist(\mathcal{T}_{p_+}(t),\mathcal{T}_{p_-}(t)),
\end{align}
throughout~$\mathrm{im}(\Psi_{\mathcal{T}_c})$ as well as \eqref{RequirementForGamma} on $\mathcal{T}_c$,
so that at each of the two triple junctions~$\mathcal{T}_{p}$, $p\in\{p_+,p_-\}$, the
local vector fields~$(\xi^{p}_{i,j},B^{p})$ from~Proposition~\ref{prop:xi_triple_junction} 
(at scale~$\hat r_{p}$) may be chosen so that they are locally compatible
with~$(\xi^c_{i,j},B^c)$ in the sense that
\begin{align}
\label{eq:localComp1}
\big|\xi_{i,j}^c{-}\xi_{i,j}^p\big| 
+ \hat r_p\big|(\nabla\xi^c_{i,j}{-}\nabla\xi_{i,j}^p)^\mathsf{T}\xi_{i,j}^c\big|
&\leq C\hat r_p^{-1}\dist(\cdot,\bar I_{i,j}),
\\
\label{eq:localComp2}
\big|(\xi_{i,j}^c{-}\xi_{i,j}^p)\cdot\xi_{i,j}^c\big|
&\leq C\hat r_p^{-2}\dist^2(\cdot,\bar I_{i,j}),
\\
\label{eq:localComp3}
\big|B^p {-} B ^c\big|
&\leq C\hat r_p^{-3}\dist^2(\cdot,\bar I_{i,j}),
\\
\label{eq:localComp4}
\big|\nabla B^p {-} \nabla B ^c\big|
&\leq C\hat r_p^{-3}\dist(\cdot,\bar I_{i,j})
\end{align}
in the region~$B_{\frac{1}{2}(\hat r_p \wedge \ell)}(\mathcal{T}_p(t)) \cap 
\big(W^p_{i,j}(t)\cup W^p_{i}(t) \cup W^p_{j}(t)\big)$
for all~$t\in [0,T]$ (where the wedges~$W^p_{i,j},W^p_{i},W^p_{j}$ are the ones
from Definition~\ref{def:locRadiusTripleJunction} 
with respect to the triple junction~$\mathcal{T}_p$).
The constant~$C>0$ in the above estimates~\emph{\eqref{eq:estimateDerivTangentialComp}--\eqref{eq:localComp4}}
may depend on~$\bar\Omega$, but is independent of~$\hat r_{p_+}$, $\hat r_{p_-}$ and~$\ell$.
\end{proposition}

\begin{proof}
The proof is split into three steps.

\textit{Step 1: Choice of vector fields.} We take~$(\xi_{i,j}^{p_\pm},B^{p_\pm})$
as constructed in the proof of Proposition~\ref{prop:xi_triple_junction}.
Moreover, we take~$(\xi_{i,j}^c,B^c)$ as defined in Lemma~\ref{LemmaBoundsLocalConstructionsTwoPhase}
with the following choice of the tangential component~$\gamma_c$. 
Let~$\theta$ be a smooth cutoff function with~$\theta(r)=1$ for~$|r|\leq \frac{1}{2}$
and~$\theta\equiv 0$ for~$|r|\geq 1$. We then define
\begin{align}
\label{def:tangentialVelocity}
\gamma_c := \theta\Big(\frac{\dist(\cdot,\mathcal{T}_{p_+})}{\ell\wedge\hat r_{p_+}}\Big)B^{p_+}\cdot\bar\tau_{i,j} 
					 + \theta\Big(\frac{\dist(\cdot,\mathcal{T}_{p_-})}{\ell\wedge\hat r_{p_-}}\Big)B^{p_-}\cdot\bar\tau_{i,j}
					\quad \text{on } \mathcal{T}_c,
\end{align}
and extend this definition to $\mathrm{im}(\Psi_{\mathcal{T}_c})$ by a suitable Taylor expansion to match~\eqref{RequirementForGamma}.
By the choice of the cutoff~$\theta$, this is indeed well-defined. The regularity
estimate~\eqref{eq:estimateDerivTangentialComp} is a direct consequence of the definition~\eqref{def:tangentialVelocity}
and the estimates~\eqref{Bound2ndDerivativeNormalTangent} and~\eqref{boundDerivativeB}.
Note that~\eqref{eq:estimateDerivTangentialComp} 
in turn updates the estimate~\eqref{eq:estimatesTwoPhaseVel} to
\begin{align}
\label{eq:estimatesTwoPhaseVelFinal}
\max_{k=0,1,2} (\hat r_{p_+} \wedge \hat r_{p_-} \wedge \ell)^{k+1} |\nabla^kB^c| \leq C
\quad\text{in } \mathrm{im}(\Psi_{\mathcal{T}_c}),
\end{align}
with the constant~$C>0$ being independent of~$\hat r_{p_+}$, $\hat r_{p_-}$ and~$\ell$.

\textit{Step 2: Proof of~\eqref{eq:localComp3} and~\eqref{eq:localComp4}.}
Let~$p\in \{p_+,p_-\}$. First, we note that for all~$t\in [0,T]$ it 
holds~$B_{\frac{1}{2}(\hat r_p \wedge \ell)}(\mathcal{T}_p(t)) \cap 
\big(W^p_{i,j}(t)\cup W^p_{i}(t) \cup W^p_{j}(t)\big) \subset \mathrm{im}(\Psi_{\mathcal{T}_c})$
due to~\eqref{eq:inlcusionInterfaceWedge}--\eqref{eq:inclusionInterpolWedge}.
By means of the regularity estimates~\eqref{boundDerivativeB} 
and~\eqref{eq:estimatesTwoPhaseVelFinal}, the choice of the cutoff function~$\theta$, and
the definition~\eqref{def:tangentialVelocity} of the tangential velocity of~$B^c$, it thus suffices to prove~$B^c = B^p$
within the interface wedge~$W_{i,j}^p(t)\cap B_{\frac{1}{2}(\hat r_p \wedge \ell)}(\mathcal{T}_p(t))$
for all~$t\in [0,T]$. However, by~\eqref{def:tangentialVelocity} the two
vector fields agree in tangential direction. Their normal component in turn equals~$H_{i,j}\bar{\vec{n}}_{i,j}$,
which is evident for~$B^c$ from definition~\eqref{DefinitionVelocityTwoPhase}, and for~$B^p$ from
the definitions~\eqref{AnsatzAuxiliaryVelocityHalfSpace} and~\eqref{def:velocity_triple_junction}.

\textit{Step 3: Proof of~\eqref{eq:localComp1} and~\eqref{eq:localComp2}.}
Let again~$p\in \{p_+,p_-\}$. Thanks to the regularity estimates~\eqref{eq:estimatesTwoPhaseXi}
resp.\ \eqref{boundDerivativesXi} and the fact~$(\nabla\xi^c_{i,j})^\mathsf{T}\xi^c_{i,j}
=\frac{1}{2}\nabla|\xi^c_{i,j}|^2=0$,
the asserted bounds~\eqref{eq:localComp1} and~\eqref{eq:localComp2}
follow once we assured ourselves of the validity of~$\xi^c_{i,j}-\xi^p_{i,j}=0$
and~$(\nabla\xi_{i,j}^p)^\mathsf{T}\xi_{i,j}^c=0$
along the local interface segment~$\mathcal{T}_{c}(t)\cap B_{\frac{1}{2}(\hat r_p \wedge \ell)}(\mathcal{T}_p(t))$
for all~$t\in [0,T]$. The former is immediate from both vector fields being extensions of
the unit normal~$\bar{\vec{n}}_{i,j}|_{\bar I_{i,j}}$, whereas the latter then follows
from adding zero and~$|\xi_{i,j}^p|^2\equiv 1$: $(\nabla\xi_{i,j}^p)^\mathsf{T}\xi_{i,j}^c=
(\nabla\xi_{i,j}^p)^\mathsf{T}\xi_{i,j}^p=\frac{1}{2}\nabla|\xi^p_{i,j}|^2=0$.
\end{proof}

\section{Gradient flow calibrations for a regular network}
\label{sec:networkConstruction}
The aim of this section is to prove Theorem~\ref{TheoremExistenceCalibration}: Given 
a strong solution to multiphase mean curvature flow (in the sense of an evolving network of 
smooth curves meeting at triple junctions), we construct a gradient flow calibration by gluing 
together the local constructions from the previous two sections.

More precisely, in Section~\ref{SectionPartitionOfUnity} we define a partition of unity which 
allows us to localize around each topological feature $\mathcal{T}_n$, i.e., a two phase interface 
or a triple junction, for some suitable index $n \in \mathbb{N}$.
We then define the global vector fields $\xi_{i,j}$ for $i,j \in \{1,\ldots, P\}$ with $i \neq j$
and $B$ in Section~\ref{SectionGlobalDefinitions} by gluing together suitable locally defined vector 
fields $\xi^n_{i,j}$ and $B^n$. Most of these vector fields were already constructed in 
Sections~\ref{SectionLocalConstructionsTwoPhase} and \ref{SectionLocalConstructionsTriod}, so that in 
Section~\ref{SectionGlobalDefinitions} we only need to define those vector fields $\xi_{i,j}^n$ 
for which at least one of the two phases $i$ or $j$ is not present at the selected topological feature $\mathcal{T}_n$. 
For their construction we crucially use the coercivity condition of Definition~\ref{DefinitionAdmissibleSurfaceTensions}
on the matrix of surface tensions. 
In Section~\ref{SectionCompatibilityBoundsGluing}, we prove the compatibility between the 
local constructions of the vector fields of adjacent topological features, which then 
allows us in Section~\ref{SectionBoundsTimeEvolutionGlobalConstruction} 
to prove Theorem~\ref{TheoremExistenceCalibration}.

We first describe the necessary notation. Let $\bar\Omega=(\bar{\Omega}_1,\ldots,\bar{\Omega}_P)$ 
be a strong solution for multiphase
mean curvature flow in the sense of Definition~\ref{DefinitionStrongSolution} on some time interval $[0,T]$.
In particular, the family~$\bar\Omega$ is a smoothly evolving regular partition
and the family~$\mathcal{I}=\bigcup_{i\neq j}\bar I_{i,j}$ is a smoothly evolving regular network of interfaces 
in the sense of Definition~\ref{DefinitionSmoothlyEvolvingPartition}. 

We decompose the network of interfaces of the strong solution according to its topological features, 
i.e., into smooth two-phase interfaces on the one hand and triple junctions on the other hand. 
Suppose that the strong solution has $N$ of such topological features $\mathcal{T}_n$, $n \in \{1,\ldots,N\}$. 
We then split $\{1,\ldots,N\}=:\mathcal{C} \cupdot\mathcal{P}$ with the convention 
that $\mathcal{C}$ enumerates the connected components in space-time of the smooth two-phase 
interfaces (being time-evolving curves) and $\mathcal{P}$ enumerates the triple 
junctions (being time-evolving points).
If $p\in \mathcal{P}$, we define $\mathcal{T}_p:=\bigcup_{t\in [0,T]}\mathcal{T}_p(t){\times}\{t\}$
to be the trajectory in space-time described by the triple junction. 
If $c\in \mathcal{C}$, we define $\mathcal{T}_c := \bigcup_{t\in [0,T]}\mathcal{T}_c(t){\times}\{t\} \subset {\bar{I}}_{i,j}$ 
for some $i,j \in \{1,\ldots,P\}$ with $i\neq j$ to be the corresponding
space-time connected component of a two-phase interface ${\bar{I}}_{i,j}$.
We say that the $i$-th phase of the strong solution is
\emph{present at the topological feature $\mathcal{T}_{n}$} for $n \in \{1,\ldots,N\}$
if $\partial\bar\Omega_i\cap\mathcal{T}_{n}\neq\emptyset$.
Otherwise, we say that the phase is \emph{absent at $\mathcal{T}_n$}.
Finally, we write~$c\sim p$ for~$c\in\mathcal{C}$ and~$p\in\mathcal{P}$
if and only if~$\mathcal{T}_c$ has an endpoint at~$\mathcal{T}_p$.
Otherwise, we write~$c\not\sim p$.

For each~$p\in\mathcal{P}$, let~$\hat r_p\in (0,1]$ denote the 
localization scale provided by Proposition~\ref{prop:xi_triple_junction},
and for each~$i,j\in\{1,\ldots,P\}$ such that~$i\neq j$ let~$r_{i,j}\in (0,1]$
be an admissible localization scale for the interface~$\bar I_{i,j}$ in the sense
of Definition~\ref{DefinitionStrongSolutionTwoPhase}. We also define
\begin{align*}
3\ell_{\mathcal{P}} := 1\wedge \min_{t\in [0,T]}\min_{p,p'\in\mathcal{P},\, p\neq p'} 
\dist(\mathcal{T}_{p}(t),\mathcal{T}_{p'}(t)).
\end{align*} 
In words, $\ell_{\mathcal{P}}$ keeps track of the separation of the triple junctions.
Moreover, for each~$c\in\mathcal{C}$ we let
\begin{align*}
3\ell_c &:= 1\wedge \min_{t\in [0,T]}\min_{c'\in\mathcal{C}\setminus\{c\}
\colon \mathcal{T}_{c}\cap\mathcal{T}_{c'}=\emptyset}
\dist(\mathcal{T}_{c}(t),\mathcal{T}_{c'}(t)).
\end{align*}
If~$c\in\mathcal{C}$ refers to a closed loop, then~$\ell_c$ measures
the separation to all other topological features. Otherwise, $c\in\mathcal{C}$
refers to a two-phase interface with two triple junction endpoints, and in this case
$\ell_c$ represents the minimal distance to all other topological features
except for the two triple junctions at its endpoints and the set of two-phase interfaces
also having an endpoint at these triple junctions. 
We then define
\begin{align}
\label{def:minLocScaleTripleJunction}
2r_{\mathcal{P}} &:=
\min_{p\in\mathcal{P}} \hat r_p \wedge \ell_{\mathcal{P}} 
\wedge \min_{c\in\mathcal{C}} \ell_c 
\in (0,1].
\end{align}
Note that~$r_{\mathcal{P}}$ allows for the application of all the 
results from Section~\ref{SectionLocalConstructionsTriod}, and that distinct
triple junctions are well separated. In addition, the $r_{\mathcal{P}}$-ball
around a triple junction~$\mathcal{T}_p$ intersects with the $r_{\mathcal{P}}$-neighborhood
of a two-phase interface~$\mathcal{T}_c$ if and only if~$c\sim p$.

Next, in case~$c\in\mathcal{C}$ does not refer to a closed loop, i.e., 
there exists exactly two $p_{+},p_{-}\in\mathcal{P}$ such that
$c\sim p_{+}$ and $c\sim p_{-}$, we consider
\begin{align*}
3\ell'_c &:= 1\wedge\min_{t\in [0,T]}\min_{\substack{c'\in\mathcal{C}\setminus\{c\} \\
							c'\sim p,\,p\in\{p_{\pm}\}}}\dist\Big(\mathcal{T}_c(t)\setminus 
							\bigcup_{p\in\{p_{\pm}\}}B_{r_{\mathcal{P}}}(\mathcal{T}_p(t)),\mathcal{T}_{c'}(t) \Big).
\end{align*}
The purpose of~$\ell'_c$
is to separate interfaces which meet at the same triple junction; at least outside
of a neighborhood of the latter. We then define
\begin{align*}
2r_{\mathcal{C}} := \min_{i,j\in\{1,\ldots,P\},\,i\neq j} r_{i,j} 
\wedge \min_{c\in\mathcal{C}} \ell_c 
\wedge \min_{c\in\mathcal{C}\colon 
\exists p\in\mathcal{P}\text{ s.t. } c\sim p} \ell'_c \in (0,1]	.
\end{align*} 
Observe that the scale~$r_{\mathcal{C}}$ allows for the application
of all the results from Section~\ref{SectionLocalConstructionsTwoPhase},
and that distinct interfaces are well separated at this scale in the previously described sense.

Finally, it is convenient to define
a minimal localization scale by means of
\begin{align}
\label{def:minLocScale}
\bar r_{\mathrm{min}} := r_{\mathcal{C}} \wedge r_{\mathcal{P}} > 0.
\end{align}

\subsection{Localization of topological features}\label{SectionPartitionOfUnity}
We now introduce a partition of unity $(\eta_{\mathrm{bulk}},\eta_1,\ldots,\eta_N)$, 
where each $\eta_n$ for $n=1,\ldots,N$ localizes in a neighborhood of the corresponding 
topological feature~$\mathcal{T}_n$ as follows:

\begin{lemma}\label{LemmaPartitionOfUnity}
Let $d=2$ and $P \in \mathbb{N}$, $P\geq 2$. Let $\bar\Omega=(\bar{\Omega}_1,\ldots,\bar{\Omega}_P)$ 
be a strong solution to multiphase mean curvature flow in the sense of Definition~\ref{DefinitionStrongSolution},
whose network of interfaces decomposes into $N$ topological features $\mathcal{T}_n$, $n\in\{1,\ldots,N\}$.
Let~$r_{\mathcal{P}},\bar r_{\mathrm{min}}\in (0,1]$ be the localization scales 
defined by~\eqref{def:minLocScaleTripleJunction} and~\eqref{def:minLocScale},
and let~$\mathcal{T}_{\mathcal{P}}:=\bigcup_{p\in\mathcal{P}}\mathcal{T}_p$.

Then, for each $n\in\{1,\ldots,N\}$ there exists a continuous function
\begin{align*}
\eta_n\colon \Rd[2]\times [0,T] \to [0,1]
\end{align*} 
satisfying~$\eta_n\in (C^0_tC^2_x\cap C^1_tC^0_x)(\Rd[2]{\times} [0,T]
\setminus \mathcal{T}_{\mathcal{P}})$
with corresponding estimates
\begin{align}
\label{eq:estimatesDerivEta}
\max_{k=1,2} \bar r_{\mathrm{min}}^k|\nabla^k\eta_n|
+ \bar r_{\mathrm{min}}^2|\partial_t\eta_n| \leq C
\quad\text{in } \Rd[2]{\times} [0,T] \setminus\mathcal{T}_{\mathcal{P}}, 
\end{align}
for some constant~$C>0$, depending only on~$\bar\Omega$ but
not on~$\bar r_{\mathrm{min}}$,
so that the family $(\eta_1,\ldots,\eta_N)$ is a partition of unity in
the following sense: 
\begin{itemize}[leftmargin=0.7cm]
\item[i)] Let $\eta_{\mathrm{bulk}}:=1-\sum_{n=1}^N \eta_n$. Then $\eta_{\mathrm{bulk}}\in [0,1]$
					throughout~$\Rd[2]{\times }[0,T]$. On the evolving network of interfaces
					$\mathcal{I}:=\bigcup_{i\neq j}{\bar{I}}_{i,j}$ we have~$\eta_{\mathrm{bulk}} \equiv 0$.
					Moreover, there exists a constant~$C\geq 1$, depending only on~$\bar\Omega$ but not on~$\bar r_{\mathrm{min}}$,
					such that it holds
					\begin{align}
					\label{eq:coercivityBulkCutOff}
					C^{-1}\big(\bar r_{\mathrm{min}}^{-2}\dist^2(\cdot,\mathcal{I})\wedge 1\big)
					&\leq \eta_{\mathrm{bulk}}
					&&\text{in } \Rd[2]{\times} [0,T] \setminus\mathcal{T}_{\mathcal{P}},
					\\
					\label{eq:upperBoundBulkCutOff}
					\eta_{\mathrm{bulk}} &\leq C\big(\bar r_{\mathrm{min}}^{-2}\dist^2(\cdot,\mathcal{I})\wedge 1\big)
					&&\text{in } \Rd[2]{\times} [0,T] \setminus\mathcal{T}_{\mathcal{P}},
					\\
					\label{eq:boundDerivBulkCutoff}
					|\nabla\eta_{\mathrm{bulk}}| &\leq C\bar r_{\mathrm{min}}^{-1}
					\big(\bar r_{\mathrm{min}}^{-1}\dist(\cdot,\mathcal{I})\wedge 1\big)
					&&\text{in } \Rd[2]{\times} [0,T] \setminus\mathcal{T}_{\mathcal{P}},
					\\
					\label{eq:boundTimeDerivBulkCutoff}
					|\partial_t\eta_{\mathrm{bulk}}| &\leq C\bar r_{\mathrm{min}}^{-2}
					\big(\bar r_{\mathrm{min}}^{-1}\dist(\cdot,\mathcal{I})\wedge 1\big)
					&&\text{in } \Rd[2]{\times} [0,T] \setminus\mathcal{T}_{\mathcal{P}},
					\end{align}
					and if either phase~$i$ or phase~$j$ is absent at a given
					topological feature $n\in\{1,\ldots,N\}$ we have the estimates
					\begin{align}
					\label{GlobalEquationsBoundMinorityPhases}
					\eta_n &\leq C\big(\bar r_{\mathrm{min}}^{-2}\dist^2(\cdot,{\bar{I}}_{i,j})\wedge 1\big)
					&&\text{in } \Rd[2]{\times} [0,T] \setminus\mathcal{T}_{\mathcal{P}},
					\\
					\label{GlobalEquationsBoundGradientMinorityPhases}
					|\nabla\eta_n| &\leq C\bar r_{\mathrm{min}}^{-1}
					\big(\bar r_{\mathrm{min}}^{-1}\dist(\cdot,{\bar{I}}_{i,j})\wedge 1\big)
					&&\text{in } \Rd[2]{\times} [0,T] \setminus\mathcal{T}_{\mathcal{P}},
					\\
					\label{GlobalEquationsBoundTimeDerivMinorityPhases}
					|\partial_t\eta_n| &\leq C\bar r_{\mathrm{min}}^{-2}
					\big(\bar r_{\mathrm{min}}^{-1}\dist(\cdot,{\bar{I}}_{i,j})\wedge 1\big)
					&&\text{in } \Rd[2]{\times} [0,T] \setminus\mathcal{T}_{\mathcal{P}}.
					\end{align}
\item[ii)] For all~$c\in \mathcal{C}$ and~$t\in [0,T]$ it holds
					 \begin{align}
					 \label{LocalizationTwoPhase}
					 \supp\eta_c(\cdot,t)\subset \Psi_{\mathcal{T}_c}(\mathcal{T}_c(t){\times}\{t\}
					 {\times}[\bar r_{\mathrm{min}},\bar r_{\mathrm{min}}])
					 =:\mathrm{im}_{\bar r_{\mathrm{min}}}(\Psi_{\mathcal{T}_c})(t),
					 \end{align}
					 with~$\Psi_{\mathcal{T}_c}$ denoting the restriction to~$\mathcal{T}_c$
					 of the diffeomorphism~\eqref{DiffeoTubularNeighborhood} (assuming
					that~$\mathcal{T}_c\subset\bar I_{i,j}$).
\item[iii)] For all~$p \in  \mathcal{P}$ and~$t\in [0,T]$ it holds
						\begin{align}\label{LocalizationTripleJunction}
						\supp\eta_p(\cdot,t)\subset B_{r_{\mathcal{P}}}(\mathcal{T}_p (t)).
						\end{align}
\item[iv)] Let $p, p'\in\mathcal{P}$ be two distinct triple junctions. Then for all~$t\in [0,T]$ we have
				   \begin{align}\label{LocalizationTwoTripleJunctions}
					 \supp\eta_p(\cdot,t)\cap\supp\eta_{p'}(\cdot,t)\subset 
					 B_{r_{\mathcal{P}}}(\mathcal{T}_p(t))\cap B_{r_{\mathcal{P}}}(\mathcal{T}_{p'}(t))
					 =\emptyset.
					 \end{align}
\item[v)] Let $p \in \mathcal{P}$ be a triple junction and let $c \in \mathcal{C}$ be 
					a two-phase interface. Then $\supp\eta_p\cap\supp\eta_c\neq\emptyset$ if and only if
					$\mathcal{T}_{c}$ has an endpoint at $\mathcal{T}_{p}$.
					In this case and assuming $\mathcal{T}_c\subset {\bar{I}}_{i,j}$ 
					for $i\neq j \in \{1,\ldots, P\}$, it holds for all $t\in [0,T]$ that
					\begin{align}\label{LocalizationTripleJunctionTwoPhase}
					\supp\eta_p(\cdot,t)\cap\supp\eta_c(\cdot,t)
					\subset B_{r_{\mathcal{P}}}(\mathcal{T}_p(t))\cap(W_{i,j}(t)\cup W_{i}(t)\cup W_{j}(t)),
					\end{align}
					where $W_{i,j}$, $W_i$ and $W_j$ are as in Definition~\ref{def:locRadiusTripleJunction}.
\item[vi)] Let $c,c' \in\mathcal{C}$ be two distinct two-phase interfaces. Then we have
					 $\supp\eta_c\cap\supp\eta_{c'}\neq\emptyset$ if and only if both interfaces
					 have an endpoint at the same triple junction $\mathcal{T}_{p},\,p\in \mathcal{P}$.
					 In this case, it holds for all $t\in [0,T]$ that
					 \begin{align}\label{LocalizationTwoInterfaces}
					 \supp\eta_c(\cdot,t)\cap\supp\eta_{c'}(\cdot,t)\subset 
					 B_{r_{\mathcal{P}}}(\mathcal{T}_p(t))\cap W_{i}(t),
					 \end{align}
					 where we assume that $\mathcal{T}_{c}\subset {\bar{I}}_{i,j}$ and $\mathcal{T}_{c'}\subset {\bar{I}}_{k,i}$.
\end{itemize}					
\end{lemma}

\begin{proof}
An illustration of the constructed functions close to a triple junction can be found in Figure~\ref{fig:cutoff}.
For the definition of a partition of unity $(\eta_{\mathrm{bulk}},\eta_1,\ldots,\eta_N)$
with the required localization and coercivity properties we proceed in several steps.

\textit{Step 1: Definition of auxiliary cutoffs.} 
Let $\theta$ be a smooth and even cutoff function with $\theta(s)=1$ 
for $|s|\leq\frac{1}{2}$ and $\theta\equiv 0$ for $|s|\geq 1$. 
Let $\zeta\colon\Rd[]\to[0,\infty)$ be another smooth cutoff function defined by
\begin{align}
\label{QuadraticInterfaceCutOff}
\zeta(s) = (1-s^2)\theta(s^2),
\end{align}
see Figure~\ref{fig:zetas}. Let~$\delta \in (0,1]$ be a constant to be determined later
(independent of~$\bar r_{\mathrm{min}}$).
Based on the profile~$\zeta$, we then introduce for each topological feature~$\mathcal{T}_n$, $n\in\{1,\ldots,N\}$, 
a corresponding cutoff function~$\zeta_n$ as follows.
First, for a given triple junction~$p\in\mathcal{P}$ we define the
associated triple junction cutoff 
\begin{align}
\label{DefinitionAuxiliaryCutoffHalfSpace}
\zeta_p(x,t) := \zeta\Big(\frac{\dist(x,\mathcal{T}_p(t))}{r_{\mathcal{P}}}\Big), \quad
(x,t)\in \Rd[2]{\times} [0,T].
\end{align}
Second, for a given connected component~$c\in\mathcal{C}$ of a
two-phase interface, say $\mathcal{T}_c\subset\bar I_{i,j}$ for some $i,j\in\{1,\ldots,P\}$
with $i\neq j$, we define the 
associated interface cutoff function
\begin{align}
\label{DefinitionAuxiliaryCutoffInterface}
\zeta_c(x,t) := 
\begin{cases}
\zeta\big(\frac{s_{i,j}(x,t)}{\delta \bar r_{\mathrm{min}}}\big),
& (x,t)\in \overline{\mathrm{im}(\Psi_{\mathcal{T}_c})}, \\
0 & \text{else},
\end{cases}
\end{align}
where~$s_{i,j}$ is the signed distance function 
defined in~\eqref{SignedDistanceTwoPhase} 
and~$\mathrm{im}(\Psi_{\mathcal{T}_c})$ is the image of the diffeomorphism~$\Psi_{\mathcal{T}_c}$,
i.e., the restriction to~$\mathcal{T}_c$
of the diffeomorphism~\eqref{DiffeoTubularNeighborhood}.

It follows directly from the definitions~\eqref{QuadraticInterfaceCutOff}--\eqref{DefinitionAuxiliaryCutoffInterface},
the regularity of the signed distance in form of~\eqref{eq:regSignedDistanceProjection}, 
\eqref{Bound2ndDerivativeSignedDistance} and~\eqref{boundTimeDerivatives}, 
as well as~\eqref{eq:boundsInitialValues} that
\begin{align}
\label{eq:suppTripleJunctionCutoff}
\supp \zeta_p(\cdot,t) &\subset B_{r_{\mathcal{P}}}(\mathcal{T}_p(t)),
&& t\in [0,T],
\\
\label{eq:suppInterfaceCutoff}
\supp \zeta_c(\cdot,t) &\subset \Psi_{\mathcal{T}_c}(\mathcal{T}_c(t){\times}\{t\}
					 {\times}[-\delta\bar r_{\mathrm{min}},\delta\bar r_{\mathrm{min}}]),
&& t\in [0,T],
\end{align}
and~$\zeta_p \in (C^0_tC^2_x \cap C^1_tC^0_x)(\Rd[2]{\times}[0,T]\setminus\mathcal{T}_p)$ as well as
$\zeta_c \in (C^0_tC^2_x \cap C^1_tC^0_x)(\overline{\mathrm{im}(\Psi_{\mathcal{T}_c})})$
with corresponding estimates (assuming~$\mathcal{T}_c\subset\bar I_{i,j}$)
\begin{align}
\label{eq:regEstimatesTripleJunctionCutoff}
|1{-}\zeta_p| &\leq C\big(\bar r_{\mathrm{min}}^{-2}\dist^2(\cdot,\mathcal{T}_p) \wedge 1\big)
&&\text{on } \Rd[2]{\times}[0,T]\setminus\mathcal{T}_p,
\\
\label{eq:regEstimatesTripleJunctionCutoff1}
|\nabla^k\zeta_p| &\leq C\bar r_{\mathrm{min}}^{-k}
\big(\bar r_{\mathrm{min}}^{-(2-k)}\dist^{2-k}(\cdot,\mathcal{T}_p) \wedge 1)
&&\text{on } \Rd[2]{\times}[0,T]\setminus\mathcal{T}_p,\,k\in\{1,2\},
\\
\label{eq:regEstimatesTripleJunctionCutoff2}
|\partial_t\zeta_p| &\leq C\bar r_{\mathrm{min}}^{-2}
\big(\bar r_{\mathrm{min}}^{-1}\dist(\cdot,\mathcal{T}_p) \wedge 1\big)
&&\text{on } \Rd[2]{\times}[0,T]\setminus\mathcal{T}_p,
\\
\label{eq:regEstimatesInterfaceCutoff}
|1{-}\zeta_c| &\leq C\big(\bar r_{\mathrm{min}}^{-2}\dist^2(\cdot,\bar I_{i,j}) \wedge 1 \big)
&&\text{on } \overline{\mathrm{im}(\Psi_{\mathcal{T}_c})},
\\
\label{eq:regEstimatesInterfaceCutoff1}
|\nabla^k\zeta_c| &\leq C\bar r_{\mathrm{min}}^{-k}
\big(\bar r_{\mathrm{min}}^{-(2-k)}\dist^{2-k}(\cdot,\bar I_{i,j}) \wedge 1 \big)
&&\text{on } \overline{\mathrm{im}(\Psi_{\mathcal{T}_c})},\,k\in\{1,2\},
\\
\label{eq:regEstimatesInterfaceCutoff2}
|\partial_t\zeta_c| &\leq C\bar r_{\mathrm{min}}^{-2}
\big(\bar r_{\mathrm{min}}^{-1}\dist(\cdot,\bar I_{i,j}) \wedge 1\big)
&&\text{on } \overline{\mathrm{im}(\Psi_{\mathcal{T}_c})}.
\end{align}

\begin{figure}
		\begin{tikzpicture}[scale=2]
			\draw[->] (-1.3,0)--(1.3,0) node [below] {$r$};
			\draw (0.5,-0.01)--(0.5,0.01) node [below] {$1/2$};
			\draw (1,-0.01)--(1,0.01) node [below] {$1$};
			\draw (-0.5,-0.01)--(-0.5,0.01) node [below] {$-1/2$};
			\draw (-1,-0.01)--(-1,0.01) node [below] {$-1$};
			\draw[->] (0,-0.1)--(0,1.3) node [left] {$\zeta(r)$};
			
			\draw [dashed] (0,1)--(-0.3,1) node [left] {$1$};
			
			\draw [thick, color=black, domain=-0.5:0.5, samples=200, smooth]
			plot (\x,{1-\x*\x});
			\draw [dashed, color=black, domain=-1:-0.5, samples=200, smooth]
			plot (\x,{1-\x*\x});
			\draw [dashed, color=black, domain=0.5:1, samples=200, smooth]
			plot (\x,{1-\x*\x});
			\draw[thick,smooth] (.5,.75) to[out=-45,in=180] (0.9,0);
			\draw[thick,smooth] (0.9,0)--(1.2,0);
			\draw[thick,smooth] (-.5,.75) to[out=-135,in=0] (-0.9,0);
			\draw[thick,smooth] (-0.9,0)--(-1.2,0);
		\end{tikzpicture}	
	\caption{The profile $\zeta$ used to construct 
	the cutoff functions for two-phase interfaces and triple junctions.}
	\label{fig:zetas}
\end{figure}
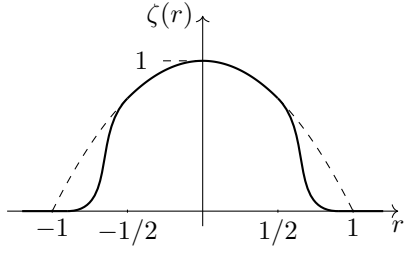

\textit{Step 2: Define $\eta_p$ for triple junctions $p\in\mathcal{P}$.} 
Let us assume that the phases $i,j,k\in\{1,\ldots,P\}$ are present
at the triple junction $\mathcal{T}_p$, and the corresponding interfaces are
denoted by $\mathcal{T}_{c_{i,j}}\subset \bar{I}_{i,j}$, $\mathcal{T}_{c_{j,k}}\subset  \bar{I}_{j,k}$
and $\mathcal{T}_{c_{k,i}}\subset \bar{I}_{k,i}$.

We want to define~$\eta_p$ such that~\eqref{LocalizationTripleJunction} holds true.
Recall from Definition~\ref{def:locRadiusTripleJunction} that $B_{r_{\mathcal{P}}}(\mathcal{T}_p)$
decomposes into six wedges. Three of them, namely the interface wedges $W_{i,j}$, $W_{j,k}$ resp.\ $W_{k,i}$,
contain the interfaces $\mathcal{T}_{c_{i,j}}$, $\mathcal{T}_{c_{j,k}}$ resp.\ $\mathcal{T}_{c_{k,i}}$.
The other three are interpolation wedges denoted by $W_{i}$, $W_{j}$ resp.\ $W_{k}$.

We now have everything in place to move on with the definition of $\eta_p$.
We note that $B_{r_{\mathcal{P}}}(\mathcal{T}_p(t))\cap W_{i,j}(t) \subset 
\mathrm{im}(\Psi_{\mathcal{T}_{c_{i,j}}})$ for all~$t\in [0,T]$ 
due to~\eqref{eq:inlcusionInterfaceWedge} and~\eqref{def:minLocScaleTripleJunction}. 
Therefore, we can begin by setting
\begin{align}\label{DefinitionEtaTripleJunctionWedgeInterface}
\eta_p(x,t) := \zeta_p(x,t)\zeta_{c_{i,j}}(x,t),\quad
t\in [0,T],\,x\in B_{r_{\mathcal{P}}}(\mathcal{T}_p(t))\cap W_{i,j}(t) ,
\end{align}
and analogously on the other interface wedges~$W_{j,k}$ and~$W_{k,i}$. To define~$\eta_p$ on
the interpolation wedges, we use the interpolation parameter built in Lemma~\ref{lemma:interpolation_functions}. 
To clarify the direction of interpolation, i.e., on which boundary of the interpolation wedge the corresponding
interpolation function is equal to one or zero, we make use of the following
notational convention. For the interpolation wedge~$W_i$, say, we denote 
by~$\lambda_{i}^{j,k}$ the interpolation function as built in Lemma~\ref{lemma:interpolation_functions}
and which interpolates from~$j$ to~$k$ in the sense that it is equal 
to one on $(\partial W_{i,j}\cap\partial W_{i})\setminus\mathcal{T}_{p}$ and 
which vanishes on $(\partial W_{k,i}\cap\partial W_{i})\setminus\mathcal{T}_{p}$. We also define
$\lambda_{i}^{k,j}:= 1-\lambda_{i}^{j,k}$ which interpolates on $W_{i}$ in the opposite direction from~$k$ to~$j$.
Analogously, one introduces the interpolation functions on the other interpolation wedges.
We may then define
\begin{equation}
\label{DefinitionEtaTripleJunctionInterpolationWedge}
\begin{aligned}
\eta_p(x,t) &:= \lambda_{i}^{j,k}(x,t)\zeta_p(x,t)\zeta_{c_{i,j}}(x,t)
+(1{-}\lambda_{i}^{j,k})(x,t)\zeta_p(x,t)\zeta_{c_{k,i}}(x,t),
\\&~~~~~
t\in [0,T],\,x\in B_{r_{\mathcal{P}}}(\mathcal{T}_p(t))\cap W_{i}(t),
\end{aligned}
\end{equation}
due to $B_{r_{\mathcal{P}}}(\mathcal{T}_p(t))\cap W_{i}(t) \subset \mathrm{im}(\Psi_{\mathcal{T}_{c_{i,j}}}) 
\cap \mathrm{im}(\Psi_{\mathcal{T}_{c_{k,i}}})$ for all~$t\in [0,T]$, which follows from~\eqref{eq:inclusionInterpolWedge}
and~\eqref{def:minLocScaleTripleJunction}.
We can analogously define~$\eta_p$ on the other two interpolation wedges~$W_{j}$ and~$W_{k}$.
Finally, we define
\begin{align}
\label{DefinitionEtaTripleJunctionAway}
\eta_p(x,t) &:= 0, \quad t\in [0,T],\, x\notin B_{r_{\mathcal{P}}}(\mathcal{T}_p(t)).
\end{align}
We refer to Figure~\ref{fig:cutoff} for an illustration of the construction. 

The localization property~\eqref{LocalizationTripleJunction} is
immediate from the definitions~\eqref{DefinitionEtaTripleJunctionWedgeInterface}--\eqref{DefinitionEtaTripleJunctionAway}
and the property~\eqref{eq:suppTripleJunctionCutoff}, whereas~\eqref{LocalizationTwoTripleJunctions} 
follows from the definition~\eqref{def:minLocScaleTripleJunction}
of the localization scale~$r_{\mathcal{P}}$. Moreover, as a consequence of the
estimates~\eqref{boundslambda1}--\eqref{boundslambda2} for the interpolation parameter, 
the estimates~\eqref{eq:regEstimatesTripleJunctionCutoff}--\eqref{eq:regEstimatesInterfaceCutoff2}
for the auxiliary cutoffs,
the definitions~\eqref{DefinitionEtaTripleJunctionWedgeInterface}--\eqref{DefinitionEtaTripleJunctionAway}
and the trivial estimate~$\dist(\cdot,\bar I_{i,j})\vee \dist(\cdot,\bar I_{j,k}) \vee
\dist(\cdot,\bar I_{k,i}) \leq \dist(\cdot,\mathcal{T}_p)$ throughout~$B_{r_{\mathcal{P}}}(\mathcal{T}_p(t))$
for all~$t\in [0,T]$ (assuming that the phases~$i,j,k\in\{1,\ldots,P\}$ are present at~$\mathcal{T}_p$)
we obtain
\begin{align}
\label{eq:regEstimateTripleJunctionCutoff}
|1{-}\eta_p| &\leq C\big(\bar r_{\mathrm{min}}^{-2}\dist^2(\cdot,\mathcal{T}_p) \wedge 1\big)
&& \text{on } \Rd[2]{\times}[0,T]\setminus\mathcal{T}_p,
\\
\label{eq:regEstimateTripleJunctionCutoff1}
|\nabla^k\eta_p| &\leq C\bar r_{\mathrm{min}}^{-k}
\big(\bar r_{\mathrm{min}}^{-(2-k)}\dist^{2-k}(\cdot,\mathcal{T}_p) \wedge 1\big)
&& \text{on } \Rd[2]{\times}[0,T]\setminus\mathcal{T}_p,\,k\in\{1,2\},
\\
\label{eq:regEstimateTripleJunctionCutoff2}
|\partial_t\eta_p| &\leq C\bar r_{\mathrm{min}}^{-2}
\big(\bar r_{\mathrm{min}}^{-1}\dist(\cdot,\mathcal{T}_p) \wedge 1\big)
&& \text{on } \Rd[2]{\times}[0,T]\setminus\mathcal{T}_p.
\end{align}
These estimates of course imply the asserted bound~\eqref{eq:estimatesDerivEta}
for~$n=p\in\mathcal{P}$. Note also that the error 
estimates~\eqref{GlobalEquationsBoundMinorityPhases}--\eqref{GlobalEquationsBoundTimeDerivMinorityPhases}
are trivially fulfilled by definition~\eqref{def:minLocScaleTripleJunction}
of the localization scale~$r_{\mathcal{P}}$, the property~\eqref{LocalizationTripleJunction}
and the estimate~\eqref{eq:estimatesDerivEta}.

\textit{Step 3: Define $\eta_c$ for $c\in\mathcal{C}$.}
Let $i,j \in \{1,\ldots,P\}$ with $i \neq j$ be such that $\mathcal{T}_c\subset {\bar{I}}_{i,j}$.
If the interface $\mathcal{T}_c$ has no endpoint at a triple junction, i.e., it is a closed loop, we simply set
\begin{align}\label{DefinitionEtaTwoPhaseNoTripleJunction}
\eta_c(x,t) :=\begin{cases}
	 \zeta_{c}(x,t) & \text{ if }
(x,t)\in  \mathrm{im}(\Psi_{\mathcal{T}_c}),\\
 0, &  \text{ else,}
	\end{cases}
\end{align}
where the cutoff~$\zeta_c$ was already defined in~\eqref{DefinitionAuxiliaryCutoffInterface}.

Otherwise, the interface ends in two different triple junctions corresponding to $p, p'\in \mathcal{P}$ with $p\neq p'$.
We will only describe the construction close to $\mathcal{T}_p$, 
as by~\eqref{def:minLocScaleTripleJunction} the triple junctions are separated
on scale~$r_{\mathcal{P}}$ and can thus also be treated separately.
Away from the triple junctions $\mathcal{T}_p$ and $\mathcal{T}_{p'}$, we still define
\begin{align}\label{DefinitionEtaTwoPhaseAwayTripleJunction}
\eta_c(x,t) := \begin{cases}
	\zeta_{c}(x,t) & (x,t) \in
	\mathrm{im}(\Psi_{\mathcal{T}_c})	\setminus \bigcup_{t\in [0,T]}
	\big(B_{r_{\mathcal{P}}}(\mathcal{T}_p(t)) \cup B_{r_{\mathcal{P}}}(\mathcal{T}_{p'}(t))\big) {\times}\{t\}\\
	0 & \text{in } \big(\Rd[2]{\times}[0,T]\setminus \mathrm{im}(\Psi_{\mathcal{T}_c})\big)\setminus\bigcup_{t\in [0,T]} 
	\big(B_{r_{\mathcal{P}}}(\mathcal{T}_p(t)) \cup B_{r_{\mathcal{P}}}(\mathcal{T}_{p'}(t))\big) {\times}\{t\}.
\end{cases}
\end{align}
Near the triple junction, i.e., on $B_{r_{\mathcal{P}}}(\mathcal{T}_p(t))$
for all~$t\in [0,T]$, we aim to modify the definition 
such that~$\eta_c$ is supported within the set $W_{i}\cup W_{j}\cup W_{i,j}$.
To this end, we define
\begin{align}\label{DefinitionEtaTwoPhaseWedgeInterface}
\eta_c(x,t) := \big(1{-}\zeta_p(x,t)\big)\zeta_{c}(x,t),\quad
t\in [0,T],\,x\in B_{r_{\mathcal{P}}}(\mathcal{T}_p(t))\cap W_{i,j}(t),
\end{align}
which is indeed possible in analogy to \eqref{DefinitionEtaTripleJunctionWedgeInterface}, and
where the auxiliary cutoff~$\zeta_p$ was introduced in \eqref{DefinitionAuxiliaryCutoffHalfSpace}.
On the interpolation wedges~$W_i$ resp.\ $W_j$, we again make use of the 
arguments enabling~\eqref{DefinitionEtaTripleJunctionInterpolationWedge} and set
\begin{equation}\label{DefinitionEtaTwoPhaseInterpolationWedge}
\begin{aligned}
\eta_c(x,t) &:= \lambda_{i}^{j,k}(x,t)\big(1{-}\zeta_p(x,t)\big)\zeta_c (x,t),\quad
t\in [0,T],\,x\in B_{r_{\mathcal{P}}}(\mathcal{T}_p(t))\cap W_{i}(t), \\
\eta_c(x,t) &:= \lambda_{j}^{i,k}(x,t)\big(1{-}\zeta_p(x,t)\big)\zeta_c(x,t),\quad
t\in [0,T],\,x\in B_{r_{\mathcal{P}}}(\mathcal{T}_p(t))\cap W_{j}(t), \\
\eta_c(x,t) &:= 0, \qquad
t\in [0,T],\, x\in B_{r_{\mathcal{P}}}(\mathcal{T}_p(t)) \setminus
\big(W_{i,j}(t) \cup W_{i}(t) \cup W_{j}(t)\big),
\end{aligned}
\end{equation}
where $k \in \{1,\ldots,P\}$ corresponds to the third phase present at~$p$.
We refer again to Figure~\ref{fig:cutoff} for an illustration of the construction.

In terms of the required qualitative regularity for~$\eta_c$, the only obstruction
might be the compatibility of~\eqref{DefinitionEtaTwoPhaseAwayTripleJunction}
with~\eqref{DefinitionEtaTwoPhaseInterpolationWedge}.
This is precisely the point where we rely on a suitable choice of the scale~$\delta\in (0,1]$.
As we have seen in the proof of Lemma~\ref{lem:existenceLocRadius}, 
the curve trapping condition of~\eqref{eq:inlcusionInterfaceWedge}
in fact holds on scale~$r_{\mathcal{P}}$ for a wedge strictly contained in the interface wedge~$W_{i,j}$
(e.g., a wedge obtained by angle bisection).
Hence, due to the ball condition of Definition~\ref{DefinitionStrongSolutionTwoPhase},
this improved curve trapping condition, and the
definition~\eqref{def:minLocScaleTripleJunction} of the localization scale~$r_{\mathcal{P}}$ 
we may choose the constant~$\delta\in (0,1]$ small enough, depending only on the surface tensions
associated with~$\bar\Omega$, such that
\begin{align*}
&\overline{\Psi_{\mathcal{T}_c}(\mathcal{T}_c(t){\times}\{t\}
					 {\times}[-\delta r_{\mathcal{P}},\delta r_{\mathcal{P}}])}
\cap \partial B_{r_{\mathcal{P}}}(\mathcal{T}_p(t))
\subset\subset W_{i,j}(t)
\end{align*}
uniformly over all~$t\in [0,T]$. This choice in turn 
ensures continuity of~$\eta_c$, and then based on the 
definitions~\eqref{DefinitionEtaTwoPhaseNoTripleJunction}--\eqref{DefinitionEtaTwoPhaseInterpolationWedge} 
that $\eta_c\in (C^0_tC^2_x\cap C^1_tC^0_x)(\Rd[2]{\times} [0,T]
\setminus \mathcal{T}_{\mathcal{P}})$ since all the constituents of~$\eta_c$
enjoy this regularity (cf.\ \textit{Step~1} for the auxiliary cutoffs
and Lemma~\ref{lemma:interpolation_functions} for the interpolation parameter, respectively).

Next, we may infer the localization property~\eqref{LocalizationTwoPhase} 
from the definitions~\eqref{DefinitionEtaTwoPhaseNoTripleJunction}--\eqref{DefinitionEtaTwoPhaseInterpolationWedge}
and the property~\eqref{eq:suppInterfaceCutoff}.
Moreover, based on the choice~\eqref{def:minLocScale} of the localization scale~$\bar r_{\mathrm{min}}$,
the localization property~\eqref{LocalizationTripleJunction} and the 
definitions~\eqref{DefinitionEtaTwoPhaseWedgeInterface}--\eqref{DefinitionEtaTwoPhaseInterpolationWedge},
one may deduce~\eqref{LocalizationTripleJunctionTwoPhase} and~\eqref{LocalizationTwoInterfaces}.

We move on with the proof of the estimates~\eqref{eq:estimatesDerivEta}
and~\eqref{GlobalEquationsBoundMinorityPhases}--\eqref{GlobalEquationsBoundTimeDerivMinorityPhases}
in terms of~$n=c\in\mathcal{C}$. First, a straightforward application of the 
definitions~\eqref{DefinitionEtaTwoPhaseNoTripleJunction}--\eqref{DefinitionEtaTwoPhaseInterpolationWedge},
the estimates~\eqref{boundslambda1}--\eqref{boundslambda2} for the interpolation parameter, 
and the estimates~\eqref{eq:regEstimatesTripleJunctionCutoff}--\eqref{eq:regEstimatesInterfaceCutoff2}
for the auxiliary cutoffs implies~\eqref{eq:estimatesDerivEta}. Consider then~$c\in\mathcal{C}$
and distinct~$i,j\in\{1,\ldots,P\}$ such that~$\mathcal{T}_c\not\subset \bar I_{i,j}$, i.e.,
either phase~$i$ or phase~$j$ is absent at~$\mathcal{T}_c$. Without loss of generality,
we may assume that there exists~$c'\in\mathcal{C}\setminus\{c\}$ and~$p\in\mathcal{P}$
such that~$\mathcal{T}_{c'}\subset\bar I_{i,j}$, $c\sim p$ and~$c'\sim p$;
and in this regime, it even suffices to restrict to the domain~$B_{r_{\mathcal{P}}}(\mathcal{T}_p(t))$
for all~$t\in [0,T]$. Otherwise, the error
estimates~\eqref{GlobalEquationsBoundMinorityPhases}--\eqref{GlobalEquationsBoundTimeDerivMinorityPhases}
are trivially fulfilled because of~\eqref{LocalizationTwoPhase}, the estimate~\eqref{eq:estimatesDerivEta} and
definition~\eqref{def:minLocScale} of the localization scale~$\bar r_{\mathrm{min}}$.

To prove the error estimates in the remaining regime, we now fully exploit the
fact that a factor of~$1-\zeta_p$ always appears in the definitions~\eqref{DefinitionEtaTwoPhaseWedgeInterface}
and~\eqref{DefinitionEtaTwoPhaseInterpolationWedge}. In particular, by means of
the estimates~\eqref{boundslambda1}--\eqref{boundslambda2} for the interpolation parameter, 
the estimates~\eqref{eq:regEstimatesTripleJunctionCutoff}--\eqref{eq:regEstimatesInterfaceCutoff2}
for the auxiliary cutoffs, and the trivial estimate~$\dist(\cdot,\mathcal{T}_c)\leq \dist(\cdot,\mathcal{T}_p)$
throughout~$B_{r_{\mathcal{P}}}(\mathcal{T}_p(t))$ for all~$t\in [0,T]$, it follows
\begin{align}
\eta_c &\leq C\big(\bar r_{\mathrm{min}}^{-2}\dist^2(\cdot,\mathcal{T}_p) \wedge 1\big)
&& \text{in } B_{r_{\mathcal{P}}}(\mathcal{T}_p(t)),\,t\in [0,T],
\\
|\nabla\eta_c| &\leq C\bar r_{\mathrm{min}}^{-1}
\big(\bar r_{\mathrm{min}}^{-1}\dist(\cdot,\mathcal{T}_p) \wedge 1\big)
&& \text{in } B_{r_{\mathcal{P}}}(\mathcal{T}_p(t)),\,t\in [0,T],
\\
|\partial_t\eta_c| &\leq C\bar r_{\mathrm{min}}^{-2}
\big(\bar r_{\mathrm{min}}^{-1}\dist(\cdot,\mathcal{T}_p) \wedge 1\big)
&& \text{in } B_{r_{\mathcal{P}}}(\mathcal{T}_p(t)),\,t\in [0,T].
\end{align}
These estimates upgrade 
to~\eqref{GlobalEquationsBoundMinorityPhases}--\eqref{GlobalEquationsBoundTimeDerivMinorityPhases}
thanks to the bounds~\eqref{eq:compDistances2} and~\eqref{eq:compDistances1}.

\textit{Step 4: Partition of unity.}
Next, we validate the partition of unity property  
for the family of localization functions $(\eta_1,\ldots,\eta_N)$. 
First of all, it is clear from our definitions 
\eqref{DefinitionEtaTripleJunctionWedgeInterface}--\eqref{DefinitionEtaTwoPhaseInterpolationWedge} 
that~$\eta_n\in [0,1]$ for each topological feature~$n\in\{1,\ldots,N\}$.
Together with the already established localization properties 
\eqref{LocalizationTwoPhase}--\eqref{LocalizationTwoInterfaces} and the definitions
\eqref{DefinitionEtaTripleJunctionWedgeInterface}--\eqref{DefinitionEtaTwoPhaseInterpolationWedge}, 
it also follows that $\sum_{n=1}^N\eta_n\leq 1$ on $\Rd[2]\times [0,T]$
as well as $\sum_{n=1}^N\eta_n\equiv 1$ on the evolving network of interfaces~$\mathcal{I}=\bigcup_{i\neq j}\bar I_{i,j}$.
Hence, we may define the bulk term $\eta_{\mathrm{bulk}}:=1-\sum_{n=1}^N\eta_n\in [0,1]$ and obtain that the
extended family $(\eta_{\mathrm{bulk}},\eta_1,\ldots,\eta_N)$ is indeed a partition of unity on $\Rd[2]\times [0,T]$.

\textit{Step 5: Estimates for the bulk cutoff.}
By the localization properties~\eqref{LocalizationTwoPhase}--\eqref{LocalizationTwoInterfaces}
as well as the choices~\eqref{def:minLocScaleTripleJunction} and~\eqref{def:minLocScale} of the 
localization scales~$r_{\mathcal{P}}$ and~$\bar r_{\mathrm{min}}$,
it suffices to prove~\eqref{eq:upperBoundBulkCutOff}--\eqref{eq:boundDerivBulkCutoff}
in~$\bigcup_{c\in\mathcal{C}}\mathrm{im}_{\bar r_{\mathrm{min}}}(\Psi_{\mathcal{T}_c})
\setminus \bigcup_{p\in\mathcal{P}} \bigcup_{t\in [0,T]}B_{r_{\mathcal{P}}}(\mathcal{T}_p(t)){\times}\{t\}$
and in $\bigcup_{p\in\mathcal{P}}\bigcup_{t\in [0,T]}B_{r_{\mathcal{P}}}(\mathcal{T}_p(t)){\times}\{t\}$,
respectively. We in fact may argue separately for each~$c\in\mathcal{C}$
and each~$p\in\mathcal{P}$. Moreover, for all~$c\in\mathcal{C}$ and all distinct~$i,j\in\{1,\ldots,P\}$
such that~$\mathcal{T}_c\subset\bar I_{i,j}$ it holds
\begin{align}
\label{eq:compNetworkDistanceAwayTripleJunctions}
\dist(\cdot,\bar I_{i,j}) = \dist(\cdot,\mathcal{I})
\quad\text{in } \mathrm{im}_{\bar r_{\mathrm{min}}}(\Psi_{\mathcal{T}_c})
\setminus \bigcup_{p\in\mathcal{P}} \bigcup_{t\in [0,T]}B_{r_{\mathcal{P}}}(\mathcal{T}_p(t)){\times}\{t\},
\end{align}
and similarly for all~$p\in\mathcal{P}$ with present phases~$i,j,k\in\{1,\ldots,P\}$, it holds
\begin{align}
\label{eq:compNetworkDistanceTripleJunctions}
\dist(\cdot,\bar I_{i,j}) \wedge \dist(\cdot,\bar I_{j,k})
\wedge \dist(\cdot,\bar I_{k,i}) = \dist(\cdot,\mathcal{I})
\,\text{ in } \bigcup_{t\in [0,T]}B_{r_{\mathcal{P}}}(\mathcal{T}_p(t)){\times}\{t\}.
\end{align}

First, let~$c\in\mathcal{C}$. Due to the localization
properties~\eqref{LocalizationTwoPhase}--\eqref{LocalizationTwoInterfaces},
the choices~\eqref{def:minLocScaleTripleJunction} and~\eqref{def:minLocScale} of the 
localization scales~$r_{\mathcal{P}}$ and~$\bar r_{\mathrm{min}}$,
as well as the definitions~\eqref{DefinitionEtaTwoPhaseNoTripleJunction}
and~\eqref{DefinitionEtaTwoPhaseAwayTripleJunction} it holds
\begin{align}
\label{eq:bulkCutoffAwayTripleJunctions}
\eta_{\mathrm{bulk}}= 1 {-} \eta_c 
= 1 {-} \zeta_c
\quad\text{in } \mathrm{im}_{\bar r_{\mathrm{min}}}(\Psi_{\mathcal{T}_c})
\setminus \bigcup_{p\in\mathcal{P}} \bigcup_{t\in [0,T]}B_{r_{\mathcal{P}}}(\mathcal{T}_p(t)){\times}\{t\}.
\end{align}
The upper bounds~\eqref{eq:upperBoundBulkCutOff}--\eqref{eq:boundTimeDerivBulkCutoff}
are therefore an immediate consequence 
of the bounds~\eqref{eq:regEstimatesInterfaceCutoff}--\eqref{eq:regEstimatesInterfaceCutoff2}, respectively, together 
with~\eqref{eq:compNetworkDistanceAwayTripleJunctions} and~\eqref{eq:compNetworkDistanceTripleJunctions}.
The coercivity estimate~\eqref{eq:coercivityBulkCutOff} in turn follows
from the choice~\eqref{QuadraticInterfaceCutOff} of the quadratic cutoff profile.

Second, consider~$p\in\mathcal{P}$ and assume that the pairwise distinct phases~$i,j,k\in\{1,\ldots,P\}$
are present at~$\mathcal{T}_p$. Modulo a permutation of the indices, it suffices to
consider the two unique two-phase interfaces~$\mathcal{T}_{c_{i,j}}\subset\bar I_{i,j}$ 
and~$\mathcal{T}_{c_{k,i}}\subset\bar I_{k,i}$
so that~$c_{i,j}\sim p$ and~$c_{k,i}\sim p$, and then to prove the desired estimates on the interface wedge~$W_{i,j}$
and the interpolation wedge~$W_{i}$. In this regime, due to the localization
properties~\eqref{LocalizationTwoPhase}--\eqref{LocalizationTwoInterfaces},
the choices~\eqref{def:minLocScaleTripleJunction} and~\eqref{def:minLocScale} of the 
localization scales~$r_{\mathcal{P}}$ and~$\bar r_{\mathrm{min}}$,
as well as the 
definitions~\eqref{DefinitionEtaTripleJunctionWedgeInterface}--\eqref{DefinitionEtaTripleJunctionInterpolationWedge}
resp.\ \eqref{DefinitionEtaTwoPhaseWedgeInterface}--\eqref{DefinitionEtaTwoPhaseInterpolationWedge},
it holds
\begin{align}
\label{eq:bulkCutoffTripleJunctionsInterfaceWedge}
\eta_{\mathrm{bulk}} &= 1 {-} \eta_{c_{i,j}} {-} \eta_{p}
= 1 {-} \zeta_{c_{i,j}}
&&\text{in } B_{r_{\mathcal{P}}}(\mathcal{T}_p(t)) \cap W_{i,j}(t),
\\
\label{eq:bulkCutoffTripleJunctionsInterpolationWedge}
\eta_{\mathrm{bulk}} &= 1 {-} \eta_{c_{i,j}} {-} \eta_{c_{k,i}} {-} \eta_{p}
\\ \nonumber
&= \lambda^{j,k}_{i}(1 {-} \zeta_{c_{i,j}}) + (1{-}\lambda^{j,k}_{i})(1 {-} \zeta_{c_{k,i}})
&&\text{in } B_{r_{\mathcal{P}}}(\mathcal{T}_p(t)) \cap W_{i}(t)
\end{align}
for all~$t\in [0,T]$.
The upper bounds~\eqref{eq:upperBoundBulkCutOff}--\eqref{eq:boundTimeDerivBulkCutoff}
therefore follow from the estimates~\eqref{eq:regEstimatesInterfaceCutoff}--\eqref{eq:regEstimatesInterfaceCutoff2}, 
the bound~\eqref{boundslambda1}
for the interpolation parameter, the estimates~\eqref{eq:compDistances1} and~\eqref{eq:compDistances3},
as well as the estimates~\eqref{eq:compNetworkDistanceAwayTripleJunctions} 
and~\eqref{eq:compNetworkDistanceTripleJunctions}. The coercivity estimate~\eqref{eq:coercivityBulkCutOff} in turn
is again implied by~\eqref{QuadraticInterfaceCutOff}.
\end{proof}

\begin{figure}
\includegraphics[scale=0.12]{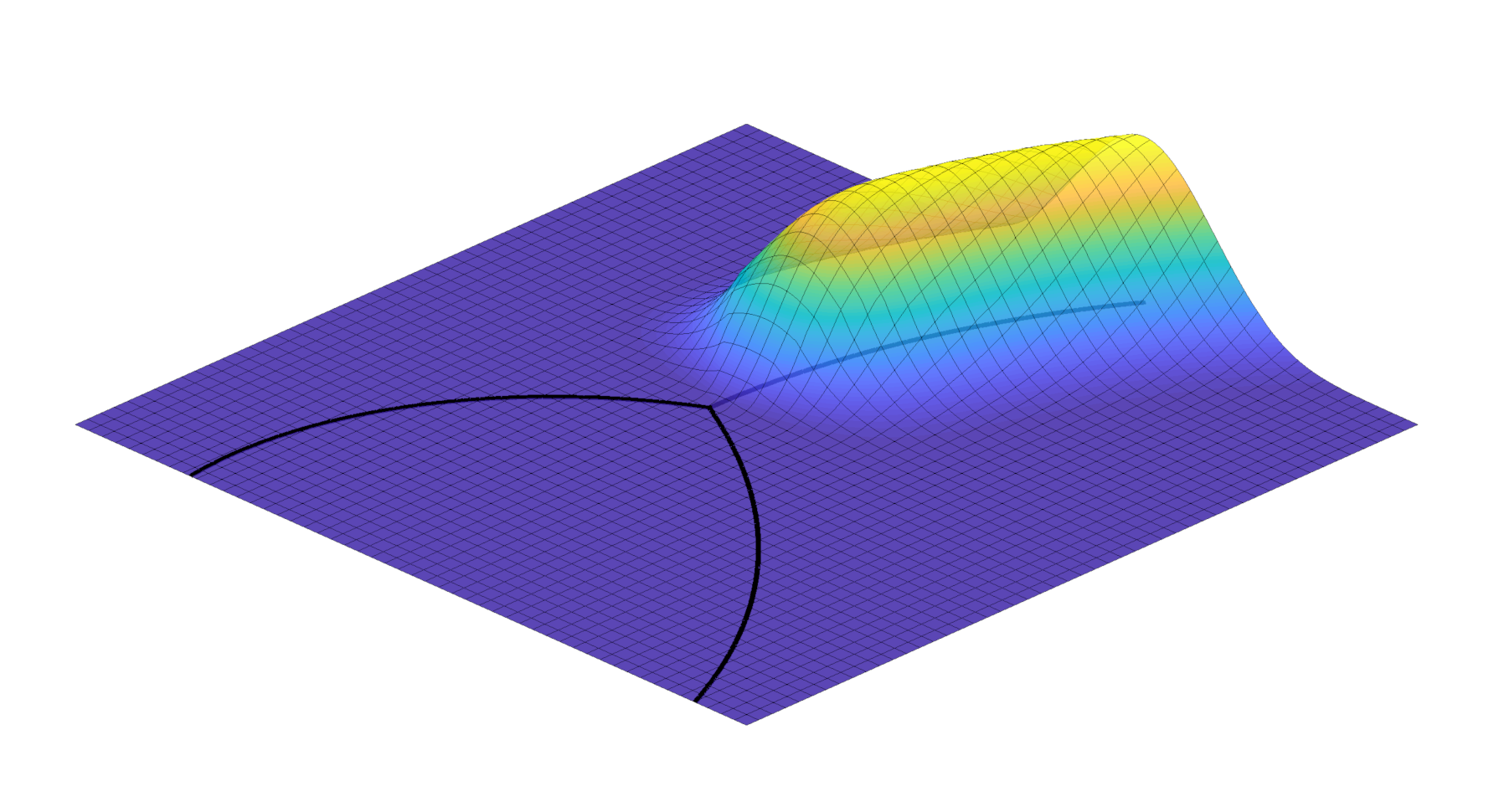}
\,
\includegraphics[scale=0.12]{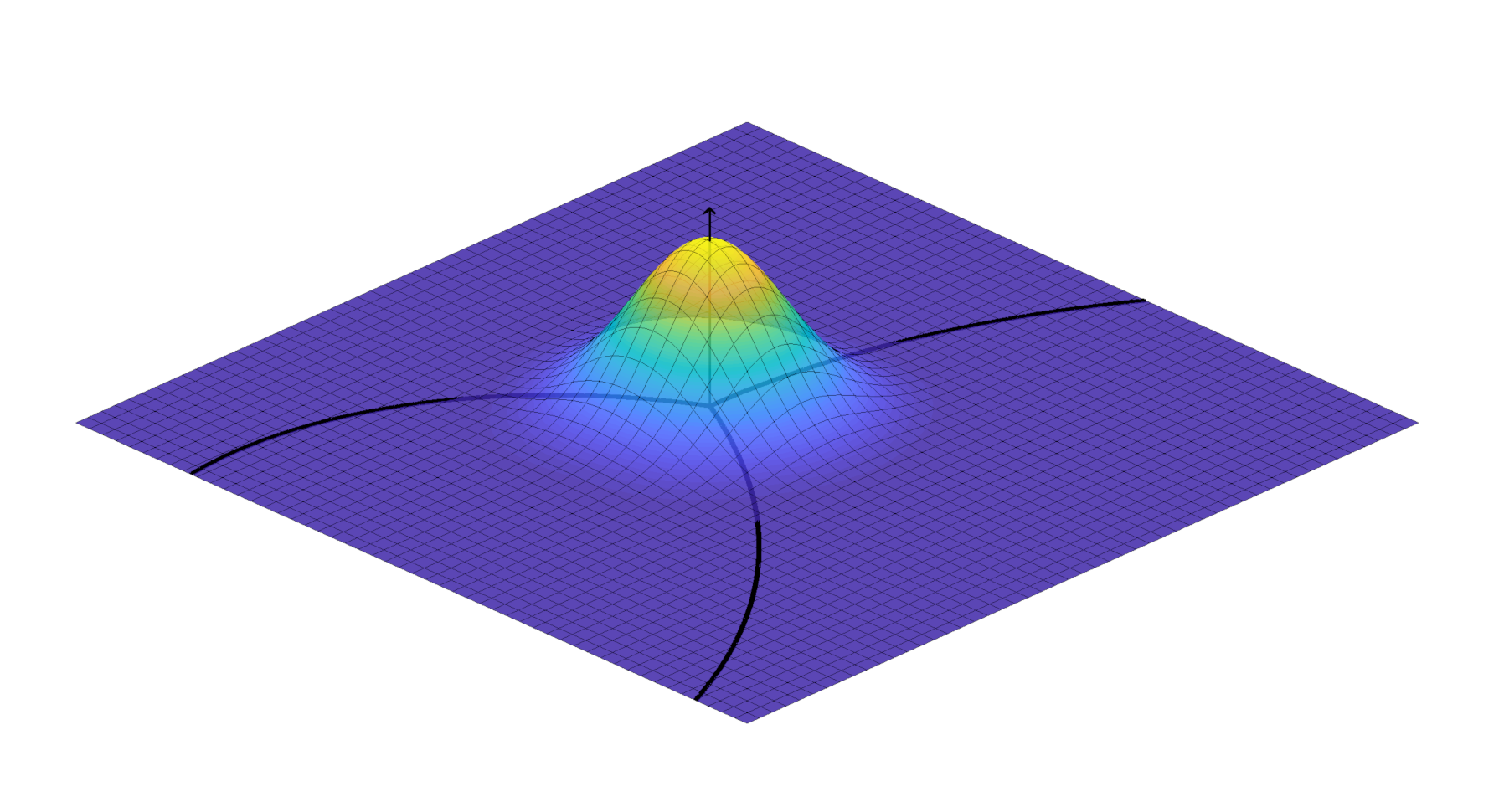}
\\
\includegraphics[scale=0.12]{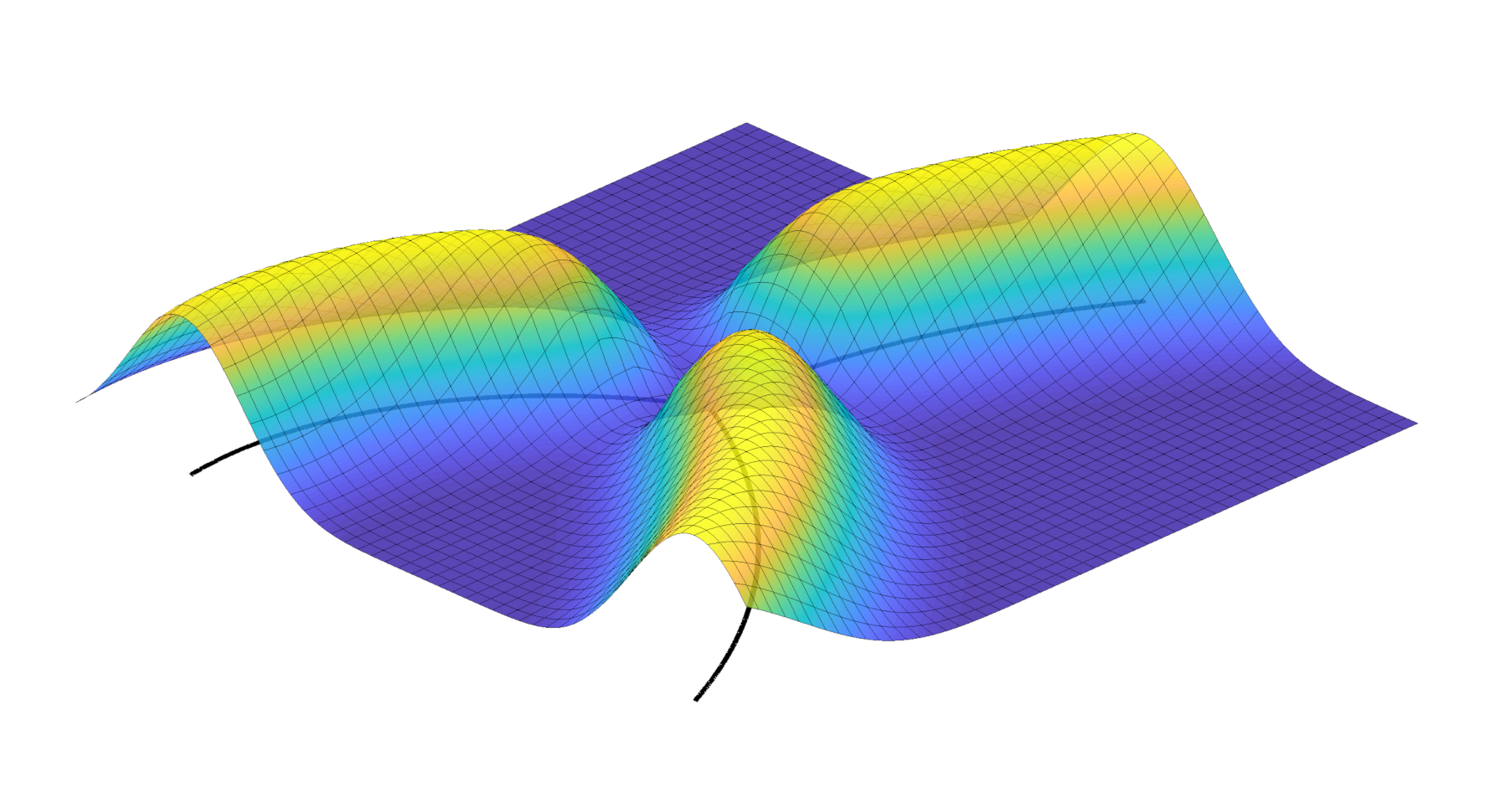}
\,
\includegraphics[scale=0.12]{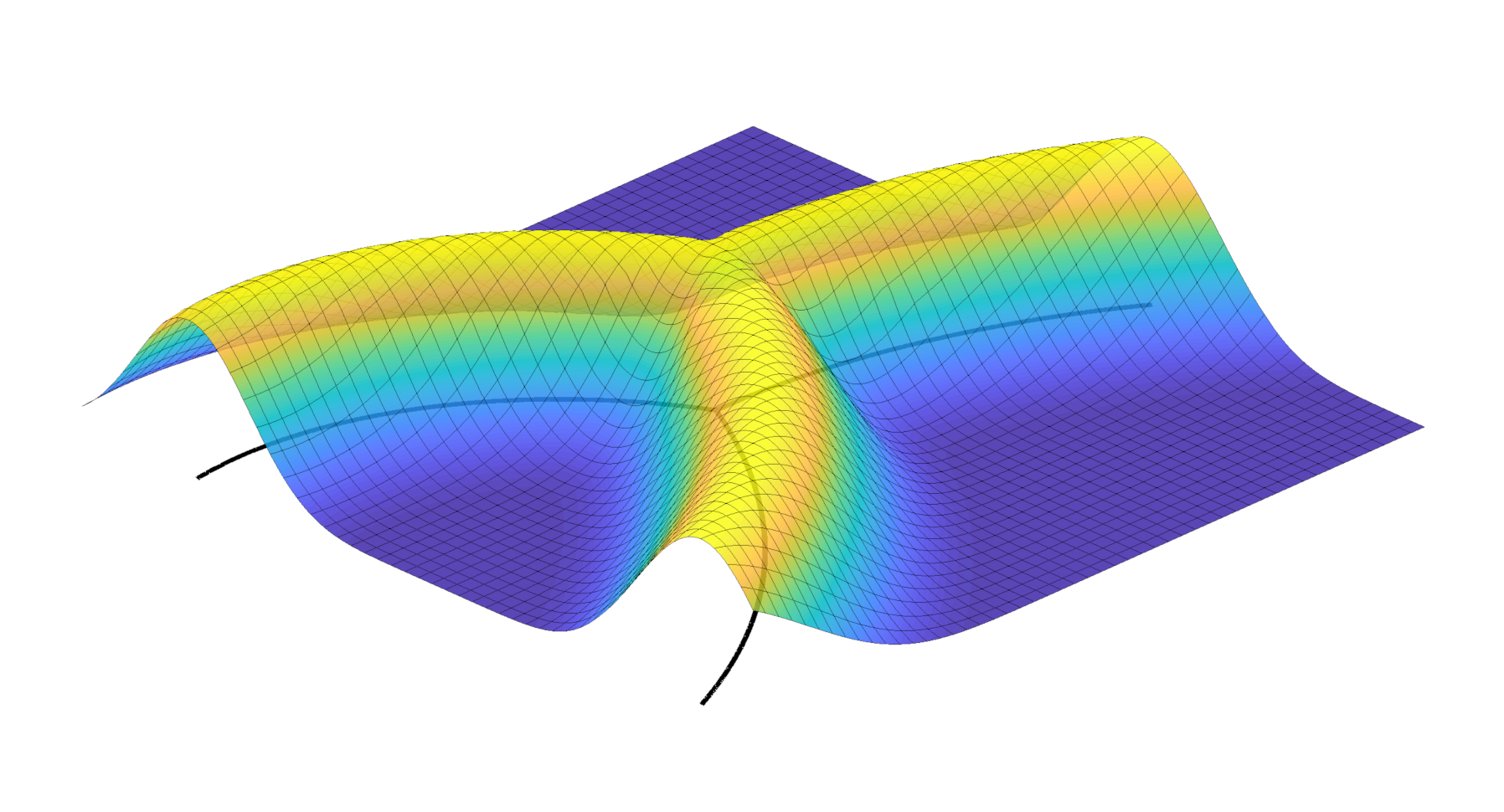}
\caption{The different functions $\eta_n$ for $n \in \mathcal{C}\cup \mathcal{P}$ in the partition of unity at a single triple junction $\mathcal{T}_p$ for $p \in \mathcal{P}$: The function $\eta_c$ for a single two-phase interface $c\in \mathcal{C}$ ending at the triple junction (top left), the function $\eta_p$ for the triple junction itself (top right), the sum of all two-phase localization functions at a triple junction (bottom left), and the sum of all localization functions $\sum_n \eta_n$ 
(bottom right). Observe that the sum of all localization functions equals $1$ on the interfaces in the strong solution, but decays quadratically away from them.}
\label{fig:cutoff}
\end{figure}

\subsection{Global construction of the calibration}\label{SectionGlobalDefinitions}
In this section, we glue together the local constructions to define the global 
extensions~$\xi_{i,j}$ and~$B$ of the normal vector fields and velocity field, respectively.

The idea for the construction of the vector fields~$\xi_{i,j}$ for $i,j \in \{1,\ldots, P\}$ with $i \neq j$ is as follows.
First, we provide the definition of local vector fields~$\xi^n_{i,j}$ for~$n \in \{1,\ldots,N\}$ 
in the support of the associated localization function~$\eta_n$ for each topological 
feature~$\mathcal{T}_n$. If both phases~$i$ and~$j$
are \emph{present} at~$\mathcal{T}_n$, 
we define~$\xi^n_{i,j}$ by means of the local constructions provided in Section~\ref{SectionLocalConstructionsTwoPhase} 
for the model problem of a smooth manifold and Section~\ref{SectionLocalConstructionsTriod} for the model problem 
of a triple junction. This, however, leaves open the question of the definition of the vector fields~$\xi_{i,j}^n$
for phases \emph{absent} at~$\mathcal{T}_n$. It turns out that this issue is related to 
the conditions of global stability between the phases.
In particular, we would like to ensure that at a given topological feature~$\mathcal{T}_n$, 
our relative entropy functional provides a length control for those interfaces which are not present at~$\mathcal{T}_n$. 
For this purpose, we rely on the stability condition for an admissible matrix of surface tensions in the sense 
of Definition~\ref{DefinitionAdmissibleSurfaceTensions} \emph{iii)}.

\begin{lemma}\label{LemmaLocalConstructionsNetwork}
Let $d=2$ and $P \in \mathbb{N}$, $P\geq 2$.
Let $\bar\Omega=(\bar{\Omega}_1,\ldots,\bar{\Omega}_P)$ 
be a strong solution to multiphase mean curvature flow in the sense of Definition~\ref{DefinitionStrongSolution}.
Let $(\eta_{\mathrm{bulk}},\eta_1,\ldots,\eta_N)$ be a partition of unity as constructed in 
Lemma~\ref{LemmaPartitionOfUnity}. In particular, let~$\bar r_{\mathrm{min}}\in (0,1]$ 
be the localization scale defined by~\eqref{def:minLocScale},
and~$\mathcal{T}_{\mathcal{P}}:=\bigcup_{p\in\mathcal{P}}\mathcal{T}_p$. 
Let $i,j\in\{1,\ldots,P\}$ be distinct phases 
and let $n\in\{1,\ldots,N\}$ correspond to a topological feature.
Given
\begin{align}\label{def_support_space_time}
	\mathcal{U}_n := \bigcup_{t\in [0,T]} \{x \in \Rd[2]: \eta_n(x,t) >0 \} \times \{t\}
\end{align}
there exist continuous vector fields
\begin{align*}
\xi^n_{i,j}\colon \mathcal{U}_n &\to\Rd[2], \\
\xi^n_{i}\colon \mathcal{U}_n&\to\Rd[2],
\end{align*}
satisfying the following properties:
\begin{itemize}[leftmargin=0.7cm]
\item[i)] It holds $\xi^n_{i,j},\xi^n_{i}\in \big(C^0_tC^2_x\cap C^1_tC^0_x\big)
				  (\overline{\mathcal{U}_n}\setminus\mathcal{T}_{\mathcal{P}})$,
					and there exists~$C>0$, which may depend on~$\bar\Omega$ but not 
					on~$\bar r_{\mathrm{min}}$, such that 
					throughout~$\mathcal{U}_n\setminus\mathcal{T}_{\mathcal{P}}$
					\begin{align}
					\label{eq:regularityEstimatesLocalXi}
					\max_{k=0,1,2} \bar r_{\mathrm{min}}^k|\nabla^k\xi^n_{i,j}|
					+ \bar r_{\mathrm{min}}^2 |\partial_t\xi^n_{i,j}| \leq C.
					\end{align}
\item[ii)] On~$\mathcal{U}_n$ we have $\xi^n_{i,j}=-\xi^n_{j,i},\,|\xi^n_{i,j}|\leq 1$ as well as 
					 \begin{align}\label{CalibrationLocalConstruction}
					 \sigma_{i,j}\xi^n_{i,j}=\xi_i^n-\xi_j^n.
					 \end{align}
\item[iii)] If the phases~$i$ and~$j$ are both present at the topological feature~$\mathcal{T}_n$, 
						then~$\xi^n_{i,j}$ coincides on~$\mathcal{U}_n$ with the explicit two-phase construction from 
						Lemma~\ref{LemmaBoundsLocalConstructionsTwoPhase} in case of~$n\in \mathcal{C}$, 
						respectively the triple junction construction from Proposition~\ref{prop:xi_triple_junction} 
						in case of~$n\in \mathcal{P}$. 
\item[iv)] There exists a constant~$b=b(\sigma)\in (0,1)$, depending only on the surface tension
					 matrix associated with the strong solution~$\bar\Omega$, with the property that
					 if either phase~$i$ or~$j$ is absent at the topological feature~$\mathcal{T}_n$, 
					 then throughout~$\mathcal{U}_n$ we have
					 \begin{align}\label{CoercivityMinorityPhase}
					 |\xi^n_{i,j}|\leq b<1.
					 \end{align}
\item[v)] In case of equal surface tensions $\sigma_{i,j}=\sigma_{j,k}=\sigma_{k,i}$, we have $\xi_k^n \cdot \xi_{i,j}^n=0$.
\end{itemize}
\end{lemma}

\begin{proof}
The proof consists of two parts distinguishing between the topological
features present in the network of interfaces of the strong solution.

\textit{Step 1: Consider the case $n=c \in \mathcal{C}$.}
We first assume that both phases~$i$ and~$j$ are present at the two-phase
interface $\mathcal{T}_{c}$, i.e., $\mathcal{T}_c\subset {\bar{I}}_{i,j}$. We then define the
vector field~$\smash{\xi_{i,j}^c}$ on~$\smash{\mathcal{U}_c}$ as in Lemma~\ref{LemmaBoundsLocalConstructionsTwoPhase}.
Note that by the localization property~\eqref{LocalizationTwoPhase} 
and the definition~\eqref{def:minLocScale},
we are indeed in the setting of Section~\ref{SectionLocalConstructionsTwoPhase}.
In particular, $\xi^c_{i,j}=-\xi^c_{j,i}$ and~$\xi^c_{i,j}$ coincides with~$\bar{\vec{n}}_{i,j}$ 
on $\supp\eta_c\cap {\bar{I}}_{i,j}$.
Furthermore, let us define the vector fields $\xi^c_{i}$ and $\xi^c_{j}$
as $\xi_i^c:=\smash{\frac{\sigma_{i,j}}{2} \xi_{i,j}^c}$ resp.\ as
$\xi_j^c:=\smash{\frac{\sigma_{i,j}}{2} \xi_{j,i}^c}$. This ensures that the 
desired formula~\eqref{CalibrationLocalConstruction} is indeed satisfied. 
Moreover, the regularity estimate~\eqref{eq:regularityEstimatesLocalXi}
follows from~\eqref{eq:estimatesTwoPhaseXi} and~\eqref{eq:estimatesTwoPhaseXiTimeDeriv}.

Now, let us assume that at least one of the phases $i$ or $j$ is absent at the
two-phase interface $\mathcal{T}_{c}$. To be specific, we fix $m,l\in\{1,\ldots,P\}$
with $m\neq l$ such that $\mathcal{T}_{c}\subset {\bar{I}}_{m,l}$. The idea now is to
first define vector fields $\xi^c_{i}$ and $\xi^c_{j}$ and then define $\xi^c_{i,j}$
by means of \eqref{CalibrationLocalConstruction} such that \eqref{CoercivityMinorityPhase}
holds true. To this end, we rely on the strict triangle inequality \eqref{TriangleInequalitySurfaceTensions}
for the given matrix of surface tensions, a direct 
consequence of our stability assumption Definition~\ref{DefinitionAdmissibleSurfaceTensions}~\emph{iii)}. 
Let us define
\begin{align*}
\xi^c_i &:= \frac{1}{2}(\sigma_{l,i}\xi^c_{m,l}+\sigma_{m,i}\xi^c_{l,m}),
\end{align*}
and analogously for $\xi^c_{j}$. Note that this is indeed well-defined since we have already provided a definition
of the vector fields $\xi^c_{m,l}=-\xi^c_{l,m}$ on the right-hand side as 
they are assumed to be associated to phases present at $\mathcal{T}_c$. 
This definition is also consistent with the previous one because of the convention $\sigma_{l,l}=\sigma_{m,m}=0$. 
We may then compute plugging in the definitions
\begin{align*}
\xi^c_{i,j}:=\frac{\xi^c_i-\xi^c_j}{\sigma_{i,j}} = 
\frac{1}{2}\Big(\frac{\sigma_{l,i}-\sigma_{l,j}}{\sigma_{i,j}}\xi^c_{m,l}
+\frac{\sigma_{m,i}-\sigma_{m,j}}{\sigma_{i,j}}\xi^c_{l,m}\Big).
\end{align*}
Hence, \eqref{CoercivityMinorityPhase} holds true because we have 
$|\frac{\sigma_{l,i}-\sigma_{l,j}}{\sigma_{i,j}}|<1$ and  
$|\frac{\sigma_{m,i}-\sigma_{m,j}}{\sigma_{i,j}}|<1$ due to 
the strict triangle inequality \eqref{TriangleInequalitySurfaceTensions},
whereas~\eqref{eq:regularityEstimatesLocalXi} follows
because~$\xi^c_{m,l}=-\xi^c_{l,m}$ is already subject to the same bound.

We note that in case of equal surface tensions $\sigma_{i,j}=\sigma_{j,k}=\sigma_{k,i}$, these definitions ensure that
\begin{align}
\label{OrthogonalityTwoPhase}
\xi_k^c \cdot \xi_{i,j}^c =0
\end{align}
holds for $\mathcal{T}_c\subset\bar I_{i,j}$.

\textit{Step 2: Consider the case $n=p \in \mathcal{P}$.}
Again, we first assume that both phases $i$ and $j$ are present
at the triple junction $\mathcal{T}_{p}$, i.e., a connected component of the 
interface ${\bar{I}}_{i,j}$ has an endpoint at $\mathcal{T}_{p}$. Note that by the 
localization property~\eqref{LocalizationTripleJunction} and the definition~\eqref{def:minLocScaleTripleJunction},
we may apply Proposition~\ref{prop:xi_triple_junction}.
Therein, we constructed a vector field in the support of $\eta_p$ we now call $\xi^p_{i,j}$.
In particular, $\xi^p_{i,j}=-\xi^p_{j,i}$ and $\xi^p_{i,j}$ coincides with $\bar{\vec{n}}_{i,j}$ 
on $\supp\eta_p\cap {\bar{I}}_{i,j}$.

Assume now that $k\in\{1,\ldots,P\}$ is the third phase being present 
at the triple junction $\mathcal{T}_{p}$. By construction, we have
$\sigma_{i,j}\xi^{p}_{i,j}+\sigma_{j,k}\xi^{p}_{j,k}+\sigma_{k,i}\xi^{p}_{k,i}=0$
on the support of $\eta_p$. Defining then the vector field $\xi^p_{i}$
as $\xi^p_{i}:=\frac{1}{3}(\sigma_{i,j}\xi^p_{i,j}+\sigma_{i,k}\xi^p_{i,k})$,
and analogously for~$\xi^p_{j}$ and~$\xi^p_{k}$, we indeed obtain \eqref{CalibrationLocalConstruction}.
The remaining claimed properties follow from Proposition~\ref{prop:xi_triple_junction}.

We note that for equal surface tensions $\sigma_{i,j}=\sigma_{j,k}=\sigma_{k,i}$, these definitions imply directly
\begin{align}
\label{OrthogonalityThreePhase}
\xi_k^p \cdot \xi_{i,j}^p =0.
\end{align}

In order to define~$\xi^p_{i,j}$ if at least one of the phases~$i$ or~$j$ 
is absent at the triple junction, we define the vector fields~$\xi_i^p$ and~$\xi_j^p$ 
as time-independent affine combinations of the previously defined vector fields 
using the stability condition Definition~\ref{DefinitionAdmissibleSurfaceTensions}~\emph{iii)}.

To be specific, we assume that the distinct phases $k,l,m\in\{1,\ldots,P\}$ are present at~$\mathcal{T}_p$.
We then employ the stability condition Definition~\ref{DefinitionAdmissibleSurfaceTensions}~\emph{iii)}, 
that is, there exists a non-degenerate $(P-1)$-simplex $(q_1,\ldots, q_P)$ in~$\Rd[P-1]$ 
such that $\sigma_{i',j'} = | q_{i'}-q_{j'}|$ for all $i',j' \in \{1,\ldots,P\}$. In particular, 
the triangle $(q_k,q_l,q_m)$ is non-degenerate and spans a plane~$E$ in~$\Rd[P-1]$, 
which we may isometrically identify with~$\Rd[2]$ via an affine map~$\phi\colon E \to \Rd[2]$. 
We furthermore denote the orthogonal projection onto~$E$ by~$\pi$.
See Figure~\ref{fig:embedding} for a sketch.

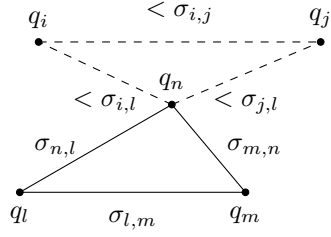
\begin{figure}
 	\begin{tikzpicture}[scale=5]
			\tikzmath{\l = .6;};
			\coordinate (A) at (0,0);
			\coordinate (B) at (30:\l);
			\draw [name path=A--B,color=white] (A) -- (B);
			\coordinate (C) at (0:\l);
			\coordinate (D) at ($(0:\l)+(130:\l)$);
			\draw [name path=C--D,color=white] (C) -- (D);
			\coordinate (E) at (intersection of A--B and C--D);
			
		  	\draw (0,0) -- node[label= below:{$\sigma_{l,m}$}] {} (0:\l);
		  	\draw (0:\l) --  node[label= right:{$\sigma_{m,n}$}] {}  (E);
		  	\draw (E) --node[label= left:{$\sigma_{n,l}$}] {}  (0,0);
		  	
		  	\filldraw (0,0) circle[radius=.25pt];
			\node[below, outer sep=2pt] (0,0) {$q_l$};
			\filldraw (\l,0) circle[radius=.25pt];
			\node[below, outer sep=2pt] at (\l,0) {$q_m$};
			\filldraw (E) circle[radius=.25pt];
			\node[above, outer sep=2pt] at (E) {$q_n$};
			
			\filldraw (.05,.4) circle[radius=.25pt];
			\node[above, outer sep=2pt] at (.05,.4) {$q_i$};
			\draw[dashed] (E) --node[label= below:{$<\sigma_{i,l}$}] {} (.05,.4);
			\filldraw (.8,.4) circle[radius=.25pt];
			\node[above, outer sep=2pt] at (.8,.4) {$q_j$};
			\draw[dashed] (E) --node[label= below:{$<\sigma_{j,l}$}] {} (.8,.4);
			\draw[dashed] (.05,.4) --node[label= above:{$<\sigma_{i,j}$}] {} (.8,.4);
		\end{tikzpicture}
		\caption{Sketch of the $l^2$-embedding of~$\sigma$ in the case that~$i$ and~$j$ 
		correspond to absent phases, projected into the plane $E$ containing $q_k$, $q_l$ and $q_m$. \label{fig:embedding}}
\end{figure}

In order to prepare the proof of the coercivity condition~\eqref{CoercivityMinorityPhase} we claim
\begin{align}\label{claim_chortness}
	\left|\pi q_{i} - \pi q_{j}\right| < b \sigma_{i,j}
\end{align}
for some $b \in (0,1)$, which we prove by considering two cases:

If exactly one of the two indices, say, $j$ corresponds to a phase being present at~$\mathcal{T}_p$, then $ \pi q_j = q_j$. 
Note that due to the simplex $(q_1,\ldots,q_P)$ being non-degenerate, also the $3$-simplex $(q_k,q_l,q_m,q_i)$ 
is non-degenerate, so that~$q_i$ cannot lie in the plane~$E$.
Therefore, we have $ \pi q_i \neq q_i$, so that
\begin{align*}
	 | \pi q_i -  \pi q_j|^2 
	< |  q_i -  \pi q_j |^2+ | \pi q_j - q_j|^2 
	= |q_i - q_j|^2 = \sigma_{i,j}^2,
\end{align*}
the latter by Definition~\ref{DefinitionAdmissibleSurfaceTensions}~\emph{iii)}. 
This implies the strict inequality in this subcase.
 
If both~$i$ and~$j$ correspond to phases being absent at~$\mathcal{T}_p$, 
we consider the orthogonal projection on the three dimensional affine
space~$\tilde E$ spanned by $(q_j,q_k,q_l,q_m)$, as well as the orthogonal 
projection~$\tilde \pi$ onto~$\tilde E$. As the $4$-simplex $(q_i,q_j,q_k,q_l,q_m)$ is non-degenerate, 
we have $\tilde \pi q_i \neq q_i$ and $\pi q_i = \pi \circ \tilde \pi q_i$.
Therefore, we have
\begin{align*}
	| \pi q_i -  \pi q_j|^2 \leq | q_i - \tilde \pi q_j  |^2 
	< | q_i - \tilde \pi q_j  |^2 + | \tilde \pi q_j - q_i|^2 
	= |q_i - q_j|^2 = \sigma_{i,j}^2,
\end{align*}
allowing us to conclude as in the previous case.

We now proceed with the definition of $\xi^p_{i'}$ for all $i'\in \{1,\ldots,P\}$.
As $(q_k,q_l,q_m)$ is non-degenerate and $\phi$ is isometric, also the 
triangle $(\phi q_k, \phi q_l, \phi q_m)$ is non-degenerate.
Therefore, there exist unique $\hat \lambda^{i'}_k,\hat \lambda^{i'}_l,\hat \lambda^{i'}_m \in \Rd[]$ 
such that $\hat \lambda^{i'}_k+\hat \lambda^{i'}_l + \hat \lambda^{i'}_m =1$ and
\begin{align*}
	\phi\circ \pi q_{i'} = \hat \lambda^{i'}_k \phi q_{k} + \hat \lambda^{i'}_l  \phi q_l + \hat \lambda^{i'}_m  \phi q_m.
\end{align*}
We may then on $\mathcal{U}_p$ define
\begin{align}\label{affine_definition}
	\xi^p_{i'} := \hat \lambda^{i'}_k \xi^p_{k} + \hat \lambda^{i'}_l  \xi^p_{l} + \hat \lambda^{i'}_m \xi^p_{m},
\end{align}
as well as $\xi^p_{i,j}$ and $\xi^p_{j,i}$ via \eqref{CalibrationLocalConstruction}.
By uniqueness of the coefficient, these definitions are consistent with the previous ones.

The claimed properties~\emph{i)} and~\emph{iii)} immediately follow from 
Proposition~\ref{prop:xi_triple_junction}. The identity~$\xi^p_{i,j}=-\xi^p_{j,i}$, 
\eqref{CalibrationLocalConstruction}, and~\eqref{eq:regularityEstimatesLocalXi} 
are straightforward consequences of the definition and again Proposition~\ref{prop:xi_triple_junction}. 
Therefore, we only have to prove~\eqref{CoercivityMinorityPhase} in order to get $|\xi^p_{i,j}| \leq 1$. 
To this end, we argue as follows:

Again by non-degeneracy of $(\phi q_k, \phi q_l, \phi q_m)$ for 
all $(x,t) \in \mathcal{U}_p$ there exist a unique matrix  $A(x,t) \in \Rd[2\times2]$ and $y(x,t) \in \Rd[2]$ such that
\begin{align}\label{rotation_plus_shift}
	\xi^p_{i'}(x,t) = A(x,t) \phi\circ \pi q_{i'} + y(x,t).
\end{align}
for all $i'=k,l,m$.
As \eqref{affine_definition} constitutes an affine combination, this equality even holds for all $i'\in \{1,\ldots,P\}$.
Furthermore, we have that the matrix $A$ is orthogonal, i.e., $A(x,t) \in \mathcal{O}$ 
for all $(x,t) \in \mathcal{U}_p$, since by Proposition~\ref{prop:xi_triple_junction}~\emph{i)} we have
\begin{align*}
	|A ( \phi\circ \pi q_{i'} -\phi\circ \pi q_{j'})| = |\xi^p_{i'} -\xi^p_{j'}| 
	= \sigma_{i',j'} |\xi^p_{i',j'}| = \sigma_{i',j'} = | \phi\circ \pi q_{i'} -\phi\circ \pi q_{j'}|
\end{align*}
and the triangle $(\phi q_k, \phi q_l, \phi q_m)$ is non-degenerate.
As $A$ is orthogonal and $\phi$ is isometric, we have by~\eqref{claim_chortness} that
\begin{align}
	|\xi^p_{i} - \xi^p_{i}| = |A ( \phi\circ \pi q_{i} -\phi\circ \pi q_{j})|
	= |\pi q_{i} - \pi q_{j}| < b |q_i - q_j| = b \sigma_{ij},
\end{align}
which together with~\eqref{CalibrationLocalConstruction} gives~\emph{iv)}.
\end{proof}

\begin{figure}
\includegraphics[scale=0.15]{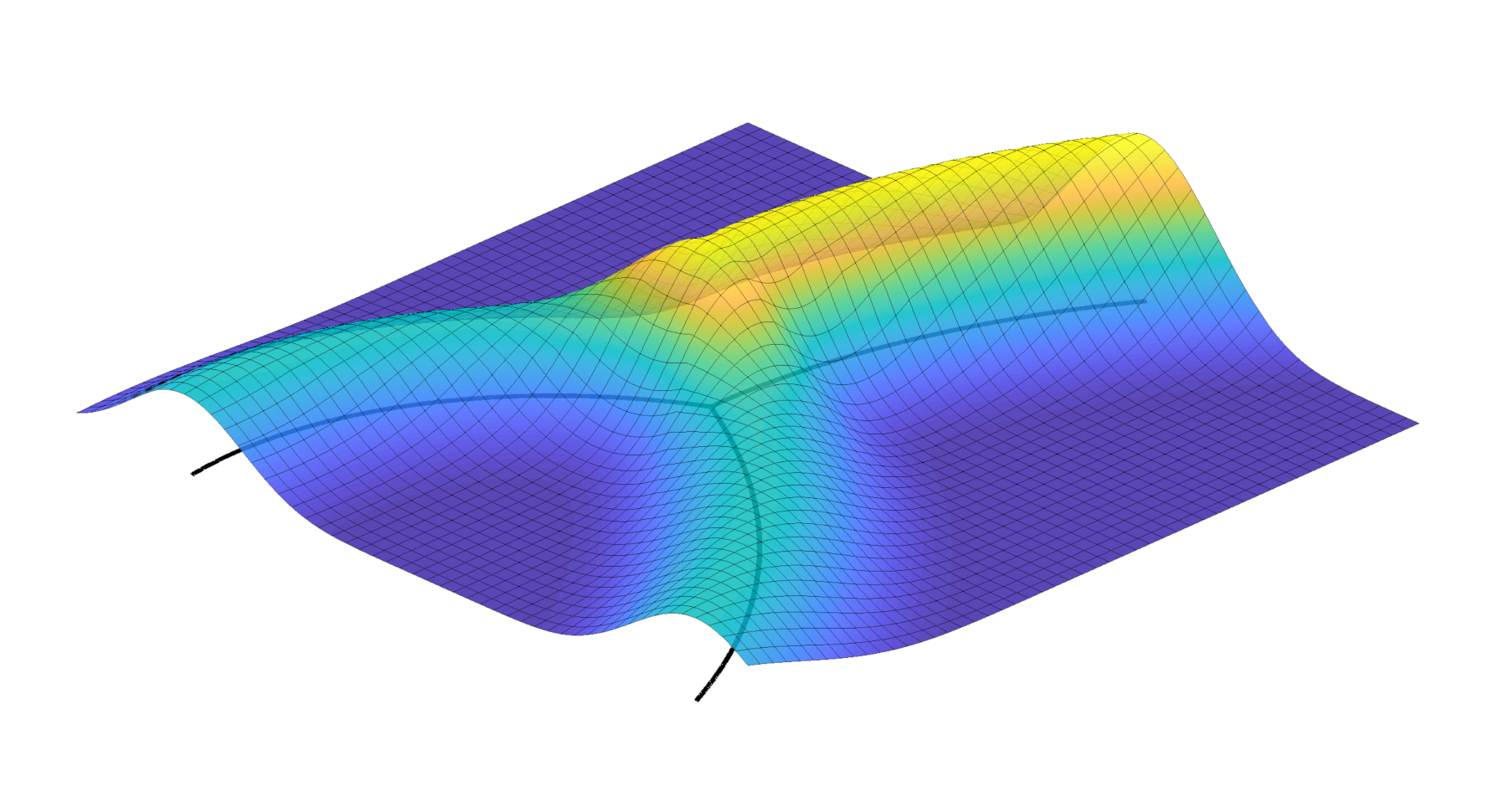}
\caption{Plot of the length of the vector field $\xi_{i,j}$. Observe that the length is~$1$ 
on the interface ${\bar{I}}_{i,j}$ of the strong solution, but decays quadratically away from 
it to a value strictly smaller than~$1$, even on the other interfaces ${\bar{I}}_{i,p}$ and 
${\bar{I}}_{j,p}$. As a consequence, the integral $\int_{I_{i,j}} 1-\vec{n}_{i,j}\cdot \xi_{i,j} \,\mathrm{d}\mathcal{H}^{1}$
 provides an upper bound for the interface error functional 
$c\smash{\int_{I_{i,j}} \min\{\dist^2(x,{\bar{I}}_{i,j}),\,1\} \,\mathrm{d}\mathcal{H}^{1}}$.}
\label{fig:lengthXi}
\end{figure}

Now we may define the global extensions $\xi_{i,j}=-\xi_{j,i}$ of the unit normal 
vector fields between the phases~$i$ and~$j$ in the strong solution by gluing the local definitions by 
means of the partition of unity $(\eta_{\mathrm{bulk}},\eta_1,\ldots,\eta_N)$ from Lemma~\ref{LemmaPartitionOfUnity}. 

\begin{construction}
\label{DefinitionGlobalXi}
Let $d=2$ and $P \in \mathbb{N}$, $P\geq 2$.
Let $\bar\Omega=(\bar{\Omega}_1,\ldots,\bar{\Omega}_P)$ 
be a strong solution to multiphase mean curvature flow in the sense of Definition~\ref{DefinitionStrongSolution}.
Let $(\eta_{\mathrm{bulk}},\eta_1,\ldots,\eta_N)$ be a partition of unity as constructed in 
Lemma~\ref{LemmaPartitionOfUnity}.
Let $i,j\in\{1,\ldots,P\}$ with $i\neq j$, and let for all $n \in \{1,\ldots, N\}$ 
the local vector fields $\xi^n_{i,j}=-\xi^n_{j,i}$
be given as in Lemma~\ref{LemmaLocalConstructionsNetwork}.
We then define
\begin{align}\label{GlobalXi}
\xi_{i,j}(x,t) &:= \sum_{n=1}^N\eta_n(x,t) \xi_{i,j}^n(x,t)
\end{align}
for all $x\in \Rd[2]$ and all $t\in [0,T]$.
\end{construction}

We proceed with the derivation of the coercivity condition provided by
the length of the vector fields $\xi_{i,j}$ as defined by 
Construction~\ref{DefinitionGlobalXi}. For an illustration we
refer to Figure~\ref{fig:lengthXi}.

\begin{lemma}
\label{lem:coercivityLenghtXi}
Let $d=2$ and $P \in \mathbb{N}$, $P\geq 2$.
Let $\bar\Omega=(\bar{\Omega}_1,\ldots,\bar{\Omega}_P)$ 
be a strong solution to multiphase mean curvature flow in the sense of Definition~\ref{DefinitionStrongSolution}.
Let $(\eta_{\mathrm{bulk}},\eta_1,\ldots,\eta_N)$ be a partition of unity as constructed in 
Lemma~\ref{LemmaPartitionOfUnity}. In particular, let~$\bar r_{\mathrm{min}}\in (0,1]$ 
be the localization scale defined by~\eqref{def:minLocScale}.
Let~$\xi_{i,j}$ for $i,j \in \{1,\ldots,P\}$ with $i\neq j$ be the family of vector
fields provided by Construction~\ref{DefinitionGlobalXi}. Then there exists 
a constant $C\geq 1$, depending only on~$\bar\Omega$ but not on~$\bar r_{\mathrm{min}}$,
such that for all $i,j\in\{1,\ldots,P\}$ with $i\neq j$ it holds
\begin{align}
\label{eq:coercivityLengthXi}
\frac{1}{C}\big(\bar r_{\mathrm{min}}^{-2}\dist^2(\cdot,{\bar{I}}_{i,j})\wedge 1 \big)
\leq 1-|\xi_{i,j}|.
\end{align}
Furthermore, in case of equal surface tensions $\sigma_{i,j}=\sigma_{j,k}=\sigma_{k,i}$, we have
\begin{align}
\label{Orthogonalityxik}
|\xi_k \cdot \xi_{i,j}| \leq C \dist(\cdot,\bar I_{i,j}).
\end{align}
\end{lemma}

\begin{proof}
Let~$(x,t)\in\Rd[2]{\times}[0,T]$
and $i,j\in\{1,\ldots,P\}$ with $i\neq j$. The asserted estimate~\eqref{eq:coercivityLengthXi} 
is trivially fulfilled for~$(x,t)\notin\supp\xi_{i,j}$. By the definition~\eqref{GlobalXi} 
we may therefore assume that there exists a topological feature~$n \in \{1,\ldots,N\}$ such that
$(x,t)\in\supp\eta_n$ and that $\eta_n(x,t) = \max\{\eta_{n'}(x,t): 1\leq n'\leq N \}$.
Because of the localization properties~\eqref{LocalizationTwoTripleJunctions}--\eqref{LocalizationTwoInterfaces},
we may additionally assume $\eta_n(x,t) \geq \frac{1}{4}$.
Otherwise, $|\xi_{i,j}| \leq \frac{3}{4}$ on account of the local vector fields having
at most unit length.

If either phase~$i$ or phase~$j$ is absent at the topological feature~$\mathcal{T}_n$, 
we argue as follows. Using $b\in (0,1)$ from~\eqref{CoercivityMinorityPhase} we compute
\begin{align*}
	|\xi_{i,j}|  & =\left|\eta_n\xi^n_{i,j} 
	+ \sum_{n'\in\{1,\ldots,N\}\setminus\{n\}}\eta_{n'}\xi^{n'}_{i,j} \right|  
	\leq \eta_n b +\sum_{n'\in\{1,\ldots,N\}\setminus\{n\}}\eta_{n'}\\
	& \leq 1 - \eta_n (1-b).
\end{align*}
Due to $\eta_n(x,t) \geq \frac{1}{4}$ we deduce 
$1- |\xi_{i,j}(x,t)| \geq  \frac{1}{4}(1-b) \in (0,1) $. 
Therefore the estimate~\eqref{eq:coercivityLengthXi} holds in this case.

Next, we assume that both phases~$i$ and~$j$ are present at~$\mathcal{T}_n$.
In the regime~$n=c\in\mathcal{C}$, it follows from~$(x,t)\in\supp\eta_c$, the localization
properties~\eqref{LocalizationTwoPhase} and~\eqref{LocalizationTripleJunctionTwoPhase},
the definitions~\eqref{def:minLocScale} and~\eqref{def:minLocScaleTripleJunction} of the localization
scales~$r_{\mathcal{P}}$ and~$\bar r_{\mathrm{min}}$, as well as 
the estimates~\eqref{eq:compDistances1} and~\eqref{eq:compDistances3}
that $\dist(x,\bar I_{i,j}(t)) \leq C\dist(x,\mathcal{I}(t))$.
Hence, \eqref{eq:coercivityLengthXi} is implied 
by the coercivity estimate~\eqref{eq:coercivityBulkCutOff} for the bulk cutoff
and the definition~\eqref{GlobalXi}.

If~$n=p\in\mathcal{P}$, denote by~$k\in\{1,\ldots,P\}$ the third phase present
at~$\mathcal{T}_p$ next to the phases~$i$ and~$j$. If~$x\in B_{r_{\mathcal{P}}}(\mathcal{T}_p(t))
\setminus \big(W_{j,k}(t)\cup W_{k,i}(t) \cup W_{k}(t)\big)$, 
then by~\eqref{eq:compDistances1} and~\eqref{eq:compDistances3}
it again holds $\dist(x,\bar I_{i,j}(t)) \leq C\dist(x,\mathcal{I}(t))$
so that~\eqref{eq:coercivityLengthXi} follows as before.
Thus, assume that~$x\in B_{r_{\mathcal{P}}}(\mathcal{T}_p(t))
\cap \big(W_{j,k}(t)\cup W_{k,i}(t) \cup W_{k}(t)\big)$. Figure~\ref{fig:lengthXi}
serves as an illustration for the subsequent argument, for which we
in fact assume that~$x\in W_k(t)$ (the argument in case of interface
wedges is similar). Based on the definition~\eqref{GlobalXi},
the localization properties~\eqref{LocalizationTwoTripleJunctions}--\eqref{LocalizationTwoInterfaces},
the coercivity estimate~\eqref{CoercivityMinorityPhase},
and the definitions~\eqref{DefinitionEtaTripleJunctionInterpolationWedge}, \eqref{DefinitionEtaTwoPhaseInterpolationWedge}
as well as~\eqref{DefinitionAuxiliaryCutoffHalfSpace},
we estimate at~$(x,t)$
\begin{align*}
1 - |\xi_{i,j}| &\geq 1 - \big(\eta_p + b\eta_{c_{k,i}} + b\eta_{c_{j,k}}\big)
\\&
= \lambda^{i,j}_k \Big(1 - \big(b\zeta_{c_{k,i}} + (1{-}b)\zeta_p\zeta_{c_{k,i}}\big)\Big)
+ (1{-}\lambda^{i,j}_k) \Big(1 - \big(b\zeta_{c_{j,k}} + (1{-}b)\zeta_p\zeta_{c_{j,k}}\big)\Big)
\\&
\geq (1-b)(1-\zeta_p) \geq (1-b)
\big(\bar r_{\mathrm{min}}^{-2}\dist^2(x,\mathcal{T}_p) \wedge 1\big).
\end{align*}
The trivial estimate~$\dist(x,\mathcal{T}_p(t)) \geq \dist(x,\bar I_{i,j}(t))$
therefore allows to conclude.

To show \eqref{Orthogonalityxik}, we simply use item v) from Lemma~\ref{LemmaLocalConstructionsNetwork} as well as the compatibility conditions on the vector fields $(\xi_{i,j})_{i\neq j}$.
\end{proof}

For a global definition of the velocity field $B$, we proceed analogously, i.e.,
we first provide a definition for local velocity fields $B^n$ for each topological
feature $\mathcal{T}_n$ with $n\in \{1,\ldots,N\}$ and then glue them together by means of the partition of unity
$(\eta_{\mathrm{bulk}},\eta_1,\ldots,\eta_N)$ from Lemma~\ref{LemmaPartitionOfUnity}. 

\begin{construction}\label{DefinitionGlobalB}
Let $d=2$ and $P \in \mathbb{N}$, $P\geq 2$.
Let $\bar\Omega=(\bar{\Omega}_1,\ldots,\bar{\Omega}_P)$ 
be a strong solution to multiphase mean curvature flow in the sense of Definition~\ref{DefinitionStrongSolution}.
Let $(\eta_{\mathrm{bulk}},\eta_1,\ldots,\eta_N)$ be a partition of unity as constructed in 
Lemma~\ref{LemmaPartitionOfUnity}. 

Let~$n \in \{1,\ldots,N\}$, and recalling the notation \eqref{def_support_space_time},
we define a continuous vector field
\begin{align*}
B^n\colon\mathcal{U}_n &\mapsto\Rd[2]
\end{align*}
as follows: in case of~$n\in\mathcal{C}$ we take~$B^n$ as the restriction to~$\mathcal{U}_n$
of the two-phase velocity field from Lemma~\ref{LemmaBoundsLocalConstructionsTwoPhase}.
More precisely, in case the curve~$\mathcal{T}_c$ connects two triple junctions,
the tangential component of~$B^n$ is chosen as in Proposition~\ref{prop:localCompatibilityEstimates};
otherwise, we simply let the tangential component vanish.
In case of~$n\in\mathcal{P}$ we take~$B^n$ as the restriction to~$\mathcal{U}_n$
of the triple junction velocity field from Proposition~\ref{prop:xi_triple_junction}.

We finally define a global velocity field by means of
\begin{align}\label{GlobalVelocityField}
B(x,t):=\sum_{n=1}^N \eta_n(x,t) B^n(x,t)
\end{align}
for all $x\in \Rd[2]$ and all $t\in [0,T]$.
\end{construction}

We briefly present the regularity properties of the family of 
local velocity fields from Construction~\ref{DefinitionGlobalB}.

\begin{lemma}
\label{LemmaLocalConstructionVelocityNetwork}
In the setting of Construction~\ref{DefinitionGlobalB}, for all~$n\in\{1,\ldots,N\}$
the associated local velocity field satisfies $B^n\in C^0_tC^2_x(\overline{\mathcal{U}_n}\setminus\mathcal{T}_{\mathcal{P}})$,
$\mathcal{T}_{\mathcal{P}}:=\bigcup_{p\in\mathcal{P}}\mathcal{T}_p$.
Moreover, there exists~$C>0$, which may depend on~$\bar\Omega$ but not on
the localization scale~$\bar r_{\mathrm{min}}$ from~\eqref{def:minLocScale}, 
such that throughout~$\mathcal{U}_n\setminus\mathcal{T}_{\mathcal{P}}$ it holds
\begin{align}
\label{eq:regularityEstimatesLocalVelocities}
\max_{k=0,1,2} \bar r_{\mathrm{min}}^k|\nabla^k B^n|
\leq C\bar r_{\mathrm{min}}^{-1}.
\end{align}
\end{lemma}

\begin{proof}
For~$n=c\in\mathcal{C}$ the estimate~\eqref{eq:regularityEstimatesLocalVelocities}
follows from~\eqref{eq:estimatesTwoPhaseVel} and~\eqref{eq:estimateDerivTangentialComp},
which in turn are indeed applicable thanks to the localization property~\eqref{LocalizationTwoPhase} 
and the definition~\eqref{def:minLocScale}. In case of~$n=p\in\mathcal{P}$,
we may apply Proposition~\ref{prop:xi_triple_junction} due to the 
localization property~\eqref{LocalizationTripleJunction} and the definition~\eqref{def:minLocScaleTripleJunction},
so that~\eqref{boundDerivativeB} implies~\eqref{eq:regularityEstimatesLocalVelocities}.
\end{proof}

Equipped with the definition of the global velocity field~$B$, we may now prove
a suitable estimate on the advective derivative of the bulk cutoff~$\eta_{\mathrm{bulk}}$ from 
Lemma~\ref{LemmaPartitionOfUnity}.

\begin{lemma}\label{LemmaAdvectionPartitionOfUnity}
Let $d=2$ and $P \in \mathbb{N}$, $P\geq 2$. Let $\bar\Omega=(\bar{\Omega}_1,\ldots,\bar{\Omega}_P)$ 
be a strong solution to multiphase mean curvature flow in the sense of Definition~\ref{DefinitionStrongSolution}. 
Let~$\eta_{\mathrm{bulk}}$ be the bulk cutoff from 
Lemma~\ref{LemmaPartitionOfUnity}, $\bar r_{\mathrm{min}}\in (0,1]$
the localization scale defined by~\eqref{def:minLocScale},
and $\mathcal{T}_{\mathcal{P}}:=\bigcup_{p\in\mathcal{P}}\mathcal{T}_p$.
Let~$B$ be the global velocity field from Construction~\ref{DefinitionGlobalB}. 
Denote by $\mathcal{I}:=\bigcup_{t\in [0,T]}\bigcup_{i\neq j}{\bar{I}}_{i,j}(t)\times\{t\}$
the evolving network of interfaces.
Then there exists a constant~$C>0$, depending only on the strong solution~$\bar\Omega$
but not on~$\bar r_{\mathrm{min}}$, such that
\begin{align}
\label{AdvectionEquationBulkLocalization}
|\partial_t\eta_{\mathrm{bulk}}+(B\cdot\nabla)\eta_{\mathrm{bulk}}|
&\leq C\bar r_{\mathrm{min}}^{-2}
\big(\bar r_{\mathrm{min}}^{-2}\dist^2(\cdot,\mathcal{I})\wedge 1\big)
\end{align}
in $\Rd[2]{\times} [0,T]\setminus \mathcal{T}_{\mathcal{P}}$. Moreover, 
for all~$n\in\{1,\ldots,N\}$ and all distinct~$i,j\in\{1,\ldots,P\}$
such that either phase~$i$ or phase~$j$ is absent at~$\mathcal{T}_n$ it holds
\begin{align}
\label{eq:advectionWrongCutoffs}
|\partial_t\eta_{n}+(B\cdot\nabla)\eta_{n}|
&\leq C\bar r_{\mathrm{min}}^{-2}
\big(\bar r_{\mathrm{min}}^{-2}\dist^2(\cdot,\bar I_{i,j})\wedge 1\big)
\end{align}
in $\Rd[2]{\times} [0,T]\setminus \mathcal{T}_{\mathcal{P}}$.
\end{lemma}

\begin{proof}
The estimate~\eqref{eq:advectionWrongCutoffs} is trivially fulfilled
in case of~$n=p\in\mathcal{P}$ by~\eqref{eq:estimatesDerivEta}, \eqref{LocalizationTripleJunction} 
and the definition~\eqref{def:minLocScaleTripleJunction} of the localization scale~$r_{\mathcal{P}}$.
Hence, let us reserve notation for the proof of~\eqref{eq:advectionWrongCutoffs} by
fixing~$c''\in\mathcal{C}$ and distinct phases~$i',j'\in\{1,\ldots,P\}$
such that at least one of them is absent at~$\mathcal{T}_{c''}$.

We now split the proof into two parts, first establishing the asserted 
estimates along two-phase interfaces~$\mathcal{T}_{c}$ and away from triple
junctions, and second in the vicinity of triple junctions adjacent to~$\mathcal{T}_{c}$. 
More precisely, by the localization properties~\eqref{LocalizationTwoPhase}--\eqref{LocalizationTwoInterfaces}
and the choices~\eqref{def:minLocScaleTripleJunction}--\eqref{def:minLocScale} of the 
localization scales~$r_{\mathcal{P}}$ and~$\bar r_{\mathrm{min}}$,
it suffices to prove~\eqref{AdvectionEquationBulkLocalization}
in~$\bigcup_{c\in\mathcal{C}}\mathrm{im}_{\bar r_{\mathrm{min}}}(\Psi_{\mathcal{T}_{c}})
\setminus \bigcup_{p\in\mathcal{P}} \bigcup_{t\in [0,T]}B_{r_{\mathcal{P}}}(\mathcal{T}_p(t)){\times}\{t\}$
and in $\bigcup_{p\in\mathcal{P}}\bigcup_{t\in [0,T]}B_{r_{\mathcal{P}}}(\mathcal{T}_p(t)){\times}\{t\}$,
respectively. We in fact may argue separately for each~$c\in\mathcal{C}$
and each~$p\in\mathcal{P}$.

\textit{Step 1: Estimates close to $\mathcal{T}_{c}$ and away from triple junctions.}
	In this step, we restrict ourselves to the region 
	$\mathrm{im}_{\bar r_{\mathrm{min}}}(\Psi_{\mathcal{T}_{c}})
	\setminus \bigcup_{p\in\mathcal{P}} \bigcup_{t\in [0,T]}B_{r_{\mathcal{P}}}(\mathcal{T}_p(t)){\times}\{t\}$. 
	To fix notation, let $i,j\in \{1,\ldots,P\}$ be such that~$c$ refers to a 
	two-phase interface $\mathcal{T}_{c} \subset {\bar{I}}_{i,j}$.
	Recalling~\eqref{eq:bulkCutoffAwayTripleJunctions}, we register that
	\begin{align}
		\label{eq:away from tj eta}
		\eta_{\mathrm{bulk}}& = 1- \eta_{c} ,
		\\ 
		\label{eq:away from tj etac and zeta}
		\eta_{c} &= \zeta_{c}= \zeta\Big(\frac{s_{i,j}}{\delta \bar r_{\mathrm{min}}}\Big),
		\\
		\label{eq:away from tj B}
		B&= \eta_{c} B^{c},
		\end{align}
		in $\mathrm{im}_{\bar r_{\mathrm{min}}}(\Psi_{\mathcal{T}_{c}})
	\setminus \bigcup_{p\in\mathcal{P}} \bigcup_{t\in [0,T]}B_{r_{\mathcal{P}}}(\mathcal{T}_p(t)){\times}\{t\}$.

For~\eqref{AdvectionEquationBulkLocalization}, we first observe that the signed 
distance function is transported by~$B^{c}$, cf.\ \eqref{eq:transportSignedDistanceB}. By the chain rule, 
this also holds for~$\zeta_{c}$, i.e.,
\begin{align}
\label{eq:transportInterfaceCutoff}
\partial_t\zeta_{c} + (B^{c}\cdot\nabla)\zeta_{c} = 0
\quad\text{in } \mathrm{im}(\Psi_{\mathcal{T}_{c}}).
\end{align} 
Hence, using~\eqref{eq:away from tj B}, \eqref{eq:away from tj eta}, 
the quadratic order of~$\eta_{\mathrm{bulk}}$ from~\eqref{eq:upperBoundBulkCutOff},
and the regularity estimates~\eqref{eq:estimatesDerivEta}
and~\eqref{eq:regularityEstimatesLocalVelocities} we obtain
\begin{align}
\label{eq:auxEstimateAdvectiveDerivative}
|\partial_t \zeta_c + (B \cdot \nabla ) \zeta_c | = 
\eta_{\mathrm{bulk}} | (B^{c} \cdot \nabla) \zeta_{c} | \leq C\bar r_{\mathrm{min}}^{-2}
\big(\bar r_{\mathrm{min}}^{-2}\dist^2(\cdot,\mathcal{I})\wedge 1\big)
\end{align}
in the region~$\mathrm{im}_{\bar r_{\mathrm{min}}}(\Psi_{\mathcal{T}_{c}})
\setminus \bigcup_{p\in\mathcal{P}} \bigcup_{t\in [0,T]}B_{r_{\mathcal{P}}}(\mathcal{T}_p(t)){\times}\{t\}$.
By~\eqref{eq:away from tj eta} and~\eqref{eq:away from tj etac and zeta},
this is equivalent to~\eqref{AdvectionEquationBulkLocalization}.

For a proof of~\eqref{eq:advectionWrongCutoffs} throughout
$\mathrm{im}_{\bar r_{\mathrm{min}}}(\Psi_{\mathcal{T}_{c}})
\setminus \bigcup_{p\in\mathcal{P}} \bigcup_{t\in [0,T]}B_{r_{\mathcal{P}}}(\mathcal{T}_p(t)){\times}\{t\}$, 
we may assume without loss of generality that~$c''=c$;
otherwise, the estimate~\eqref{eq:advectionWrongCutoffs} is trivially fulfilled
by~\eqref{LocalizationTwoPhase} and the definition~\eqref{def:minLocScale} 
of the localization scale~$\bar r_{\mathrm{min}}$. However, if~$c''=c$
then the above argument already yields the claim thanks to~\eqref{eq:away from tj eta}, \eqref{eq:away from tj etac and zeta}
and~\eqref{eq:auxEstimateAdvectiveDerivative}.

\textit{Step 2: Estimates close to $\mathcal{T}_c$ and in the vicinity of triple junctions.}
Now, consider~$p\in\mathcal{P}$ and assume that the pairwise distinct phases~$i,j,k\in\{1,\ldots,P\}$
are present at~$\mathcal{T}_p$. Modulo a permutation of the indices, it suffices to
consider the two unique two-phase interfaces~$\mathcal{T}_{c_{i,j}}\subset\bar I_{i,j}$ 
and~$\mathcal{T}_{c_{k,i}}\subset\bar I_{k,i}$ so that~$c:=c_{i,j}\sim p$ and~$c':=c_{k,i}\sim p$, 
and then to prove the desired estimate~\eqref{AdvectionEquationBulkLocalization} on the interface wedge~$W_{i,j}$
and the interpolation wedge~$W_{i}$. 

In this step, let us turn to the interface wedge~$W_{i,j}$. 
The interpolation wedge~$W_i$ will be discussed in \textit{Step~3}. 
With respect to~\eqref{eq:advectionWrongCutoffs},
it then suffices to work in the regime~$c''\sim p$ and~$c''=c$;
otherwise, the estimate~\eqref{eq:advectionWrongCutoffs}
is again fulfilled for trivial reasons thanks to~\eqref{LocalizationTripleJunctionTwoPhase}
and~\eqref{LocalizationTwoInterfaces}.
Based on~\eqref{DefinitionEtaTwoPhaseWedgeInterface}
and~\eqref{eq:bulkCutoffTripleJunctionsInterfaceWedge} we then have
\begin{align}
	\label{eq:at tj etac}
	\eta_c &= (1-\zeta_p) \zeta_c ,
	\\
	\label{eq:at tj eta etas}
	\eta_{\mathrm{bulk}} &= 1-\eta_c-\eta_p = 1-\zeta_c,
	\\
	\label{eq:at tj B Bc Bp}
	B&= \eta_c B^c + \eta_p B^p,
\end{align}
throughout~$B_{r_{\mathcal{P}}}(\mathcal{T}_p(t)) \cap W_{i,j}(t)$ for all~$t\in [0,T]$.

For the estimate on the advective derivative of the bulk cutoff, using~\eqref{eq:at tj eta etas}
and the transport equation for the interface cutoff~\eqref{eq:transportInterfaceCutoff}
(which is applicable throughout~$B_{r_{\mathcal{P}}}(\mathcal{T}_p(t)) \cap W_{i,j}(t)$ for all~$t\in [0,T]$
due to~\eqref{eq:inlcusionInterfaceWedge}) we obtain
\begin{align*}
 \partial_t\zeta_c 
&= -(B^c\cdot\nabla)\zeta_c
= -\big(B\cdot\nabla\big)\zeta_c 
- \eta_{\mathrm{bulk}}(B^c\cdot\nabla)\zeta_c
-\eta_p\big((B^p{-}B^c)\cdot\nabla\big)\zeta_c
\end{align*}
in $B_{r_{\mathcal{P}}}(\mathcal{T}_p(t)) \cap W_{i,j}(t)$ for all~$t\in [0,T]$.
In particular, because of~\eqref{eq:upperBoundBulkCutOff}, \eqref{eq:regularityEstimatesLocalVelocities},
\eqref{eq:regEstimatesInterfaceCutoff1}, \eqref{eq:inlcusionInterfaceWedge}, \eqref{eq:estimatesDerivEta},
\eqref{eq:localComp3}, \eqref{eq:compDistances3}, and finally~\eqref{eq:compNetworkDistanceTripleJunctions}
this entails
\begin{align}\label{eq:dt zetac Wij} 
\big|\partial_t\zeta_c+\big(B\cdot\nabla\big)\zeta_c\big|
\leq C\bar r_{\mathrm{min}}^{-2}
\big(\bar r_{\mathrm{min}}^{-2}\dist^2(\cdot,\mathcal{I})\wedge 1\big)
\end{align}
in $B_{r_{\mathcal{P}}}(\mathcal{T}_p(t)) \cap W_{i,j}(t)$ for all~$t\in [0,T]$.
By the representation~\eqref{eq:at tj eta etas}, 
this is equivalent to~\eqref{AdvectionEquationBulkLocalization}.

To obtain the asserted bound on the advective derivative of the interface cut-off~$\eta_c$, 
we use that since~$\zeta_p$ is only a smooth function of the distance to the triple 
point~$\mathcal{T}_p(t)=\{p(t)\}$ (performing an excusable abuse of notation), it satisfies the transport 
equation~$\partial_t \zeta_p + (\frac{\mathrm{d}}{\mathrm{d}t}p(t)\cdot \nabla) \zeta_p =0$
throughout~$\Rd[2]{\times}[0,T]\setminus\mathcal{T}_{\mathcal{P}}$.
By Proposition~\ref{prop:xi_triple_junction}~\textit{i)}, the partition of unity
property of the family~$(\eta_1,\ldots,\eta_N)$, and the
regularity estimates~\eqref{eq:regularityEstimatesLocalVelocities} resp.\
\eqref{eq:estimatesDerivEta}, it follows that~$|B-B(p(t),t)| \leq 
C\bar r_{\mathrm{min}}^{-2}\dist(\cdot,\mathcal{T}_p)$
in $B_{r_{\mathcal{P}}}(\mathcal{T}_p(t)) \cap 
(W_{i,j}(t) \cup W_{i}(t) \cup W_{j}(t))$ for all~$t\in [0,T]$. This in turn implies
by means of~\eqref{DefinitionAuxiliaryCutoffHalfSpace}
\begin{align}
\label{eq:transportTripleJunctionCutoff}
|\partial_t\zeta_p + \big(B\cdot\nabla)\zeta_p| 
\leq C\bar r_{\mathrm{min}}^{-2}r_{\mathcal{P}}^{-2}\dist^2(\cdot,\mathcal{T}_p(t))
\leq C\bar r_{\mathrm{min}}^{-2}(1 - \zeta_p)
\end{align}
in $B_{r_{\mathcal{P}}}(\mathcal{T}_p(t)) \cap 
(W_{i,j}(t) \cup W_{i}(t) \cup W_{j}(t))$ for all~$t\in [0,T]$. 
Hence, when restricting to the interface wedge we obtain from
the combination of~\eqref{eq:at tj etac}, the product rule, \eqref{eq:dt zetac Wij},
\eqref{eq:transportTripleJunctionCutoff} and finally~\eqref{GlobalEquationsBoundMinorityPhases}
that the desired estimate~\eqref{eq:advectionWrongCutoffs} indeed holds true
in $B_{r_{\mathcal{P}}}(\mathcal{T}_p(t)) \cap W_{i,j}(t)$ for all~$t\in [0,T]$.

\textit{Step 3: Estimates in interpolation wedges at triple junctions.}
We turn to the proof of~\eqref{AdvectionEquationBulkLocalization}
and~\eqref{eq:advectionWrongCutoffs} on the interpolation wedge~$W_i$. 
Recall to this end the notation fixed at the beginning of \textit{Step~2}.
With respect to proving~\eqref{eq:advectionWrongCutoffs},
it suffices to consider~$c''\sim p$ and~$c''\in\{c,c'\}$,
and thus up to a relabeling~$c''=c$; otherwise, the estimate~\eqref{eq:advectionWrongCutoffs}
follows trivially because of~\eqref{LocalizationTripleJunctionTwoPhase}
and~\eqref{LocalizationTwoInterfaces}. 

Because of~\eqref{DefinitionEtaTwoPhaseInterpolationWedge}
and~\eqref{eq:bulkCutoffTripleJunctionsInterpolationWedge}, it then holds
(abbreviating~$\lambda:=\lambda^{j,k}_i$)
\begin{align}
\label{eq:repWrongCutoffInterpolationWedge2}
\eta_c &= \lambda(1-\zeta_p)\zeta_c,
\\
\label{eq:repBulkCutoffInterpolationWedge2}
\eta_{\mathrm{bulk}} &= 1 - \eta_c - \eta_{c'} - \eta_{p}
											= \lambda (1 {-} \zeta_c) + (1 {-} \lambda)(1 {-} \zeta_{c'}),
\\
\label{eq:repVelocityInterpolationWedge2}
B &= \eta_cB^c + \eta_{c'}B^{c'} + \eta_p B^p,
\end{align}
throughout~$B_{r_{\mathcal{P}}}(\mathcal{T}_p(t)) \cap W_{i}(t)$ for all~$t\in [0,T]$.

Based on the second identity of~\eqref{eq:repBulkCutoffInterpolationWedge2}
and~\eqref{eq:repVelocityInterpolationWedge2},
we may split the task of estimating the advective derivative of the bulk cutoff
as follows:
\begin{align*}
\partial_t\eta_{\mathrm{bulk}} + (B\cdot\nabla)\eta_{\mathrm{bulk}}
=: I + II,
\end{align*} 
where we defined
\begin{align*}
I &:= (1{-}\zeta_c)(\partial_t {+} B\cdot\nabla)\lambda
			+ (1{-}\zeta_{c'})(\partial_t {+} B\cdot\nabla)(1{-}\lambda),
\\
II &:= \lambda\big(\partial_t{+}B\cdot\nabla\big)(1{-}\zeta_c)
			 +(1{-}\lambda)\big(\partial_t{+}B\cdot\nabla\big)(1{-}\zeta_{c'}).
\end{align*}

We estimate term by term. For an estimate of~$II$, we argue in a similar fashion to \textit{Step~2}.
More precisely, applying~\eqref{eq:repBulkCutoffInterpolationWedge2}
and the transport equation for the interface cutoff~\eqref{eq:transportInterfaceCutoff}
(which is applicable throughout~$B_{r_{\mathcal{P}}}(\mathcal{T}_p(t)) \cap W_{i}(t)$ for all~$t\in [0,T]$
due to~\eqref{eq:inclusionInterpolWedge}) we have
\begin{align*}
 \partial_t\zeta_c 
= -\big(B\cdot\nabla\big)\zeta_c 
- \eta_{\mathrm{bulk}}(B^c\cdot\nabla)\zeta_c
- \eta_{c'}\big((B^{c'}{-}B^c)\cdot\nabla\big)\zeta_c
- \eta_p\big((B^p{-}B^c)\cdot\nabla\big)\zeta_c
\end{align*}
in $B_{r_{\mathcal{P}}}(\mathcal{T}_p(t)) \cap W_{i}(t)$ for all~$t\in [0,T]$.
Replacing the use of~\eqref{eq:inlcusionInterfaceWedge} by~\eqref{eq:inclusionInterpolWedge}
and the use of~\eqref{eq:compDistances3} by~\eqref{eq:compDistances1}, we may rely
on the otherwise same argument entailing~\eqref{eq:dt zetac Wij} to deduce that
(adding also zero in form of~$B^{c'}{-}B^c=(B^{c'}{-}B^p)+(B^p{-}B^c)$)
\begin{align}\label{eq:dt zetac Wi} 
\big|\partial_t\zeta_c+\big(B\cdot\nabla\big)\zeta_c\big|
\leq C\bar r_{\mathrm{min}}^{-2}
\big(\bar r_{\mathrm{min}}^{-2}\dist^2(\cdot,\mathcal{I})\wedge 1\big)
\end{align}
in $B_{r_{\mathcal{P}}}(\mathcal{T}_p(t)) \cap W_{i}(t)$ for all~$t\in [0,T]$.
Of course, the same estimate holds true in terms of~$\zeta_{c'}$.
Hence, $|II| \leq C\bar r_{\mathrm{min}}^{-2}
\big(\bar r_{\mathrm{min}}^{-2}\dist^2(\cdot,\mathcal{I})\wedge 1\big)$
in~$B_{r_{\mathcal{P}}}(\mathcal{T}_p(t)) \cap W_{i}(t)$ for all~$t\in [0,T]$
as desired.

We turn to the estimate of~$I$. Adding zero and relying on~\eqref{eq:repBulkCutoffInterpolationWedge2}
as well as~\eqref{eq:repVelocityInterpolationWedge2}, we observe that it holds
\begin{align*}
(\partial_t {+} B\cdot\nabla)\lambda
&= (\partial_t{+}B^p\cdot\nabla)\lambda 
+ \big(\eta_c(B^c{-}B^p) {+} \eta_{c'}(B^{c'}{-}B^p) 
{-}\eta_{\mathrm{bulk}}B^p\big)\cdot\nabla\lambda.
\end{align*}
By familiar arguments in combination with the controlled blowup~\eqref{boundslambda1}
of the derivative of the interpolation parameter, one checks that the
second right hand side term of the previous display is of the order~$O(\bar r_{\mathrm{min}}^{-2})$.
The first right hand side term is of the same order thanks to~\eqref{eq:compDistances1} and the bound~\eqref{advectionlambda}
on the advective derivative of the interpolation parameter
(for which we may freely pass from~$B^p$ to~$B^p(p(t),t)$, abusing
again notation in form of~$\mathcal{T}_p(t)=\{p(t)\}$, cf.\
Proposition~\ref{prop:xi_triple_junction}~\textit{i)} and
the estimate~\eqref{boundDerivativeB}). Hence,
\begin{align}
\label{eq:advectionInterpolParameterGlobalVel}
|\partial_t\lambda + (B\cdot\nabla)\lambda| \leq C\bar r_{\mathrm{min}}^{-2}.
\end{align}
By~\eqref{eq:regEstimatesInterfaceCutoff} and~\eqref{eq:compDistances1}, we thus obtain
$|(1{-}\zeta_c)(\partial_t {+} B\cdot\nabla)\lambda|
\leq C\bar r_{\mathrm{min}}^{-2}(r_{\mathrm{min}}^{-2}\dist^2(\cdot,\mathcal{I}) \wedge 1)$.
Arguing analogously one also bounds the term~$(1{-}\zeta_{c'})(\partial_t {+} B\cdot\nabla)(1{-}\lambda)$
to the same order, so that in summary~\eqref{AdvectionEquationBulkLocalization}
follows in the region~$B_{r_{\mathcal{P}}}(\mathcal{T}_p(t)) \cap W_{i}(t)$ for all~$t\in [0,T]$.

We finally provide the proof of~\eqref{eq:advectionWrongCutoffs} in the given interpolation wedge.
When computing the advective derivative of~$\eta_c$, it follows
from~\eqref{eq:repWrongCutoffInterpolationWedge2},  the product rule, 
\eqref{eq:dt zetac Wi}, \eqref{eq:transportTripleJunctionCutoff}
and~\eqref{GlobalEquationsBoundMinorityPhases} that we only need to additionally control the 
term when the derivative falls onto the interpolation parameter. 
However, since we already have~\eqref{eq:advectionInterpolParameterGlobalVel}
at our disposal, it follows from~\eqref{eq:regEstimatesTripleJunctionCutoff} that
\begin{align*}
|(\partial_t{+}B\cdot\nabla)\lambda| (1-\zeta_p) \zeta_c
\leq C\bar r_{\mathrm{min}}^{-2}\big(\bar r_{\mathrm{min}}^{-2}
\dist^2(\cdot,\mathcal{T}_p) \wedge 1\big),
\end{align*}
which by~\eqref{eq:compDistances1} (or a trivial argument if either~$i'$
or~$j'$ is absent at~$\mathcal{T}_p$) entails a bound of required order.
This in turn concludes the proof.
\end{proof}

\subsection{Global compatibility estimates}\label{SectionCompatibilityBoundsGluing}
We next lift the local compatibility estimates from Proposition~\ref{prop:localCompatibilityEstimates}
to compatibility estimates between the global and local constructions.
These technical estimates will be needed in order to derive
the estimates~\eqref{TransportEquationXi}--\eqref{Dissip} for the global constructions from
the corresponding ones for the local constructions in Lemma~\ref{LemmaBoundsLocalConstructionsTwoPhase}
and Proposition~\ref{prop:xi_triple_junction}.

\begin{lemma}\label{LemmaCompatibilityLocalConstructions}
Let $d=2$ and $P \in \mathbb{N}$, $P\geq 2$. Let $\bar\Omega=(\bar{\Omega}_1,\ldots,\bar{\Omega}_P)$ 
be a strong solution to multiphase mean curvature flow in the sense of Definition~\ref{DefinitionStrongSolution}. 
Let $(\eta_{\mathrm{bulk}},\eta_1,\ldots,\eta_N)$ be a partition of unity as constructed in 
Lemma~\ref{LemmaPartitionOfUnity}. In particular, let~$\bar r_{\mathrm{min}}\in (0,1]$ 
be the localization scale defined by~\eqref{def:minLocScale} and~$\mathcal{T}_{\mathcal{P}}
:=\bigcup_{p\in\mathcal{P}}\mathcal{T}_p$.
Let $(\xi^n_{i,j})_{n\in\{1,\ldots,N\}}$ be the local vector fields 
from Lemma~\ref{LemmaLocalConstructionsNetwork} as well as $(B^n)_{n\in\{1,\ldots,N\}}$ be the local velocity fields 
from Construction~\ref{DefinitionGlobalB}. Let $\xi_{i,j}$ be the global vector 
fields from Construction~\ref{DefinitionGlobalXi}, and let $B$ be the global velocity field from 
Construction~\ref{DefinitionGlobalB}.

Then, the local and global constructions are compatible in the sense that
for all topological features~$n\in\{1,\ldots,N\}$, and all distinct
phases~$i,j\in\{1,\ldots,P\}$ such that both~$i$ and~$j$ are present at~$\mathcal{T}_n$, 
the following estimates are satisfied
\begin{align}
\label{CompatibilityBound1}
\mathds{1}_{\supp\eta_n}\big|\xi_{i,j}-\xi^n_{i,j}\big|
&\leq C\big(\bar r_{\mathrm{min}}^{-1}\dist(\cdot,{\bar{I}}_{i,j})\wedge 1\big),
\\\label{CompatibilityBound2}
\mathds{1}_{\supp\eta_n}\big|(\xi_{i,j}-\xi^n_{i,j})\cdot\xi^n_{i,j}\big|
&\leq C\big(\bar r_{\mathrm{min}}^{-2}\dist^2(\cdot,{\bar{I}}_{i,j})\wedge 1\big),
\\\label{CompatibilityBound3}
\mathds{1}_{\supp\eta_n}\big|B-B^n\big|
&\leq C\bar r_{\mathrm{min}}^{-1}
\big(\bar r_{\mathrm{min}}^{-2}\dist^2(\cdot,{\bar{I}}_{i,j})\wedge 1\big),+
\\\label{CompatibilityBound4}
\mathds{1}_{\supp\eta_n}\big|\nabla B-\nabla B^n\big|
&\leq C\bar r_{\mathrm{min}}^{-2}
\big(\bar r_{\mathrm{min}}^{-1}\dist(\cdot,{\bar{I}}_{i,j})\wedge 1\big)
\end{align}
throughout~$\Rd[2]{\times} [0,T]\setminus \mathcal{T}_{\mathcal{P}}$. The constant~$C>0$
may depend on the strong solution~$\bar\Omega$, but is independent
of~$\bar r_{\mathrm{min}}$.
\end{lemma}

For the proof of Lemma~\ref{LemmaCompatibilityLocalConstructions}, recall that we decomposed 
$\{1,\ldots,N\}=:\mathcal{C} \cupdot\mathcal{P}$ with the convention that $\mathcal{C}$ enumerates the connected components in space-time of the smooth two-phase interfaces and $\mathcal{P}$ enumerates the triple junctions.
If $p\in \mathcal{P}$, we defined $\mathcal{T}_p$ to be the trajectory in space-time described by the triple junction. 
If $c\in \mathcal{C}$, we defined $\mathcal{T}_c \subset {\bar{I}}_{i,j}$ for some $i,j \in \{1,\ldots,P\}$ with $i\neq j$ to be the corresponding space-time connected component of a two-phase interface ${\bar{I}}_{i,j}$.
We further write~$c\sim p$ for~$c\in\mathcal{C}$ and~$p\in\mathcal{P}$
if and only if~$\mathcal{T}_c$ has an endpoint at~$\mathcal{T}_p$. 
Note finally that two distinct phases~$i,j\in\{1,\ldots,P\}$ are simultaneously present at
a topological feature~$\mathcal{T}_n$, $n\in\{1,\ldots,N\}$, if and only if~$\mathcal{T}_n\subset\bar I_{i,j}$.

\begin{proof}
We aim to reduce the situation to the local compatibility estimates from
Proposition~\ref{prop:localCompatibilityEstimates}. Such a reduction
argument turns out to be possible due to the localization
properties~\eqref{LocalizationTwoTripleJunctions}--\eqref{LocalizationTwoInterfaces},
the estimates~\eqref{eq:upperBoundBulkCutOff}--\eqref{GlobalEquationsBoundGradientMinorityPhases},
and our assumption that both phases~$i$ and~$j$ are present at the selected topological feature.
For all what follows, let~$n\in\{1,\ldots,N\}$ and~$i,j\in\{1,\ldots,P\}$ such that~$i\neq j$
as well as~$\mathcal{T}_n\subset\bar I_{i,j}$.
For notational convenience, we abbreviate for the purpose of the proof~$\bar r:= \bar r_{\mathrm{min}}$
and~$d_{i,j}:=\dist(\cdot,\bar I_{i,j})$.

\textit{Step 1: Proof of~\eqref{CompatibilityBound1}.}
We insert the definition~\eqref{GlobalXi}
which in combination with the estimates~\eqref{eq:upperBoundBulkCutOff},
\eqref{GlobalEquationsBoundMinorityPhases} and~\eqref{eq:regularityEstimatesLocalXi} yields
\begin{align}
\nonumber
\mathds{1}_{\supp\eta_n}(\xi_{i,j} {-} \xi^n_{i,j})
&= -\mathds{1}_{\supp\eta_n}\eta_{\mathrm{bulk}}\xi^n_{i,j}
+ \sum_{n'=1,n'\neq n}^N \mathds{1}_{\supp\eta_n}\eta_{n'}(\xi^{n'}_{i,j} {-} \xi^n_{i,j})
\\& \label{eq:auxReduction1}
= \sum_{\substack{n'=1,n'\neq n \\ \mathcal{T}_{n'}\subset\bar I_{i,j}}}^N 
\mathds{1}_{\supp\eta_n}\eta_{n'}(\xi^{n'}_{i,j} {-} \xi^n_{i,j})
+ O(\bar r^{-2}d_{i,j}^2\wedge 1).
\end{align}
Next, the localization
properties~\eqref{LocalizationTwoTripleJunctions}--\eqref{LocalizationTwoInterfaces}
allow to represent the remaining right hand side terms in form of
\begin{align*}
\sum_{\substack{n'=1,n'\neq n \\ \mathcal{T}_{n'}\subset\bar I_{i,j}}}^N 
\mathds{1}_{\supp\eta_n}\eta_{n'}(\xi^{n'}_{i,j} {-} \xi^n_{i,j})
&
= \sum_{p\in\mathcal{P},\mathcal{T}_p\subset\bar I_{i,j}}\sum_{c\in\mathcal{C},c\sim p}
\mathds{1}_{n=c}\mathds{1}_{\supp\eta_c}\eta_p(\xi^p_{i,j} {-} \xi^c_{i,j})
\\&\hspace*{-4ex}
+ \sum_{c\in\mathcal{C},\mathcal{T}_c\subset\bar I_{i,j}}\sum_{p\in\mathcal{P},c\sim p}
\mathds{1}_{n=p}\mathds{1}_{\supp\eta_p}\eta_c(\xi^c_{i,j} {-} \xi^p_{i,j})
\\&\hspace*{-4ex}
+ \sum_{c\in\mathcal{C},\mathcal{T}_c\subset\bar I_{i,j}}\sum_{p\in\mathcal{P},c\sim p}
\sum_{\substack{c'\in\mathcal{C},c'\neq c \\ c'\sim p}}
\mathds{1}_{n=c'}\mathds{1}_{\supp\eta_{c'}}\eta_c(\xi^c_{i,j} {-} \xi^{c'}_{i,j}).
\end{align*}
The \textit{assumption}~$\mathcal{T}_n\subset\bar I_{i,j}$ furthermore
enables us to post-process the previous identity as follows
\begin{align*}
\sum_{\substack{n'=1,n'\neq n \\ \mathcal{T}_{n'}\subset\bar I_{i,j}}}^N 
\mathds{1}_{\supp\eta_n}\eta_{n'}(\xi^{n'}_{i,j} {-} \xi^n_{i,j})
&= \sum_{p\in\mathcal{P},\mathcal{T}_p\subset\bar I_{i,j}}
\sum_{\substack{c\in\mathcal{C},\mathcal{T}_c\subset\bar I_{i,j} \\ c\sim p}}
\mathds{1}_{n=c}\mathds{1}_{\supp\eta_c}\eta_p(\xi^p_{i,j} {-} \xi^c_{i,j})
\\&~~~
+ \sum_{c\in\mathcal{C},\mathcal{T}_c\subset\bar I_{i,j}}
\sum_{\substack{p\in\mathcal{P},\mathcal{T}_p\subset\bar I_{i,j} \\ c\sim p}}
\mathds{1}_{n=p}\mathds{1}_{\supp\eta_p}\eta_c(\xi^c_{i,j} {-} \xi^p_{i,j}).
\end{align*}
We are now in a position to apply Proposition~\ref{prop:localCompatibilityEstimates}.
More precisely, thanks to the localization property~\eqref{LocalizationTripleJunctionTwoPhase}
and the definition~\eqref{def:minLocScale} we have the estimate~\eqref{eq:localComp1} 
at our disposal, implying that
\begin{align*}
\sum_{\substack{n'=1,n'\neq n \\ \mathcal{T}_{n'}\subset\bar I_{i,j}}}^N 
\mathds{1}_{\supp\eta_n}\eta_{n'}(\xi^{n'}_{i,j} {-} \xi^n_{i,j})
= O(\bar r^{-1}d_{i,j}\wedge 1),
\end{align*}
at least under our assumption of~$\mathcal{T}_n\subset\bar I_{i,j}$.
This concludes the argument for~\eqref{CompatibilityBound1}.

\textit{Step 2: Proof of~\eqref{CompatibilityBound2}.}
Multiplying~\eqref{eq:auxReduction1} by~$\xi^n_{i,j}$
and afterwards running through the same argument
as in \textit{Step~1} entails
\begin{align*}
&\sum_{\substack{n'=1,n'\neq n \\ \mathcal{T}_{n'}\subset\bar I_{i,j}}}^N 
\mathds{1}_{\supp\eta_n}\eta_{n'}(\xi^{n'}_{i,j} {-} \xi^n_{i,j})\cdot\xi^n_{i,j}
\\&
= \sum_{p\in\mathcal{P},\mathcal{T}_p\subset\bar I_{i,j}}
\sum_{\substack{c\in\mathcal{C},\mathcal{T}_c\subset\bar I_{i,j} \\ c\sim p}}
\mathds{1}_{n=c}\mathds{1}_{\supp\eta_c}\eta_p(\xi^p_{i,j} {-} \xi^c_{i,j})\cdot\xi^c_{i,j}
\\&~~~
+ \sum_{c\in\mathcal{C},\mathcal{T}_c\subset\bar I_{i,j}}
\sum_{\substack{p\in\mathcal{P},\mathcal{T}_p\subset\bar I_{i,j} \\ c\sim p}}
\mathds{1}_{n=p}\mathds{1}_{\supp\eta_p}\eta_c(\xi^c_{i,j} {-} \xi^p_{i,j})\cdot\xi^p_{i,j}
+ O(\bar r^{-2}d^2_{i,j}\wedge 1).
\end{align*}
Adding zero in the second right hand side term of the previous display 
in form of $(\xi^c_{i,j} {-} \xi^p_{i,j})\cdot\xi^p_{i,j} = - |\xi^c_{i,j} {-} \xi^p_{i,j}|^2
+ (\xi^c_{i,j} {-} \xi^p_{i,j})\cdot\xi^c_{i,j}$, and then applying the local 
compatibility estimates~\eqref{eq:localComp2} and~\eqref{eq:localComp1},
we deduce~\eqref{CompatibilityBound2}.

\textit{Step 3: Proof of~\eqref{CompatibilityBound3}.}
Using the definition~\eqref{GlobalVelocityField}, the regularity 
estimates~\eqref{eq:regularityEstimatesLocalVelocities}
and the local compatibility estimate~\eqref{eq:localComp3}
instead of~\eqref{GlobalXi}, \eqref{eq:regularityEstimatesLocalXi} 
and~\eqref{eq:localComp1}, respectively, and substituting~$(B,B^n)$
for~$(\xi_{i,j},\xi^n_{i,j})$ in the argument of \textit{Step~1}
directly implies~\eqref{CompatibilityBound3}.

\textit{Step 4: Proof of~\eqref{CompatibilityBound4}.}
We give some details here, as in comparison to \textit{Step~1} or \textit{Step~3} the argument
in favor of~\eqref{CompatibilityBound4} involves an additional (though simple) reduction step. 
Starting with the definition~\eqref{GlobalVelocityField},
the estimates~\eqref{eq:upperBoundBulkCutOff}, \eqref{GlobalEquationsBoundMinorityPhases} 
and~\eqref{eq:regularityEstimatesLocalVelocities}, and in addition the product rule we obtain
\begin{align*}
&\mathds{1}_{\supp\eta_n} (\nabla B {-} \nabla B^n)
\\&
= -\mathds{1}_{\supp\eta_n} \eta_{\mathrm{bulk}} \nabla B^n
+ \sum_{n'=1,n'\neq n}^N \mathds{1}_{\supp\eta_n}\eta_{n'}(\nabla B^{n'} {-} \nabla B^n)
+ \sum_{n'=1}^N\mathds{1}_{\supp\eta_n} B^{n'}\otimes\nabla\eta_{n'}
\\&
= \sum_{\substack{n'=1,n'\neq n \\ \mathcal{T}_{n'}\subset\bar I_{i,j}}}^N 
\mathds{1}_{\supp\eta_n}\eta_{n'}(\nabla B^{n'} {-} \nabla B^n)
+ \sum_{n'=1}^N\mathds{1}_{\supp\eta_n} B^{n'}\otimes\nabla\eta_{n'}
\\&~~~
+ O\big(\bar r^{-2}(\bar r^{-2}d_{i,j}^2\wedge 1)\big).
\end{align*}
The first right hand side term is estimated to desired 
order~$O\big(\bar r^{-2}(\bar r^{-1}d_{i,j}\wedge 1)\big)$
based on the local compatibility estimate~\eqref{eq:localComp4}
and the above familiar reduction arguments. Adding zero in the second right
hand side term moreover entails
\begin{align*}
&\sum_{n'=1}^N\mathds{1}_{\supp\eta_n} B^{n'}\otimes\nabla\eta_{n'}
\\&
= \sum_{n'=1,n'\neq n}^N\mathds{1}_{\supp\eta_n} (B^{n'}{-}B^n)\otimes\nabla\eta_{n'}
-\mathds{1}_{\supp\eta_n} B^n\otimes\nabla\eta_{\mathrm{bulk}}.
\end{align*}
The previous reduction arguments in combination with
the local compatibility estimate~\eqref{eq:localComp3}, 
the upper bound~\eqref{eq:boundDerivBulkCutoff} for the gradient 
of the bulk cutoff, as well as the
regularity estimates~\eqref{eq:estimatesDerivEta} 
and~\eqref{eq:regularityEstimatesLocalVelocities}
thus show that $\sum_{n'=1}^N\mathds{1}_{\supp\eta_n} B^{n'}\otimes\nabla\eta_{n'}$
is of order $O\big(\bar r^{-2}(\bar r^{-1}d_{i,j} \wedge 1)\big)$. This concludes the proof.
\end{proof}

\subsection{Approximate transport and mean curvature flow equations}
\label{SectionBoundsTimeEvolutionGlobalConstruction}
We derive the global (or network) version of our previous bounds
from Lemma~\ref{LemmaBoundsLocalConstructionsTwoPhase} and  Proposition~\ref{prop:xi_triple_junction}, 
which are valid for the model problem of a smooth manifold and a triple junction, respectively.

\begin{lemma}
\label{LemmaGlobalCoercivityBounds}
Let $d=2$ and $P \in \mathbb{N}$, $P\geq 2$. Let $\bar\Omega=(\bar{\Omega}_1,\ldots,\bar{\Omega}_P)$ 
be a strong solution to multiphase mean curvature flow in the sense of Definition~\ref{DefinitionStrongSolution}. 
Let next~$\bar r_{\mathrm{min}}\in (0,1]$ be the localization scale defined by~\eqref{def:minLocScale}
and~$\mathcal{T}_{\mathcal{P}}:=\bigcup_{p\in\mathcal{P}}\mathcal{T}_p$.
Let $(\xi^n_{i,j})_{n\in\{1,\ldots,N\}}$ be the local vector fields 
from Lemma~\ref{LemmaLocalConstructionsNetwork} as well as $(B^n)_{n\in\{1,\ldots,N\}}$ be the local velocity fields 
from Construction~\ref{DefinitionGlobalB}. Let $\xi_{i,j}$ be the global vector 
fields from Construction~\ref{DefinitionGlobalXi}, and let $B$ be the global velocity field from 
Construction~\ref{DefinitionGlobalB}.

Then there exists a constant~$C>0$, depending only on the strong solution~$\bar\Omega$
but not on~$\bar r_{\mathrm{min}}$, so that we have the estimates
\begin{align}
\label{TransportEquationXiGlobal}
|\partial_t\xi_{i,j} + ( B\cdot\nabla)\xi_{i,j} + (\nabla B)^\mathsf{T}\xi_{i,j}| &\leq 
C\bar r_{\mathrm{min}}^{-2}\big(\bar r_{\mathrm{min}}^{-1}\dist(\cdot,{\bar{I}}_{i,j})\wedge 1\big), 
\\
\label{DissipGlobal}
|(\nabla\cdot\xi_{i,j}) +  B\cdot\xi_{i,j}| &\leq 
C\bar r_{\mathrm{min}}^{-1}\big(\bar r_{\mathrm{min}}^{-1}\dist(\cdot,{\bar{I}}_{i,j})\wedge 1\big),
\\
\label{LengthConservationGlobal}
\big|\xi_{i,j}\cdot\partial_t\xi_{i,j} + \xi_{i,j}\cdot(B\cdot\nabla)\xi_{i,j}\big| &\leq 
C\bar r_{\mathrm{min}}^{-2}\big(\bar r_{\mathrm{min}}^{-2}\dist^2(\cdot,{\bar{I}}_{i,j})\wedge 1\big) 
\end{align}
in $\Rd[2]{\times} [0,T]\setminus \mathcal{T}_{\mathcal{P}}$, for all $i,j\in\{1,\ldots,P\}$ with $i\neq j$.

Furthermore, the additional estimates mentioned in Remark~\ref{RemarkFurtherCalibrationProperties}
\begin{align}
\label{EstimateBnn}
|\nabla B : \xi_{i,j} \otimes \xi_{i,j}|(x,t)
&\leq C 
\dist(x,{\bar{I}}_{i,j}(t))
\qquad\text{for all }i\neq j,
\\
\label{EstimateBSkew}
\big|\nabla B : \big(\xi_{i,j} \otimes J\xi_{i,j}+J\xi_{i,j} \otimes \xi_{i,j}\big)\big|(x,t)
&\leq C 
\dist(x,{\bar{I}}_{i,j}(t))
\qquad\text{for all }i\neq j,
\end{align}
may be enforced, where the matrix $J$ denotes the counter-clockwise
rotation by $90^\circ$.
\end{lemma}
\begin{proof}
Let~$i,j\in\{1,\ldots,P\}$ such that~$i\neq j$.
For notational convenience, we again abbreviate for the purpose of the proof~$\bar r:= \bar r_{\mathrm{min}}$
and~$d_{i,j}:=\dist(\cdot,\bar I_{i,j})$. Recall that the distinct phases~$i$ and~$j$
are both present at a given topological feature~$\mathcal{T}_n$, $n\in\{1,\ldots,N\}$,
if and only if~$\mathcal{T}_n\subset\bar I_{i,j}$.

\textit{Step 1: Proof of~\eqref{TransportEquationXiGlobal}.}
By the product rule, the definition~\eqref{GlobalXi}, the regularity
estimates~\eqref{eq:regularityEstimatesLocalXi} and~\eqref{eq:regularityEstimatesLocalVelocities},
as well as the error 
estimates~\eqref{GlobalEquationsBoundMinorityPhases}--\eqref{GlobalEquationsBoundTimeDerivMinorityPhases}
we compute
\begin{align*}
\partial_t\xi_{i,j} {+} (B\cdot\nabla)\xi_{i,j}
&= \sum_{n=1,\mathcal{T}_n\subset\bar I_{i,j}}^N
\eta_n(\partial_t{+}B\cdot\nabla)\xi^n_{i,j}
+ \sum_{n=1,\mathcal{T}_n\subset\bar I_{i,j}}^N
\xi^n_{i,j}(\partial_t{+}B\cdot\nabla)\eta_n
\\&~~~
+ O\big(\bar r^{-2}(\bar r^{-1}d_{i,j}\wedge 1)\big).
\end{align*}
Next, it follows from adding zero, the compatibility estimate~\eqref{CompatibilityBound1},
the regularity bound~\eqref{eq:estimatesDerivEta},
and again~\eqref{eq:regularityEstimatesLocalVelocities}, \eqref{GlobalEquationsBoundGradientMinorityPhases}
and~\eqref{GlobalEquationsBoundTimeDerivMinorityPhases} that
\begin{align*}
\sum_{n=1,\mathcal{T}_n\subset\bar I_{i,j}}^N
\xi^n_{i,j}(\partial_t{+}B\cdot\nabla)\eta_n
&= \sum_{n=1,\mathcal{T}_n\subset\bar I_{i,j}}^N
\xi_{i,j} (\partial_t{+}B\cdot\nabla)\eta_n
+ O\big(\bar r^{-2}(\bar r^{-1}d_{i,j}\wedge 1)\big)
\\&
= -\xi_{i,j}(\partial_t{+}B\cdot\nabla)\eta_{\mathrm{bulk}}
+ O\big(\bar r^{-2}(\bar r^{-1}d_{i,j}\wedge 1)\big).
\end{align*}
Thanks to the compatibility estimate~\eqref{CompatibilityBound3}
and the regularity estimate~\eqref{eq:regularityEstimatesLocalXi},
we also have
\begin{align*}
\sum_{n=1,\mathcal{T}_n\subset\bar I_{i,j}}^N
\eta_n(B\cdot\nabla)\xi^n_{i,j}
= \sum_{n=1,\mathcal{T}_n\subset\bar I_{i,j}}^N
\eta_n(B^n\cdot\nabla)\xi^n_{i,j} 
+ O\big(\bar r^{-2}(\bar r^{-1}d_{i,j}\wedge 1)\big).
\end{align*}
Together with the upper bounds~\eqref{eq:boundDerivBulkCutoff} resp.\ \eqref{eq:boundTimeDerivBulkCutoff} 
for the bulk cutoff and the regularity estimate~\eqref{eq:regularityEstimatesLocalVelocities},
the previous three displays combine to
\begin{align}
\label{eq:auxTransportXi1}
\partial_t\xi_{i,j} {+} (B\cdot\nabla)\xi_{i,j}
&= \sum_{n=1,\mathcal{T}_n\subset\bar I_{i,j}}^N
\eta_n(\partial_t{+}B^n\cdot\nabla)\xi^n_{i,j}
+ O\big(\bar r^{-2}(\bar r^{-1}d_{i,j}\wedge 1)\big).
\end{align}

In a next step, we compute based on the product rule,
the definitions~\eqref{GlobalXi} and~\eqref{GlobalVelocityField},
the error estimate~\eqref{GlobalEquationsBoundMinorityPhases},
the regularity estimates~\eqref{eq:regularityEstimatesLocalVelocities}
and~\eqref{eq:estimatesDerivEta},
as well as the compatibility estimate~\eqref{CompatibilityBound4}
\begin{align}
\nonumber
(\nabla B)^\mathsf{T}\xi_{i,j} 
&= \sum_{n=1,\mathcal{T}_n\subset\bar I_{i,j}}^N
\eta_n (\nabla B)^\mathsf{T}\xi^n_{i,j}
+ O\big(\bar r^{-2}(\bar r^{-1}d_{i,j}\wedge 1)\big)
\\& \label{eq:auxTransportXi2}
= \sum_{n=1,\mathcal{T}_n\subset\bar I_{i,j}}^N
\eta_n (\nabla B^n)^\mathsf{T}\xi^n_{i,j}
+ O\big(\bar r^{-2}(\bar r^{-1}d_{i,j}\wedge 1)\big).
\end{align}
Hence, in view of~\eqref{eq:auxTransportXi1} and~\eqref{eq:auxTransportXi2}
we reduced the task to the local evolution equations at topological
features for which both phases~$i$ and~$j$ are present:
\begin{align*}
\partial_t\xi_{i,j} {+} (B\cdot\nabla)\xi_{i,j}
{+} (\nabla B)^\mathsf{T}\xi_{i,j}
&= \sum_{n=1,\mathcal{T}_n\subset\bar I_{i,j}}^N
\eta_n \big(\partial_t\xi_{i,j}^n {+} (B^n\cdot\nabla)\xi^n_{i,j}
{+}(\nabla B^n)^\mathsf{T}\xi^n_{i,j}\big)
\\&~~~
+ O\big(\bar r^{-2}(\bar r^{-1}d_{i,j}\wedge 1)\big).
\end{align*}
To conclude that~\eqref{TransportEquationXiGlobal} holds, it thus only
remains to observe that the bounds on the local evolution equations~\eqref{TransportEquationXiTwoPhase}
and~\eqref{TransportEquationXiTriod},
respectively, are applicable due to the localization
properties~\eqref{LocalizationTwoPhase}--\eqref{LocalizationTripleJunction} 
and the definitions~\eqref{def:minLocScaleTripleJunction}--\eqref{def:minLocScale}.

\textit{Step 2: Proof of~\eqref{DissipGlobal}.}
We proceed in the same style as for the proof of~\eqref{TransportEquationXiGlobal}.
On one side, it is immediate from the definitions~\eqref{GlobalXi} and~\eqref{GlobalVelocityField},
the error estimate~\eqref{GlobalEquationsBoundMinorityPhases},
the regularity estimates~\eqref{eq:regularityEstimatesLocalXi}
and~\eqref{eq:regularityEstimatesLocalVelocities},
as well as the compatibility estimate~\eqref{CompatibilityBound3}
\begin{align*}
B\cdot\xi_{i,j} &= \sum_{n=1,\mathcal{T}_n\subset\bar I_{i,j}}^N
\eta_n B\cdot\xi^n_{i,j} + O\big(\bar r^{-1}(\bar r^{-1}d_{i,j}\wedge 1)\big)
\\&
= \sum_{n=1,\mathcal{T}_n\subset\bar I_{i,j}}^N
\eta_n B^n\cdot\xi^n_{i,j} + O\big(\bar r^{-1}(\bar r^{-1}d_{i,j}\wedge 1)\big).
\end{align*}
On the other side, we have by the definition~\eqref{GlobalXi},
the product rule, the error 
estimates~\eqref{GlobalEquationsBoundMinorityPhases}--\eqref{GlobalEquationsBoundGradientMinorityPhases},
the regularity estimates~\eqref{eq:regularityEstimatesLocalXi} 
and~\eqref{eq:estimatesDerivEta}, the compatibility estimate~\eqref{CompatibilityBound1},
and finally the upper bound~\eqref{eq:boundDerivBulkCutoff} for the bulk cutoff
\begin{align*}
\nabla\cdot\xi_{i,j} &= 
\sum_{n=1,\mathcal{T}_n\subset\bar I_{i,j}}^N \eta_n(\nabla\cdot\xi^n_{i,j}) 
+ \sum_{n=1,\mathcal{T}_n\subset\bar I_{i,j}}^N (\xi^n_{i,j}\cdot\nabla)\eta_n
+ O\big(\bar r^{-1}(\bar r^{-1}d_{i,j}\wedge 1)\big)
\\&
= \sum_{n=1,\mathcal{T}_n\subset\bar I_{i,j}}^N \eta_n(\nabla\cdot\xi^n_{i,j}) 
+ \sum_{n=1,\mathcal{T}_n\subset\bar I_{i,j}}^N (\xi_{i,j}\cdot\nabla)\eta_n
+ O\big(\bar r^{-1}(\bar r^{-1}d_{i,j}\wedge 1)\big)
\\&
= \sum_{n=1,\mathcal{T}_n\subset\bar I_{i,j}}^N \eta_n(\nabla\cdot\xi^n_{i,j}) 
-(\xi_{i,j}\cdot\nabla)\eta_{\mathrm{bulk}}
+ O\big(\bar r^{-1}(\bar r^{-1}d_{i,j}\wedge 1)\big)
\\&
= \sum_{n=1,\mathcal{T}_n\subset\bar I_{i,j}}^N \eta_n(\nabla\cdot\xi^n_{i,j}) 
+ O\big(\bar r^{-1}(\bar r^{-1}d_{i,j}\wedge 1)\big).
\end{align*}
The previous two displays in total imply
\begin{align*}
\nabla\cdot\xi_{i,j} + B\cdot\xi_{i,j}
&= \sum_{n=1,\mathcal{T}_n\subset\bar I_{i,j}}^N 
\eta_n \big(\nabla\cdot\xi^n_{i,j} {+} B^n\cdot\xi^n_{i,j}\big) 
+ O\big(\bar r^{-1}(\bar r^{-1}d_{i,j}\wedge 1)\big),
\end{align*}
so that~\eqref{DissipGlobal} follows due to
its local counterparts~\eqref{DissipTwoPhase} 
and~\eqref{DissipTriod}, respectively.

\textit{Step 3: Proof of~\eqref{LengthConservationGlobal}.}
We first claim that
\begin{equation}
\begin{aligned}
\label{eq:auxTransportXi3}
&\xi_{i,j}\cdot(\partial_t{+}B\cdot\nabla)\xi_{i,j}
\\
&= \sum_{\substack{n,n'=1 \\ \mathcal{T}_n,\mathcal{T}_{n'}\subset\bar I_{i,j}}}^N
\eta_n\eta_{n'}\xi^n_{i,j}\cdot(\partial_t{+}B^{n'}\cdot\nabla)\xi^{n'}_{i,j}
+ O\big(\bar r^{-2}(\bar r^{-2}d_{i,j}^2 \wedge 1)\big).
\end{aligned}
\end{equation}
For a proof of~\eqref{eq:auxTransportXi3} one may argue as follows.
First, plugging in the definition~\eqref{GlobalXi}, applying the product rule, and making use
of the error estimate~\eqref{GlobalEquationsBoundMinorityPhases}
as well as the regularity estimates~\eqref{eq:regularityEstimatesLocalXi}
and~\eqref{eq:estimatesDerivEta} entails
\begin{align*}
\xi_{i,j}\cdot\partial_t\xi_{i,j}
&=
\sum_{n=1,\mathcal{T}_n\subset\bar I_{i,j}}^N \eta_n\xi^n_{i,j}\cdot\partial_t\xi_{i,j}
+ O\big(\bar r^{-2}(\bar r^{-2}d_{i,j}^2 \wedge 1)\big)
\\&
= \sum_{\substack{n,n'=1 \\ \mathcal{T}_n,\mathcal{T}_{n'}\subset\bar I_{i,j}}}^N
\eta_n\eta_{n'}\xi^n_{i,j}\cdot\partial_t\xi^{n'}_{i,j}
+ \sum_{n=1,\mathcal{T}_n\subset\bar I_{i,j}}^N\sum_{n'=1}^N
\eta_{n}(\xi^n_{i,j}\cdot\xi^{n'}_{i,j})\partial_t\eta_{n'}
\\&~~~
+ O\big(\bar r^{-2}(\bar r^{-2}d_{i,j}^2 \wedge 1)\big).
\end{align*}
Substituting the differential operator~$(B\cdot\nabla)$ for~$\partial_t$,
and recalling in addition to the above ingredients the regularity
estimate~\eqref{eq:regularityEstimatesLocalVelocities} 
as well as the compatibility estimate~\eqref{CompatibilityBound3}
(which allows to switch from~$B$ to~$B^{n'}$)
then also yields
\begin{align*}
\xi_{i,j}\cdot(B\cdot\nabla)\xi_{i,j}
&=
\sum_{\substack{n,n'=1 \\ \mathcal{T}_n,\mathcal{T}_{n'}\subset\bar I_{i,j}}}^N
\eta_n\eta_{n'}\xi^n_{i,j}\cdot(B^{n'}\cdot\nabla)\xi^{n'}_{i,j}
\\&~~~
+ \sum_{n=1,\mathcal{T}_n\subset\bar I_{i,j}}^N\sum_{n'=1}^N
\eta_{n}(\xi^n_{i,j}\cdot\xi^{n'}_{i,j})(B\cdot\nabla)\eta_{n'}
+ O\big(\bar r^{-2}(\bar r^{-2}d_{i,j}^2 \wedge 1)\big).
\end{align*}
Observe that the combination of the previous two displays
already generates the first right hand side term of~\eqref{eq:auxTransportXi3}.

We proceed by first splitting the sum over topological features~$n'\in\{1,\ldots,N\}$,
adding zero several times in the resulting first term,
then applying the compatibility estimates~\eqref{CompatibilityBound1}, \eqref{CompatibilityBound2}
and~\eqref{eq:localComp1}, and finally recalling the regularity estimate~\eqref{eq:estimatesDerivEta}
which results in the estimate (of course, only terms with~$\supp\eta_n\cap\supp\eta_{n'}\neq\emptyset$
are relevant in the subsequent sums)
\begin{align*}
&\sum_{n=1,\mathcal{T}_n\subset\bar I_{i,j}}^N\sum_{n'=1}^N
\eta_{n}(\xi^n_{i,j}\cdot\xi^{n'}_{i,j})\partial_t\eta_{n'}
\\&
=\sum_{\substack{n,n'=1 \\ \mathcal{T}_n,\mathcal{T}_{n'}\subset\bar I_{i,j}}}^N
\eta_{n}(\xi^n_{i,j}\cdot\xi^{n'}_{i,j})\partial_t\eta_{n'}
+ \sum_{\substack{n,n'=1 \\ \mathcal{T}_n\subset\bar I_{i,j},
\mathcal{T}_{n'}\not\subset\bar I_{i,j}}}^N
\eta_{n}(\xi^n_{i,j}\cdot\xi^{n'}_{i,j})\partial_t\eta_{n'}
\\&
=\sum_{\substack{n,n'=1 \\ \mathcal{T}_n,\mathcal{T}_{n'}\subset\bar I_{i,j}}}^N
\eta_{n}\big(|\xi_{i,j}|^2{-}|\xi^{n}_{i,j}{-}\xi_{i,j}|^2
+(\xi^n_{i,j}{-}\xi_{i,j})\cdot\xi^n_{i,j}
+\xi^{n'}_{i,j}\cdot (\xi^{n'}_{i,j}{-}\xi_{i,j})\big)\partial_t\eta_{n'}
\\&~~~
+ \sum_{\substack{n,n'=1 \\ \mathcal{T}_n,\mathcal{T}_{n'}\subset\bar I_{i,j}}}^N
\eta_{n}(\xi^n_{i,j}{-}\xi^{n'}_{i,j})\cdot(\xi^{n'}_{i,j}{-}\xi_{i,j})\partial_t\eta_{n'}
+ \sum_{\substack{n,n'=1 \\ \mathcal{T}_n\subset\bar I_{i,j},
\mathcal{T}_{n'}\not\subset\bar I_{i,j}}}^N
\eta_{n}(\xi^n_{i,j}\cdot\xi^{n'}_{i,j})\partial_t\eta_{n'}
\\&
= \sum_{\substack{n,n'=1 \\ \mathcal{T}_n,\mathcal{T}_{n'}\subset\bar I_{i,j}}}^N
\eta_n|\xi_{i,j}|^2\partial_t\eta_{n'}
+ \sum_{\substack{n,n'=1 \\ \mathcal{T}_n\subset\bar I_{i,j},
\mathcal{T}_{n'}\not\subset\bar I_{i,j}}}^N
\eta_{n}(\xi^n_{i,j}\cdot\xi^{n'}_{i,j})\partial_t\eta_{n'}
+ O\big(\bar r^{-2}(\bar r^{-2}d_{i,j}^2 \wedge 1)\big).
\end{align*}
Based on the regularity estimates~\eqref{eq:estimatesDerivEta} 
and~\eqref{eq:regularityEstimatesLocalVelocities}, we may again
substitute the differential operator~$(B\cdot\nabla)$ for~$\partial_t$
in the previous computation, which in turn by two applications of the crucial 
estimate~\eqref{eq:advectionWrongCutoffs} and finally
an application of the bulk cutoff estimates~\eqref{AdvectionEquationBulkLocalization}
resp.\ \eqref{eq:upperBoundBulkCutOff} allows to deduce
\begin{align*}
&\sum_{n=1,\mathcal{T}_n\subset\bar I_{i,j}}^N\sum_{n'=1}^N
\eta_{n}(\xi^n_{i,j}\cdot\xi^{n'}_{i,j})(\partial_t {+} B\cdot\nabla)\eta_{n'}
\\&
= \sum_{n,\mathcal{T}_n\subset\bar I_{i,j}}^N\sum_{n'=1}^N
\eta_n|\xi_{i,j}|^2(\partial_t {+} B\cdot\nabla)\eta_{n'}
\\&~~~
+ \sum_{\substack{n,n'=1 \\ \mathcal{T}_n\subset\bar I_{i,j},
\mathcal{T}_{n'}\not\subset\bar I_{i,j}}}^N
\eta_{n}(\xi^n_{i,j}\cdot\xi^{n'}_{i,j})(\partial_t{+}B\cdot\nabla)\eta_{n'}
+ O\big(\bar r^{-2}(\bar r^{-2}d_{i,j}^2 \wedge 1)\big)
\\&
= -(1{-}\eta_{\mathrm{bulk}})|\xi_{i,j}|^2(\partial_t{+}B\cdot\nabla)\eta_{\mathrm{bulk}}
+ O\big(\bar r^{-2}(\bar r^{-2}d_{i,j}^2 \wedge 1)\big)
=  O\big(\bar r^{-2}(\bar r^{-2}d_{i,j}^2 \wedge 1)\big).
\end{align*} 
In particular, we obtain the asserted estimate~\eqref{eq:auxTransportXi3}.

It remains to post-process the right hand side term of~\eqref{eq:auxTransportXi3}.
In view of~\eqref{LengthConservationTwoPhase} and~\eqref{LengthConservationXiTriodNormalized},
it suffices to get rid of the ``off-diagonal'' terms~$n\neq n'\in\{1,\ldots,N\}$
with~$\mathcal{T}_n\subset\bar I_{i,j}$, $\mathcal{T}_{n'}\subset\bar I_{i,j}$
and~$\supp\eta_{n}\cap\supp\eta_{n'}\neq\emptyset$.
For each such pair of topological features we may add zero several times to rewrite
(recall again
the local identities~\eqref{LengthConservationTwoPhase} and~\eqref{LengthConservationXiTriodNormalized})
\begin{align*}
&\xi^n_{i,j}\cdot(\partial_t{+}B^{n'}\cdot\nabla)\xi^{n'}
\\&
= \xi^n_{i,j}\cdot\big(\partial_t\xi^{n'}_{i,j}
{+}(B^{n'}\cdot\nabla)\xi^{n'}_{i,j}{+}(\nabla B^{n'})^{\mathsf{T}}\xi^{n'}_{i,j}\big)
- \xi^n_{i,j}(\nabla B^{n'})^{\mathsf{T}}\xi^{n'}_{i,j}
\\&
= (\xi^n_{i,j}{-}\xi^{n'}_{i,j})\cdot\big(\partial_t\xi^{n'}_{i,j}
{+}(B^{n'}\cdot\nabla)\xi^{n'}_{i,j}{+}(\nabla B^{n'})^{\mathsf{T}}\xi^{n'}_{i,j}\big)
+(\xi^{n'}_{i,j} {-} \xi^n_{i,j})(\nabla B)^{\mathsf{T}}\xi^{n'}_{i,j}
\\&~~~
+ (\xi^{n'}_{i,j} {-} \xi^n_{i,j})(\nabla B^{n'}{-}\nabla B)^{\mathsf{T}}\xi^{n'}_{i,j}.
\end{align*}
Hence, summing the previous identity over the relevant topological features,
then matching terms which correspond to the previous computation but with
the roles of~$n$ and~$n'$ being reversed, and finally using the 
compatibility estimates~\eqref{CompatibilityBound4} resp.\ \eqref{eq:localComp1}
as well as the local evolution equations~\eqref{TransportEquationXiTwoPhase}
and~\eqref{TransportEquationXiTriod} we infer that
\begin{align*}
\sum_{\substack{n,n'=1 \\ \mathcal{T}_n,\mathcal{T}_{n'}\subset\bar I_{i,j}}}^N
\eta_n\eta_{n'}\xi^n_{i,j}\cdot(\partial_t{+}B^{n'}\cdot\nabla)\xi^{n'}_{i,j}
= O\big(\bar r^{-2}(\bar r^{-2}d_{i,j}^2 \wedge 1)\big).
\end{align*}
This in turn constitutes the required upgrade of~\eqref{eq:auxTransportXi3}.

Similarly, the bounds \eqref{EstimateBnn} and \eqref{EstimateBSkew} are a consequence of the corresponding bounds from Lemma~\ref{LemmaBoundsLocalConstructionsTwoPhase}, Proposition~\ref{prop:xi_triple_junction} and the compatibility estimates.
\end{proof}

\subsection{Existence of gradient flow calibrations: Proof of Theorem~\ref{TheoremExistenceCalibration}}
Let us summarize our results from the previous sections to conclude with
a proof of the main result.

\begin{proof}[Proof of Theorem~\ref{TheoremExistenceCalibration}]
Let $(\xi_{i,j})_{i\neq j}$ be the family of global vector fields 
from Construction~\ref{DefinitionGlobalXi}.
Let $i,j \in \{1,\ldots,P\}$ with $i\neq j$. The coercivity condition 
\eqref{LengthControlXi} immediately follows from Lemma~\ref{lem:coercivityLenghtXi}.
The formula \eqref{Calibrations} follows from the corresponding local version \eqref{CalibrationLocalConstruction}
and the definition \eqref{GlobalXi}. Moreover, that $\xi_{i,j}(x,t)=\bar{\vec{n}}_{i,j}(x,t)$ holds true 
for all $t\in [0,T]$ and $x\in {\bar{I}}_{i,j}(t)$
is a consequence of Lemma~\ref{LemmaLocalConstructionsNetwork}~\textit{iii)}  
and that $(\eta_1,\ldots,\eta_N)$ is a partition of unity on the network of interfaces of the strong solution 
(see Lemma~\ref{LemmaPartitionOfUnity}~\textit{i)}).

Finally, let $B$ be the global velocity field from Construction~\ref{DefinitionGlobalB}.
The validity of the equations \eqref{TransportEquationXi}, \eqref{LengthConservation}
and \eqref{Dissip} is then the content of Lemma~\ref{LemmaGlobalCoercivityBounds}.
\end{proof}

\section{Existence of transported weights: Proof of Lemma~\ref{LemmaWeightedVolumeControl}}
\label{SectionVolumeError}
The aim of this section is to establish the existence of a family of transported weights
in the case of $d=2$ and an underlying strong solution of multiphase
mean curvature flow.

\begin{proof}[Proof of Lemma~\ref{LemmaWeightedVolumeControl}]
We again make use of the description of the network of interfaces
of the strong solution in terms of its underlying topological features,
namely two-phase interfaces and triple junctions. Assume that there
is a total of~$N\in\mathbb{N}$ such topological features present.
Recall then that we decomposed 
$\{1,\ldots,N\}=:\mathcal{C} \cupdot\mathcal{P}$ with the convention that $\mathcal{C}$ 
enumerates the connected components in space-time of the smooth two-phase interfaces 
and $\mathcal{P}$ enumerates the triple junctions.
If $p\in \mathcal{P}$, we defined $\mathcal{T}_p$ to be the trajectory in space-time described by the triple junction. 
If $c\in \mathcal{C}$, we defined $\mathcal{T}_c \subset {\bar{I}}_{i,j}$ 
for some $i,j \in \{1,\ldots,P\}$ with $i\neq j$ to be the corresponding space-time 
connected component of a two-phase interface ${\bar{I}}_{i,j}$.
We further write~$c\sim p$ for~$c\in\mathcal{C}$ and~$p\in\mathcal{P}$
if and only if~$\mathcal{T}_c$ has an endpoint at~$\mathcal{T}_p$. 

Let now~$r_{\mathcal{P}}$ and~$\bar r_{\mathrm{min}}$ be the localization
scales from~\eqref{def:minLocScaleTripleJunction} and~\eqref{def:minLocScale}.
We then choose a large-scale cutoff~$R>0$ such that for all~$t\in [0,T]$ a suitable
neighborhood of the network of interfaces at time~$t$ is compactly
supported in the ball $B_R(0)$:
\begin{align}
\label{eq:cutoffScale}
\bigcup_{p\in\mathcal{P}} B_{r_{\mathcal{P}}}(\mathcal{T}_p(t)) \cup
\Big(\bigcup_{c\in\mathcal{C}}\mathrm{im}_{\bar r_{\mathrm{min}}}(\Psi_{\mathcal{T}_c})(t)
\setminus\bigcup_{p\in\mathcal{P}} B_{r_{\mathcal{P}}}(\mathcal{T}_p(t))\Big)
\subset\subset B_R(0),
\end{align}
where we abbreviated~$\mathrm{im}_{\bar r_{\mathrm{min}}}(\Psi_{\mathcal{T}_c})(t)
:=\Psi_{\mathcal{T}_c}(\mathcal{T}_c(t)
{\times}\{t\}{\times}[-\bar r_{\mathrm{min}},\bar r_{\mathrm{min}}])$
for~$t\in [0,T]$, and where~$\Psi_{\mathcal{T}_c}$ refers to the restriction
of the diffeomorphism~\eqref{DiffeoTubularNeighborhood} to~$\mathcal{T}_c$
(assuming~$\mathcal{T}_c\subset\bar I_{i,j}$).

The idea for the proof is to construct in the first part a family of
weight functions $(\hat\vartheta_i)_{i\in\{1,\ldots,P\}}$ which satisfies 
all the requirements of Definition~\ref{def:transportedWeights}  
but violates the integrability condition $\hat\vartheta_i\in L^1_{x,t}(\Rd[2]\times [0,T])$.
To overcome the integrability issue at the end of the proof, we introduce
a smooth and concave function $\kappa\colon [0,\infty)\to [0,1]$
such that $\kappa(r)=1$ for $r\geq 1$, $\kappa'(r)\in (0,2)$ for $r\in (0,1)$
and $\kappa(0)=0$. Note that $\kappa$ represents an upper concave approximation
of $r\mapsto r\wedge 1$ on the interval $[0,\infty)$. We next define an integrable weight
$\eta_R\in W^{1,\infty}_x(\Rd[2])\cap W^{1,1}_x(\Rd[2])$
by means of
\begin{align}
\label{eq:intWeight}
\eta_R(x) := \kappa(\exp(R-|x|)),\quad x\in\Rd[2],
\end{align}
whose spatial gradient is now subject to the following convenient estimate
\begin{align}
\label{eq:gradientBoundIntWeight}
|\nabla\eta_R| \leq C|\eta_R| \quad\text{in } \Rd[2].
\end{align}
We will then define $\vartheta_i:=\eta_R\hat\vartheta_i$,
and verify in a second part that all the requirements of Definition~\ref{def:transportedWeights} 
are indeed satisfied for this choice of weight functions. 

\textit{Step 1: Construction of $(\hat\vartheta_i)_{i\in\{1,\ldots,P\}}$.}
Let $\vartheta\colon\Rd[]\to\Rd[]$ be a 
truncation of the identity with $\vartheta(r)=r$ for $|r|\leq\frac{1}{2}$, 
$\vartheta(r)= -1$ for $r\leq -1$, $\vartheta(r)=1$ for $r\geq 1$,  
$0\leq\vartheta'\leq 2$ as well as $|\vartheta''|\leq C$.
Fix $i\in\{1,\ldots,P\}$.
For purely technical reasons (similar to the one described in \textit{Step~3},
Proof of Lemma~\ref{LemmaPartitionOfUnity}), we need to introduce
another constant~$\delta\in (0,1]$ which will be determined
in the course of the proof (depending only on the surface tensions
associated with the strong solution).

We start with the definition of~$\hat\vartheta_i$ away from the
(relevant part of the) network of interfaces. To this end,
we define subsets~$\mathcal{P}_i\subset\mathcal{P}$ and~$\mathcal{C}_i\subset\mathcal{C}$
which collect those triple junctions and two-phase interfaces for which
the phase~$i$ is present, respectively. We then define for all~$t\in [0,T]$
\begin{align}
\label{eq:defAuxWeightBulk1}
&\hat\vartheta_i(\cdot,t) := -1 
\\& \nonumber
\text{in } \bar\Omega_i(t)\setminus \bigcup_{p\in\mathcal{P}_i} B_{r_{\mathcal{P}}}(\mathcal{T}_p(t)) \cup
\Big(\bigcup_{c\in\mathcal{C}_i}\mathrm{im}_{\bar r_{\mathrm{min}}}(\Psi_{\mathcal{T}_c})(t)
\setminus\bigcup_{p\in\mathcal{P}_i} B_{r_{\mathcal{P}}}(\mathcal{T}_p(t))\Big),
\\
\label{eq:defAuxWeightBulk2}
&\hat\vartheta_i(\cdot,t) := \phantom{+}1 
\\& \nonumber
\text{in } \big(\Rd[2]\setminus\bar\Omega_i(t)\big)
\setminus \bigcup_{p\in\mathcal{P}_i} B_{r_{\mathcal{P}}}(\mathcal{T}_p(t)) \cup
\Big(\bigcup_{c\in\mathcal{C}_i}\mathrm{im}_{\bar r_{\mathrm{min}}}(\Psi_{\mathcal{T}_c})(t)
\setminus\bigcup_{p\in\mathcal{P}_i} B_{r_{\mathcal{P}}}(\mathcal{T}_p(t))\Big).
\end{align}
By the definitions~\eqref{def:minLocScaleTripleJunction} and~\eqref{def:minLocScale}
of the scales~$r_{\mathcal{P}}$ and~$\bar r_{\mathrm{min}}$, we may
provide the further construction of~$\hat\vartheta_i$ separately within
$\mathrm{im}_{\bar r_{\mathrm{min}}}(\Psi_{\mathcal{T}_c})(t)
\setminus\bigcup_{p\in\mathcal{P}_i} B_{r_{\mathcal{P}}}(\mathcal{T}_p(t))$
for each~$c\in\mathcal{C}_i$ and within~$B_{r_{\mathcal{P}}}(\mathcal{T}_p(t))$
for each~$p\in\mathcal{P}_i$, respectively.

For each~$c\in\mathcal{C}_i$, and assuming for notational concreteness 
that~$\mathcal{T}_c\subset\bar I_{i,j}$ for some~$j\in\{1\ldots,P\}\setminus\{i\}$,
we simply define for all~$t\in [0,T]$
\begin{align}
\label{DefinitionWeightTwoPhase}
\hat\vartheta_i(\cdot,t) := \vartheta\Big(\frac{s_{i,j}(\cdot,t)}
{\delta\bar r_{\mathrm{min}}}\Big),
\quad \text{in } \mathrm{im}_{\bar r_{\mathrm{min}}}(\Psi_{\mathcal{T}_c})(t)
\setminus\bigcup_{p\in\mathcal{P}_i} B_{r_{\mathcal{P}}}(\mathcal{T}_p(t)),
\end{align} 
where the signed distance~$s_{i,j}$ was introduced in~\eqref{SignedDistanceTwoPhase}.

Now, consider a triple junction~$p\in\mathcal{P}_i$. We assume that
the pairwise distinct phases present at~$\mathcal{T}_p$ are given by~$i,j,k\in\{1,\ldots,P\}$.
Recall from Definition~\ref{def:locRadiusTripleJunction} that $B_{r_{\mathcal{P}}}(\mathcal{T}_p)$
decomposes into six wedges. Three of them, namely the interface wedges $W_{i,j}$, $W_{j,k}$ resp.\ $W_{k,i}$,
contain the interfaces $\mathcal{T}_{c_{i,j}}$, $\mathcal{T}_{c_{j,k}}$ resp.\ $\mathcal{T}_{c_{k,i}}$.
The other three are interpolation wedges denoted by $W_{i}$, $W_{j}$ resp.\ $W_{k}$.
For the definition of~$\hat\vartheta_i$ on the latter wedges, we rely on
the interpolation parameter built in Lemma~\ref{lemma:interpolation_functions}. 
To clarify the direction of interpolation, i.e., on which boundary of the interpolation wedge the corresponding
interpolation function is equal to one or zero, we make use of the following
notational convention. For the interpolation wedge~$W_i$, say, we denote 
by~$\lambda_{i}^{j,k}$ the interpolation function as built in Lemma~\ref{lemma:interpolation_functions}
and which interpolates from~$j$ to~$k$ in the sense that it is equal 
to one on $(\partial W_{i,j}\cap\partial W_{i})\setminus\mathcal{T}_{p}$ and 
which vanishes on $(\partial W_{k,i}\cap\partial W_{i})\setminus\mathcal{T}_{p}$. We also define
$\lambda_{i}^{k,j}:= 1-\lambda_{i}^{j,k}$ which interpolates on $W_{i}$ in the opposite direction from~$k$ to~$j$.
Analogously, one introduces the interpolation functions on the other interpolation wedges.

We now define the weight function~$\hat\vartheta_i$ 
for all~$t\in [0,T]$ on the ball~$B_{r_{\mathcal{P}}}(\mathcal{T}_p(t))$
as follows:
\begin{align}
\label{DefinitionWeightWedgeInterface}
\hat\vartheta_i(\cdot,t) := \vartheta\Big(\frac{s_{i,j}(\cdot,t)}
{\delta\bar r_{\mathrm{min}}}\Big),
\quad \text{in } W_{i,j}(t)\cap B_{r_{\mathcal{P}}}(\mathcal{T}_p(t)),
\end{align}
and analogously on the interface wedge~$W_{i,k}$, whereas we interpolate on
the interpolation wedge~$W_{i}$ by means of
\begin{align}
\label{DefinitionWeightInterpolationWedge}
\hat\vartheta_i(\cdot,t) := 
\lambda_{i}^{j,k}(\cdot,t)\vartheta\Big(\frac{s_{i,j}(\cdot,t)}
{\delta\bar r_{\mathrm{min}}}\Big)
+\lambda_{i}^{k,j}(\cdot,t)\vartheta\Big(\frac{s_{i,k}(\cdot,t)}
{\delta\bar r_{\mathrm{min}}}\Big),
\,\,\text{in } W_{i}(t)\cap B_{r_{\mathcal{P}}}(\mathcal{T}_p(t)).
\end{align} 
Furthermore, we define
\begin{align}
\label{DefinitionWeightComplementWedge}
\hat\vartheta_i(\cdot,t) := \vartheta\Big(\frac{\dist(\cdot,\mathcal{T}_p(t))}{\delta\bar r_{\mathrm{min}}}\Big),
\quad \text{in } W_{j,k}(t)\cap B_{r_{\mathcal{P}}}(\mathcal{T}_p(t)),
\end{align}
whereas we again interpolate
on the interpolation wedge~$W_j$ via
\begin{align}
\label{DefinitionWeightInterpolationWedge2}
\hat\vartheta_i(\cdot,t) := 
\lambda_{j}^{k,i}(\cdot,t)\vartheta\Big(\frac{\dist(\cdot,\mathcal{T}_p(t))}{\delta\bar r_{\mathrm{min}}}\Big)
+\lambda_{j}^{i,k}(\cdot,t)\vartheta\Big(\frac{s_{i,j}(\cdot,t)}{\delta\bar r_{\mathrm{min}}}\Big),
\,\, \text{in } W_{j}(t)\cap B_{r_{\mathcal{P}}}(\mathcal{T}_p(t)),
\end{align} 
and analogously for the interpolation wedge~$W_k$.

\textit{Step 2: Regularity of $(\hat\vartheta_i)_{i\in\{1,\ldots,P\}}$.}
First of all, it is immediate from the above 
definitions~\eqref{eq:defAuxWeightBulk1}--\eqref{DefinitionWeightInterpolationWedge2} 
that the coercivity properties of Definition~\ref{def:transportedWeights}
hold true as required. Choosing~$\delta\in (0,1]$ as in
\textit{Step~3}, Proof of Lemma~\ref{LemmaPartitionOfUnity},
ensures that the definitions~\eqref{DefinitionWeightWedgeInterface}--\eqref{DefinitionWeightInterpolationWedge2}
close to triple junctions are compatible with the bulk 
definitions~\eqref{eq:defAuxWeightBulk1}--\eqref{eq:defAuxWeightBulk2}.
In particular, the asserted regularity $\hat\vartheta_i\in W^{1,\infty}_{x,t}(\Rd[2]\times [0,T])$
for the auxiliary weight functions is now a consequence of the regularity~\eqref{eq:regSignedDistanceProjection} 
of the signed distance functions as well as the controlled blowup \eqref{boundslambda1}
of the first-order derivatives of the interpolation parameter. In terms of estimates, it holds
\begin{align}
\label{eq:regEstimateAuxWeight}
\max_{k=0,1} \bar r_{\mathrm{min}}^k|\nabla^k\hat\vartheta_i|
+ \bar r_{\mathrm{min}}^2|\partial_t\hat\vartheta_i|
\leq C \quad\text{in } \Rd[2]{\times} [0,T]\setminus\bigcup_{p\in\mathcal{P}_i}\mathcal{T}_p,
\end{align}
for a constant~$C>0$ which may depend on the strong solution~$\bar\Omega$,
but which is independent of~$\bar r_{\mathrm{min}}$.

\textit{Step 3: Estimate for the advective derivatives of~$(\hat\vartheta_i)_{i\in\{1,\ldots,P\}}$.}
For a proof of the bound~\eqref{AdvectionEquationVolumeControl}
on the advective derivative with respect to the auxiliary weight~$\hat\vartheta_i$, 
it suffices to work
in the regions~$\bigcup_{c\in\mathcal{C}_i}\mathrm{im}_{\bar r_{\mathrm{min}}}(\Psi_{\mathcal{T}_{c}})
\setminus \bigcup_{p\in\mathcal{P}_i} \bigcup_{t\in [0,T]}B_{r_{\mathcal{P}}}(\mathcal{T}_p(t)){\times}\{t\}$
and~$\bigcup_{p\in\mathcal{P}_i}\bigcup_{t\in [0,T]}B_{r_{\mathcal{P}}}(\mathcal{T}_p(t)){\times}\{t\}$,
respectively. We in fact may argue separately for each~$c\in\mathcal{C}_i$
and each~$p\in\mathcal{P}_i$. The argument turns out to be almost analogous
to the one for the proof of~\eqref{AdvectionEquationBulkLocalization}; a connection which we will make
precise in the subsequent steps to avoid unnecessary repetition.

\textit{Substep 1: Estimate near~$\partial\bar\Omega_i$ but away from triple junctions.}
Let~$c\in\mathcal{C}_i$, and assume for concreteness that~$\mathcal{T}_c\subset\bar I_{i,j}$.
It follows from the definition~\eqref{DefinitionWeightTwoPhase} that~$\hat\vartheta_i$
is a smooth function of the signed distance~$s_{i,j}$ throughout the space-time
domain~$\mathrm{im}_{\bar r_{\mathrm{min}}}(\Psi_{\mathcal{T}_{c}})
\setminus \bigcup_{p\in\mathcal{P}_i} \bigcup_{t\in [0,T]}B_{r_{\mathcal{P}}}(\mathcal{T}_p(t)){\times}\{t\}$. 
Hence, due to~\eqref{eq:regEstimateAuxWeight}
the otherwise exact same argument guaranteeing~\eqref{eq:auxEstimateAdvectiveDerivative} entails
\begin{align}
\label{eq:auxAdvectionWeightSubstep1}
|\partial_t\hat\vartheta_i + (B\cdot\nabla)\hat\vartheta_i|
\leq C\bar r_{\mathrm{min}}^{-2}(\bar r_{\mathrm{min}}^{-1}\dist(\cdot,\bar I_{i,j}) \wedge 1)
\leq C\bar r_{\mathrm{min}}^{-2}|\hat\vartheta_i|
\end{align} 
in~$\mathrm{im}_{\bar r_{\mathrm{min}}}(\Psi_{\mathcal{T}_{c}})
\setminus \bigcup_{p\in\mathcal{P}_i} \bigcup_{t\in [0,T]}B_{r_{\mathcal{P}}}(\mathcal{T}_p(t)){\times}\{t\}$.
The last inequality follows due to~$\vartheta$ being a truncation of unity.

\textit{Substep 2: Estimate at triple junction in interface wedges containing~$\partial\bar\Omega_i$.}
Consider~$p\in\mathcal{P}_i$, and let $c\in\mathcal{C}$ such that~$c\sim p$
and~$\mathcal{T}_c\subset\bar I_{i,j}$. We provide the required estimate
in the interface wedge~$W_{i,j}(t)\cap B_{r_{\mathcal{P}}}(\mathcal{T}_p(t))$ for all~$t\in [0,T]$.
In this case, definition~\eqref{DefinitionWeightWedgeInterface} applies so that~$\hat\vartheta$
is again a smooth function of the signed distance~$s_{i,j}$. Recalling~\eqref{eq:regEstimateAuxWeight},
we may thus apply the argument in favor of~\eqref{eq:dt zetac Wij} to deduce again
\begin{align}
\label{eq:auxAdvectionWeightSubstep2}
|\partial_t\hat\vartheta_i + (B\cdot\nabla)\hat\vartheta_i|
\leq C\bar r_{\mathrm{min}}^{-2}(\bar r_{\mathrm{min}}^{-1}\dist(\cdot,\bar I_{i,j}) \wedge 1)
\leq C\bar r_{\mathrm{min}}^{-2}|\hat\vartheta_i|,
\end{align}
this time throughout~$W_{i,j}(t)\cap B_{r_{\mathcal{P}}}(\mathcal{T}_p(t))$ for all~$t\in [0,T]$.

\textit{Substep 3: Estimate at triple junction in interface wedge not containing~$\partial\bar\Omega_i$.}
Let~$p\in\mathcal{P}_i$, and let~$j,k\in\{1,\ldots,P\}$ denote the other two distinct
phases which are present at~$\mathcal{T}_p$ next to~$i$. We aim to estimate
the advective derivative of~$\hat\vartheta_i$ in the interface wedge
$W_{j,k}(t)\cap B_{r_{\mathcal{P}}}(\mathcal{T}_p(t))$ for all~$t\in [0,T]$.
Note that thanks to~\eqref{DefinitionWeightComplementWedge}, the
auxiliary weight~$\hat\vartheta_i$ is a smooth function of the distance to
the triple junction. Hence, we may simply follow the argument resulting in~\eqref{eq:transportTripleJunctionCutoff}
and obtain together with~\eqref{eq:regEstimateAuxWeight} that
\begin{align}
\label{eq:auxAdvectionWeightSubstep3}
|\partial_t\hat\vartheta_i + (B\cdot\nabla)\hat\vartheta_i|
\leq C\bar r_{\mathrm{min}}^{-2}(\bar r_{\mathrm{min}}^{-1}\dist(\cdot,\mathcal{T}_p) \wedge 1)
\leq C\bar r_{\mathrm{min}}^{-2}|\hat\vartheta_i|
\end{align}
in the region~$W_{j,k}(t)\cap B_{r_{\mathcal{P}}}(\mathcal{T}_p(t))$ for all~$t\in [0,T]$.

\textit{Substep 4: Estimate at triple junction in interpolation wedges.}
Let the notation of \textit{Substep~3} in place. On the interpolation wedge~$W_i$,
the auxiliary weight is defined by means of~\eqref{DefinitionWeightInterpolationWedge},
i.e., one interpolates between two smooth functions of the signed distances~$s_{i,j}$ and~$s_{k,i}$, respectively.
Hence, we may estimate based on the product rule, the estimate~\eqref{eq:advectionInterpolParameterGlobalVel},
the bound~\eqref{eq:regEstimateAuxWeight}, the fact that~$\lambda_{i}^{j,k}=1{-}\lambda_{i}^{k,j}$,
the argument establishing~\eqref{eq:dt zetac Wi}, and finally~\eqref{eq:compDistances1}
\begin{align}
\nonumber
|\partial_t\hat\vartheta_i + (B\cdot\nabla)\hat\vartheta_i|
&\leq C\bar r_{\mathrm{min}}^{-2}\Big|\vartheta\Big(\frac{s_{i,j}(\cdot,t)}
{\delta\bar r_{\mathrm{min}}}\Big) {-} \vartheta\Big(\frac{s_{k,i}(\cdot,t)}
{\delta\bar r_{\mathrm{min}}}\Big)\Big|
\\&~~~\nonumber
+ C\bar r_{\mathrm{min}}^{-2}\lambda_{i}^{j,k}
(\bar r_{\mathrm{min}}^{-1}\dist(\cdot,\bar I_{i,j}) \wedge 1)
\\&~~~\nonumber
+ C\bar r_{\mathrm{min}}^{-2}\lambda_{i}^{k,j}
(\bar r_{\mathrm{min}}^{-1}\dist(\cdot,\bar I_{k,i}) \wedge 1)
\\& \label{eq:auxAdvectionWeightSubstep4}
\leq C\bar r_{\mathrm{min}}^{-2}
(\bar r_{\mathrm{min}}^{-1}\dist(\cdot,\mathcal{T}_p) \wedge 1)
+ C\bar r_{\mathrm{min}}^{-2}|\hat\vartheta_i|
\leq C\bar r_{\mathrm{min}}^{-2}|\hat\vartheta_i|
\end{align}
throughout~$W_{i}(t)\cap B_{r_{\mathcal{P}}}(\mathcal{T}_p(t))$ for all~$t\in [0,T]$.
In view of the definition~\eqref{DefinitionWeightInterpolationWedge2}
and the argument for~\eqref{eq:transportTripleJunctionCutoff} (carefully noting
that the latter is established also on interpolation wedges), the otherwise
same ingredients and computations employed for the proof of~\eqref{eq:auxAdvectionWeightSubstep4}
also imply
\begin{align}
\label{eq:auxAdvectionWeightSubstep5}
|\partial_t\hat\vartheta_i + (B\cdot\nabla)\hat\vartheta_i|
\leq C\bar r_{\mathrm{min}}^{-2}|\hat\vartheta_i|
\end{align}
in~$W_{j}(t)\cap B_{r_{\mathcal{P}}}(\mathcal{T}_p(t))$ for all~$t\in [0,T]$.

\textit{Substep 5: Conclusion.}
In summary, the estimates~\eqref{eq:auxAdvectionWeightSubstep1}--\eqref{eq:auxAdvectionWeightSubstep5}
imply the asserted bound~\eqref{AdvectionEquationVolumeControl}
for the advective derivative in terms of the auxiliary weights~$\hat\vartheta_i$.
In particular, the family of auxiliary weights $(\hat\vartheta_i)_{i\in\{1,\ldots,P\}}$
satisfies all the required properties of Definition~\ref{def:transportedWeights}
with the only exception being $\hat\vartheta_i\in L^1_{x,t}(\Rd[2]{\times}[0,T])$.

\textit{Step 4: Construction and properties of $\vartheta_i$.}
As already mentioned at the beginning of the proof, we may now define
$\vartheta_i:=\eta_R\hat\vartheta_i$ for all~$i\in\{1,\ldots,P\}$. 
The regularity and the required coercivity properties for~$\vartheta_i$ are then immediate
consequences of its definition and the previous step. The estimate~\eqref{AdvectionEquationVolumeControl}
on the advective derivative also carries over since~$\eta_R$ is time-independent and by~\eqref{eq:gradientBoundIntWeight}
\begin{align*}
|\hat\vartheta_i||(B\cdot\nabla)\eta_R| \leq C|\vartheta_i|
\quad\text{in } \Rd[2]\times [0,T],
\end{align*}
so that the product rule together with the previous step implies~\eqref{AdvectionEquationVolumeControl}
on the level of the weight~$\vartheta_i$. This in turn concludes the proof of Lemma~\ref{LemmaWeightedVolumeControl}.
\end{proof}

\section{Admissibility of a class of Read-Shockley type surface tensions}

Finally, we provide the proof that the surface tensions given by the 
Read-Shockely formulas \eqref{eq:read-shockley sigma} and \eqref{eq:read-shockley f} 
are admissible, as stated in Lemma \ref{lemma:read-shockley}. 
Here, we are inspired by \cite{EsedogluOtto15} and partly follow the 
general strategy of Theorem~5.5 in~\cite{EsedogluOtto15}. However, the situation is more complicated in our setting as our integrand is not concave. In fact, $f^2$ is strictly convex close to the origin.

\begin{proof}[Proof of Lemma \ref{lemma:read-shockley}] 
	We first prove the embeddability for general surface tensions $\sigma_{i,j}$ coming from \eqref{eq:read-shockley sigma} with $f$ satisfying the negativity condition \eqref{eq:read shockley negative fourier}. In the second step, we then verify this negativity condition for the particular choice of the Read-Shockley formula \eqref{eq:read-shockley f}. 
	
	\textit{Step 1: Embeddability under negativity condition.}
	To show the embaddability, we prove the equivalent negative definiteness \eqref{eq:Q negative}.
	To this end, we define the symmetric bilinear form
	\[
	Q(u) := \dashint_{-\pi/4}^{\pi/4} \dashint_{-\pi/4}^{\pi/4} f^2(x-y) u(x) u(y) \dx \dy \quad \text{for $u$ with } \dashint_{-\pi/4}^{\pi/4} u(x)\dx =0,
	\] 
	where $f$ is extended evenly to $(-\pi/4,\pi/4)$ and periodically to $(-\pi/2,\pi/2)$. 
	Denoting $g:=f^2$ and using the orthonormal system $(e^{4k\mathbf{i}})_{k\in \mathbb{Z}}$, where $\mathbf{i}=\sqrt{-1}$, by Plancherel, we may write $Q$ in Fourier space as
	\begin{align*}
		Q(u) = \sum_{k\in \mathbb{Z}} \widehat{g}_k |\widehat{u}_k|^2.
	\end{align*}
	By assumption, for all $k\in \mathbb{Z}\setminus \{0\}$, $\widehat{g}_k$ is a negative real number. In addition, we have $\widehat u_0= \dashint_{-\pi/4}^{\pi/4} u(x) \dx =0$. 
	Hence, the weaker version $z^\mathsf{T} Qz\leq 0$ for all $z\in \Rd[d]$ with $\sum_{j=1}^P z_j=0$ follows from plugging in the measure
	\begin{align}\label{eq:read shockley test function}
		u(x) = \sum_{j=1}^P \delta_{x=\theta_j} z_j.	
	\end{align} 
	
	To obtain the strict inequality $z^\mathsf{T} Qz<0$, we will quantify this argument as follows.
	First, we approximate $u$ in \eqref{eq:read shockley test function} by
	\begin{align*}
		u^N(x) = \sum_{j=1}^PF_N(4(x-\theta_j)) z_j,
	\end{align*} 
	where $F_N$ is the F\'ejer kernel. In other words, the $k$-th Fourier coefficient of $u^N$ is given by
	\begin{align*}
		\widehat{(u^N)}_k = 
		\begin{cases}
			\left(1-\frac{|k|}{N}\right) \sum_{j=1}^P e^{-4k\theta_j \mathbf{i}} z_j,& |k|<N\\
			0,&|k|\geq N.
		\end{cases}
	\end{align*}
	Since $F_N \stackrel{\ast}{\rightharpoonup} \delta_0$ in the sense of measures, we have $u^N  \stackrel{\ast}{\rightharpoonup} u$ as measures. Hence also the product measure converges, $u^N(x)\dx 
	\otimes u^N(y) \dy  \stackrel{\ast}{\rightharpoonup} u(\dx) \otimes u(\dy)$ as measures on $(-\pi/4,\pi/4)^2$.
	Therefore, since $g=f^2$ is continuous,
	\begin{align*}
		\lim_{N\to \infty} Q(u^N) = Q(u) = \sum_{i,j=1}^P z_i f^2(\theta_i-\theta_j) z_j = z^\mathsf{T} Q z.
	\end{align*}
	In order to conclude, we use the assumption on $\widehat g_k$ and $\widehat {(u^N)}_0=\sum_{j=1}^P z_j =0$ to obtain for any $N$ larger than, say, $2P$,
	\begin{align*}
	Q(u^N) &= \sum_{|k|<N}  \widehat{g}_k	\left(1-\frac{|k|}{N}\right) ^2  |\widehat{u}_k|^2
	\\& \leq \sum_{k=1}^P \widehat{g}_k	\left(1-\frac{P}{N}\right) ^2  |\widehat{u}_k|^2
	\leq -\frac14 \left( \min_{1\leq k \leq P} |  \widehat{g}_k|\right) \sum_{k=1}^P  |\widehat{u}_k|^2.
	\end{align*}
	Now we observe that
	\begin{align*}
		\widehat{u}_k = \sum_{j=1}^P e^{-4k\theta_j \mathbf{i}} z_j.
	\end{align*}
	In other words,  the $P$-vector $(\widehat{u}_1,\ldots, \widehat{u}_P)$ is given by the matrix-vector product $Mz$, where $M$ is a $(P\times P)$-Vandermonde matrix with entries $M_{kj} = (e^{-4\theta_j \mathbf{i}})^k.$
	Then the claim follows from the fact that 
	\begin{align*}
		|\det M| 
		= \left| \prod_{1\leq j < k \leq P} \left(e^{-4\theta_j \mathbf{i}}-e^{-4\theta_k \mathbf{i}}\right)\right| >0,
	\end{align*}
	where we have used our assumption $\theta_k \neq \theta_j \mod \frac{\pi}2$ for $k\neq j$. 
	
	\textit{Step 2: Negativity condition $\widehat{g}_k<0$ for the Read-Shockley formula.}
	Let us now turn to the specific Read-Shockley profile \eqref{eq:read-shockley f}. We aim to show that $g=f^2$, extended evenly from $(0,\pi/4)$ to $(-\pi/4,\pi/4)$, satisfies 
	\begin{align*}
		\widehat g_k  \text{ is a negative real number for all } k \in \mathbb{Z}\setminus \{0\}.
	\end{align*}
	Since $g$ is even, $\widehat g_k \in \Rd[]$ for all $k$, and by symmetry we only need to show 
	\begin{align*}
		\widehat g_k<0 \quad \text{ for all } k =1,2,3, \ldots.
	\end{align*}
	Two integrations by parts yield
	\begin{align*}
		\widehat g_k
		=&\dashint_{-\pi/4}^{\pi/4}g(x) \cos(4kx) \dx
		=2\frac{2}{\pi}\int_0^{\pi/4}g(x) \cos(4kx) \dx
		\\=& \frac1\pi \left[\frac1k g(x) \sin(4kx) - \frac1{4k^2} g'(x)\cos(4kx) \right]_{x=0}^{\pi/4}- \frac1{4\pi k^2} \int_0^{\pi/4} g''(x)  \cos(4kx) \dx.
	\end{align*}
	Since $\sin(0)=\sin(n\pi)=0$ and by assumption $g'(0) =g'(0+)=\lim_{x\to0} 2f(x)f'(x) =0$ and $g'(\pi/4)=f(\pi/4)f'(\pi/4)=0$, the boundary terms vanish and therefore $\widehat g_k<0$ is equivalent to
	\begin{equation*}
		\int_0^{\pi/4} g''(x) \cos(4kx) \dx > 0.
	\end{equation*}
	In case of the particular structure \eqref{eq:read-shockley f} of $f$, using the change of variable $x\mapsto \theta_\ast x$ this may be written as
	\begin{equation}\label{eq:log2log-1}
			I(\alpha)=\int_0^{1} \big( \log^2(x) + \log(x) -1 \big) \cos(\alpha x) \dx > 0
	\end{equation}
	for $\alpha =4 n\theta_\ast$, $n=1,2,3,...$.
	
	We will show that \eqref{eq:log2log-1} in fact holds for all $\alpha>0$. 
	Using the series representation of the cosine and integrating by parts twice each term of the series in the integral \eqref{eq:log2log-1}, we obtain the (absolutely convergent) series representation
	\begin{align*}
		I(\alpha) & = \sum_{n=1}^\infty \frac{4n^2+6n}{(2n+1)^2}\frac{(-1)^{n+1}}{(2n+1)!} \alpha^{2n}\\
		& = \sum_{k=1}^\infty \left(\frac{4 (2k-1)^2 + 6(2k-1)}{(2(2k-1) +1)^2} - \frac{4(2k)^2 - 6(2k)}{(2(2k)+1)^2} \frac{\alpha^2}{(4k+1)4k} \right) \frac{\alpha^{4k-2}}{(4k-1)!}.
	\end{align*}
	As the map $x\mapsto \frac{4x^2+6x}{(2x+1)^2}$ is strictly decreasing for $x\geq 2$, we have
	\begin{align*}
		\sum_{k=2}^\infty \left(\frac{4 (2k-1)^2 + 6(2k-1)}{2(2k-1) +1)^2} - \frac{4(2k)^2 - 6(2k)}{(2(2k)+1)^2} \frac{\alpha^2}{(4k+1)4k} \right) \frac{\alpha^{4k-2}}{(4k-1)!} >0
	\end{align*}
	provided $0<\alpha^2 \leq 72$.
	The remaining term for $k=1$ can be seen to be strictly positive provided $0<\alpha^2 < \frac{500}{11}$.
	Therefore, we proved $I(\alpha) >0$ under the condition $0 <\alpha^2 \geq 45< \frac{500}{11}$.
	
	Furthermore, it can be seen that the map $\alpha \mapsto \int_0^\alpha \frac{\sin(t)}{t} \dt$ is non-negative and maximal when $\alpha=\pi$.
	Consequently, for $\alpha \geq \sqrt{45}$ we have the estimate
	\begin{align*}
		I(\alpha) \geq \frac{1}{\alpha}\left( -1 - \int_0^\pi \frac{\sin(t)}{t} \dt + 2 \int_0^{\sqrt{45}} \frac{1}{x} \int_0^x \frac{\sin(t)}{t} \dt \dx\right).
	\end{align*}
	Numerical integration yields
	\begin{align*}
		I(\alpha) \geq \frac{4}{\alpha},
	\end{align*}
	concluding the proof.
\end{proof}

\section*{Glossary of notation}

\begin{longtabu}{ll}
\endfirsthead
\endhead
\endfoot
\endlastfoot

$d\geq 2$ & ambient dimension
\vspace{2mm}
\\
$D$ & open set
\vspace{2mm}
\\
$\partial_t v$ & distributional partial derivative w.r.t.\ time
\\
&  of $v: D \times [0,T) \to \Rd$
\vspace{2mm}
\\
$\nabla v$ & distributional partial derivative w.r.t.\ space, $(\nabla v)_{i,j} = \partial_j v_i$
\vspace{2mm}
\\
$C_{\mathrm{cpt}}^\infty(D)$ & space of compactly 
supported and infinitely\\
& differentiable functions on $D$
\vspace{2mm}
\\
$C_{t}^lC_{x}^k(U)$ & space of functions on~$U\subset\Rd{\times}[0,T]$ with continuous  \\
& and bounded partial derivatives~$\partial_t^{l'}\partial_x^{k'},\,0\leq l'\leq l,\,0\leq k'\leq k$.
\vspace{2mm}
\\
$u \otimes v$ & tensor product of $u, v\in \Rd$, $(u \otimes v)_{i,j} = u_i v_j$
\vspace{2mm}
\\
$A : B$ & $\sum_{i,j} A_{ij}B_{ij}$, scalar product of tensors
\vspace{2mm}
\\
$\mathcal{L}^d$ & $d$-dimensional 
Lebesgue measure
\vspace{2mm}
\\
$\mathcal{H}^k$ & $k$-dimensional Hausdorff measure on $\Rd$ for $k\in [0,d]$
\vspace{2mm}
\\
$L^p(\Omega,\mu)$ & Lebesgue space w.r.t.\ to a measure $\mu$ on $\Omega \subset \Rd$ for $p\in [1,\infty]$
\vspace{2mm}
\\
$L^p(D)$ & Lebesgue space w.r.t.\ Lebesgue measure
\vspace{2mm}
\\
$L^p(D;\Rd)$ & Lebesgue space for vector valued functions
\vspace{2mm}
\\
$L^p([0,T];X)$ & Bochner--Lebesgue space for a Banach space $X$ and $T \in (0,\infty)$
\vspace{2mm}
\\
$W^{k,p}(D)$ & Sobolev spaces with $p\in [1,\infty)$ and $k\in \mathbb{N}$
\vspace{2mm}
\\
$BV(D)$
& Functions of bounded variation \cite{AmbrosioFuscoPallara} on Lipschitz domain $D\subset \Rd$
\vspace{2mm}
\\
$\partial^* \Omega$ & reduced boundary of a set of finite perimeter $\Omega \subset D$
\vspace{2mm}
\\
$\smash{\vec{n}=-\frac{\nabla\chi_\Omega}{|\nabla\chi_\Omega|}}$ &  outward pointing unit normal vector field along $\partial^*\Omega$
\vspace{2mm}
\\
$s_{i,j}$ & signed distance function to ${\bar{I}}_{i,j}$ with $\nabla s_{i,j} = \bar{\vec{n}}_{i,j}$
\vspace{2mm}
\\
$\dist(\cdot, A)$ & distance function $\Rd{\times}[0,T]\ni(x,t)\mapsto \dist(x,A(t))$
for a domain
\\ & $A=\bigcup_{t\in [0,T]}A(t){\times}\{t\}$, $A(t)\subset\Rd$, $t\in [0,T]$.
\vspace{2mm}
\\
$P \geq 2$ & number of phases
\vspace{2mm}
\\
$\Omega_i$ & region occupied by phase $i =1,\ldots,P$ in \emph{weak solutions}
\vspace{2mm}
\\
$\chi_i$ & characteristic function of $\Omega_i$
\vspace{2mm}
\\
$I_{i,j}$ & interface between phases $\Omega_i$ and $\Omega_j$
\vspace{2mm}
\\
$\vec{n}_{i,j}$ & unit normal vectors along $I_{i,j}$ pointing from phase $i$ to phase $j$
\vspace{2mm}
\\
$V_i$ & normal velocity of $I_{i,j}$ with $V_i>0$ for expanding $\Omega_i$, see \eqref{EvolutionPhasesBVSolution}
\vspace{2mm}
\\
$\bar \Omega_i$, $\bar \chi_i$, \ldots & corresponding quantities of the \emph{strong solution}
\vspace{2mm}
\\
$\vec{H}_{i,j}$ & mean curvature vector of ${\bar{I}}_{i,j}$
\vspace{2mm}
\\
$H_{i,j}$ & scalar mean curvature of ${\bar{I}}_{i,j}$ given by\\
& $\vec{H}_{i,j} \cdot {\bar{\vec{n}}}_{i,j}=-\nabla^{\mathrm{tan}}\cdot\bar{\vec{n}}_{i,j}=-\Delta s_{i,j}$
\vspace{2mm}
\\
$s_{i,j}$ & signed distance function to ${\bar{I}}_{i,j}$ with $\nabla s_{i,j} = \bar{\vec{n}}_{i,j}$
\vspace{2mm}
\\
$\smash{J= \big(\begin{smallmatrix} 0 & -1 \\ 1 & 0	\end{smallmatrix}\big)}$ & counter-clockwise rotation by $90^\circ$
\vspace{2mm}
\\
$\bar{\tau}_{i,j}$ & tangent vector along ${\bar{I}}_{i,j}$ given by $ J^{-1}\bar{\vec{n}}_{i,j}$
\vspace{2mm}
\\
$O(\cdot)$ & Landau symbol, implicit constant only depends on strong solution
\end{longtabu}

\bibliographystyle{abbrv}
\bibliography{multiphase}

\end{document}